\pdfoutput=1
\documentclass[10pt, english, openany]{memo-l}
\usepackage[normalem]{ulem}
\usepackage{amsthm}
\usepackage{extpfeil}
\usepackage{tikz, stmaryrd}
\usetikzlibrary{matrix,arrows,decorations.pathmorphing}
\usepackage{appendix}
\usepackage{color}
\usepackage{amsmath,amssymb, amsfonts, amscd}
\usepackage[all]{xy}
\usepackage{verbatim}
\usepackage{enumerate}
\usepackage{yhmath}
\usepackage{xcolor}
\usepackage{scalerel,stackengine}
\stackMath
\newcommand\reallywidehat[1]{%
\savestack{\tmpbox}{\stretchto{%
  \scaleto{%
    \scalerel*[\widthof{\ensuremath{#1}}]{\kern-.6pt\bigwedge\kern-.6pt}%
    {\rule[-\textheight/2]{1ex}{\textheight}}
  }{\textheight}%
}{0.5ex}}%
\stackon[1pt]{#1}{\tmpbox}%
}
\usepackage{hyperref}
\newcommand{\adj}[4]{#1\negmedspace: #2\rightleftarrows #3:\negmedspace #4}
\newextarrow{\xbigtoto}{{20}{20}{20}{20}}
   {\bigRelbar\bigRelbar{\bigtwoarrowsleft\rightarrow\rightarrow}}

\newtheorem{thm}{Theorem}
\newtheorem{cor}[thm]{Corollary}
\newtheorem{lem}[thm]{Lemma}
\newtheorem{obs}[thm]{Observation}

\newtheorem{prop}[thm]{Proposition}

\theoremstyle{plain}
\newtheorem{defn}[thm]{Definition}
\newtheorem{notation}[thm]{Notation}
\newtheorem{example}[thm]{Example}

\def\llp{\mathrel{\ooalign{\hss$\square$\hss\cr$\diagup$}}}

\newtheorem{rem}[thm]{Remark}

\newcommand{\stkout}[1]{\ifmmode\text{\sout{\ensuremath{#1}}}\else\sout{#1}\fi}

\makeatletter\@addtoreset{thm}{chapter}\makeatother

\DeclareMathAlphabet{\mathpzc}{OT1}{pzc}{m}{it}

\newcommand*{\defeq}{\mathrel{\vcenter{\baselineskip0.5ex \lineskiplimit0pt
                     \hbox{\scriptsize.}\hbox{\scriptsize.}}}%
                     =}

\usepackage{mathrsfs}

\newcommand{\Q}{{\mathbb Q}}
\def\Z{{\mathbb Z}}

\def\C{{\mathbb C}}
\def\R{{\mathbb R}}

\makeindex
\begin{document}
\frontmatter

\title[Homotopy in Exact Categories]{
Homotopy in Exact Categories}
\author[Jack Kelly]{Jack Kelly}
\address{Jack Kelly,
The Hamilton Mathematics Institute,
School of Mathematics,
Trinity College Dublin,
Dublin 2,
Ireland}
\email{jack.kelly@tcd.ie}
\thanks{The majority of this research was conducted during the author's graduate studies, which was supported by the EPSRC studentship BK/13/007, while employed at the University of Oxford on the EPSRC grant `Symmetries and Correspondences', and while supported by the Simons Foundation at Trinity College Dublin under the
program “Targeted Grants to Institutes”.}%

\begin{abstract}
In this monograph we develop various aspects of the homotopy theory of exact categories. We introduce different notions of compactness and generation in exact categories, and use these to study model structures on categories of chain complexes $Ch_{*}(\mathpzc{E})$ which are induced by cotorsion pairs on $\mathpzc{E}$. As a special case we show that under very general conditions the categories $Ch_{+}(\mathpzc{E})$, $Ch_{\ge0}(\mathpzc{E})$, and $Ch(\mathpzc{E})$ are equipped with the projective model structure, and that a generalisation of the Dold-Kan correspondence holds. We also establish conditions under which categories of filtered objects in exact categories are equipped with natural model structures. When $\mathpzc{E}$ is monoidal we also examine when these model structures are monoidal and conclude by studying some homotopical algebra in such categories. In particular we provide conditions under which $Ch(\mathpzc{E})$ and $Ch_{\ge0}(\mathpzc{E})$ are homotopical algebra contexts, thus making them suitable settings for derived geometry.
 \end{abstract}

\maketitle
\tableofcontents

\mainmatter


\chapter{Introduction}\label{sec1}
\section{Background and Motivation}
\subsection*{Derived Geometry}
Derived geometry has proved crucial for understanding intersection theory, deformation theory and moduli theory in algebraic, smooth, and, recently, complex analytic geometry. 

There are two dominating abstract models for derived geometry. Lurie's approach \cite{lurie2009derived} uses a higher-categorical generalization of ringed spaces, namely structured $(\infty,1)$-topoi. This is an $(\infty,1)$-topos $\mathcal{X}$ together with a limit-preserving functor 
$$\mathcal{O}:\mathpzc{G}\rightarrow\mathcal{X}$$
where $\mathpzc{G}$ is a {geometry} - an $(\infty,1)$-category satisfying certain properties. For example taking $(\mathpzc{G})^{op}$ to be the $(\infty,1)$-category of simplicial rings gives a reasonable notion of derived algebraic stacks. David Spivak \cite{spivak2010derived} considers derived smooth manifolds by taking as $(\mathpzc{G})^{op}$ the category of simplicial $C^{\infty}$-rings. Mauro Porta and Tony Yue Yu \cite{porta2016higher},\cite{porta2015derived1} \cite{porta2015derived}, \cite{porta2017derived},\cite{porta2018derived}, \cite{porta2017representability}, \cite{porta2018derivedhom} are developing derived analytic geometry by taking $\mathcal{G}^{op}$ to be the category of simplicial rings equipped with a holomorphic functional calculus. In particular they have proven GAGA, base-change, and Riemann Hilbert type theorems. Thanks to their work the field of derived analytic geometry has developed rapidly in recent years, leading to GAGA, base-change, and Riemann-Hilbert type theorems. They have also announced a Hochschild-Kostant-Rosenberg theorem. Despite these crucial results there are some drawbacks to their setup. For example, there is no obvious definition of the category of quasi-coherent sheaves on a derived analytic space. Base-change theorems for coherent sheaves are difficult to formulate and prove. As in the algebraic case the push-forward functor between categories of coherent sheaves is only defined for proper maps. Moreover, constructing the pullback functor is very technical since the algebraic tensor product of sheaves is not the correct notion in the analytic setting.  We believe that a model of derived analytic geometry based on To\"{e}n and Vezzosi's formalism \cite{toen2004homotopical} of derived geometry relative to a monoidal model category $\mathpzc{M}$ would resolve many of these issues. 

To\"{e}n and Vezzosi's model for derived geometry is inspired by the theory of (non-derived) geometry relative to a symmetric monoidal category (developed for instance in \cite{deligne2007categories} and \cite{banerjee2017noetherian}). This is a category-theoretic framework which views geometry as the unification of algebra and topology. The algebra describes {local} pieces and a Grothendieck topology allows one to glue these local pieces and obtain {global} objects.  In \cite{toen2004homotopical} they introduce the notion of a homotopical algebraic geometry context. Up to some technical details, a homotopical algebraic geometry context consists of a monoidal model category $\mathpzc{M}$ such that the category $\mathpzc{Alg}_{\mathfrak{Comm}}(\mathpzc{M})$ of unital commutative monoids in $\mathpzc{M}$ is a model category with the transferred model structure, and $(\mathpzc{Alg}_{\mathfrak{Comm}}(\mathpzc{M}))^{op}$ is equipped with a homotopy Grothendieck topology $\tau$. We regard $\mathpzc{Aff}_{\mathpzc{M}}\defeq(\mathpzc{Alg}_{\mathfrak{Comm}}(\mathpzc{M}))^{op}$ as a category of affine spaces. The category of derived stacks on $\mathpzc{M}$ is then the category of functors $\mathcal{X}:\mathpzc{Aff}_{\mathpzc{M}}\rightarrow\mathpzc{sSet}$ satisfying descent for $\tau$-hypercovers. For derived algebraic geometry one considers either the category $\mathpzc{M}=Ch(R)$ of chain complexes of modules over a ring $R$ (in characteristic zero), or the category $\mathpzc{M}=s{}_{R}\mathpzc{Mod}$ of simplicial $R$-modules.  

\subsection*{Monoidal Categories and Analytic Geometry}
For the purposes of motivation we will give a brief overview of one approach to a formulation of derived analytic geometry. Details will appear in a forthcoming paper \cite{dang}. Let $(X,\mathcal{O}_{X})$ be a complex manifold. For each open set $U$ the set $\mathcal{O}_{X}(U)$ has a canonical structure of a Fr\'{e}chet space. Moreover, the restriction maps $\mathcal{O}_{X}(V)\rightarrow\mathcal{O}_{X}(U)$ are continuous. Let $\mathpzc{F}$ be a coherent sheaf on $X$. Cartan's Theorem B implies that on a coordinate  neighbourhood (or more generally a Stein neighbourhood), there is an exact sequence
$$\mathcal{O}_{X}^{m}(V)\rightarrow\mathcal{O}_{X}^{n}(V)\rightarrow\mathpzc{F}(V)\rightarrow0$$
The quotient topology on $\mathpzc{F}(V)$ makes it a Fr\'{e}chet space. Thus sheaves on complex spaces have natural topological structures. 
	
	It is therefore tempting to view (non-derived) complex analytic geometry as geometry relative to the symmetric monoidal category of Fr\'{e}chet spaces. Unfortunately this does not seem possible. However in \cite{bambozzi2015stein} the authors construct a Grothendieck topology $\tau^{fhZ}$ on a subcategory $\mathpzc{St}$ of $(\mathpzc{Alg}_{\mathfrak{Comm}}(\mathpzc{Fr}))^{op}$. $\mathpzc{St}$ is equivalent to the category of (dagger) Stein spaces and when $k=\mathbb{C}$ the coverings in their topology correspond to coverings of Stein spaces by Stein spaces. In particular the category of complex analytic spaces embeds in the category of schemes on this site. 
	
	This construction is somewhat ad hoc but as usual passing to the derived world proves enlightening. The Grothendieck topology $\tau^{fhZ}$ of \cite{bambozzi2015stein} makes use of the homological structure on $\mathpzc{Fr}$ which is a quasi-abelian, and therefore exact, category. It is an additive category with classes of admissible monomorphisms and admissible epimorphisms which provide a well-defined notion of homology. There are also notions of projective objects, exact functors, derived categories, and derived functors. If $\mathpzc{E}$ is a monoidal exact category with a left-derivable tensor product $\otimes$, then we say a map $A\rightarrow B$ of commutative monoids in $\mathpzc{E}$ is a \textbf{homotopy epimorphism} if the map $B\otimes_{A}^{\mathbb{L}} B\rightarrow B$ is a quasi-isomorphism. The opposites of these maps make up the covers in $\tau^{fhZ}$. The obstacle to such covers defining a topology on the entire category $(\mathpzc{Alg}_{\mathfrak{Comm}}(\mathpzc{Fr}))^{op}$	is that they are not stable under base-change (because of the derived tensor product). 
	
	If $Ch(\mathpzc{Fr})$ were a good enough monoidal model category then we could easily extend the definition of a homotopy epimorphism. Moreover as a homotopy cover in such a model category the issue of base change would disappear and would give a genuine model topology on $(\mathpzc{Alg}_{\mathfrak{Comm}}(Ch(\mathpzc{Fr})))^{op}$. Tragically $\mathpzc{Fr}$ is not good enough. It is neither complete nor cocomplete and does not have enough projectives. Fortunately it does nicely embed in a complete and cocomplete exact category with enough projectives, namely the category $CBorn_{\C}$ of complete bornological spaces over $\C$. It is sometimes convenient to pass to the even bigger category $Ind(Ban_{\C})$, the formal completion of the category of Banach spaces by filtered colimits. 
	
\section{Goals and Layout}
The central goal of this monograph is to show that $Ind(Ban_{\C})$ and $CBorn_{\C}$ are good categories for developing homotopical algebra, i.e. local derived geometry. More generally we show that under very general conditions on an exact category $\mathpzc{E}$, the category $Ch(\mathpzc{E})$ admits a good homotopy theory of algebras.
\subsection*{Exact Category Generalities}

Building on work of \cite{Buehler} in Chapter \ref{sec2} we begin this work by establishing some technical results about exact categories in general which we will need in subsequent chapters.   After recalling some basic facts we introduce various useful notions of acyclicity. We then discuss bounded and unbounded resolutions in exact categories. In particular we generalise the famous result of Spaltenstein \cite{spaltenstein} to exact categories satisfying very general conditions, following a similar generalisation of  this result to relative homological algebra \cite{chacholski2017relative}.

\begin{thm}[Corollary \ref{Kproj}]\label{thm:Kprojintro}
Let $\mathpzc{E}$ be an exact category with kernels in which. Let $\mathcal{P}$ be a class of objects such that for each object $X$ in $\mathpzc{E}$ there is an object $P$ in $\mathcal{P}$ together with an admissible epimorphism $P\twoheadrightarrow X$. Suppose further that $\mathcal{P}$ is closed under countable coproducts and satisfy axiom $AB4-k$ for some $k$. Then for any complex $X_{\bullet}$ in $ Ch(\mathpzc{E})$ there is a complex $P_{\bullet}$ in $ Ch(\mathcal{P})$ and an admissible epimorphism $P_{\bullet}\rightarrow X_{\bullet}$ which is a quasi-isomorphism. Moreover, $X_{\bullet}$ is the limit of a $ Ch_{+}(\mathcal{P})$-special direct system.
\end{thm}

We then  recall a suitable idea of generators, before defining so-called elementary  and weakly elementary exact categories. These technical notions will be crucial for controlling the homotopy theory of an exact category and avoiding set-theoretic smallness concerns.  Next we define monoidal exact categories and establish some basic properties of them. In particular we prove the existence of an induced exact structure on modules for commutative monoids internal to such categories. More generally we study monads on exact categories and their categories of algebras. 

In Chapter \ref{Secex1} we study numerous examples, including the category of complete bornological spaces over a Banach field, and the non-expanding normed and Banach categories over a non-Archimedean field. We also relate our work to Gillespie's work on derived categories of Grothendieck abelian categories with respect to generators in \cite{gillespie2016derived}, and to the work of Gillespie \cite{gillespie2016derived}, Estrada-Gillespie-Odaba{\c{s}}i \cite{estrada2017pure}, and Krause \cite{krause2012approximations} on pure exact structures.

\subsection*{Model Structures on Exact Categories}

In Chapter \ref{chmodelexact} we discuss model structures on exact categories. There is a general theory of model structures on weakly idempotent complete exact categories due to \cite{hovey}, \cite{gillespie} and \cite{vst2012exact} using \textbf{cotorsion pairs}. A pair of classes of objects $(\mathfrak{L},\mathfrak{R})$ in an exact category $\mathpzc{E}$ is said to be a cotorsion pair if $L\in\mathfrak{L}$ if and only if ${Ext}^{1}(L,R)=0$ for all $R\in\mathfrak{R}$, and $R\in\mathfrak{R}$ if and only if $Ext^{1}(L,R)=0$ for all $L\in\mathfrak{L}$. In \cite{Gillespie2} Gillespie suggests a strategy for producing a model structure on $Ch(\mathpzc{E})$, given a cotorsion pair on an abelian category $\mathpzc{E}$, which can easily be adapted to exact categories more generally. When $\mathpzc{E}$ is monoidal we give conditions on $(\mathfrak{L},\mathfrak{R})$ such that the induced model structure is monoidal and satisfies the monoid axiom.

However we give very general conditions on an exact category $\mathpzc{E}$ such that it does work for the projective cotorsion pair $(\textbf{Proj}(\mathpzc{E}),\textbf{Ob}(\mathpzc{E}))$, where $\textbf{Proj}(\mathpzc{E})$ is the class of projective objects in $\mathpzc{E}$. In particular we prove the following

\begin{thm}[Theorem \ref{projmod}]\label{projmodintro}
Let $\mathpzc{E}$ be an exact category satisfying the following conditions
\begin{enumerate}
\item
$\mathpzc{E}$ has enough projectives.
\item
$\mathpzc{E}$ has kernels.
\item
We suppose that countable coproducts of projectives exist and satisfy axiom $AB4-k$ for some $k\in\mathbb{Z}$.
\end{enumerate}
Then, applied to the cotorsion pair $(\textbf{Proj}(\mathpzc{E}),\textbf{Ob}(\mathpzc{E}))$, Gillespie's strategy produces a model structure on $Ch(\mathpzc{E})$.
\end{thm}
 We call this the \textbf{projective model structure} on $Ch(\mathpzc{E})$. If $E={}_{R}\mathpzc{Mod}$ is the category of $R$-modules over a ring $R$ then this is the usual projective model structure. Under some stronger assumptions, namely that the category $\mathpzc{E}$ has generators which are presented relative to the class of admissible monics, the result of the above theorem can be deduced from results of \cite{vst2012exact}. Moreover, Gillespie proved Theorem \ref{projmod} in \cite{gillespie2016derived} in the case that $\mathpzc{E}$ is the $G$-exact structure on a Grothendieck abelian category. An advantage of our result is that it avoids many set-theoretic concerns and we suggest examples where this is useful. In particular at the end of the chapter we give natural examples of exact categories with very different set-theoretic properties which satisfy the conditions of Theorem \ref{projmod}. This result should also be compared with the model structures of \cite{christensen2002quillen}. Indeed, if an exact category has enough projectives then the class of exact sequences on $\mathpzc{E}$ are the $\mathcal{P}$-exact sequences for $\mathcal{P}$ the class of projectives, in the sense of \cite{christensen2002quillen} Section 1.1. When the underlying category of $\mathpzc{E}$ is abelian Christensen and Hovey show that the projective model structure exists on $Ch(\mathpzc{E})$ (and $Ch_{\ge0}(\mathpzc{E})$). We suspect that the fact that $\mathpzc{E}$ is abelian is not necessary for their proof, and that one only needs an additive category with kernels. Although our proof of the existence of resolutions is similar to \cite{christensen2002quillen} Theorem 2.2, we actually give a more general version of their Case A. Moreover our proof strategy for the existence of the projective model structure is somewhat different - we work with cotorsion pairs. This allows us to derive useful properties of the projective model structure.

 We call a monoidal exact category \textbf{monoidal elementary} if it is elementary and its projectives are flat and closed under the tensor product.  We then prove the following result.
 
 \begin{thm}[Theorem \ref{chainmonoidal}]
Let $\mathpzc{E}$ be a monoidal elementary exact category. Then the projective model structure on $Ch(\mathpzc{E})$ is monoidal and satisfies the monoid axiom.
\end{thm}
We then prove a generalisation of the Dold-Kan correspondence.
\begin{thm}[Theorem \ref{dkex}]\label{dkexintro}
Let $\mathpzc{E}$ be a weakly idempotent complete exact category with enough projectives. Endow $ Ch_{\ge0}(\mathpzc{E})$ and $\textbf{s}\mathpzc{E}$ with their projective model structures. Then the functors
$$\Gamma: Ch_{\ge0}(\mathpzc{E})\rightarrow\textbf{s}\mathpzc{E}$$
and
$$N:\textbf{s}\mathpzc{E}\rightarrow Ch_{\ge0}(\mathpzc{E})$$
form a Quillen equivalence.
\end{thm}

In Chapter \ref{sec6} we discuss model structures for graded and filtered objects in exact categories. This will prove crucial for generalising Koszul duality results in \cite{kelly2019koszul}. Finally in Chapter \ref{homtopalgexact} we discuss model structures on categories of dg-modules and dg-algebras in monoidal exact categories. By modying results of \cite{white2019homotopical} and \cite{white2017model}  we show the following
\begin{thm}[Theorem \ref{thm:adopch}]
Let $\mathpzc{E}$ be an elementary exact category which is also a closed symmetric monoidal category and let $\mathfrak{P}$ be any non-symmetric operad in $Ch_{*}(\mathpzc{E})$ for $*\in\{\ge0,\emptyset\}$. Then the transferred model structure exists on the category $\mathpzc{Alg}_{\mathfrak{P}}(Ch_{*}(\mathpzc{E}))$ of $\mathfrak{P}$-aglebras in $Ch_{*}(\mathpzc{E})$. If $\mathpzc{E}$ is enriched over $\mathbb{Q}$ and $\mathfrak{P}$ is a symmetric operad, then the transferred model structure also exists on $\mathpzc{Alg}_{\mathfrak{P}}(Ch_{*}(\mathpzc{E}))$.
\end{thm}
We also prove a Dold-Kan correspondence for algebras over operads (Theorem \ref{thm:simpDKA}) , again following \cite{white2019homotopical}, as well as a generalisation of the cosimplicial Dold-Kan equivalence of \cite{castiglioni2004cosimplicial} (Theorem \ref{thm:simpCDKA}).

Finally in Theorem \ref{HAcont} we show that whenever $\mathpzc{E}$ is monoidal elementary, the categories $Ch(\mathpzc{E})$ and $Ch_{\ge0}(\mathpzc{E})$ are homotopical algebra contexts in the sense of  \cite{toen2004homotopical}. In particular they are suitable settings for theories of derived geometry. In future joint work with Oren Ben-Bassat and Kobi Kremnitzer we will develop a model of derived analytic geometry by applying this to the quasi-abelian category $CBorn_{k}$.
 In  appendix \ref{modelapp} we recall some basic facts and set out our conventions regarding model categories. 
 
 \section{Notation and Conventions}\label{notation}

Throughout this work we will use the following notation.
\begin{itemize}
\item
$1$-categories will be denoted using the mathpzc font $\mathpzc{C},\mathpzc{D},\mathpzc{E}$, etc. In particular we denote by $\mathpzc{Ab}$ the category of abelian groups and ${}_{\Q}\mathpzc{Vect}$ the category of $\Q$-vector spaces. If $\mathpzc{M}$ is a model category, or a category with weak equivalences, its associated $(\infty,1)$-category will be denoted $\textbf{M}$.
\item 
Operads will be denoted using capital fractal letters $\mathfrak{C},\mathfrak{P}$, etc. Algebras over an operad will generally be denoted using capital letters $X,Y$, etc. The category of algebras over an operad will be denoted $\mathpzc{Alg}_{\mathfrak{P}}$
\item
We denote the operads for unital associative algebras, unital commutative algebras, non-unital commutative algebras, and Lie algebras by $\mathfrak{Ass},\mathfrak{Comm},\mathfrak{Comm}^{nu}$, and $\mathfrak{Lie}$ respectively. 
\item
 For the operad $\mathfrak{Ass},\mathfrak{Comm},\mathfrak{Lie}$ we will denote the corresponding free algebras by $T(V),S(V)$, and $L(V)$ respectively. We also denote by $U(L)$ the universal enveloping algebra of a Lie algebra $L$.
\item
Unless stated otherwise, the unit in a monoidal category will be denoted by $k$, the tensor functor by $\otimes$, and for a closed monoidal category the internal hom functor will be denoted by $\underline{\textrm{Hom}}$. Monoidal categories will always be assumed to be symmetric, with symmetric braiding $\sigma$.
\item
Filtered colimits will be denoted by $\textrm{lim}_{\rightarrow}$. Projective limits will be denoted $\textrm{lim}_{\leftarrow}$.
\item
The first infinite ordinal will be denoted $\aleph_{0}$.
\end{itemize}

Let us now introduce some conventions for chain complexes.

\begin{defn}\label{chaincat}
A \textbf{chain complex} in a pre-additive category $\mathpzc{E}$ is a sequence

\begin{displaymath}
\xymatrix{K_{\bullet}=\ldots\ar[r] & K_{n}\ar[r]^{d_{n}} & K_{n-1}\ar[r]^{d_{n-1}} & K_{n-2}\ar[r] &\ldots}
\end{displaymath}

where the $K_{i}$ are objects and the $d_{i}$ are morphisms such that $d_{n-1}\circ d_{n}=0$. The morphisms are called \textbf{differentials}. A \textbf{morphism of chain complexes} $f_{\bullet}:K_{\bullet}\rightarrow L_{\bullet}$ is a collection of morphisms $f_{n}:K_{
n}\rightarrow L_{n}$ such that the following diagram commutes for each $n$:

\begin{displaymath}
\xymatrix{\ldots\ar[r] & K_{n+1}\ar[d]_{f_{n+1}}\ar[r]^{d^{K}_{n+1}} & K_{n}\ar[d]_{f^{n}}\ar[r]^{d^{K}_{n}} & K_{n-1}\ar[d]_{f^{n-1}}\ar[r] &\ldots\\
\ldots\ar[r] & L_{n+1}\ar[r]^{d^{L}_{n+1}} & L_{n}\ar[r]^{d^{L}_{n}} & L_{n-1}\ar[r] &\ldots}
\end{displaymath}
\end{defn}

The category whose objects are chain complexes and whose morphisms are as described above is called the \textbf{category of chain complexes in} $\mathpzc{E}$, denoted $Ch(\mathpzc{E})$. We also define $Ch_{\ge0}(\mathpzc{E})$ to be the full subcategory of $Ch(\mathpzc{E})$ on complexes $A_{\bullet}$ such that $A_{n}=0$ for $n<0$, $Ch_{\le0}(\mathpzc{E})$ to be the full subcategory of $Ch(\mathpzc{E})$ on complexes $A_{\bullet}$ such that $A_{n}=0$ for $n>0$, $Ch_{+}(\mathpzc{E})$, the full subcategory of chain complexes $A_{\bullet}$ such that $A_{n}=0$ for $n<<0$, $Ch_{-}(\mathpzc{E})$, the full subcategory of chain complexes $A_{\bullet}$ such that $A_{n}=0$ for $n>>0$ and $Ch_{b}(\mathpzc{E})$  to be the full subcategory of $Ch(\mathpzc{E})$ on complexes $A_{\bullet}$ such that $A_{n}\neq 0$ for only finitely many $n$. A lot of the statements in the rest of this document apply to several of these categories at once. In such cases we will write $Ch_{*}(\mathpzc{E})$, and specify that $*$ can be any element of some subset of $\{\ge0,\le0,+,-,b,\emptyset\}$, where by definition $Ch_{\emptyset}(\mathpzc{E})=Ch(\mathpzc{E})$.\newline
\\
We will  frequently use the following special chain complexes.

\begin{defn}
If $E$ is an object of a pointed category $\mathpzc{E}$ we let $S^{n}(E)\in Ch(\mathpzc{E})$ be the complex whose $n$th entry is $E$,  with all other entries being $0$. We also denote by $D^{n}(E)\in Ch(\mathpzc{E})$ the complex whose $n$th and $(n-1)$st entries are $E$, with all other entries being $0$, and the differential $d_{n}$ being the identity.
\end{defn}

Let us also introduce some notation for truncation functors.

\begin{defn}
Let $\mathpzc{E}$ be an additive category which has kernels. For a complex $X_{\bullet}$ we denote by $\tau_{\ge n}X$ the complex such that $(\tau_{\ge n}X)_{m}=0$ if $m<n$, $(\tau_{\ge n}X)_{m}= X_{m}$ if $m>n$ and $(\tau_{\ge n}X)_{n}=\textrm{Ker}(d_{n})$. The differentials are the obvious ones. The construction is clearly functorial. Dually we define the truncation functor $\tau_{\le k}$. 
\end{defn}

All of the above categories are naturally enriched over $Ch(\mathpzc{Ab})$. We denote the enriched hom by $\textbf{Hom}(-,-)$. For notational clarity we recall its definition here.

\begin{defn}
Let $X_{\bullet},Y_{\bullet}\in Ch(\mathpzc{E})$. We define $\textbf{Hom}(X_{\bullet},Y_{\bullet})\in Ch(\mathpzc{Ab})$ to be the complex with
$$\textbf{Hom}(X_{\bullet},Y_{\bullet})_{n}=\prod_{i\in\Z}\textrm{Hom}_{\mathpzc{E}}(X_{i},Y_{i+n})$$
and differential $d_{n}$ defined on $\textrm{Hom}_{\mathpzc{E}}(X_{i},Y_{i+n})$ by
$$df=d^{Y}_{i+n}\circ f-(-1)^{n}f\circ d^{X}_{i}$$
\end{defn}

Let $(\mathpzc{E},\otimes,k)$ be a monoidal additive category, i.e. $\otimes$ is an additive bifunctor. There is an induced monoidal structure on $Ch_{*}(\mathpzc{E})$ for $*\in\{\ge0,\le0,+,-,b,\emptyset\}$. The unit is $S^{0}(k)$. If $X_{\bullet}$ and $Y_{\bullet}$ are chain complexes then we set

$$(X_{\bullet}\otimes Y_{\bullet})_{n}=\bigoplus_{i+j=n}X_{i}\otimes Y_{j}$$
If $i+j=n$, then we define the differential on the summand $X_{i}\otimes Y_{j}$ of $(X_{\bullet}\otimes Y_{\bullet})_{n}$ by
$$d^{X_{\bullet}\otimes Y_{\bullet}}_{n}|_{X_{i}\otimes Y_{j}}=d^{X_{\bullet}}_{i}\otimes id_{Y_{\bullet}}+(-1)^{i}id_{X_{\bullet}}\otimes d_{j}^{Y_{\bullet}}$$
If $*\in\{\ge0,\le0,+,-,b,\emptyset\}$ then $(Ch_{*}(\mathpzc{E}),\otimes,S^{0}(k))$ is a monoidal additive category. 

If $(\mathpzc{E},\otimes,k,\underline{\textrm{Hom}})$ is a closed monoidal additive category then we define a functor
$$\underline{\textrm{Hom}}(-,-):Ch(\mathpzc{E})^{op}\times Ch(\mathpzc{E})\rightarrow Ch(\mathpzc{E})$$

$$\underline{\textrm{Hom}}(X_{\bullet},Y_{\bullet})_{n}=\prod_{i\in\Z}\underline{\textrm{Hom}}_{\mathpzc{E}}(X_{i},Y_{i+n})$$
and differential $d_{n}$ defined on $\textrm{Hom}_{\mathpzc{E}}(X_{i},Y_{i+n})$ by

$$d=\underline{\textrm{Hom}}(d_{i}^{X_{\bullet}},id)+(-1)^{i}\underline{\textrm{Hom}}(id,d_{i+n}^{Y_{\bullet}})$$
This does define an internal hom on the monoidal category 

$$(Ch(\mathpzc{E}),\otimes,S^{0}(k))$$
The internal hom on chain complexes also restricts to a bifunctor

$$\underline{\textrm{Hom}}(-,-):Ch_{b}(\mathpzc{E})^{op}\times Ch_{b}(\mathpzc{E})\rightarrow Ch_{b}(\mathpzc{E})$$
Then

$$(Ch_{b}(\mathpzc{E}),\otimes,S^{0}(k),\underline{\textrm{Hom}})$$
is a closed monoidal additive category. In fact, in both of these categories there are natural isomorphisms of chain complexes of abelian groups.
$$\textbf{Hom}(X_{\bullet},\underline{\textrm{Hom}}(Y_{\bullet},Z_{\bullet}))\cong\textbf{Hom}(X_{\bullet}\otimes Y_{\bullet},Z_{\bullet})$$

The categories $Ch_{*}(\mathpzc{E})$ for $*\in\{+,-,b,\emptyset\}$ also come equipped with a shift functor. It is given on objects by $(A_{\bullet}[1])_{i}=A_{i+1}$ with differential $d_{i}^{A[1]}=-d^{A}_{i+1}$. The shift of a morphism $f^{\bullet}$ is given by $(f_{\bullet}[1])_{i}=f_{i+1}$. $[1]$ is an auto-equivalence with inverse $[-1]$. We set $[0]=\textrm{Id}$ and $[n]=[1]^{n}$ for any integer $n$.\newline
\\
 Finally, we define the mapping cone as follows.

\begin{defn}
Let $X_{\bullet}$ and $Y_{\bullet}$ be chain complexes in an additive category $\mathpzc{E}$ and $f_{\bullet}:X_{\bullet}\rightarrow Y_{\bullet}$. The \textbf{mapping cone of} $f_{\bullet}$, denoted $\textrm{cone}(f_{\bullet})$ is the complex whose components are
$$\textrm{cone}(f_{\bullet})_{n}=X_{n-1}\oplus Y_{n}$$
and whose differential is
 \begin{displaymath}
d^{\textrm{cone}(f)}_{n}  = \left(
     \begin{array}{lr}
       -d_{n-1}^{X} &0\\
       -f_{n-1} & d^{Y}_{n}
   \end{array}
            \right)
\end{displaymath} 
\end{defn}
There are natural morphisms $\tau:Y_{\bullet}\rightarrow \textrm{cone}(f)$ induced by the injections $Y_{i}\rightarrow X_{i-1}\oplus Y_{i}$, and $\pi:\textrm{cone}(f)\rightarrow X_{\bullet}[-1]$ induced by the projections $X_{i-1}\oplus Y_{i}\rightarrow X_{i-1}$. The sequence
$$Y_{\bullet}\rightarrow\textrm{cone}(f)\rightarrow X_{\bullet}[-1]$$
is split exact in each degree.

\subsection*{Acknowledgements}
 The author would like to thank Kobi  Kremnizer, Kevin McGerty, Theo  B\"{u}hler, andTimoth\'{e}e Moreau for useful comments. The author is also very grateful to the referee for many helpful comments and suggestions which have greatly improved this paper.

\chapter{Exact Category Generalities}\label{sec2}

In this chapter we will establish some technicalities about exact categories which will be used throughout this work. In particular we will discuss acyclicity of complexes and the existence of unbounded resolutions. We will also discuss various notions of generation and compactness in such categories, and introduce the notion of a monoidal exact category. Finally we will see when exact structures can be lifted to categories of algebras for some monad acting on an exact category. The results in this chapter will prove crucial for studying the homotopy theory of exact categories in Chapter \ref{chmodelexact}. 

\section{Recollections on Exact Categories}

In this section we review the rudiments of exact categories, following \cite{Buehler}. In the following $\mathpzc{E}$ will be an additive category. A $\textbf{kernel-cokernel pair}$ in $\mathpzc{E}$ is a pair of composable maps $(i,p)$, $i:A\rightarrow B,p:B\rightarrow C$ such that $i=\textrm{Ker}(p)$ and $p=\textrm{Coker}(i)$. If $\mathcal{Q}$ is a class of kernel-cokernel pairs and $(i,p)\in\mathcal{Q}$, then we say that $i$ is an admissible monic and $p$ is an admissible epic with respect to $\mathcal{Q}$.

\begin{defn}
A \textbf{Quillen exact structure} on an additive category $\mathpzc{E}$ is a collection $\mathcal{Q}$ of kernel-cokernel pairs such that
\begin{enumerate}
\item
Isomorphisms are both admissible monics and admissible epics.
\item
Both the collection of admissible monics and the collection of admissible epics  are closed under composition.
\item
If
\begin{displaymath}
\xymatrix{
A\ar[d]\ar[r]^{f} & B\ar[d]\\
X\ar[r]^{f'} & Y
}
\end{displaymath}
is a pushout diagram, and $f$ is an admissible monic, then $f'$ is as well.
\item
If
\begin{displaymath}
\xymatrix{
A\ar[d]\ar[r]^{f'} & B\ar[d]\\
X\ar[r]^{f} & Y
}
\end{displaymath}
is a pullback diagram, and $f$ is an admissible epic, then $f'$ is as well.
\end{enumerate}
\end{defn}

Let $(\mathpzc{E},\mathcal{Q})$ be an exact category. We call a null sequence
\begin{displaymath}
\xymatrix{
0\ar[r] & A\ar[r]^{i} & B\ar[r]^{p} & C\ar[r] & 0
}
\end{displaymath}
\textbf{short exact} if $(i,p)$ is a kernel-cokernel pair in $\mathcal{Q}$. We will use interchangeably the notion of kernel-cokernel pair and short exact sequence. When it is not likely to cause confusion, we will suppress the notation $(\mathpzc{E},\mathcal{Q})$ to $\mathpzc{E}$.

When studying exact categories it is natural to consider so-called exact functors:

\begin{defn}
Let $(\mathpzc{E},\mathcal{P})$, $(\mathpzc{F},\mathcal{Q})$ be exact categories. A functor $F:\mathpzc{E}\rightarrow\mathpzc{F}$ is said to be \textbf{exact} (with respect to $\mathcal{P}$ and $\mathcal{Q}$) if for any short exact sequence
$$0\rightarrow X\rightarrow Y\rightarrow Z\rightarrow 0$$
in $\mathcal{P}$,
$$0\rightarrow F(X)\rightarrow F(Y)\rightarrow F(Z)\rightarrow 0$$
is a short exact sequence in $\mathcal{Q}$.
\end{defn}

\begin{defn}
Let $(\mathpzc{E},\mathcal{P})$ be an exact category. An \textbf{exact subcategory} of $(\mathpzc{E},\mathcal{P})$ is an exact category $(\mathpzc{F},\mathcal{Q})$ where $\mathpzc{F}$ is a subcategory of $\mathpzc{E}$ and the inclusion  functor is exact.
\end{defn}

On any additive category one can define the split exact structure for which the kernel-cokernel pairs are the split exact sequences. Any exact category contains this is an exact subcategory. The split exact structure is therefore the minimal exact structure on an additive category. If an additive category has all kernels and cokernels, then it also has a maximal exact structure. This is due to Sieg and Wegner \cite{sieg2011maximal}. One particularly nice class of additive categories has the class of all kernel-cokernel pairs as its maximal exact structure.

\begin{defn}\label{quas}
An additive category $\mathpzc{E}$ with all kernels and cokernels is said to be \textbf{quasi-abelian} if the class $\mathpzc{qac}$ of all kernel-cokernel pairs forms an exact structure on $\mathpzc{E}$.
\end{defn}
The following is then tautological.

\begin{prop}
Let $\mathpzc{E}$ be a quasi-abelian category, and let $\mathcal{Q}$ be a class of kernel-cokernel pairs on $\mathpzc{E}$ such that $(\mathpzc{E},\mathcal{Q})$ is an exact category. Then  the identity functor $\textrm{id}_{\mathpzc{E}}$ is an exact functor $(\mathpzc{E},\mathcal{Q})\rightarrow(\mathpzc{E},\mathpzc{qac})$.
\end{prop}

We will study quasi-abelian structures in more detail later. For now let us note that abelian categories are quasi-abelian. In an abelian category all monics are kernels of their cokernels, and all epics are cokernels of their kernels. It therefore trivially follows that both classes are closed under composition. It is also clear that both classes contain all isomorphisms. It is a standard exercise that in an abelian category, monomorphisms are pushout-stable and epimorphisms are pullback-stable. See for example \cite{freyd} Theorem 2.54. Let us now record some basic results about exact categories which will prove useful.

\begin{prop}\label{bicart}
Let
\begin{displaymath}
\xymatrix{
A\ar@{>->}[r]^{i}\ar[d]^{f} & B\ar[d]^{f'}\\
A'\ar@{>->}[r]^{i'} & B'
}
\end{displaymath}
be a commutative diagram in which the horizontal morphisms are admissible monics. Then the following are equivalent
\begin{enumerate}
\item
The square above is a pushout.
\item 
The sequence
\begin{displaymath}
\xymatrix{
0\ar[r] & A\ar[r]^{\left(\begin{smallmatrix} i\\ -f \end{smallmatrix} \right)\;\;\;\;}&B\oplus A'\ar[r]^{\;\;\;\;\;\left(\begin{smallmatrix} f' & i' \end{smallmatrix} \right)} & B'\ar[r] & 0
}
\end{displaymath}
is short exact.
\item
The square above is bicartesian.
\item
The square is part of a commutative diagram
\begin{displaymath}
\xymatrix{
A\ar@{>->}[r]^{i}\ar[d]^{f} & B\ar[d]^{f'}\ar@{->>}[r]^{p} & C\ar@{=}[d]\\
A'\ar@{>->}[r]^{i'} & B'\ar@{->>}[r]^{p'} & C
}
\end{displaymath}
with short exact rows.
\end{enumerate}
\end{prop}

\begin{proof}
See \cite{Buehler} Proposition 2.12.
\end{proof}

\begin{prop}\label{subexactkc}
Let $\mathpzc{E}$ be an exact category and $\mathpzc{A}\subset\mathpzc{E}$ a full additive subcategory. Suppose that for every morphism $f:A\rightarrow B$ which is admissible in $\mathpzc{E}$, a kernel and cokernel of $f$ in $\mathpzc{E}$ exist in $\mathpzc{A}$. Then the collection of all kernel-cokernel pairs $(i:A\rightarrow B,p:B\rightarrow C)$ which are exact in $\mathpzc{E}$ where $A,B,C\in\mathpzc{A}$ defines an exact structure on $\mathpzc{A}$ which makes it an exact subcategory of $\mathpzc{E}$.
\end{prop}

\begin{proof}
It is clearly sufficient to show that this collection of kernel-cokernel pairs endows $\mathpzc{A}$ with an exact structure. The first and second conditions are clearly satisfied. Let 
\begin{displaymath}
\xymatrix{
A\ar[d]^{f}\ar[r]^{i} & B\ar[d]^{f'}\\
A'\ar[r]^{i'} & B'
}
\end{displaymath}
be a pushout diagram in $\mathpzc{E}$ with $f$ an admissible monic, and $i$ and $i'$ in $\mathpzc{A}$. We need to show that $Y$ is (isomorphic to) an object of $\mathpzc{A}$. But there is an exact sequence
\begin{displaymath}
\xymatrix{
0\ar[r] & A\ar[r]^{\left(\begin{smallmatrix} i\\ -f \end{smallmatrix} \right)}&B\oplus A'\ar[r]^{\;\;\;\;\;\;\;\left(\begin{smallmatrix} f' & i' \end{smallmatrix} \right)} & B'\ar[r] & 0
}
\end{displaymath}
in $\mathpzc{E}$. Now a cokernel of the map $A\rightarrow B\oplus A'$ in $\mathpzc{E}$ exists in $\mathpzc{A}$, so $B'$ is isomorphic to an object of $\mathpzc{A}$. The last condition is dual to this one.
\end{proof}

For technical reasons, unless stated otherwise we will assume from now on that all exact categories are \textbf{weakly idempotent complete}. This means that every retraction has a kernel, or equivalently, that every coretraction has a cokernel. Note that the condition is self-dual. Quasi-abelian categories are in particular weakly idempotent complete. In weakly idempotent complete exact categories, we then have the following useful result, often called the \textbf{Obscure Axiom}.

\begin{prop}[The Obscure Axiom]\label{obscure}
Let $\mathpzc{E}$ be a weakly idempotent complete exact category.
\begin{enumerate}
\item
Suppose that $i:A\rightarrow B$ is a morphism. If there exists a morphism $j:B\rightarrow C$ such that the composite $ji:A\rightarrow C$ is an admissible monic, then $i$ is an admissible monic. 
\item
Suppose that $i:A\rightarrow B$ is a morphism. If there exists a morphism $j:C\rightarrow A$ such that $i\circ j$ is an admissible epic, then $i$ is an an admissible epic.
\end{enumerate}
\end{prop}

\begin{proof}
See \cite{Buehler} Proposition 2.16.
\end{proof}

Another useful result which translates from abelian categories is the $3\times3$ Lemma, \cite{Buehler} Corollary 3.6. 
\begin{prop}[$3\times 3$ Lemma]
Consider a commutative diagram in which all columns are exact sequences and all rows are null sequences.
\begin{displaymath}
\xymatrix{
A'\ar[d]^{a}\ar[r]^{f'} & B'\ar[d]^{b}\ar[r]^{g'} & C'\ar[d]^{c}\\
A\ar[d]^{a'}\ar[r]^{f} & B\ar[d]^{b'}\ar[r]^{g} & C\ar[d]^{c'}\\
A''\ar[r]^{f''} & B''\ar[r]^{g''} & C
}
\end{displaymath}
If two of the rows are exact sequences then the remaining row is also an exact sequence. 
\end{prop}

This has an immediate corollary.
\begin{cor}[\cite{Buehler} Proposition 2.9]\label{plusexact}
Let $f:X\rightarrow Y$ and $a:A\rightarrow B$ be admissible monomorphisms (resp. admissible epimorphisms). Then $f\oplus a:X\oplus A\rightarrow Y\oplus B$ is an admissible monomorphism (resp. admissible epimorphism).
\end{cor}

subsection{Abelianisations}

Let $(\mathpzc{E},\mathcal{Q})$ be an exact category. Let $\mathpzc{F}$ be a full subcategory of $\mathpzc{E}$. Suppose that $\mathpzc{F}$ is closed under extensions, that is if
$$0\rightarrow A\rightarrow B\rightarrow C\rightarrow 0$$
is a short exact sequence in $(\mathpzc{E},Q)$ with $A$ and $C$ objects of $\mathpzc{F}$, then $B$ is an object of $\mathpzc{F}$ as well. Let $\mathcal{Q}_{\mathpzc{F}}$ consist of those kernel-cokernel pairs $(i:A\rightarrow B,q:B\rightarrow C)$ in $\mathpzc{F}$ which when regarded as pairs of morphisms in $\mathpzc{E}$ are kernel-cokernel pairs in $\mathcal{Q}$. It is then straightforward to show (\cite{baum}) that $(\mathpzc{F},\mathcal{Q}_{\mathpzc{F}})$ is an exact subcategory of $(\mathpzc{E},\mathcal{Q})$. It turns out that any small exact category can be obtained as a full subcategory of an abelian category which is closed under extensions. This is the main content of the Gabriel-Quillen Embedding Theorem which provides an invaluable tool for studying exact categories.

\begin{thm}[The Gabriel-Quillen Embedding Theorem]\label{QET}
Let $\mathpzc{E}$ be a small exact category. Then there is an abelian category $\mathpzc{A}(\mathpzc{E})$ and a fully faithful additive functor $I:\mathpzc{E}\rightarrow\mathpzc{A}(\mathpzc{E})$ which is exact, reflects exactness, and preserves all kernels. Moreover the essential image of $I$ is closed under extensions. $\mathpzc{A}(\mathpzc{E})$ may be chosen to be the category of left-exact functors $\mathpzc{E}\rightarrow\mathpzc{Ab}$. If in addition $\mathpzc{E}$ is weakly idempotent complete then a morphism $f:E\rightarrow F$ in $\mathpzc{E}$ is an admissible epic if and only if $I(f)$ is an epic in $\mathpzc{A}(\mathpzc{E})$.
\end{thm}

\begin{proof}
See Appendix A in \cite{Buehler}.
\end{proof}

\begin{defn}
We call an embedding $I:\mathpzc{E}\rightarrow\mathpzc{A}$ of an exact category into an abelian category  a \textbf{left abelianisation} of $\mathpzc{E}$ if 
\begin{enumerate}
\item
$I$ is fully faithful.
\item
$I$ is exact.
\item
$I$ reflects exactness.
\item
The essential image of $I$ is closed under extensions.
\item
$I$ preserves all kernels which exist.
\item
If $f$ is a morphism in $\mathpzc{E}$, then $f$ is an admissible epic if and only if $I(f)$ is an epic.
\end{enumerate}
\end{defn}
In particular, Theorem \ref{QET} says that any weakly idempotent complete small exact category admits a left abelianisation.  In fact, if the final assumption is removed, then all small exact categories have abelianisations. There is an obvious dual notion of a \textbf{right abelianisation}.
It is clear that right abelianisations of small weakly idempotent complete exact categories also exist. Indeed, if $\mathpzc{E}^{op}\rightarrow\mathpzc{A}$ is a left-abelianisation of $\mathpzc{E}^{op}$, then $\mathpzc{E}\rightarrow\mathpzc{A}^{op}$ is a right-abelianisation of $\mathpzc{E}$.

\subsection{Generation of Exact Subcategories}

Let $\mathpzc{E}$ be a locally small additive category and $\mathpzc{A}$ a small full subcategory. By an argument similar to \cite{1006507} we can find a small full exact subcategory of $\mathpzc{E}$ containing $\mathpzc{A}$. In the rest of the section we assume that given a small subcategory $\mathpzc{E}$ of $\mathpzc{A}$ one can choose direct sums for finite collections of objects in $\mathpzc{E}$, and kernels and cokernels of morphisms in $\mathpzc{E}$. For example one might assume that such limits and colimits can be made functorial in the ambient category (e.g. if $\mathpzc{E}$ is locally presentable).
\begin{prop}\label{addgen}
There is a small full additive subcategory $\Sigma(\mathpzc{A};\mathpzc{E})$ of $\mathpzc{E}$ containing $\mathpzc{A}$.
\end{prop}

\begin{proof}
We let $\Sigma(\mathpzc{A};\mathpzc{E})$ be the full subcategory whose objects are the zero object and a choice of a direct sum in $\mathpzc{E}$ for each finite collection of objects of $\mathpzc{A}$. This is clearly additive, contains $\mathpzc{A}$, and is small.
\end{proof}

Now let $\mathpzc{E}$ be an exact category and $\mathpzc{A}$ a small full subcategory.

\begin{prop}
There is a small full exact subcategory $\textit{Ex}(\mathpzc{A},\mathpzc{E})$ of $\mathpzc{E}$ containing $\mathpzc{A}$.
\end{prop}

\begin{proof}
By Proposition \ref{addgen} we may assume that $\mathpzc{A}$ is additive. Let $\textit{Ex}^{1}(\mathpzc{A},\mathpzc{E})$ denote the full subcategory of $\mathpzc{E}$ consisting of a choice of kernels and cokernels of morphisms $f:A\rightarrow A'$ which are admissible in $\mathpzc{E}$. We set $\textit{Ex}^{n+1}(\mathpzc{A};\mathpzc{E})\defeq \textit{Ex}^{1}(\textit{Ex}^{n}(\mathpzc{A};\mathpzc{E});\mathpzc{E})$ We claim that
$$\textit{Ex}(\mathpzc{A};\mathpzc{E})\defeq\bigcup_{n=1}^{\infty}\textit{Ex}^{n}(\mathpzc{A};\mathpzc{E})$$
works. Since $\textit{Ex}^{1}(\mathpzc{A};\mathpzc{E})$ is small for $\mathpzc{A}$ small this would prove the claim.  $\bigcup_{n=1}^{\infty}\textit{Ex}^{n}(\mathpzc{A};\mathpzc{E})$ is clearly closed under taking kernels and cokernels of those morphisms which are admissible in $\mathpzc{E}$. By Proposition \ref{subexactkc} it is an exact subcategory.
\end{proof}

The point of this is that even if a category $\mathpzc{E}$ is not small, when working with small diagrams in $\mathpzc{E}$ we can pass to an abelianisation.

\section{Homological Properties of Exact Categories}
\subsection{Notions of Acyclicity}

In a general exact category, arbitrary kernels and cokernels may not exist. Therefore it is not in general possible even to write down candidates for the homology objects of a chain complex. Even if all kernels and cokernels do exist, then there are multiple candidates for the homology which are not isomorphic in general. For example, given a null sequence
\begin{displaymath}
\xymatrix{
\Gamma=E\ar[r]^{f} & F\ar[r]^{g} & G
}
\end{displaymath}
 i.e. $g\circ f=0$, one could consider both $\textrm{Coker}(\textrm{Im}(f)\rightarrow\textrm{Ker}(g))$ and $\textrm{Im}(\textrm{Ker}(g)\rightarrow\textrm{Coker}(f))$. In an abelian category these are isomorphic, but for general additive categories this is not the case. Despite these ambiguities, there are still various useful notions of acyclicity in exact categories, which we discuss below. First let us define several classes of morphisms.
 
 \begin{defn}
A morphism $f:E\rightarrow F$ in an exact category is said to be

\begin{enumerate}
\item
\textbf{weakly left admissible} if it has a kernel and the map
$$\textrm{Ker}(f)\rightarrow E$$
is admissible.
\item
 \textbf{weakly right admissible} if it has a cokernel, and the map
$$F\rightarrow\textrm{Coker}(f)$$
is admissible.
\item
\textbf{weakly admissible} if it is both  weakly left admissible and weakly right admissible.
\end{enumerate}
\end{defn}

The following characterisation of weakly admissible morphisms is immediate.

\begin{prop}\label{analysis}
A morphism $f:E\rightarrow F$ in an exact category $\mathpzc{E}$ is weakly admissible if and only if it admits a decomposition
\begin{displaymath}
\xymatrix{
 & E\ar@{->>}[dr]\ar[rrr]^{f} && & F\ar@{->>}[dr] &\\
 \textrm{Ker} (f)\ar@{>->}[ur] & & \textrm{Coim}(f)\ar[r]^{\hat{f}} & \textrm{Im}(f)\ar@{>->}[ur] & & \textrm{Coker} (f)
}
\end{displaymath}
where the sequences
$$\textrm{Ker}(f)\rightarrowtail E\twoheadrightarrow\textrm{Coim}(f)$$
and
$$\textrm{Im}(f)\rightarrowtail F\twoheadrightarrow\textrm{Coker}(f)$$
are short exact.
\end{prop}

\begin{defn}
Let $f$ be a morphism in exact category. Then $f$ is said to be \textbf{admissible} if it is weakly admissible and the map $\textrm{Coim}(f)\rightarrow\textrm{Im}(f)$ is an isomorphism.
\end{defn}

\begin{rem}
Admissible epimorphisms and admissible monomorphisms are admisssible morphisms in the sense above.
\end{rem}
This is not how admissible morphisms are usually defined (see e.g. \cite{Buehler}). However the notions are equivalent.

\begin{prop}\label{adstrict}
Let $f:E\rightarrow F$ be a morphism in an exact category $\mathpzc{E}$. Then the following are equivalent.
\begin{enumerate}
\item
 $f$ is admissible.
 \item
$f$ admits a decomposition
$$E\twoheadrightarrow I\rightarrowtail F$$
\item
There is a commutative diagram
\begin{displaymath}
\xymatrix{
 & E\ar@{->>}[dr]\ar[rr]^{f} & & F\ar@{->>}[dr] &\\
 \textrm{Ker} f\ar@{>->}[ur] & & I\ar@{>->}[ur] & & \textrm{Coker} f
}
\end{displaymath}
where the sequences
$$\textrm{Ker} f\rightarrowtail E\twoheadrightarrow I$$
and
$$I\rightarrowtail F\twoheadrightarrow\textrm{Coker}(f)$$
are short exact.
\end{enumerate}
\end{prop}

\begin{proof}
$1$ and $3$ are clearly equivalent thanks to Proposition \ref{analysis}. Also $3\Rightarrow2$ trivially. Let us show that $2\Rightarrow 1$. Since $I\rightarrowtail F$ is an admissible monic, the kernel of $f$ exists, and coincides with the kernel of $E\twoheadrightarrow I$. Hence $\textrm{Ker}(f)\rightarrow E$ is an admissible monic and in particular $E\rightarrow I$ is a coimage of $f$. Dually, the cokernel of $f$ exists, it coincides with the cokernel of $G\rightarrowtail F$, and $I\rightarrowtail F$ is an image of $f$. 
\end{proof}

\begin{cor}\label{adbi}
A morphism $f:E\rightarrow F$ in an exact category is an isomorphism if and only if it is both an admissible epic and an admissible monic.
\end{cor}

\begin{proof}
Axiomatically an isomorphism is both an admissible monic and an admissible epic. Conversely, suppose $f$ is both an admissible monic and an admissible epic. Since it is an admissible monic the map $E\rightarrow\textrm{Coim}(f)$ is an isomorphism. Since it is an admissible epic the map $\textrm{Im}(f)\rightarrow E$ is an isomorphism. Since $f$ is admissible the map $\textrm{Coim}(f)\rightarrow\textrm{Im}(f)$ is an isomorphism. The claim now follows from the commutative diagram
\begin{displaymath}
\xymatrix{
E\ar[r]^{f}\ar[d]^{\sim} & F\\
\textrm{Coim}(f)\ar[r]^{\sim} & \textrm{Im} (f)\ar[u]^{\sim}
}
\end{displaymath} 
\end{proof}

We are now ready to introduce our various notions of acyclic sequences.

\begin{defn}
A null sequence
\begin{displaymath}
\xymatrix{
X\ar[r]^{f} & Y\ar[r]^{g} & Z
}
\end{displaymath}
is said to be 
\begin{enumerate}
\item
\textbf{weakly acyclic} if $f$ is weakly right admissible, $g$ has a kernel, and the natural map $\textrm{Im}(f)\rightarrow\textrm{Ker}(g)$ is an isomorphism.
\item
\textbf{weakly coacyclic} if $g$ is weakly left admissible, $f$ has a cokernel, and the natural map $\textrm{Coker}(f)\rightarrow\textrm{Coim}(g)$ is an isomorphism.
\item
\textbf{admissibly acyclic} if it is weakly acyclic and $f$ is admissible, 
\item
 \textbf{admissibly coacyclic} if it is weakly coacyclic and $g$ is admissible
 \item
 \textbf{admissible} if both $f$ and $g$ are admissible.
 \item
 \textbf{acyclic} if it is both admissibly acyclic and admissibly coacyclic. 
\end{enumerate}
\end{defn}

\begin{rem}\label{adacleftac}
If a null sequence
\begin{displaymath}
\xymatrix{
X\ar[r]^{f} & Y\ar[r]^{g} & Z
}
\end{displaymath}
is weakly acyclic then $g$ is automatically weakly left admissible.
\end{rem}

\begin{defn}
A complex
\begin{displaymath}
\xymatrix{
X_{n}\ar[r]^{f_{n}} & X_{n-1}\ar[r]^{f_{n-1}} & \ldots\ar[r] & X_{0}
}
\end{displaymath}
is said to be weakly acyclic/ weakly coacyclic/ admissibly acyclic/ admissibly coacyclic/ admissible/ acyclic if for each $1\le i\le n-1$ each sequence
\begin{displaymath}
\xymatrix{
X_{i+1}\ar[r]^{f_{i+1}} & X_{i}\ar[r]^{f_{i}} & X_{i-1}
}
\end{displaymath}
is weakly acyclic/ weakly coacyclic/ admissibly acyclic/ admissibly coacyclic/ admissible/ acyclic.
\end{defn}

Let us now set up some tools for determining whether a complex is acyclic. We can partially test acyclicity by passing to a left abelianisation.

\begin{prop}\label{abeltest}
Let $I:\mathpzc{E}\rightarrow\mathpzc{A}$ be a left abelianisation of $\mathpzc{E}$.
\begin{enumerate}
\item
 If
\begin{displaymath}
\xymatrix{
X_{n}\ar[r]^{f_{n}} & X_{n-1}\ar[r]^{f_{n-1}} & \ldots\ar[r] & X_{0}
}
\end{displaymath}
is admissibly acyclic in $\mathpzc{E}$ then
\begin{displaymath}
\xymatrix{
I(X_{n})\ar[r]^{I(f_{n})} & I(X_{n-1})\ar[r]^{I(f_{n-1})} & \ldots\ar[r] & I(X_{0})
}
\end{displaymath}
is exact in $\mathpzc{A}$.
\item
If $f_{i}$ is weakly admissible for $2\le i\le n$ and  
\begin{displaymath}
\xymatrix{
I(X_{n})\ar[r]^{I(f_{n})} & I(X_{n-1})\ar[r]^{I(f_{n-1})} & \ldots\ar[r] & I(X_{0})
}
\end{displaymath}
is exact, then 
\begin{displaymath}
\xymatrix{
X_{n}\ar[r]^{f_{n}} & X_{n-1}\ar[r]^{f_{n-1}} & \ldots\ar[r] & X_{0}
}
\end{displaymath}
is admissibly acyclic.
\item
If $f_{i}$ is weakly left admissible for $1\le i\le n-1$ and\begin{displaymath}
\xymatrix{
I(X_{n})\ar[r]^{I(f_{n})} & I(X_{n-1})\ar[r]^{I(f_{n-1})} & \ldots\ar[r] & I(X_{0})
}
\end{displaymath}
is exact in $\mathpzc{A}$, then
\begin{displaymath}
\xymatrix{
X_{n}\ar[r]^{f_{n}} & X_{n-1}\ar[r]^{f_{n-1}} & \ldots\ar[r] & X_{0}
}
\end{displaymath}
is admissibly acyclic.
\end{enumerate}
\end{prop}

\begin{proof}
Clearly it is sufficient to prove the claims for sequences
\begin{displaymath}
\xymatrix{
X\ar[r]^{f} & Y\ar[r]^{g} & Z
}
\end{displaymath}
\begin{enumerate}
\item
Suppose the above sequence is admissibly acyclic. Since $f$ is admissible $I$ preserves $\textrm{Im}(f)$. By assumption $I$ preserves all kernels. Hence
\begin{displaymath}
\xymatrix{
I(X)\ar[r]^{I(f)} & I(Y)\ar[r]^{I(g)} & I(Z)
}
\end{displaymath}
is exact.
\item
Suppose now that
\begin{displaymath}
\xymatrix{
I(X)\ar[r]^{I(f)} & I(Y)\ar[r]^{I(g)} & I(Z)
}
\end{displaymath}
is exact and that  $f$ is weakly admissible. Since $I$ preserves all kernels, and cokernels of admissible morphisms, we have $I(\textrm{Coim} (f))\cong\textrm{Coim} I(f)$. Now 
$$\textrm{Coim} I(f)\cong\textrm{Im}I(f)\cong\textrm{Ker} I(g)$$
Since $I$ is fully faithful, $\textrm{Coim}(f)$ is a kernel of $g$. Finally, note that we have a factorisation of $\textrm{Coim}(f)\rightarrow\textrm{Ker}(g)$
$$\textrm{Coim}(f)\rightarrow\textrm{Im}(f)\rightarrowtail \textrm{Ker} (g)$$
By Proposition \ref{obscure} $\textrm{Im}(f)\rightarrowtail \textrm{Ker} (g)$ is also an (admissible) epic. By Corollary \ref{adbi} it is an isomorphism. Therefore $\textrm{Coim}(f)\rightarrow\textrm{Im}(f)$ is as well. By Proposition \ref{adstrict} we are done.
\item
We can factor $f$ as 
\begin{displaymath}
\xymatrix{
X\ar[r]^{f'} &\textrm{Ker}(g)\ar[r] &  Y
}
\end{displaymath}
with $\textrm{Ker}(g)\rightarrow  Y$ an admissible monic. We need to show $f'$ is an admissible epic. Since $I$ preserves kernels, it sends the diagram above to
\begin{displaymath}
\xymatrix{
I(X)\ar[r]^{I(f')} & \textrm{Ker}I(g)\ar[r] & I(Y)
}
\end{displaymath}
Since 
\begin{displaymath}
\xymatrix{
I(X)\ar[r]^{I(f)} & I(Y)\ar[r]^{I(g)} & I(Z)
}
\end{displaymath}
is exact, $I(f')$ is an epic. Thus $f'$ is an admissible epic, and we are done.
\end{enumerate}
\end{proof}

Let us give an immediate application.
\begin{prop}
Let $f:X\rightarrow Y$ and $g:Y\rightarrow Z$ be admissible morphisms in a weakly idempotent complete exact category $\mathpzc{E}$. Then the sequence
$$0\rightarrow Ker(f)\rightarrow Ker(g\circ f)\rightarrow Ker(g)$$
is admissibly acyclic. If $f$ is an admissible epimorphism then
$$0\rightarrow Ker(f)\rightarrow Ker(g\circ f)\rightarrow Ker(g)\rightarrow0$$
is exact.
\end{prop}
\begin{proof}
For the first claim it suffices to show that $Ker(f)\rightarrow Ker(g\circ f)$ is weakly left admissible. Then we may pass to a left abelianisation. The composition $Ker(f)\rightarrow Ker(g\circ f)\rightarrow X$ is an admissible monomorphism since $f$ is admissible. Thus $Ker(f)\rightarrow Ker(g\circ f)$ is an admissible monomorphism. The second claim follows since a left abelianisation preserves admissible epimorphisms.
\end{proof}

Although the functor $I$ reflects short exact sequences, it need not in general reflect acyclicity of unbounded complexes. However it does for a certain nice class of complexes.
\begin{defn}
A complex $X_{\bullet}$ in an exact category is said to be \textbf{good} if for each $n$ there is $m<n$ such that $d_{m}$ has a kernel. $X_{\bullet}$ is said to be \textbf{cogood} if for each $n$ there is $m>n$ such that $d_{m}$ has a cokernel.
\end{defn}
\begin{example}
Bounded below complexes are good.
\end{example}

We will  frequently use the following trick for good complexes.
\begin{prop}\label{goodtrick}
Let $X_{\bullet}$ be a good complex in an exact category. Suppose that for any $n$ such that $d^{X}_{n}$ has a kernel, the induced map
$$d'_{n+1}:X_{n+1}\rightarrow Z_{n}X$$
is an admissible epic. Then $X_{\bullet}$ is acyclic.
\end{prop}
\begin{proof}
 Suppose $d_{m}$ has a kernel. By assumption $d_{m+1}$ factors as
$$X_{m+1}\twoheadrightarrow Z_{m}X\rightarrow X_{m}$$
A priori $Z_{m}X\rightarrow X_{m}$ is not  admissible. However it is a monomorphism. Therefore, since $X_{m+1}\twoheadrightarrow Z_{m}X$ is admissible its kernel exists and it coincides with the kernel $Z_{m+1}X$ of $d_{m+1}$. Since $X_{m+1}\twoheadrightarrow Z_{m}X$ is admissible it is in particular weakly left admissible. Therefore $d_{m+1}$ is also weakly left admissible. Now consider $d_{m+2}$. By assumption it factors as 
$$d_{m+2}:X_{m+2}\twoheadrightarrow Z_{m+1}X\rightarrowtail X_{m+1}$$
Thus $d_{m+2}$ is an admissible morphism whose image is $Z_{m+1}X$. An easy induction then shows that $X_{\bullet}$ is acyclic.
\end{proof}
Since $I$ preserves kernels and reflects admissible epimorphisms, Proposition \ref{goodtrick} gives the following.
\begin{cor}\label{boundedtest}
Let $(X_{\bullet},d_{\bullet})$ be a complex in $\mathpzc{E}$. Let $I:\mathpzc{E}\rightarrow\mathcal{A}$ be a left abelianisation of $\mathpzc{E}$. Suppose $X_{\bullet}$ is good. Then $X_{\bullet}$ is acyclic if and only if $I(X_{\bullet})$ is.
\end{cor}
\begin{proof}
Suppose $I(X_{\bullet})$ is acyclic, and $d_{n}^{X}$ has a kernel $Z_{n}X$. By assumption $I(d'_{n+1}):I(X_{n+1})\rightarrow Z_{n}I(X)=I(Z_{n}X)$ is an epimorphism. Thus $d'_{n+1}:X_{n+1}\rightarrow Z_{n}X$ is an admissible epimorphism.
\end{proof}
This can be used to prove the following useful result about truncations functors
\begin{prop}
Let $f:X\rightarrow Y$ be a weak equivalence in $Ch(\mathpzc{E})$. If $d^{Y}_{n}$, $d^{X}_{n}$, and $d^{X}_{n-1}$  have kernels, then
$$\tau_{\ge n}f:\tau_{\ge n}X\rightarrow\tau_{\ge n}Y$$
is a weak equivalence.
\end{prop}
It is a nice exercise to prove this result without appealing to abelianisations.

\subsection{Exactness of Functors}
Before continuing, let us introduce some further useful notions of exactness for functors, following \cite{qacs} Section 1.1.5. Part 1) of Proposition \ref{abeltest} above says that the functor $I$ is admissibly exact. This is a stronger notion than exactness. It will be useful in later contexts, so we make a definition.
\begin{defn}\label{strongexact} 
A functor $F:\mathpzc{E}\rightarrow\mathpzc{F}$ between exact categories is said to be \textbf{admissibly (co)exact} if for any admissibly (co)acyclic sequence
$$X\rightarrow Y\rightarrow Z$$
in $\mathpzc{E}$, the sequence
$$F(X)\rightarrow F(Y)\rightarrow F(Z)$$
is admissibly (co)acyclic. A functor which is both admissibly exact and admissibly coexact is said to be \textbf{strongly exact}.
\end{defn}

\begin{prop}\label{admexact}
Let $F:\mathpzc{E}\rightarrow\mathpzc{F}$ be an additive functor between exact categories with kernels. Suppose that $F$ commutes with cokernels and preserves admissible epimorphisms. Then $F$ is admissibly coexact. Dually if $F:\mathpzc{E}\rightarrow\mathpzc{D}$ is an additive functor between categories with cokernels,  which commutes with kernels and preserves admissible monomorphisms then $F$ is admissibly exact. 
\end{prop}
\begin{proof} 
The fact that $F$ preserves cokernels implies that if $p:Y\rightarrow Z$ is a cokernel of $i:X\rightarrow Y$, then $F(p):F(Y)\rightarrow F(Z)$ is a cokernel of $F(i):F(X)\rightarrow F(Y)$. Therefore $Im(F(i))=Ker(F(p))$.  Now $F(Coim(p))\cong F(Coker(i))\cong Coker(F(i))$. Therefore it remains to show that $F(Coim(p))\cong Coim(F(p))$. But both $p$ and $F(p)$ are admissible epimorphisms, so $F(Coim(p)) \cong F(Z)\cong Coim (F(p))$
\end{proof}

\begin{example}
It is easy to show that taking finite direct sums is a strongly exact functor. Indeed being both a limit and a colimit, this functor commutes with all limits and colimits. To see that the functor is exact, just use
Corollary \ref{plusexact}
\end{example}
We also have the following notion.
\begin{defn}
A covariant functor $F:\mathpzc{E}\rightarrow\mathpzc{F}$ between exact categories is said to be \textbf{right exact} if for any short exact sequence
$$0\rightarrow X\rightarrow Y\rightarrow Z\rightarrow 0$$
in $\mathpzc{E}$, the sequence
$$F(X)\rightarrow F(Y)\rightarrow F(Z)\rightarrow 0$$
is admissibly coacyclic in $\mathpzc{F}$.\newline
\\
A contravariant functor $F:\mathpzc{E}\rightarrow\mathpzc{F}$ between exact categories is said to be \textbf{right exact} if for any short exact sequence
$$0\rightarrow X\rightarrow Y\rightarrow Z\rightarrow 0$$
in $\mathpzc{E}$, the sequence
$$F(Z)\rightarrow F(Y)\rightarrow F(X)\rightarrow 0$$
is admissibly coacyclic in $\mathpzc{F}$. Dually one defines left exactness.
\end{defn}
In particular Proposition \ref{admexact} says that a right exact functor between exact categories with kernels which commutes with cokernels is admissibly coexact.

\subsection{Homotopies and Quasi-Isomorphisms}

\subsubsection{Homotopies}
Let us now discuss homological properties of maps between complexes. 

\begin{defn}\label{homequiv}
A \textbf{homotopy} between morphisms of chain complexes $f_{\bullet},g_{\bullet}:K_{\bullet}\rightarrow L_{\bullet}$ is a collection of morphisms $D_{i}:A_{i}\rightarrow B_{i+1}$ such that
$$f_{i}-g_{i}=D_{i-1}\circ d_{i}^{K}+d_{i+1}^{L}\circ D_{i}$$
We then say $f_{\bullet}\sim g_{\bullet}$.
\end{defn}

\begin{defn}
Two complexes $K_{\bullet}$ and $L_{\bullet}$ are said to be \textbf{homotopy equivalent} if there are maps $g:K_{\bullet}\rightarrow L_{\bullet}$ and $f:L_{\bullet}\rightarrow K_{\bullet}$ such that $f\circ g\sim id_{K_{\bullet}}$ and $g\circ f\sim id_{L_{\bullet}}$.
\end{defn}

If 
\begin{displaymath}
\xymatrix{
A\ar[r]^{p}\ar[d] & B\ar[d]^{\alpha}\ar[r]^{q} & C\ar[d]\\
X\ar[r]^{f} & Y\ar[r]^{g} & Z
}
\end{displaymath}
is a diagram with the top and bottom row being null sequences, we will also say that it is homotopic to zero if there are two maps $D:B\rightarrow X$ and $D':C\rightarrow Y$ such $\alpha=f\circ D-D'\circ q$. 

We can use homotopies in an exact category to test for acyclicity.

\begin{prop}\label{homotopygood}
Let $\mathpzc{E}$ be a weakly idempotent complete  exact category, and let
\begin{displaymath}
\xymatrix{
X\ar[r]^{f} & Y\ar[r]^{g} & Z
}
\end{displaymath}
be a null sequence. Suppose that $g$ has a kernel. Then the induced map $f':X\rightarrow\textrm{Ker}(g)$ is an admissible epimorphism if and only if there is a diagram
\begin{displaymath}
\xymatrix{
A\ar[r]^{p}\ar[d] & B\ar[d]^{\alpha}\ar[r]^{q} & C\ar[d]\\
X\ar[r]^{f} & Y\ar[r]^{g} & Z
}
\end{displaymath}
which is homotopic to zero, and such that the induced map $\widetilde{\alpha}:\textrm{Ker}(q)\rightarrow\textrm{Ker}(g)$ is an admissible epic.
\end{prop}

\begin{proof}
Suppose that $g$ has a kernel and that the induced map $f':X\rightarrow\textrm{Ker}(g)$ is an admissible epimorphism. Consider the diagram
\begin{displaymath}
\xymatrix{
0\ar[r]\ar[d] & X\ar[d]^{f}\ar[r] & 0\ar[d]\\
X\ar[r]^{f} & Y\ar[r]^{g} & Z
}
\end{displaymath}
By assumption the induced map $\widetilde{f}:X\rightarrow\textrm{Ker}(g)$ is an admissible epic. Moreover the diagram is clearly homotopic to $0$ via the maps $D=id:X\rightarrow X$ and $D'=0:0\rightarrow Y$. Conversely suppose we have a diagram
\begin{displaymath}
\xymatrix{
A\ar[r]^{p}\ar[d] & B\ar[d]^{\alpha}\ar[r]^{q}\ar@{-->}[dl]_{D} & C\ar[d]\ar@{-->}[dl]_{D'}\\
X\ar[r]^{f} & Y\ar[r]^{g} & Z
}
\end{displaymath}
such that $g$ has a kernel, $\alpha=f\circ D-D'\circ q$, and $\widetilde{\alpha}$ is an admissible epic. We have the factorisation of $f$
\begin{displaymath}
\xymatrix{
X\ar[r]^{\widetilde{f}} & \textrm{Ker}(g)\ar[r] & Y
}
\end{displaymath}
Moreover, $\widetilde{\alpha}=\widetilde{f}\circ D|_{\textrm{Ker}(q)}$. By Proposition \ref{obscure} $\widetilde{f}$ is an admissible epic. \end{proof}

\begin{cor}\label{homotopy}
Let $\mathpzc{E}$ be a weakly idempotent complete exact category, and let
\begin{displaymath}
\xymatrix{
X\ar[r]^{f} & Y\ar[r]^{g} & Z
}
\end{displaymath}
be a null sequence. The sequence is admissibly acyclic if and only if $g$ is weakly left admissible and there is a diagram\begin{displaymath}
\xymatrix{
A\ar[r]^{p}\ar[d] & B\ar[d]^{\alpha}\ar[r]^{q} & C\ar[d]\\
X\ar[r]^{f} & Y\ar[r]^{g} & Z
}
\end{displaymath}
which is homotopic to zero, and such that the induced map $\widetilde{\alpha}:\textrm{Ker}(q)\rightarrow\textrm{Ker}(g)$ is an admissible epic.
\end{cor}

\begin{proof}
Suppose the sequence is admissibly acyclic. By Remark \ref{adacleftac} $g$ is weakly left admissible. 

For the converse, note that by Proposition \ref{homotopygood} and the fact that $\textrm{Ker}(g)\rightarrow Y$ is admissible, we have a decomposition of $f$
$$X\twoheadrightarrow\textrm{Ker}(g)\rightarrowtail Y$$
By Proposition \ref{adstrict} $f$ is an admissible morphism whose image is $\textrm{Ker}(g)$.
\end{proof}
We can also test split exactness by looking at homotopy.

\begin{prop}\label{splithomotopy}
Let $\mathpzc{E}$ be a weakly idempotent complete  exact category, and let 
\begin{displaymath}
\xymatrix{
\Gamma:= X\ar[r]^{f} & Y\ar[r]^{g} & Z
}
\end{displaymath}
be a null sequence. The sequence is admissibly acyclic in the split exact structure if and only if $g$ is weakly left admissible and  the diagram
\begin{displaymath}
\xymatrix{
X\ar[r]^{f}\ar[d]^{id_{X}} & Y\ar[d]^{id_{Y}}\ar[r]^{g} & Z\ar[d]^{id_{Z}}\\
X\ar[r]^{f} & Y\ar[r]^{g} & Z
}
\end{displaymath}
is homotopic to zero.
\end{prop}

\begin{proof}
Suppose the diagram is homotopic to the zero. If we can show that $g$ is also weakly left admissible in the split exact structure, then the claim follows from Corollary \ref{homotopy}.  By Corollary \ref{homotopy} we already know that the sequence is admissibly acyclic, so $\textrm{Im}(f)\cong\textrm{Ker}(g)$. Let $D:Y\rightarrow X$ and $D':Z\rightarrow Y$ be maps such that $id_{Y}=f\circ D-D'\circ g$.  The map $f\circ D:Y\rightarrow Y$ factors as
\begin{displaymath}
\xymatrix{
Y\ar[r]^{\widetilde{(f\circ D)}\;\;\;} &\textrm{Im}(f)\ar[r]^{i} & Y
}
\end{displaymath}
where $i$ is the inclusion. But 
$$f\circ D\circ i =f\circ D\circ i-D\circ g\circ i=i$$
since $g\circ i=0$. It follows that $\widetilde{(f\circ D)}\circ i=\textrm{Id}_{\textrm{Im}(f)}$. This implies that the map $\textrm{Ker}(g)\cong\textrm{Im}(f)\rightarrow Y$ is split, and so is an admissible monic in the split exact structure.
\end{proof}

\begin{cor}\label{allgoodhomotopy}
Let $X_{\bullet}$ be a good complex in a weakly idempotent complete exact category $\mathpzc{E}$ .
\begin{enumerate}
\item
$X_{\bullet}$ is acyclic whenever there is a complex $Y_{\bullet}$, a morphism of complexes $f_{\bullet}:Y_{\bullet}\rightarrow X_{\bullet}$ which is homotopic to $0$, and such that the induced maps $\widetilde{f_{n}}:\textrm{Ker}(d_{n}^{Y})\rightarrow\textrm{Ker}(d_{n}^{X})$ are admissible epimorphisms.
\item
$X_{\bullet}$ is split exact whenever $id_{X_{\bullet}}$ is homotopic to $0$.
\end{enumerate}
\end{cor}

\begin{proof}
The first assertion follows from Proposition \ref{goodtrick} and Proposition \ref{homotopygood}. For the second assertion note that $X_{\bullet}$ is acyclic by the first assertion. In particular each
$$X_{n+1}\rightarrow X_{n}\rightarrow X_{n-1}$$
is acyclic, and $X_{n}\rightarrow X_{n-1}$ is (weakly left) admissible. Thus we may use Proposition \ref{splithomotopy}.
\end{proof}

\subsubsection{Quasi-isomorphisms}

Recall that in an abelian category a map of complexes induces a map on homology. The map is said to be a quasi-isomorphism if the induced map on homology is an isomorphism. Quasi-isomorphisms can also be characterised in terms of their mapping cone. A map of chain complexes in an abelian category is a quasi-isomorphism if and only if its mapping cone is acyclic. As remarked previously, in an exact category we cannot in general  define the homology of a complex. However the construction of the mapping cone makes sense in any additive category. By the previous remarks, the following definition is sensible.

\begin{defn}
Let $\mathpzc{E}$ be an exact category. A map $f_{\bullet}:X_{\bullet}\rightarrow Y_{\bullet}$ of complexes of $\mathpzc{E}$ is said to be a \textbf{quasi-isomorphism} if $\textrm{cone}(f_{\bullet})$ is acyclic.
\end{defn}

\begin{prop}
If $\mathpzc{E}$ is idempotent complete then homotopy equivalences are quasi-isomorphisms.
\end{prop}

\begin{proof}
See \cite{Buehler} Proposition 10.9.
\end{proof}

The next proposition is an immediate consequence of Corollary \ref{boundedtest}.

\begin{prop}\label{adquas}
Let $I:\mathpzc{E}\rightarrow\mathpzc{A}$ be a left abelianisation of an exact category $\mathpzc{E}$. Let $f_{\bullet}:X_{\bullet}\rightarrow Y_{\bullet}$ be a morphism of complexes. Suppose $\textrm{cone}(f)$ is good. Then $f$ is a quasi-isomorphim if and only if $I(f)$ is.
\end{prop}

\begin{rem}
As for abelian categories, one can define the derived category $D_{*}(\mathpzc{E})$ of an exact category $\mathpzc{E}$ by localizing $ Ch_{*}(\mathpzc{E})$ at the quasi-isomorphisms. For details see for example \cite{Buehler} Section 10.
\end{rem}

\subsection{Ext Groups and Projective Objects}
In order to study cotorsion pairs in exact categories in Section \ref{chmodelexact}, we will need the notion of Ext groups in exact categories. Recall for an abelian category $\mathpzc{A}$ one can define the groups $\textrm{Ext}^{n}(A,B)$ for any pair of objects $A,B\in\mathpzc{A}$ regardless of whether $\mathpzc{A}$ has enough projectives by the Yoneda construction. This construction goes through mutatis-mutandis for exact categories. The elements are Yoneda equivalences classes of $n$-extensions and the binary operation is the Baer sum. All the proofs for the above facts work as the abelian case. The interested reader can adapt the relevant proofs in \cite{buchsbaum} for example. The first Ext group $\textrm{Ext}^{1}(A,B)$ can also be computed by passing to a left abelianisation. More generally we have the following straightforward result.

\begin{prop}\label{extab}
Let $\mathpzc{E}$ and $\mathpzc{F}$ be exact categories. Let $F:\mathpzc{E}\rightarrow\mathpzc{F}$ be a fully faithful exact functor which reflects exactness. Suppose that the essential image of $\mathpzc{E}$ is closed under extensions. Then $F$ induces a natural isomorphism of abelian groups
$$\textrm{Ext}^{1}_{\mathpzc{E}}(-,-)\cong\textrm{Ext}^{1}_{\mathpzc{F}}(F(-),F(-))$$
\end{prop}

\begin{rem}
In the above we make the implicit assumption that each $\textrm{Ext}^{n}(A,B)$ is a set. This always holds for exact categories with enough projectives, which can be seen from the discussion in the following section.
\end{rem}

At this point we recall the notion of a projective object in an exact category, and mention how they relate to the $\textrm{Ext}$ functor.

\begin{defn}
An object $P$ in an exact category $\mathpzc{E}$ is said to be \textbf{projective} if the functor $\textrm{Hom}(P,-):\mathpzc{E}\rightarrow\mathpzc{Ab}$ is exact. 
\end{defn}

\begin{rem}
By Proposition \ref{admexact}, for any projective object $P$ the functor $\textrm{Hom}(P,-)$ is admissibly exact.
\end{rem}

\begin{example}
In the split exact structure every object is projective.
\end{example}

As in the abelian case one has the following result. The equivalence of the first three conditions can be found in \cite{Buehler} Proposition 11.3.

\begin{prop}\label{projequiv}
The following are equivalent.
\begin{enumerate}
\item
$P$ is projective.
\item
Given a map $f:P\rightarrow C$ and an admissible epic $e:B\rightarrow C$, there is a morphism $g:P\rightarrow B$ such that the following diagram commutes
\begin{displaymath}
\xymatrix{
& B\ar[d]^{e}\\
P\ar[ur]^{g}\ar[r]^{f} & C
}
\end{displaymath}
\item
Any admissible epic with codomain $P$ splits.
\item
$\textrm{Ext}^{1}(P,A)$ vanishes for any object $A$.
\item
$\textrm{Ext}^{n}(P,A)$ vanishes for any object $A$ and any $n\ge 1$.
\end{enumerate}
\end{prop}

\subsection{Exact Structures on Chain Complexes}
Let $\mathpzc{E}$ be an exact category and consider the category $Ch_{*}(\mathpzc{E})$ for $*\in\{\emptyset,b,\ge,\le,+,-\}$. Say that a sequence $0\rightarrow A_{\bullet}\rightarrow B_{\bullet}\rightarrow C_{\bullet}\rightarrow0$ is exact precisely if for each $i\in\Z$ the sequence $0\rightarrow A_{i}\rightarrow B_{i}\rightarrow C_{i}\rightarrow0$ is exact. Since limits and colimits in $Ch_{*}(\mathpzc{E})$ are computed degree-wise this is an exact structure on $Ch(\mathpzc{E})$.

\begin{prop}
Let $F:\mathcal{A}\rightarrow\mathcal{B}$ be a fully faithful exact functor which reflects exactness and whose essential image is closed under extensions. Then for $*\in\{\ge0,\le0,+,-,b,\emptyset\}$ the induced functor
$$ Ch_{*}(F): Ch_{*}(\mathcal{A})\rightarrow Ch_{*}(\mathcal{B})$$
is a fully faithful exact functor which reflects exactness and whose essential image is closed under extensions.
\end{prop}

\begin{proof}
Since exactness of chain complexes is defined level wise, $ Ch_{*}(F)$ is clearly exact and reflects exactness. It is clearly faithful. Let us check that it is full. Let $(X_{\bullet},d_{\bullet})$ and $(Y_{\bullet},\delta_{\bullet})$ be chain complexes in $\mathcal{A}$. Let $f_{\bullet}:F(X_{\bullet})\rightarrow F(Y_{\bullet})$ be a chain map. For each $n$ there is some $g_{n}:X_{n}\rightarrow Y_{n}$ with $f_{n}=F(g_{n})$. Moreover
$$F(g_{n}\circ d_{n+1})=F(g_{n})\circ F(d_{n+1})=f_{n}\circ F(d_{n+1})=F(\delta_{n+1})\circ f_{n+1}=F(\delta_{n+1}\circ g_{n+1})$$
Since $F$ is faithful, $g_{n}\circ d_{n+1}=\delta_{n+1}\circ g_{n+1}$. It remains to show that the essential image of $ Ch_{*}(F)$ is closed under extensions. So suppose we have an exact sequence of chain complexes.
\begin{displaymath}
\xymatrix{
0\ar[r] & F(X_{\bullet},d_{\bullet})\ar[r]^{f_{\bullet}} & (Q_{\bullet},\gamma_{\bullet})\ar[r]^{g_{\bullet}} & F(Y_{\bullet},\delta_{\bullet})\ar[r] & 0
}
\end{displaymath}
For each $n$ pick an object $P_{n}\in\mathcal{A}$ and an isomorphism $p_{n}: Q_{n}\widetilde{\rightarrow} F(P_{n})$. Let $\gamma_{n}'=p_{n-1}\circ\gamma_{n}\circ p_{n}^{-1}:F(P_{n})\rightarrow F(P_{n-1})$. Then $(P_{\bullet},\gamma_{\bullet}')$ is a chain complex. Moreover by construction we have an isomorphism $p_{\bullet}: Q_{\bullet}\rightarrow F(P_{\bullet})$ whose $n$th component is $p_{n}$.
\end{proof}

\begin{cor}\label{chainab}
Let $I:\mathpzc{E}\rightarrow\mathcal{A}(\mathpzc{E})$ is a left abelianisation of $\mathpzc{E}$. Then $ Ch_{*}(I): Ch_{*}(\mathpzc{E})\rightarrow Ch_{*}(\mathcal{A}(\mathpzc{E}))$ is a left abelianisation of $ Ch_{*}(\mathpzc{E})$. 
\end{cor}

\begin{proof}
By the previous proposition, it remains to check that $ Ch_{*}(I)$ preserves kernels, and $ Ch_{*}(I)(f_{\bullet})$ is an admissible epimorphism if and only if $f_{\bullet}$ is. However this is clear since everything is computed degree-wise.
\end{proof}

\subsection{A Useful Example: The Degree-Wise Exact Structure}
Let $\mathpzc{E}$ be an additive category, and endow it with the split exact structure. The induced exact structure on $Ch(\mathpzc{E})$ is called the \textbf{degree-wise split} exact structure, and we denote the Ext functors in this structure by $\textrm{Ext}^{n}_{dw}$. We conclude this section with a brief discussion of the relation between extensions in the degree-wise split exact structure and the $Ch(\textbf{Ab})$-enriched structure on $Ch(\mathpzc{E})$. This is also done in a model theoretic context for modules over a ring in \cite{gillespie} Section 5.2 and the proofs are formally the same.
\begin{prop}\label{splitexact}
A sequence of chain complexes
\xymatrix{
0\ar[r] & X_{\bullet}\ar[r]^{p_{\bullet}} & Z_{\bullet}\ar[r]^{q_{\bullet}} & Y_{\bullet}\ar[r] & 0
}
 is split exact in each degree if and only if it is isomorphic to a complex of the form
 $$0\rightarrow X_{\bullet}\rightarrow\textrm{cone}(f_{\bullet})\rightarrow Y_{\bullet}\rightarrow 0$$
 for some morphism of complexes $f_{\bullet}:Y_{\bullet}[1]\rightarrow X_{\bullet}$.
\end{prop}

\begin{proof}
The sequence
 $$0\rightarrow X_{\bullet}\rightarrow\textrm{cone}(f_{\bullet})\rightarrow Y_{\bullet}\rightarrow 0$$
 is clearly split exact in each degree, so any complex isomorphic to it is split exact in each degree as well. Suppose
\begin{displaymath}
\xymatrix{
0\ar[r] & X_{\bullet}\ar[r]^{p_{\bullet}} & Z_{\bullet}\ar[r]^{q_{\bullet}} & Y_{\bullet}\ar[r] & 0
}
\end{displaymath}
is split exact in each degree. Let $\alpha_{n}:Z_{n}\rightarrow X_{n}$ be such that $\alpha_{n}\circ p_{n}=\textrm{id}_{X_{n}}$ and $\beta_{n}:Y_{n}\rightarrow Z_{n}$ be a map such that $q_{n}\circ\beta_{n}=\textrm{id}_{Y_{n}}$. We may assume also that $\alpha_{n}\circ\beta_{n}=0$. Define $f_{\bullet}:Y_{\bullet}[1]\rightarrow X_{\bullet}$ by $f_{n}=\alpha_{n}\circ d_{n+1}^{Z}\circ \beta_{n+1}$. This is easily seen to be a map of chain complexes. Let $\alpha_{n}:Z_{n}\rightarrow X_{n}\oplus Y_{n}$ denote the isomorphism induced by the degree-wise splitting. A straightforward computation shows that this gives a map of chain complexes $\alpha_{\bullet}:Z_{\bullet}\rightarrow\textrm{cone}(f_{\bullet})$. Thus we get an isomorphism of exact  sequences.
\begin{displaymath}
\xymatrix{
0\ar[r] & X_{\bullet}\ar@{=}[d]\ar[r]^{p_{\bullet}} & Z_{\bullet}\ar[r]^{q_{\bullet}}\ar[d]^{\alpha_{\bullet}} & Y_{\bullet}\ar[r]\ar@{=}[d] & 0\\
0\ar[r] & X_{\bullet}\ar[r] & \textrm{cone}(f_{\bullet})\ar[r] & Y_{\bullet}\ar[r] & 0
}
\end{displaymath}
\end{proof}

\begin{defn}\label{contract}
\begin{enumerate}
\item
A complex $X_{\bullet}$ is said to be \textbf{contractible} if the map $X_{\bullet}\rightarrow 0$ is a homotopy equivalence.
\item
A complex $X_{\bullet}$ is said to be \textbf{split acyclic} if it is acyclic in the split exact structure.
\end{enumerate}
\end{defn}

\begin{prop}\label{cone0}
A map of chain complexes $f_{\bullet}:X_{\bullet}\rightarrow Y_{\bullet}$ is homotopic to $0$ if and only if the sequence
 $$0\rightarrow Y_{\bullet}\rightarrow\textrm{cone}(f_{\bullet})\rightarrow X_{\bullet}[-1]\rightarrow 0$$
 is split exact.
\end{prop}

\begin{proof}
Suppose that $f_{\bullet}$ is homotopic to $0$. Let $\{D_{n}:X_{n}\rightarrow Y_{n+1}\}$ be a homotopy. We then get a map $\alpha_{n}=(\textrm{id}_{X_{n-1}},D_{n-1}):  X_{n-1}\rightarrow\textrm{cone}(f)_{n}$. It is straightforward to check that this gives a chain map $\alpha_{\bullet}:X_{\bullet}[-1]\rightarrow\textrm{cone}(f_{\bullet})$. Moreover it obviously gives a splitting of $\textrm{cone}(f_{\bullet})\rightarrow X_{\bullet}[-1]$. Conversely suppose the sequence is split exact. Let $\alpha_{\bullet}: X_{\bullet}[-1]\rightarrow\textrm{cone}(f_{\bullet})$ be a splitting of the map $\textrm{cone}(f_{\bullet})\rightarrow X_{\bullet}[-1]$. It is an easy computation to check that the collection of compositions $\{D_{n-1}:X_{n-1}\rightarrow\textrm{cone}(f_{\bullet})_{n}\rightarrow Y_{n}\}$ is a homotopy between $f$ and $0$.  
\end{proof}

We recover the following standard result (see e.g. \cite{gillespie2016derived} Lemma 2.8).

\begin{cor}\label{dw}
For chain complexes $X_{\bullet},Y_{\bullet}$ in an additive category $\mathpzc{E}$. we have 
$$\textrm{Ext}^{1}_{dw}(X,Y[n-1])\cong H_{n}\textbf{Hom}(X_{\bullet},Y_{\bullet})=\textrm{Hom}_{ Ch(\mathpzc{E})}(X,Y[n])\big\slash\sim$$
where $\sim$ is chain homotopy.
\end{cor}

\begin{proof}
By direct computation, one finds that  $f\in\prod_{i}\textrm{Hom}(X_{i},Y_{i+n})$ defines a chain map $f_{\bullet}:X_{\bullet}\rightarrow Y_{\bullet}[n]$ if and only if $f\in\textrm{Ker}(d_{n})$. Similarly, $f_{\bullet}$ is then null-homotopic if and only if it is in $\textrm{Im}(d_{n+1})$. This gives the isomorphism
$$H_{n}\textbf{Hom}(X_{\bullet},Y_{\bullet})=\textrm{Hom}_{ Ch(\mathpzc{E})}(X,Y[n])\big\slash\sim$$
The isomorphism $\textrm{Ext}^{1}_{dw}(X,Y[n-1])\cong \textrm{Hom}_{ Ch(\mathpzc{E})}(X,Y[n])\big\slash\sim$ follows from Proposition \ref{splitexact} and Proposition \ref{cone0}. 
\end{proof}

\section{Resolutions in Exact Categories}

We will need some results about resolutions in exact categories later.

\subsection{Bounded Resolutions}

We begin by discussing the easier case of bounded resolutions.

\begin{defn}
An exact category $\mathpzc{E}$ is said to \textbf{have enough projectives} if for any object $X$ of $\mathpzc{E}$,there is a projective object $P$ and an admissible epimorphism $P\twoheadrightarrow X$.
\end{defn}

\begin{lem}\label{enoughres}
Let $\mathcal{P}$ be a subclass of $\textbf{Ob}(\mathpzc{E})$, the objects of $\mathpzc{E}$. Assume that for any object $E$ of $\mathpzc{E}$ there is an object $P\in\mathcal{P}$ and an admissible epimorphism $P\twoheadrightarrow E$. Then, for any bounded below complex $E$ of $ Ch_{+}(\mathpzc{E})$, there is a bounded below complex $P$ whose entries are objects of $\mathcal{P}$, and a quasi-isomorphism
$$u:P\rightarrow E$$
where each $u_{k}:P_{k}\rightarrow E_{k}$ is an admissible epimorphism. Moreover, this construction can be made functorial if the choice of admissible epimorphism $P\twoheadrightarrow E$ can be made functorial.
\end{lem}

\begin{proof}
This is proved in \cite{Buehler} Theorem 12.7 for the case that $\mathcal{P}$ is the class of projectives in an exact category with enough projectives. However the proof goes through the same.
\end{proof}

\begin{lem}\label{projextend}
Let $A,B$ be objects in an exact category $\mathpzc{E}$. Let $f:A\rightarrow B$ be a morphism. Let $P_{\bullet}$ be a complex with $P_{-1}=A, P_{n}=0$ for $n<-1$ and $P_{n}$ projective for $n>0$. Also let $Q_{\bullet}$ be an acyclic complex with $Q_{-1}=B$ and $Q_{n}=0$ for $n<-1$.  Then there is a chain map $f_{\bullet}:P_{\bullet}\rightarrow Q_{\bullet}$ with $f_{-1}=f$. Moreover, $f_{\bullet}$ is unique up to homotopy.
\end{lem}

\begin{proof}
See \cite{Buehler} Theorem 12.4.
\end{proof}

As in the abelian case one can define derived functors between derived categories of exact categories. There are also notions of adapted classes for functors. Proposition \ref{projequiv} and Lemma \ref{enoughres} essentially say that as in the abelian case, if a category $\mathpzc{E}$ has enough projectives, then the class of projective objects is adapted to the functor $\textrm{Hom}(-,A):\mathpzc{E}^{op}\rightarrow\mathpzc{Ab}$. It can be shown that $R^{n}\textrm{Hom}(-,A)\defeq H_{n}(R\textrm{Hom}(-,A))\cong\textrm{Ext}^{n}(-,A)$.

\subsection{Unbounded Resolutions}
When dealing with the model structures on unbounded chain complexes., we will also need to have unbounded resolutions. For this we will modify the famous Theorem 3.4 in \cite{spaltenstein} and its proof to work for more general exact categories. In the following we shall let $\mathcal{B}$ be a class of complexes in $\mathpzc{E}$ which is stable under shifts, and we shall assume that for any bounded below complex $X_{\bullet}$ there is a bounded below complex $B_{\bullet}$ in $\mathcal{B}$ and a quasi-isomorphism $B_{\bullet}\rightarrow X_{\bullet}$ which is an admissible epimorphism in each degree. We will call such a class a \textbf{bounded resolving class}. 

 Let us now recall some notions from Spaltenstein's paper.

\begin{defn}
Let $\mathcal{B}$ be a class of complexes. A direct system $(P_{\bullet}^{n})_{n\in E}$ in $ Ch(\mathpzc{E})$ is a $\mathcal{B}$-\textbf{special direct system} if it satisfies the following conditions.
\begin{enumerate}
\item
$E$ is well-ordered.
\item
If $n\in E$ has no predecessor then $P_{\bullet}^{n}=\textrm{lim}_{\rightarrow_{m<n}} P^{m}_{\bullet}$.
\item
If $n\in E$ has a predecessor $n-1$ then the natural chain map $P^{n-1}_{\bullet}\rightarrow P^{n}_{\bullet}$ is injective, its cokernel $C^{n}_{\bullet}$ belongs to $\mathcal{B}$, and the short exact sequence
$$0\rightarrow P^{n-1}_{\bullet}\rightarrow P^{n}_{\bullet}\rightarrow C^{n}_{\bullet}\rightarrow 0$$
is split exact in each degree.
\end{enumerate}
We denote by $\textrm{lim}_{\rightarrow}\mathcal{B}$ the class of complexes which are limits of $\mathcal{B}$-special direct systems.
\end{defn}

 \begin{prop}\label{specialinductive}
Let $\mathpzc{E}$ be an exact category which has kernels. Suppose that $\mathcal{B}$ is a bounded resolving class. Then for any complex $X_{\bullet}$ there exists a $\mathcal{B}$-special direct system $(P_{\bullet}^{n})_{n\ge-1}$ and a direct system of chain maps $f^{n}:P^{n}_{\bullet}\rightarrow\tau_{\ge n} X_{\bullet}$ such that 
\begin{enumerate}
\item
$f^{n}$ is a quasi-isomorphism for every $n\ge0$.
\item
$f^{n}$ is an admissible epimorphism in each degree.
\end{enumerate}
\end{prop}

\begin{proof}
We construct the data $(P^{n}_{\bullet})_{n\ge-1}$ and $(f^{n})_{n\ge-1}$ by induction. For $n=-1$ we take $P^{-1}_{\bullet}=0$ and so $f^{-1}=0$. Let now $n\ge1$, and suppose that $P^{-1}_{\bullet},\ldots, P^{n-1}_{\bullet}$ and $f^{-1},\ldots,f^{n-1}$ have been constructed. Let $P_{\bullet}= P_{\bullet}^{n-1}$ and $Y_{\bullet}=\tau_{\ge n}X_{\bullet}$. Denote by $f$ the composite $P^{n-1}_{\bullet}\rightarrow\tau_{\ge n-1}X_{\bullet}\rightarrow Y_{\bullet}$. By assumption we can find a quasi-isomorphism $g:Q_{\bullet}\rightarrow\textrm{cone}(f)[1]$ which is an admissible epimorphism in each degree, and $Q_{\bullet}[-1]\in\mathcal{B}$. Now we have a degree-wise splitting, $\textrm{cone}(f)[1]=P_{\bullet}\oplus Y_{\bullet}[1]$. We therefore get two maps $g':Q_{\bullet}\rightarrow P_{\bullet}$ and $g'':Q_{\bullet}\rightarrow Y_{\bullet}[1]$ which are admissible epimorphisms in each degree, and such that $g'$ is a chain map. Define $P^{n}\defeq\textrm{cone}(-g')$ and let $f^{n}:\textrm{cone}(-g')=Q[1]\oplus P\rightarrow Y$ be defined by $f^{n}=g''[1]+f$. As in \cite{spaltenstein} Lemma 3.3, by direct calculation $f^{n}$ is a chain map and $\textrm{cone}(f^{n})=\textrm{cone}(g)[1]$. Since $g$ is a quasi-isomorphism $f^{n}$ is as well. Moreover the sequence
$$0\rightarrow P^{n-1}_{\bullet}\rightarrow P_{\bullet}^{n}\rightarrow Q_{\bullet}[1]\rightarrow0$$
is split exact in each degree.
\end{proof}

Following the work of  \cite{chacholski2017relative} in the relative homological algebra setting, we will give a general condition under which unbounded resolutions exist.
\begin{defn}
Let $\mathpzc{E}$ be an exact category with kernels and let $k\in\mathbb{Z}$. A complex $X$ is said to be \textbf{homologically concentrated in degrees }$\le k$ if $\tau_{\ge k}X$ is acyclic.
\end{defn}
The following definition is an exact category version \cite{chacholski2017relative} Definition 6.1.
\begin{defn}
Let $\mathpzc{E}$ be an exact category with kernels and $k\in\mathbb{Z}$. A bounded resolving class $\mathcal{B}$ is said to satisfy condition \textbf{AB4-k} if whenever $\{B_{n}\}_{n\in\mathbb{N}}$ is a collection of objects of $\mathcal{B}$ concentrated in degree $\le 0$, then $\bigoplus_{n\in\mathbb{N}}B_{n}$ exists and is homologically concentrated in degrees $\le -k+1$.
\end{defn}

\begin{lem}
Let $\mathpzc{E}$ be an exact category with kernels and countable filtered colimits, and let $\mathcal{B}$ be a bounded resolving class satisfying condition $AB4-k$. Let
$$K_{0}\rightarrow K_{1}\rightarrow\ldots\rightarrow K_{m}\rightarrow K_{m+1}\rightarrow\ldots$$
be a sequence with each $K_{m}\rightarrow K_{m+1}$ an admissible monomorphism, and each $K_{m}$ is a complex in $\mathcal{B}$ which is homologically concentrated in degrees $\le n$. Then $\textrm{lim}_{\rightarrow_{m}}K_{m}$ is homologically concentrated in degrees $\le n-k+1$.
\end{lem}
\begin{proof}
The proof is an adaptation of \cite{chacholski2017relative} Proposition 6.3. The sequence
$$0\rightarrow\coprod_{n}K_{n}\rightarrow\coprod_{n}K_{n}\rightarrow\textrm{colim}_{n}K_{n}\rightarrow0$$
is degree-wise split. Indeed this can be proven by observing that the sequence 
$$0\rightarrow\textrm{lim}_{n}Hom(K_{n},E)\rightarrow Hom(\prod_{n}K_{n},E)\rightarrow Hom(\prod_{n}K_{n},E)\rightarrow0$$
for any $E\in\mathpzc{E}$ is a degree-wise split exact sequence of complexes of abelian groups. By passing to an abelianisation and using the long exact sequence, on homology, it suffices to prove that $\coprod_{n}K_{n}$ is concentrated in degrees $\le n-k+1$. By shifting we may assume that $n=0$, and then we are done.
\end{proof}

\begin{cor}\label{Bres}
Let $\mathpzc{E}$ be an exact category with kernels. Let $\mathcal{B}$ be a bounded resolving class satisfying condition $AB4-k$ for some $k$. Then any chain complex $X_{\bullet}$ in $\mathpzc{E}$ admits a $\textrm{lim}_{\rightarrow}\mathcal{B}$ resolution which is an admissible epimorphism in each degree.
\end{cor}

\begin{proof}
Fix a $\mathcal{B}$-special direct system $(P_{\bullet}^{n})_{n\ge-1}$ and a direct system of chain maps $f^{n}:P^{n}_{\bullet}\rightarrow\tau_{\ge n} X_{\bullet}$ such that 
\begin{enumerate}
\item
$f^{n}$ is a quasi-isomorphism for every $n\ge0$.
\item
$f^{n}$ is an admissible epimorphism in each degree.
\end{enumerate}
Let $P_{\bullet}$ be the direct limit of the special direct system. For each $n$ the composition $P_{\bullet}^{n}\rightarrow P_{\bullet}\rightarrow X_{\bullet}$ is an admissible epimorphism in degrees $>n$. Thus $P_{\bullet}\rightarrow X_{\bullet}$ is an admissible epimorphism in all degrees. It remains to prove that $P_{\bullet}\rightarrow X_{\bullet}$ is a weak equivalence. This can be proven as in \cite{chacholski2017relative} Theorem 6.4, mutatis mutandis.
\end{proof}

Now let $\mathcal{P}$ be any class of objects in $\mathpzc{E}$. Suppose that for each object $X$ in $\mathpzc{E}$ there is an object $P$ in $\mathcal{P}$ together with an admissible epimorphism $P\twoheadrightarrow X$. By Lemma \ref{enoughres} the class $ Ch_{+}(\mathcal{P})$ of chain complexes with entries in $\mathcal{P}$ is a bounded resolving class.   Let us introduce the following notion.

\begin{defn}
Let $\mathpzc{E}$ be an exact category. We say that a class of objects $\mathcal{P}$ in $\mathpzc{E}$ satisfies condition $AB4-k$ for some $k$ if $Ch_{+}(\mathcal{P})$ satisfies condition $AB4-k$.
\end{defn} 

From the proof of Corollary \ref{Bres} we then immediately have the following.

\begin{cor}\label{Kproj}
Let $\mathpzc{E}$ be a weakly idempotent complete exact category with kernels. Let $\mathcal{P}$ be a class of objects such that for each object $X$ in $\mathpzc{E}$ there is an object $P$ in $\mathcal{P}$ together with an admissible epimorphism $P\twoheadrightarrow X$. Suppose further that $\mathcal{P}$  satisfies condition $AB-k$ for some $k\in\mathbb{Z}$. Then for any complex $X_{\bullet}$ in $ Ch(\mathpzc{E})$ there is a complex $P_{\bullet}$ in $ Ch(\mathcal{P})$ and an admissible epimorphism $P_{\bullet}\rightarrow X_{\bullet}$ which is a quasi-isomorphism. Moreover, $X_{\bullet}$ is the limit of a $ Ch_{+}(\mathcal{P})$-special direct system.
\end{cor}

For a class of objects $\mathcal{P}$ in an exact category $\mathpzc{E}$, let $\textbf{AdMon}_{\mathcal{P}}$ denote the class of admissible monomorphisms whose cokernel is in $\mathcal{P}$.

We will also need the following acyclicity result, also proved in \cite{spaltenstein} for abelian categories.

\begin{prop}\label{klift}
Let $\mathcal{T}$ be a class of complexes in $ Ch(\mathpzc{E})$. The class of all complexes $A_{\bullet}\in  Ch(\mathpzc{E})$ such that $\textbf{Hom}(A_{\bullet},T_{\bullet})$ is acyclic for every $T_{\bullet}$ in $\mathcal{T}$ is closed under special direct limits.
\end{prop}

\begin{proof}
It is clear from the definition of the contravariant functor $\textbf{Hom}(-,T_{\bullet})$ that it transforms colimits into limits. If $(P^{n}_{\bullet})_{n\in E}$ is a $\mathcal{B}$-special direct system then
$(\textbf{Hom}(P^{n}_{\bullet},T_{\bullet}))_{n\in E}$ is a $\mathcal{B}$-special inverse system of acyclic complexes of abelian groups, where we use the terminology of \cite{spaltenstein} Section 2. Lemma 2.3 in \cite{spaltenstein} says that the inverse limit of such a system is again acyclic.
\end{proof}

\section{Monoidal Exact Categories and Monads in Exact Categories}

We conclude this section with a brief note on monoidal exact categories, and exact structures on categories of modules over monoids. More generally we put an exact structure on the category of algebras for an additive monad which is compatible in a precise sense with the exact structure on the underlying category. First we need a general definition.

\begin{defn}\label{monoidaldef}
\begin{enumerate}
\item
A \textbf{monoidal exact category} is a symmetric monoidal category $(\mathpzc{E},\otimes,k)$ where $\mathpzc{E}$ is an exact category, and the monoidal functor $\otimes:\mathpzc{E}\times\mathpzc{E}\rightarrow\mathpzc{E}$ is cocontinous in each variable. $\mathpzc{E}$ is said to be \textbf{compatibly monoidal} if in addition $X\otimes(-)$ is right exact for any $X\in\mathpzc{E}$.
\item
A \textbf{closed monoidal exact category} is a closed symmetric monoidal category $(\mathpzc{E},\otimes,k,\underline{Hom})$ with $\mathpzc{E}$ an exact category. $\mathpzc{E}$ is said to be \textbf{compatibly closed monoidal} if in addition $X\otimes(-)$ is right exact, and $\underline{Hom}(X,-),\underline{Hom}(-,X)$ are left exact for any $X\in\mathpzc{E}$
\end{enumerate}
\end{defn}

Note that if $(\mathpzc{E},\otimes,\underline{\textrm{Hom}},k)$ is a closed monoidal exact category, then $(\mathpzc{E},\otimes,k)$ is automatically a monoidal exact category. Indeed for each object $X$, $X\otimes(-)$ is a left adjoint so it preserves colimits.

\begin{defn}
Let $(\mathpzc{E},\otimes,k)$ be an exact category equipped with a symmetric monoidal structure where the tensor functor is additive. An object $F$ of $\mathpzc{E}$ is said to be \textbf{(strongly) flat} if the functor $F\otimes(-)$ is (strongly) exact.
\end{defn}

In the familiar category of $R$-modules over some ring $R$ with the usual monoidal structure, projectives are always flat. Moreover the tensor product of two projective $R$-modules is again projective. This is not always guaranteed for an arbitrary monoidal exact category. However it is a useful property to have, in particular when dealing with the projective model structure later. We therefore make a definition.

\begin{defn}
A monoidal exact category in which projective objects are flat and $P\otimes P'$ is projective whenever both $P$ and $P'$ are is said to be \textbf{projectively monoidal} . A projectively monoidal exact category is said to be \textbf{strongly projectively monoidal} if projectives are strongly flat.
\end{defn}

In closed exact categories we have the following observation.

\begin{obs}\label{internalproj}
Let $(\mathpzc{E},\otimes,k,\textrm{\underline{Hom}})$ be a closed monoidal exact category with enough projectives such that the underlying monoidal category is projectively monoidal. Then for any projective $P$, the functor $\underline{\textrm{Hom}}(P,-):\mathpzc{E}\rightarrow\mathpzc{E}$ is exact. The proof follows immediately from the adjunction between $\otimes$ and $\underline{\textrm{Hom}}$. It is shown in the quasi-abelian case in \cite{qacs}, for example, and the proof works identically in the exact case.
\end{obs}

Now let $R$ be a unital associative monoid internal to a monoidal exact category $(\mathpzc{E},\otimes,k)$. It turns out that there is an exact structure on the additive category ${}_{R}\mathpzc{Mod}$ where a null sequence
$$0\rightarrow X\rightarrow Y\rightarrow Z\rightarrow 0$$
in ${}_{R}\mathpzc{Mod}$ is exact if and only if it is a short exact sequence when regarded as a null sequence in $\mathpzc{E}$. This follows from a more general result about cocontinuous monads in exact categories.

\begin{prop}\label{monadexact}
Let $\mathpzc{E}$ be an exact category and let $T:\mathpzc{E}\rightarrow\mathpzc{E}$ be a cocontinuous monad. There is an exact structure on $\mathpzc{E}^{T}$ where a null sequence
$$0\rightarrow X\rightarrow Y\rightarrow Z\rightarrow 0$$
in $\mathpzc{E}^{T}$ is exact if and only if it is a short exact sequence when regarded as a null sequence in $\mathpzc{E}$. We call this exact structure the \textbf{induced exact structure}.
\end{prop}

\begin{proof}
This follows from the general fact that if $T$ is a cocontinuous monad on any category $\mathpzc{E}$ then the forgetful functor $\mathpzc{E}^{T}\rightarrow\mathpzc{E}$ creates limits and colimits and reflects isomorphisms. For a proof of this see \cite{handbook} Proposition 4.3.1 and Proposition 4.3.2.
\end{proof}

This exact structure inherits a lot from the exact structure on $\mathpzc{E}$. In fact we have the following lemma.

\begin{lem}\label{reflectlots}
Let $|-|:\mathpzc{D}\rightarrow\mathpzc{E}$ be a functor between exact categories which reflects exactness and creates both kernels and cokernels. Then $|-|$ reflects admissible monomorphisms, admissible epimorphisms, weakly admissible morphisms, admissible morphisms, and admissibly acyclic sequences.
\end{lem}

\begin{proof}
Let $f:X\rightarrow Y$ be a morphism in $\mathpzc{D}$. Supposethat $|f|$ is an admissible monomorphism. Then there is an exact sequence
$$0\rightarrow |X|\rightarrow |Y|\rightarrow\textrm{Coker}(|f|)\rightarrow0$$
Since $|-|$ creates cokernels and reflects exactness
$$0\rightarrow X\rightarrow Y\rightarrow\textrm{Coker}(f)\rightarrow 0$$
is an exact sequence in $\mathpzc{D}$. Thus $f$ is an admissible monomorphism. That $|-|$ reflects admissible epimorphisms is proved similarly. Note in particular that this means $|-|$ reflects isomorphisms.

Suppose now that $|f|:|X|\rightarrow|Y|$ is weakly admissible. Then there is a decomposition
\begin{displaymath}
\xymatrix{
 & |X|\ar@{->>}[dr]\ar[rrr]^{|f|} && & |Y|\ar@{->>}[dr] &\\
 \textrm{Ker} (|f|)\ar@{>->}[ur] & & \textrm{Coim}(|f|)\ar[r]^{\hat{|f|}} & \textrm{Im}(|f|)\ar@{>->}[ur] & & \textrm{Coker} (|f|)
}
\end{displaymath}
Since $|-|$ reflects exactness and creates both kernels and cokernels there is a decomposition in $\mathpzc{D}$
\begin{displaymath}
\xymatrix{
 & X\ar@{->>}[dr]\ar[rrr]^{f} && & Y\ar@{->>}[dr] &\\
 \textrm{Ker} (f)\ar@{>->}[ur] & & \textrm{Coim}(f)\ar[r]^{\hat{f}} & \textrm{Im}(f)\ar@{>->}[ur] & & \textrm{Coker} (f)
}
\end{displaymath}
Thus $f$ is weakly admissible. If in addition $|f|$ is admissible then $|\hat{f}|$ is an isomorphism. Since $|-|$ reflects isomorphisms $\hat{f}$ is an isomorphism, so $f$ is admissible.

Finally suppose
\begin{displaymath}
\xymatrix{
|X|\ar[r]^{|f|}& |Y|\ar[r]^{|g|} & |Z|\\
}
\end{displaymath}
is admissibly acyclic. Then $|f|$ is admissible, $|g|$ has a kernel and the map $\textrm{Im}(|f|)\rightarrow\textrm{Ker}(|g|)$ is an isomorphism. By the above $f$ is admissible. Since $|-|$ creates kernels and cokernels and also reflects isomorphisms $\textrm{Im}(f)\rightarrow\textrm{Ker}(g)$ is also an isomorphism.
\end{proof}

As a consequence of this and Remark \ref{quasweak} later, if $\mathpzc{E}$ is (quasi)-abelian, then so is $\mathpzc{E}^{T}$. In particular categories of modules for monoid objects in monoidal (quasi)-abelian categories are themselves (quasi)-abelian.

Before concluding this discussion of monoidal exact categories, let us briefly mention induced monoidal structures on chain complexes. So, let $(\mathpzc{E},\otimes,k)$ be a monoidal exact category. Recall from Section \ref{notation} there is an induced monoidal exact structure $( Ch_{*}(\mathpzc{E}),\otimes,S^{0}(k))$ on $ Ch_{*}(\mathpzc{E})$ for $*\in\{\ge0,\le0,+,-,b,\emptyset\}$. Since colimits of chain complexes are computed degree-wise, finite direct sums are strongly exact, and a null sequence of chain complexes is admissibly coacyclic if and only if it is so in each degree, it is clear that this monoidal structure is compatible if the one on $\mathpzc{E}$ is, so that $( Ch_{*}(\mathpzc{E}),\otimes,S^{0}(k))$ is a compatible monoidal exact category for $*\in\{\ge0,\le0,+,-,b\}$. 

Now suppose $(\mathpzc{E},\otimes,S^{0}(k),\underline{\textrm{Hom}})$ is a closed monoidal exact category. Then

$$( Ch_{b}(\mathpzc{E}),\otimes,S^{0}(k),\underline{\textrm{Hom}})$$
is a closed monoidal exact category which is compatible if the closed monoidal exact structure on $\mathpzc{E}$ is. Note that the closed symmetric monoidal exact category

$$( Ch(\mathpzc{E}),\otimes,S^{0}(k),\underline{\textrm{Hom}})$$
need not be a compatible closed monoidal exact category since infinite direct sums/ products need not be admissibly coexact/ admissibly exact. When we deal with unbounded complexes later we shall assume this to be the case. However it will still be a monoidal exact category, and we shall see shortly that compatibility is guaranteed for a closed monoidal structure on a quasi-abelian category.

\subsection{The $\otimes$-Pure Exact Structure}
In this section we discuss pure exact structures on exact categories relative to tensor products. This gives a method for modifying an exact structure in order to make a certain class of objects flat. This is done in \cite{estrada2017pure} Section 3 for the case of the $G$-exact structure on a Grothendieck abelian category. 

\begin{defn}
Let $\mathpzc{E}$ be a monoidal exact category, and $\mathcal{S}$ a class of objects in $\mathpzc{E}$. A short exact sequence in $\mathpzc{E}$ is said to be $\mathcal{S}$-\textbf{pure} if it remains exact after tensoring with any object of $\mathpzc{E}$. An $Ob(\mathpzc{E})$-pure exact sequence is said to be $\otimes$-\textbf{pure}

\end{defn}

\begin{lem}\label{flatcofib}
Suppose $\mathpzc{E}$ is a symmetric monoidal weakly idempotent complete exact category with enough flat objects. If $C\in\mathpzc{E}$ is flat then every short exact sequence
$$0\rightarrow A\rightarrow B\rightarrow C\rightarrow 0$$
is pure.
\end{lem}
\begin{proof}
Suppose $Z$ is arbitrary and let 
$$0\rightarrow X\rightarrow Y\rightarrow Z\rightarrow0$$
be a short exact sequence with $Y$ flat. Let $J:\mathpzc{E}\rightarrow\mathpzc{A}$ be a right abelianisation. Since the functor $J(-\otimes-)$ commutes with cokernels in each variable, we have a diagram
\begin{displaymath}
\xymatrix{
J(A\otimes X)\ar[r]\ar[d] & J(A\otimes Y)\ar[r]\ar[d] & J(A\otimes Z)\ar[d]\ar[r] & 0\\
J(B\otimes X)\ar[r]\ar[d] & J(B\otimes Y)\ar[r]\ar[d] & J(B\otimes Z)\ar[d]\ar[r] & 0\\
J(C\otimes X)\ar[r]\ar[d] & J(C\otimes Y)\ar[r]\ar[d] & J(C\otimes Z)\ar[r]\ar[d] & 0\\
0 & 0 & 0
}
\end{displaymath}
with coacyclic rows and columns. The bottom row is short exact since $C$ is flat. Since $Y$ is flat the middle column is short exact. Then the argument becomes a simple diagram chase.
\end{proof}

\begin{prop}\label{pushoutpure}
The class of $\mathcal{S}$-pure exact sequences defines an exact structure on $\mathpzc{E}$ for $\mathpzc{E}$ weakly idempotent complete. 
\end{prop}
\begin{proof}
Clearly split exact sequences are $\mathcal{S}$-pure, and isomorphisms are both $\mathcal{S}$-pure monomorphisms and $\mathcal{S}$-pure epimorphisms. 
Next we show show that $\mathcal{S}$-pure monomorphisms are stable under pushout.
Let $i:A\rightarrow B$ be a $\mathcal{S}$-pure monomorphism. Consider a pushout diagram
\begin{displaymath}
\xymatrix{
A\ar[d]\ar[r] & B\ar[d]\\
X\ar[r] & Y
}
\end{displaymath}
Since tensoring with $Z$ preserves colimits,
\begin{displaymath}
\xymatrix{
A\otimes Z\ar[d]\ar[r] & B\otimes Z\ar[d]\\
X\otimes Z\ar[r] & Y\otimes Z
}
\end{displaymath}
is a pushout. But by assumption $A\otimes Z\rightarrow B\otimes Z$ is an admissible monic. Hence $X\otimes Z\rightarrow Y\otimes Z$ is also an admissible monic.
Now let 
\begin{displaymath}
\xymatrix{
0\ar[r] & X\ar[r]\ar[d] & P\ar[r]\ar[d] & Z'\ar[d]\ar[r] &0\\
0\ar[r] & X\ar[r] & Y\ar[r]^{g} & Z\ar[r] & 0
}
\end{displaymath}
be a commutative diagram where, $g$ is a pure epimorphism, both rows are exact, and the right-hand square is a pullback. We need to show that the top row is pure-exact. Let $C$ be an object of $\mathcal{S}$. The composite map $X\otimes C\rightarrow P\otimes C\rightarrow Y\otimes C$ is an admissible monomorphism, so $X\otimes C\rightarrow P\otimes C$ is an admissible monomorphism. The cokernel is $P\otimes C\rightarrow Z'\otimes C$. In particular the sequence
$$0\rightarrow X\otimes C\rightarrow P\otimes C\rightarrow Z'\otimes C\rightarrow 0$$
is exact, and $P\rightarrow Z'$ is a pure epimorphism. 
It is clear that the composition of two pure monomorphisms is a pure monomorphism. Let us show that the composition of two pure epimorphisms is a pure epimorphism. Let $f:X\rightarrow Y$ and $g:Y\rightarrow Z$ be pure epimorphisms, and $C$ an object of $\mathcal{S}$. We have that $Ker(f)\otimes C$ is the kernel of the admissible epimorphism $X\otimes C\rightarrow Y\otimes C$. Therefore $Ker(f)\otimes C\rightarrow X\otimes C$ is an admissible monomorphism. This map factors through $Ker(f)\otimes C\rightarrow Ker(g\circ f)\otimes C$ and is therefore an admissible monomorphism. It follows that 
$$0\rightarrow Ker(f)\otimes C\rightarrow Ker(g\circ f)\otimes C\rightarrow  Ker(g)\otimes C\rightarrow0$$
is a short exact sequence. Consider the diagram
\begin{displaymath}
\xymatrix{
& 0\ar[d] & 0\ar[d] & 0\ar[d] &\\
0\ar[r] & Ker(f)\otimes C\ar[d]\ar[r] & Ker(f)\otimes C\ar[d]\ar[r] & 0\ar[r]\ar[d] & 0\\
0\ar[r] & Ker(g\circ f)\otimes C\ar[r]\ar[d] & X\otimes C\ar[d]\ar[r] & Z\otimes C\ar[d]\ar[r] & 0\\
0\ar[r] & Ker(g)\otimes C\ar[d]\ar[r] & Y\otimes C\ar[d]\ar[r] & Z\otimes C\ar[d]\ar[r] & 0\\
& 0 &0 & 0 & 
}
\end{displaymath}
The top row is obviously exact, and the bottom is exact since $g$ is a pure epimorphism. We have shown that the left-hand column is exact, and the right is obviously exact. The middle is exact since $f$ is a pure epimorphism. Therefore the middle row is exact. 
\end{proof}
\begin{notation}
If $(\mathpzc{E},\mathcal{Q})$ is an exact category equipped with a closed symmetric monoidal structure, we will denote the $\otimes$-pure exact structure by $(\mathpzc{E},\mathcal{Q}_{\otimes})$.
\end{notation}
\begin{rem}
In the $\otimes$-pure exact structure all objects are flat. 
\end{rem}

\section{Quasi-Abelian Categories}

Let us apply what we have seen so far to the particular case of quasi-abelian categories. The theory of quasi-abelian categories is developed significantly in \cite{qacs} which is our main reference here. Applications to categories of topological vector spaces can be found in \cite{dcfapp}.

\subsection{Strict Morphisms}
First we explain as in \cite{qacs} Remark 1.1.11  that Definition \ref{quas} is equivalent to the one given in \cite{qacs}. Recall that in a finitely complete and cocomplete additive category, any morphism $f:E\rightarrow F$ gives rise to a commutative diagram

\begin{displaymath}
\xymatrix{
E\ar[r]^{f}\ar[d] & F\\
\textrm{Coim}(f)\ar[r] & \textrm{Im} f\ar[u]
}
\end{displaymath}

In any abelian category the map $\textrm{Coim}(f)\rightarrow\textrm{Im}(f)$ is an isomorphism. However this is not true in general. For example, consider the standard example of the category $\mathpzc{Fr}$ of Fr\'{e}chet spaces. Then $\textrm{Coim}(f)=E\big\slash f^{-1}(0)$, $\textrm{Im}(f)=\overline{f(E)}$ and the natural map  $E\big\slash f^{-1}(0)\rightarrow\overline{f(E)}$ is the obvious one. By the Open Mapping Theorem $\textrm{Coim}(f)\rightarrow\textrm{Im}(f)$ is an isomorphism if and only if $f$ has closed range, which is not always the case.

\begin{defn}
Let $\mathpzc{E}$ be an additive category with all kernels and cokernels. A morphism $f:E\rightarrow F$ in $\mathpzc{E}$ is said to be \textbf{strict} if $\textrm{Coim}(f)\rightarrow\textrm{Im}(f)$ is an isomorphism.
\end{defn}

\begin{prop}[\cite{qacs} Remark 1.1.2]\label{strictmonoepic}
Let $\mathpzc{E}$ be a finitely complete and cocomplete additive category. 
\begin{enumerate}
\item
A monic is strict if and only if it is the kernel of some morphism. In this case it is the kernel of its cokernel.
\item
An epic is strict if and only if it is the cokernel some morphism. In this case it is the cokernel of its kernel.
\end{enumerate}
\end{prop}

\begin{proof}
\begin{enumerate}
\item
Let $f:E\rightarrow F$ and write $i_{f}:\textrm{Ker}(f)\rightarrow E$ for the canonical map. Let us show that $i_{f}$ is strict. First note that for any monic $A\rightarrow B$, the coimage is $\textrm{id}:A\rightarrow A$. Let us compute the image of $i_{f}$. It is given by
$$\textrm{Ker}(\textrm{Coker}(\textrm{Ker}(f)\rightarrow E)\rightarrow E$$
By some abstract nonsense this is just $\textrm{Ker}(f)\rightarrow E$. Conversely suppose $m:X\rightarrow E$ is a strict monic. Then the maps $E\rightarrow\textrm{Coim}(m)\rightarrow\textrm{Im}(m)$ are all isomorphisms, i.e. we get a commutative diagram
\begin{displaymath}
\xymatrix{
X\ar[d]^{\sim}\ar[r]^{m} & E\\
\textrm{Coim}(m)\ar[r]^{\sim} & \textrm{Im}(m)\ar[u]
}
\end{displaymath}
Since $\textrm{Im}(m)\rightarrow E$ is a kernel of $\textrm{Coker}(m)$, so is $m:X\rightarrow E$.
\item
This is dual to the first part.
\end{enumerate}
\end{proof}

\begin{prop}
The class of strict epics (resp. monics) in a quasi-abelian category $\mathpzc{E}$ is stable by composition.
\end{prop}

\begin{proof}
See \cite{qacs} Proposition 1.1.7.
\end{proof}

\begin{cor}
A finitely complete and cocomplete additive category $\mathpzc{E}$ is quasi-abelian if and only if the following two conditions hold.
\begin{enumerate}
\item
If
\begin{displaymath}
\xymatrix{
A\ar[d]\ar[r]^{f} & B\ar[d]\\
X\ar[r]^{f'} & Y
}
\end{displaymath}
is a pushout diagram, and $f$ is a strict monic, then $f'$ is as well.
\item
If
\begin{displaymath}
\xymatrix{
A\ar[d]\ar[r]^{f'} & B\ar[d]\\
X\ar[r]^{f} & Y
}
\end{displaymath}
is a pullback diagram, and $f$ is a strict epic, then $f'$ is as well.
\end{enumerate}
\end{cor}

Let us now describe the admissible morphisms in the quasi-abelian exact structure.

\begin{prop}[\cite{qacs} Remark 1.1.2]
Let $\mathpzc{E}$ be a finitely complete and cocomplete additive category. A morphism $f:E\rightarrow F$ in $\mathpzc{E}$ is strict if and only if it can be written as $f=i\circ p$ where $p:E\rightarrow I$ is a strict epic and $i:I\rightarrow F$ is a strict monic.
\end{prop}

\begin{proof}
Suppose $f$ admits a decomposition $f=i\circ p$ as in the statement. Then $\textrm{Ker}(f)=\textrm{Ker}(p)$. So $\textrm{Coim}(f)=\textrm{Coim}(p)$. Since $p$ is strict $\textrm{Coim}(p)\cong\textrm{Im}(p)$. Since $p$ is an epic, $\textrm{Im}(p)= I$. Similarly $\textrm{Im}(f)=\textrm{Im}(i)=I$. Conversely suppose $f$ is a strict morphism. Now  $E\rightarrow\textrm{Coim}(f)$ is a strict epic, and $\textrm{Im}(f)\rightarrow F$ is a strict monic. But since $f$ is strict, $\textrm{Coim}(f)\cong\textrm{Im}(f)$, so this gives the decomposition of $f$.
\end{proof}

\begin{cor}
A morphism in a quasi-abelian category is admissible in the quasi-abelian exact structure if and only if it is strict.
\end{cor}

\begin{rem}\label{quasweak}
An exact structure on a finitely complete and cocomplete additive category coincides with the quasi-abelian structure if and only if every morphism is weakly admissible. Then as a consequence of Proposition \ref{strictmonoepic}, a finitely complete and cocomplete additive category is abelian if and only if every morphism is admissible.
\end{rem}

\subsection{The Left Heart}

Homology in quasi-abelian categories is significantly easier than in more general exact categories. For example, there is an even stronger abelian embedding.

\begin{thm}\label{leftheart}
Let $\mathpzc{E}$ be a quasi-abelian category. There exists a left abelianisation $I:\mathpzc{E}\rightarrow\textit{LH}(\mathpzc{E})$ of $\mathpzc{E}$ such that $I$ has a left adjoint $C:\textit{LH}(\mathpzc{E})\rightarrow\mathpzc{E}$ with $C\circ I\cong \textrm{id}_{\mathpzc{E}}$, i.e. $\mathpzc{E}$ is a reflective subcategory of $\textit{LH}(\mathpzc{E})$. Moreover the induced functor on derived categories
$$D(I):D(\mathpzc{E})\rightarrow D(\textit{LH}(\mathpzc{E}))$$
is an equivalence.
\end{thm}

\begin{proof}
See \cite{qacs} Proposition 1.1.26,  Corollary 1.2.27, Proposition 1.2.28, and Proposition 1.2.31.
\end{proof}

$\textit{LH}(\mathpzc{E})$ is called the \textbf{left heart} of $\mathpzc{E}$. The embedding of $\mathpzc{E}$ into its left heart also behaves extremely well with respect to projectives.

\begin{prop}\label{lhproj}
\begin{enumerate}
\item
An object $P$ of $\mathpzc{E}$ is projective if and only if $I(P)$ is projective in $\textit{LH}(\mathpzc{E})$.
\item
$\mathpzc{E}$ has enough projectives if and only if $\textit{LH}(\mathpzc{E})$ has enough projectives. In this case an object of $\textit{LH}(\mathpzc{E})$ is projective if and only if it is isomorphic to $I(P)$ where $P$ is projective in $\mathpzc{E}$.
\end{enumerate}
\end{prop}

\begin{proof}
See \cite{qacs} Proposition 1.3.24.
\end{proof}

Moreover left abelianisations of quasi-abelian categories allow us to test acyclicity of any unbounded complex. Indeed as a consequence of Remark \ref{quasweak} and Corollary \ref{boundedtest} we get.

\begin{cor}\label{quasboundtest}
Let $I:\mathpzc{E}\rightarrow\mathpzc{A}$ be a left abelianisation of $\mathpzc{E}$ where $\mathpzc{E}$ is a quasi-abelian category. Then a complex $X_{\bullet}$ in $\mathpzc{E}$ is acyclic if and only if $I(X_{\bullet})$ is acyclic. In particular a map of complexes $f:X\rightarrow Y$ is a quasi-isomorphism if and only if $I(f)$ is.
\end{cor}

\subsection{Monoidal Quasi-Abelian Categories}

Let us briefly discuss (strongly) projectively monoidal quasi-abelian categories, i.e. a (strongly) projectively monoidal exact category in which the underlying exact category is quasi-abelian. We first make the following observation.

\begin{obs}
An additive functor $F:\mathpzc{E}\rightarrow\mathpzc{F}$ between quasi-abelian categories is right exact if and only if it preserves cokernels of strict morphisms. See for example Section 1.1 in \cite{qacs}.
\end{obs}

This implies that if $(\mathpzc{E},\otimes,k)$ is a monoidal category with $\mathpzc{E}$ quasi-abelian and $\otimes$ additive, then it is a monoidal quasi-abelian category if and only if $X\otimes(-)$ preserves colimits for each object $X$ of $\mathpzc{E}$. In particular if $(\mathpzc{E},\otimes,\underline{\textrm{Hom}},k)$ is a closed monoidal category with $\mathpzc{E}$ quasi-abelian and $\otimes$, $\underline{\textrm{Hom}}$ additive functors, then $(\mathpzc{E},\otimes,\underline{\textrm{Hom}},k)$ is in fact a closed monoidal quasi-abelian category.

\begin{prop}
Let $(\mathpzc{E},\otimes,\underline{\textrm{Hom}},k)$ be a complete and cocomplete closed monoidal quasi-abelian category which is also projectively monoidal. Then there is a monoidal structure $\widetilde{\otimes},\widetilde{\underline{\textrm{Hom}}}$ on $\textit{LH}(\mathpzc{E})$ such that $(\textit{LH}(\mathpzc{E}),\widetilde{\otimes},\widetilde{\underline{\textrm{Hom}}},I(k))$ is a closed monoidal abelian category. Moreover $I:\mathpzc{E}\rightarrow\textit{LH}(\mathpzc{E})$ is a lax monoidal functor. If $(\mathpzc{E},\otimes,\underline{\textrm{Hom}},k)$ is strongly projectively monoidal then $(\textit{LH}(\mathpzc{E}),\widetilde{\otimes},\widetilde{\underline{\textrm{Hom}}},I(k))$ is projectively monoidal.
\end{prop}

\begin{proof}
See \cite{qacs} Proposition 1.5.3 and Corollary 1.5.4.
\end{proof}

\section{Generators in Exact Categories}\label{sec3}
In this recall the definition of an admissible generating collection from \cite{saorin2011exact}, which follows Schneiders'  deinition of strict generators in quasi-abelian categories \cite{qacs} Definition 2.1.5/ Proposition 2.1.7. This will come into play later when we discuss cofibrant generation of model structures, where some compactness assumptions are required. If $\mathcal{G}$ is a collection of objects in an exact category we denote by $\bigoplus\mathcal{G}$ the collection of all small coproducts of objects in $\mathcal{G}$. We will use the word `collection' because we will also be interested in proper classes of generators.

\begin{defn}
A collection of objects $\mathcal{G}$ in an exact category $\mathpzc{E}$ is said to be an \textbf{admissible generating collection} if for each object $E$ of $\mathpzc{E}$ there is an object $Q$ of $\bigoplus\mathcal{G}$ and an admissible epimorphism $Q\twoheadrightarrow E$. An admissible generating collection $\mathcal{G}$ is said to be a \textbf{projective generating collection} if all objects in $\mathcal{G}$ are projective.
\end{defn}

The next two results are adaptations of the proof of \cite{qacs} Proposition 1.3.23 to the exact case.

\begin{prop}\label{reflectsurj}
Let $\mathcal{G}$ be an admissible generating collection in a weakly idempotent complete exact category $\mathpzc{E}$. Suppose $f:E\rightarrow F$ is a morphism such that for each $G$ in $\mathcal{G}$ then map $\textrm{Hom}(G,E)\rightarrow\textrm{Hom}(G,F)$ is an epimorphism. Then $f$ is an admissible epimorphism.
\end{prop}

\begin{proof}
 Pick an admissible epimorphism $\epsilon:P\rightarrow F$ where $P\in\bigoplus\mathcal{G}$. By assumption there is a morphism $\epsilon':P\rightarrow E$ such that $\epsilon=f\circ\epsilon'$. By Proposition \ref{obscure} $f$ is then an admissible epimorphism.
\end{proof}

\begin{prop}\label{reflexact}
Let $\mathcal{G}$ be a generating collection in a weakly idempotent complete exact category $\mathpzc{E}$. A  complex
\begin{displaymath}
\xymatrix{
0\ar[r] & E\ar[r]^{e'} & E\ar[r]^{e''} & E''
}
\end{displaymath}
with $e''$ weakly left admissible is admissibly acyclic if and only if for each $G\in\mathcal{G}$ the sequence
\begin{displaymath}
\xymatrix{
0\ar[r] & \textrm{Hom}(G,E')\ar[r] & \textrm{Hom}(G,E)\ar[r] & \textrm{Hom}(G,E'')
}
\end{displaymath}
is acyclic in $\mathpzc{Ab}$. If in addition the objects of $\mathcal{G}$ are projective, then a sequence\begin{displaymath}
\xymatrix{
E\ar[r]^{e'} & E\ar[r]^{e''} & E''
}
\end{displaymath}
with $e''$ weakly left admissible is admissibly acyclic if and only if for each $P\in\mathcal{G}$ the sequence
\begin{displaymath}
\xymatrix{
\textrm{Hom}(P,E')\ar[r] & \textrm{Hom}(P,E)\ar[r] & \textrm{Hom}(P,E'')
}
\end{displaymath}
is acyclic in $\mathpzc{Ab}$.
\end{prop}

\begin{proof}
Suppose that for each $G\in\mathcal{G}$ the sequence
\begin{displaymath}
\xymatrix{
0\ar[r] & \textrm{Hom}(G,E')\ar[r] & \textrm{Hom}(G,E)\ar[r] & \textrm{Hom}(G,E'')
}
\end{displaymath}
is acyclic in $\mathpzc{Ab}$. Since $e''$ is weakly left admissible it is sufficient to show that $e'$ is a kernel of $e''$. Then $e'$ is automatically an admissible monic. To show this one can follow the proof in \cite{qacs}. At one point in that proof the existence of a resolution of $X$ by objects of $\oplus\mathcal{G}$ is used.  Here instead we may use Lemma \ref{enoughres}

Finally let us consider the assertion about projective generators. Proposition \ref{admexact} implies that
\begin{displaymath}
\xymatrix{
\textrm{Hom}(P,E')\ar[r] & \textrm{Hom}(P,E)\ar[r] & \textrm{Hom}(P,E'')
}
\end{displaymath}
is acyclic. For the converse first consider the sequence
\begin{displaymath}
\xymatrix{
0\ar[r]&\textrm{Ker}(e'')\ar[r] & E\ar[r]^{e''} & E''
}
\end{displaymath}
Since $\textrm{Hom}(P,-)$ preserves kernels, Proposition \ref{reflectsurj} implies that 
\begin{displaymath}
\xymatrix{
E\ar[r]^{e'} & E\ar[r]^{e''} & E''
}
\end{displaymath}
is admissibly acyclic.
\end{proof}

In particular if $\mathpzc{E}$ is quasi-abelian, then every morphism is weakly admissible, so in this case one has that a sequence\begin{displaymath}
\xymatrix{
E\ar[r]^{e'} & E\ar[r]^{e''} & E''
}
\end{displaymath}
 is admissibly acyclic if and only if for each $P\in\mathcal{P}$ the sequence
\begin{displaymath}
\xymatrix{
\textrm{Hom}(P,E')\ar[r] & \textrm{Hom}(P,E)\ar[r] & \textrm{Hom}(P,E'')
}
\end{displaymath}
is acyclic in $\mathpzc{Ab}$. For general exact categories we still have the following result.

\begin{cor}\label{reflshexact}
Let $\mathcal{G}$ be a projective generating collection in a weakly idempotent complete exact category $\mathpzc{E}$. Let $X_{\bullet}$ be a complex. Suppose that $X_{\bullet}$ is good. Then $X_{\bullet}$ is acyclic if and only if $\textrm{Hom}(G,X_{\bullet})$ is acyclic for each $G\in\mathcal{G}$.
\end{cor}

\begin{proof}
Since each $G\in\mathcal{G}$ is projective the functors $\textrm{Hom}(G,-)$ preserve acyclic complexes. Conversely suppose $\textrm{Hom}(G,X_{\bullet})$ is acyclic for each $G\in\mathcal{G}$, and $d_{n}^{X}$ has a kernel $Z_{n}X$. By assumption $\textrm{Hom}(G,d'_{n+1}):\textrm{Hom}(G,X_{n+1})\rightarrow Z_{n}\textrm{Hom}(G,X)=\textrm{Hom}(G,Z_{n}X)$ is an epimorphism for each $n$. Thus $d'_{n+1}:X_{n+1}\rightarrow Z_{n}X$ is an admissible epimorphism. Now apply Proposition \ref{goodtrick}.
\end{proof}

Given a set of generators $\mathcal{G}$ of an exact category $\mathpzc{E}$ it is often possible, as in \cite{gillespie2016derived} to construct an exact structure on $\mathpzc{E}$ such that $\mathcal{G}$ is a set of \textit{projective} generators. We will discuss this in the next chapter.

\subsection{Elementary Exact Categories}

It is convenient to have generators satisfying some compactness conditions. Recall that a poset $\mathcal{J}$ is said to be $\lambda$-\textbf{filtered} for a cardinal $\lambda$ if any subset $S$ of $\mathcal{J}$ with $|S|<\lambda$ has an upper bound.

Recall (\cite{christensen2002quillen} Definition 4.1) that the \textbf{cofinality} of a limit ordinal $\gamma$ is the smallest cardinal $\kappa$ such that there exists a subset $T$ of $\gamma$ of cardinality $\kappa$ with $sup(T)=\gamma$. If $\gamma$ is a successor ordinal its cofinality is defined to be $1$.
\begin{defn}\label{defsmallnesscond}
Let $\mathpzc{E}$ be a category, $\mathcal{S}$ a class of morphisms in $\mathpzc{E}$, and $\kappa$ a cardinal. An object $E$ of $\mathpzc{E}$ is said to be
\begin{enumerate}
\item
$(\kappa,\mathcal{S})$-\textbf{small} if the canonical map 
$$\varinjlim_{\beta\in\lambda}\textrm{Hom}(E,F_{\beta})\rightarrow\textrm{Hom}(E,\varinjlim_{\beta\in\mathcal{\lambda}}F_{\beta})$$ is an isomorphism for any cardinal $\lambda$ with $cofin(\lambda)\ge\kappa$ and any $\lambda$-indexed transfinite sequence where $F_{i}\rightarrow F_{i+1}$ is in $\mathcal{S}$.
\item
$\mathcal{S}$-\textbf{small} if it is $(\kappa,\mathcal{S})$-\textbf{small} for some cardinal $\kappa$
\item
$(\kappa,\mathcal{S})$-\textbf{compact} if the canonical map 
$$\varinjlim_{\beta\in\lambda}\textrm{Hom}(E,F_{\beta})\rightarrow\textrm{Hom}(E,\varinjlim_{\beta\in\mathcal{\lambda}}F_{\beta})$$ 
 is an isomorphism for any regular cardinal $\lambda\ge\kappa$ and any $\lambda$-indexed transfinite sequence where $F_{i}\rightarrow F_{i+1}$ is in $\mathcal{S}$.
\item
$\mathcal{S}$-\textbf{compact} if it is $(\kappa,\mathcal{S})$-\textbf{small} for some cardinal $\kappa$
\item
$(\kappa,\mathcal{S})$-\textbf{presented} if the natural map 
$$\varinjlim_{i\in\mathcal{I}}\textrm{Hom}(E,F_{i})\rightarrow\textrm{Hom}(E,\varinjlim_{i\in\mathcal{I}}F_{i})$$
 is an isomorphism for any $\lambda$-filtered inductive system $F:\mathcal{I}\rightarrow\mathpzc{E}$ whose colimit exists where $\lambda\ge\kappa$ is regular, and such that $F(\alpha)\in\mathcal{S}$ for any morphism $\alpha$ in $\mathcal{I}$.
 \item
 $\mathcal{S}$-\textbf{presented} if it is $(\kappa,\mathcal{S})$-presented for some cardinal $\kappa$.
 \item
 $\mathcal{S}$-\textbf{tiny} if it is $(0,\mathcal{S})$-presented, where $0$ is the first ordinal.
 \item
 \textbf{tiny} if it is $\mathcal{S}$-\textbf{tiny} for $\mathcal{S}=\mathpzc{Mor}(\mathpzc{E})$.
\end{enumerate}
\end{defn}

The terminology `tiny', and both `elementary' and `quasi-elementary' below  are following the corresponding definitions for quasi-abelian categories in \cite{qacs}. The  definition of `small' is following Definition 4.1 in \cite{christensen2002quillen}, and the definition of `compact' is \cite{vst2012exact} Definition 3.3 (where it is called small). The notion of `presented' is following the discussion in \cite{lurie2006higher} Section A.1.1.

\begin{defn}\label{smallnessconditions}
Let $\mathpzc{E}$ be an exact category, $\mathcal{S}$ a collection of morphisms in $\mathpzc{E}$, and 
$\star\in\{\textrm{small, presented, compact tiny}\}$. $\mathpzc{E}$ is said to be
\begin{enumerate}
\item
 \textbf{projectively generated} if it has a projective generating set.
 \item
  $(\kappa,\mathcal{S})^{\star}$-\textbf{elementary} if it is complete, cocomplete and has a projective generating set consisting of $(\kappa,\mathcal{S})-\star$ objects.
 \item
   $\mathcal{S}^{\star}$-\textbf{elementary} if it is   $(\kappa,\mathcal{S})^{\star}$-elementary for some $\kappa$.
 \item
 $\mathcal{S}$-\textbf{elementary} if it is complete, cocomplete, and has a projective generating set consisting of $\mathcal{S}$-tiny objects.
 \item
 \textbf{quasi-elementary} if it is complete, cocomplete and has a projective generating set consisting of $\mathcal{S}$-tiny objects, where $\mathcal{S}$ is the class of split monomorphisms.
 \item
 $\textbf{AdMon}$-\textbf{elementary} if it is elementary for the class of admissible monomorphisms.
 \item
 \textbf{elementary} if it is $\mathcal{S}$-elementary for $\mathcal{S}=\mathpzc{Mor}(\mathpzc{E})$.
 \end{enumerate}
\end{defn}

\begin{prop}\label{qacel}
A cocomplete quasi-abelian category is (quasi)-elementary if and only if its left heart is (quasi)-elementary.
\end{prop}

\begin{proof}
See \cite{qacs} Proposition 2.1.12.
\end{proof}

 Let $\mathcal{I}$ be a category,  $\mathpzc{E}$ an exact category and $\mathpzc{S}$ a class of morphisms in $\mathpzc{E}$. Denote by $\mathpzc{Fun}_{\mathpzc{S}}(\mathcal{I};\mathpzc{E})$ the category of functors $D:\mathcal{I}\rightarrow\mathpzc{E}$ such that $D(i\rightarrow j)$ is in $\mathpzc{S}$ for any morphism $i\rightarrow j$ in $\mathcal{I}$. Denote also by $\mathpzc{Fun}_{\mathpzc{S}}(\mathcal{I};\mathpzc{E})^{cont}$ and $\mathpzc{Fun}_{\mathpzc{S}}(\mathcal{I};\mathpzc{E})^{cocont}$ the full subcategories of $\mathpzc{Fun}_{\mathpzc{S}}(\mathcal{I};\mathpzc{E})$ consisting of functors which are continuous, and cocontinuous respectively. Note that if $\mathcal{I}=\aleph_{0}$, then $\mathpzc{Fun}_{\mathpzc{S}}^{cocont}(\aleph_{0};\mathpzc{E})=\mathpzc{Fun}_{\mathpzc{S}}(\aleph_{0};\mathpzc{E})$.

\begin{defn}
We say that $\mathpzc{E}$ has $(\mathcal{I};\mathpzc{S})$-\textbf{(co)limits} if for any functor $D\in\mathpzc{Fun}_{\mathpzc{S}}(\mathcal{I};\mathpzc{E})$, a (co)limit of $D$ exists. We say that $\mathpzc{E}$ has $(\mathcal{I};\mathpzc{S})^{cont}$-\textbf{limits} if for any functor $D\in\mathpzc{Fun}^{cont}_{\mathpzc{S}}(\mathcal{I};\mathpzc{E})$, a limit of $D$ exists. Finally we say that $\mathpzc{E}$ has $(\mathcal{I};\mathpzc{S})^{cocont}$-\textbf{colimits} if for any functor $D\in\mathpzc{Fun}^{cocont}_{\mathpzc{S}}(\mathcal{I};\mathpzc{E})$, a colimit of $D$ exists.
\end{defn}

Let $ Ch(\mathpzc{S})$ denote the class of morphisms in $ Ch(\mathpzc{E})$ consisting of those morphisms $f_{\bullet}:A_{\bullet}\rightarrow B_{\bullet}$ such that $f_{n}\in\mathpzc{S}$ for each $n$. Clearly if  $\mathpzc{E}$ has $(\mathcal{I};\mathpzc{S})$-limits, $(\mathcal{I};\mathpzc{S})^{cont}$-limits, $(\mathcal{I};\mathpzc{S})$-colimits, or $(\mathcal{I};\mathpzc{S})^{cocont}$-colimits then $ Ch(\mathpzc{E})$ also has corresponding (co)limits for the class $Ch(\mathcal{S})$.

\begin{defn}
Suppose that $\mathpzc{E}$ has $(\mathcal{I};\mathpzc{S})$-colimits. We say that $(\mathcal{I};\mathpzc{S})$-colimits are exact in $\mathpzc{E}$ if for any functor
$F\in \mathpzc{Fun}_{\mathpzc{S}}(\mathcal{I};Ch(\mathpzc{E}))$
such that $F(i)$ acyclic for any object $i$ in $\mathcal{I}$, the colimit $\textrm{lim}_{\rightarrow_{\mathcal{I}}}F(i)$
is acyclic. Similarly one defines exactness of $(\mathcal{I};\mathpzc{S})^{cocont}$-colimits, $(\mathcal{I};\mathpzc{S})$-limits, and $(\mathcal{I};\mathpzc{S})^{cont}$-limits.
\end{defn}
We will be particularly interested in the cases $\mathcal{S}=\textbf{AdMon}$ is the class of admissible monomorphism, or $\mathcal{S}=\textbf{SplitMon}$ is the class of split monomorphisms. The following proposition is immediate from Proposition \ref{reflexact} and Corollary \ref{reflshexact} but it has a useful consequence.

\begin{prop}\label{elementinduct} 
Let $\mathpzc{E}$ be a complete and cocomplete elementary exact category, $\mathcal{I}$ a  filtered category, and 
$$0\rightarrow F\rightarrow G\rightarrow H\rightarrow 0$$
a null sequence of functors $\mathcal{I}\rightarrow\mathpzc{E}$ such that for each $i\in\mathcal{I}$,
$$0\rightarrow F(i)\rightarrow G(i)\rightarrow H(i)\rightarrow 0$$
is exact. Suppose there is a class $\mathcal{P}$ of projective generators of $\mathpzc{E}$ such that the maps
$$\varinjlim_{i\in\mathcal{I}}\textrm{Hom}(P,F(i))\rightarrow\textrm{Hom}(P,\varinjlim_{i\in\mathcal{I}}F(i))$$
$$\varinjlim_{i\in\mathcal{I}}\textrm{Hom}(P,G(i))\rightarrow\textrm{Hom}(P,\varinjlim_{i\in\mathcal{I}}G(i))$$
$$\varinjlim_{i\in\mathcal{I}}\textrm{Hom}(P,H(i))\rightarrow\textrm{Hom}(P,\varinjlim_{i\in\mathcal{I}}H(i))$$
are isomorphisms for any $P\in\mathcal{P}$. Then the sequence 
$$0\rightarrow \varinjlim_{i\in\mathcal{I}}F(i)\rightarrow \varinjlim_{i\in\mathcal{I}}G(i)\rightarrow \varinjlim_{i\in\mathcal{I}}H(i)\rightarrow 0$$
is exact.
\end{prop}
In particular if, for example, $\mathpzc{E}$ is a  $(\kappa,\mathcal{S})^{small}$-elementary, then for any regular $\lambda\ge\kappa$, $(\lambda,\mathcal{S})^{cocont}$-colimits in $\mathpzc{E}$ exist and are exact. If $\mathpzc{E}$ is elementary then all filtered colimits are exact. This also motivates a more general definition \ref{weaklyelementary} below.

\begin{defn}\label{weaklyelementary}
Let $\mathpzc{E}$ be an exact category and $\mathpzc{S}$ a collection of morphisms in $\mathpzc{E}$. $\mathpzc{E}$ is said to be
\begin{enumerate}
\item
\textbf{weakly} $(\lambda;\mathcal{S})$-elementary for an ordinal $\lambda$ if $\mathpzc{E}$ has $(\lambda;\mathpzc{S})^{cocont}$-colimits and $(\lambda;\mathpzc{S})^{cocont}$-colimits are exact. 
\item
\textbf{weakly }$\mathpzc{S}$-\textbf{elementary} if for any ordinal $\lambda$ $\mathpzc{E}$ is weakly $(\lambda;\mathcal{S})$-elementary. 
\item
\textbf{weakly }\textbf{AdMon}-\textbf{elementary} if it is weakly $\mathpzc{S}$-elementary for $\mathpzc{S}=\textbf{AdMon}$ of admissible monomorphisms.
\item
\textbf{weakly elementary} if it is weakly $\mathpzc{S}$-elementary for $\mathpzc{S}=\mathpzc{Mor}(\mathpzc{E})$.
\end{enumerate}
\end{defn}

In particular $\mathpzc{S}$-elementary exact categories are weakly $\mathpzc{S}$-elementary. Another useful fact is that if $\mathpzc{E}$ has enough injectives then it is weakly $(\aleph_{0};\textbf{AdMon})$-elementary.
\begin{prop}\label{prop:injadel}
If a weakly idempotent complete exact category $\mathpzc{E}$ has enough injectives then  it is weakly $(\aleph_{0};\textbf{AdMon})$-elementary.
\end{prop}
\begin{proof}
In fact we shall prove the following. Let $0\rightarrow A\rightarrow B\rightarrow C\rightarrow 0$ be an exact sequence in $\mathpzc{Fun}(\aleph_{0};\mathpzc{E})$. Suppose that $C$ is in $\mathpzc{Fun}_{\textbf{AdMon}}(\aleph_{0};\mathpzc{E})$. We claim that 
$$0\rightarrow\textrm{lim}_{\rightarrow_{n}}A_{n}\rightarrow\textrm{lim}_{\rightarrow_{n}}B_{n}\rightarrow\textrm{lim}_{\rightarrow_{n}}C_{n}\rightarrow 0$$
is exact. Noting that $Hom(-,I)$ sends colimits to limits, by the dual of Proposition \ref{reflexact} it suffices to prove that for each injective $I$, the sequence
$$0\rightarrow \textrm{lim}_{\leftarrow_{n}} Hom(C_{n},I)\rightarrow  \textrm{lim}_{\leftarrow_{n}} Hom(B_{n},I)\rightarrow  \textrm{lim}_{\leftarrow_{n}} Hom(A_{n},I)\rightarrow 0$$
is exact. Since $I$ is injective $\textrm{lim}_{\leftarrow_{n}} Hom(C_{n},I)$ is the limit of a Mittag-Leffler system. Thus the sequence of abelian groups is exact (\cite{stacks-project} Section 10.86).
\end{proof}
\begin{prop}\label{transcomp}
Let $\lambda$ be an ordinal, and $\mathpzc{E}$ be a weakly $(\lambda';\textbf{AdMon})$-elementary exact category for all $\lambda'\le\lambda$. Then $\lambda$-transfinite compositions of admissible monics are admissible monics.
\end{prop}
\begin{proof}
The proof is by transfinite induction. Since finite compositions of admissible monics are admissible, the successor case is clear. For the limit case let $\Lambda$ be a limit ordinal, and consider the commutative diagram
\begin{displaymath}
\xymatrix{
E_{0}\ar[r]\ar[d] & E_{0}\ar[d]\ar[r] & E_{0}\ar[r]\ar[d] & \ldots\\
E_{0}\ar[d]\ar[r]^{c_{\lambda}} & E_{\lambda}\ar[r]\ar[d] & E_{\lambda'}\ar[r]\ar[d] & \ldots\\
0\ar[r] &\textrm{Coker}(c_{\lambda})\ar[r] & \textrm{Coker}(c_{\lambda'})\ar[r] &\ldots
}
\end{displaymath}
with short exact columns. Taking the direct limit over $\Lambda$, we get a short exact sequence
$$0\rightarrow E_{0}\rightarrow E\rightarrow C\rightarrow 0$$
In particular $E_{0}\rightarrow E$ is admissible.
\end{proof}
Note that Proposition \ref{elementinduct} and Proposition \ref{transcomp} are proven for the $G$-exact structure (see the next chapter) in \cite{gillespie2016derived} Corollary 5.3.
Before concluding this section, let us say something brief about \textit{projective} limits. 
\begin{defn}
Let $\mathpzc{E}$ be an exact category. A projective diagram $A\in\mathpzc{Fun}(\aleph_{0};\mathpzc{E}^{op})$ 
in $\mathpzc{E}$ is said to be $\textrm{lim}_{\leftarrow}$-acyclic if for any exact sequence
$$0\rightarrow A\rightarrow B\rightarrow C\rightarrow 0$$
in $\mathpzc{Fun}(\aleph_{0};\mathpzc{E}^{op})$ the sequence
$$0\rightarrow\textrm{lim}_{\leftarrow_{n}}A_{n}\rightarrow\textrm{lim}_{\leftarrow_{n}}B_{n}\rightarrow\textrm{lim}_{\leftarrow_{n}}C_{n}\rightarrow0$$
is exact in $\mathpzc{E}$.
\end{defn}
The following is dual to Proposition \ref{prop:injadel}.
\begin{prop}\label{prop:projML}
If $\mathpzc{E}$ be a weakly idempotent complete exact category. If $\mathpzc{E}$ has enough projectives then any sequence $A\in\mathpzc{Fun}(\aleph_{0};\mathpzc{E}^{op})$ such that $A_{n}\rightarrow A_{n-1}$ is an admissible epimorphism is $\textrm{lim}_{\leftarrow}$-acyclic.
\end{prop}

\subsection{Generators in Categories of Chain Complexes}

Our goal now is to show that if $\mathpzc{E}$ is an elementary exact category then so is $ Ch_{*}(\mathpzc{E})$, for $*\in\{+,\le0,-,b,\ge0,\emptyset\}$. Much of this is based on the following technical result.

\begin{lem}\label{extstuff}
Let $\mathpzc{E}$ be a weakly idempotent complete exact category. For any object $C\in\mathpzc{E}$ and $X,Y\in Ch(\mathpzc{E})$ we have natural isomorphisms
\begin{enumerate}
\item
$\textrm{Hom}_{\mathpzc{E}}(C,Y_{n})\cong\textrm{Hom}_{ Ch(\mathpzc{E})}(D^{n}(C),Y)$
\item
$\textrm{Hom}_{\mathpzc{E}}(X_{n-1},C)\cong \textrm{Hom}_{ Ch(\mathpzc{E})}(X,D^{n}(C))$
\item
$\textrm{Ker}(\textrm{Hom}_{\mathpzc{E}}(C,d_{n}^{Y}))\cong\textrm{Hom}_{ Ch(\mathpzc{E})}(S^{n}(C),Y)$. 
In particular if $\textrm{Ker}(d_{n}^{Y})$ exists\newline 
then $\textrm{Hom}_{\mathpzc{E}}(C,\textrm{Ker}(d_{n}^{Y}))\cong\textrm{Hom}_{ Ch(\mathpzc{E})}(S^{n}(C),Y)$
\item
$\textrm{Ker}(\textrm{Hom}_{\mathpzc{E}}(d_{n+1}^{X},C))\cong\textrm{Hom}_{ Ch(\mathpzc{E})}(X,S^{n}(C))$
In particular if $\textrm{Coker}(d_{n+1}^{X})$ exists then \newline
$\textrm{Hom}_{\mathpzc{E}}(\textrm{Coker}(d_{n+1}^{X}),C)\cong\textrm{Hom}_{ Ch(\mathpzc{E})}(X,S^{n}(C))$
\item
$\textrm{Ext}^{1}_{\mathpzc{E}}(C,Y_{n})\cong\textrm{Ext}^{1}_{ Ch(\mathpzc{E})}(D^{n}(C),Y)$
\item
$\textrm{Ext}^{1}_{\mathpzc{E}}(X_{n},C)\cong\textrm{Ext}^{1}_{ Ch(\mathpzc{E})}(X,D^{n+1}(C))$
\item
Let $X$ be a complex such that $\textrm{Ker}(d_{n}^{X})$ exists. Then there is a monomorphism
$$\textrm{Ext}^{1}(C,\textrm{Ker}(d_{n}^{X}))\hookrightarrow\textrm{Ext}^{1}(S^{n}(C),X)$$
If $X$ is acyclic then this is an isomorphism.
\item
Let $X$ be a complex such that $\textrm{Coker}(d_{n+1}^{X})$ exists. Then there is a monic
$$\textrm{Ext}^{1}(\textrm{Coker}(d_{n+1}^{X}),C)\hookrightarrow\textrm{Ext}^{1}(X,S^{n}(C))$$
If $X$ is acyclic then this is an isomorphism.
\end{enumerate}
\end{lem}

\begin{proof}
By Proposition \ref{extab} and Corollary \ref{chainab} it is sufficient to prove statements $1-3,5,6,7$ under the assumption that $\mathpzc{E}$ is abelian. In this context the result is Lemma 3.1 in \cite{Gillespie2} and Lemma 4.2 in \cite{gillespie2008cotorsion}. Statement $4$ is dual to to $3$, and statement $8$ is dual to $7$.
\end{proof}

\begin{rem}
It is possible to prove most of this lemma internally in an exact category without passing to an abelianisation.
\end{rem}

At this point we can prove the following lemma. It provides one of our main applications of generating sets, namely a convenient method for testing acyclicity. It is a modification of Lemma 3.7 in \cite{kaplansky}.

\begin{lem}\label{genexact}
Let $\mathpzc{E}$ be an exact category with a collection of  generators $\mathcal{G}$. Let $X$ be a chain complex. Suppose that $X_{\bullet}$ is good. If for every $G\in\mathcal{G}$ each map $f:S^{n}(G)\rightarrow X$ extends to $D^{n+1}(G)$, then $X$ is acyclic.
\end{lem}

\begin{proof}
By Proposition \ref{goodtrick} it is enough to show that whenever $d_{m}$ has a kernel, the induced map
$$d':X_{m+1}\rightarrow Z_{m}X$$
is an admissible epic. For this it is enough to show that for each $G\in\mathcal{G}$, 
$$\textrm{Hom}(G,d'):\textrm{Hom}(G,X_{m+1})\rightarrow\textrm{Hom}(G,Z_{m}X)$$ 
is surjective, i.e. that any map $f:G\rightarrow Z_{m}X$ lifts to a diagram
\begin{displaymath}
\xymatrix{
& X_{m+1}\ar[d]^{d'}\\
G\ar[r]^{f}\ar[ur] & Z_{n}X
}
\end{displaymath}
But this is equivalent to showing that the chain map $S^{n}(G)\rightarrow X$ induced by $f$ extends to a morphism $D^{n+1}(G)\rightarrow X$.

\end{proof}
Since there is a bijective correspondence between diagrams of the form
\begin{displaymath}
\xymatrix{
S^{n-1}(G)\ar[d]\ar[r] & X\ar[d]^{f}\\
D^{n}(G)\ar[r] & Y
}
\end{displaymath}
and maps of the form $S^{n}(G)\rightarrow\textrm cone(f)$, which induces a bijection between lifts in the above diagram and extensions of the map $S^{n}(G)\rightarrow\textrm cone(f)$ to a map $D^{n+1}(G)\rightarrow\textrm{cone}(f)$, we immediately get the following.
\begin{cor}\label{liftquasi}
Let $\mathpzc{E}$ be a weakly idempotent complete exact category with a collection of generators $\mathcal{G}$. Let $g:X\rightarrow Y$ be a morphism of complexes. Then $g$ is acyclic if and only if $f$ has the right lifting property with respect to all maps of the form $S^{n}(G)\rightarrow D^{n+1}(G)$ for $n\in\Z, G\in\mathcal{G}$.
\end{cor}

Next we characterise projective objects in categories of chain complexes. It is well known that projective objects in the category of chain complexes in an abelian category are precisely the split exact complexes with projective entries. See for example \cite{hoveybook} Proposition 2.3.10. We generalise the result to exact categories. The fact that in $Ch(\mathpzc{E})$ the projective complexes are contractible complexes with projective components, and that if $\mathpzc{E}$ has enough projectives then so does $Ch(\mathpzc{E})$ is proven in \cite{gillespie2016exact} Proposition 2.6 and Corollary 2.7. 
\begin{prop}\label{chainproj} 
Let $\mathpzc{E}$ be an exact category, and let $*\in\{\ge0,\le0,+,-,b,\emptyset\}$. Then contractible complexes of projectives are projective objects in $ Ch_{*}(\mathpzc{E})$. In addition, if $P$ is projective in $\mathpzc{E}$ then $S^{0}(P)$ is projective in $ Ch_{\ge0}(\mathpzc{E})$. Conversely, if a complex $X_{\bullet}$ is a projective in $ Ch_{*}(\mathpzc{E})$ for $*\in\{+,-,b,\ge0,\le0,\emptyset\}$ then every $X_{n}$ is projective. Moreover, if $*\in\{+,-,b,\emptyset\}$ then $X_{\bullet}$ is contractible. In particular if $\mathpzc{E}$ is weakly idempotent complete then  $X_{\bullet}$ is projective if and only if it is a split exact complex of projective objects of $\mathpzc{E}$. 
\end{prop}

\begin{proof}
The claims for $*\in\{+,-,b,\emptyset\}$ are established in \cite{gillespie2016exact} Proposition 2.6 and Corollary 2.7. In fact the proof there is restricted to $*\in\{\emptyset\}$but the proof is the same for the other cases. Let us show that $S^{0}(P)$ is a projective object in $ Ch_{\ge0}(\mathpzc{E})$ whenever $P$ is projective in $\mathpzc{E}$. Indeed in this case, Lemma \ref{extstuff} implies that $\textrm{Hom}_{ Ch(\mathpzc{E})}(S^{0}(P),Y_{\bullet})\cong\textrm{Hom}_{\mathpzc{E}}(P,Y_{0})$. Since $P$ is projective, $S^{0}(P)$ is as well. 
\end{proof}

We can now show that $ Ch_{*}(\mathpzc{E})$ has enough projectives. (This is well known for $ Ch(\mathpzc{A})$ with $\mathpzc{A}$ abelian. See for example \cite{Weibel} Exercise 2.2.2).

\begin{cor}\label{enoughprojchain}
Let $\mathpzc{E}$ be an exact category with enough projectives. Then $ Ch_{*}(\mathpzc{E})$ has enough projectives for $*\in\{+,-,b,\le0,\ge0,\emptyset\}$
\end{cor}

\begin{proof}
The case for $*\in\{+,-,b,\emptyset\}$ is established in  \cite{gillespie2016exact} Corollary 2.7. 
Now let $X_{\bullet}\in Ch_{\ge0}(\mathpzc{E})$. For $n>0$ the object $D^{n}(P)$ is projective in $ Ch_{\ge0}(\mathpzc{E})$. $S^{0}(P)$ is also projective in $ Ch_{\ge0}(\mathpzc{E})$. For $n>0$, as before there is a projective object $P_{n}$ and a morphism $D^{n}(P_{n})\rightarrow X_{\bullet}$ which is an admissible epimorphism in degree $n$. For $n=0$ pick a projective object $P_{0}$ and an admissible epimorphism $P_{0}\rightarrow X_{0}$. Since $X_{-1}=0$, this induces a map $S^{0}(P_{0})\rightarrow X_{\bullet}$ which is an admissible epimorphism in degree $0$. Let $P_{\bullet}=\Bigr(\bigoplus_{n>0}D^{n}(P_{n})\Bigr)\oplus S^{0}(P_{0})$. Then we have an admissible epimorphism $P_{\bullet}\twoheadrightarrow X_{\bullet}$.
\end{proof}
In particular we have shown that $ Ch_{*}(\mathpzc{E})$ has a set of projective generators whenever $\mathpzc{E}$ does.

\begin{cor}\label{chaingen}
Suppose $\mathcal{P}$ is a collection of admissible generators for an exact category $\mathpzc{E}$. Then $D^{*}(\mathcal{P})=\{D^{n}(P):P\in\mathcal{P},n\in\Z\}\cap Ch_{*}(\mathpzc{E})$ is a collection of generators for $ Ch_{*}(\mathpzc{E})$ and $*\in\{+,-,b,\le0,\emptyset\}$. For $*\in\{\ge0\}$, $\widetilde{D}^{*}(\mathcal{P})\defeq D^{*}(\mathcal{P})\cup \{S^{0}(P):P\in\mathcal{P}\}$ is a collection of generators for $ Ch_{*}(\mathpzc{E})$. They are projective generating collections if $\mathcal{P}$ is.
\end{cor}

\begin{proof}
The proof of Corollary \ref{enoughprojchain}  shows that the collection in the statement of the proposition are admissible generating collection. Proposition \ref{chainproj}  establishes the second assertion.
\end{proof}

We are nearly ready to show that $ Ch_{*}(\mathpzc{E})$ is elementary for $*\in\{+,\ge0,\le0,-,b,\emptyset\}$. It remains to identify some suitably compact objects in complexes. However by Lemma \ref{extstuff} we have the following.
\begin{prop}\label{diagrampresented}
Let $E$ be an object satisfying one of the smallness conditions of Definition \ref{defsmallnesscond}. Then $D^{n}(E)$ and $S^{n}(E)$ satisfy the same smallness condition in $Ch(\mathpzc{E})$.
\end{prop}

As a consequence we have

\begin{cor}\label{chel}
For $\star\in\{\textrm{small, compact, presented, teeny}\}$ let $\mathpzc{E}$ be a $(\kappa,\mathcal{S})^{\star}$-elementary exact category. Then $ Ch_{*}(\mathpzc{E})$ is  $(\kappa,Ch(\mathcal{S}))^{\star}$-elementary for $*\in\{+,\le0,\ge0,-,b,\emptyset\}$. In particular if $\mathpzc{E}$ is $\mathcal{S}$-elementary then $Ch(\mathpzc{E})$ is $Ch(\mathcal{S})$-elementary
\end{cor}

\begin{proof}
Let $\mathcal{P}$ be a  projective generating set consisting of tiny objects. The sets $D^{*}(\mathcal{P})$ (resp. $\widetilde{D}^{*}(\mathcal{P})$) are projective generating sets in $ Ch_{*}(\mathpzc{E})$ for $*\in\{\le0,+,-,b,\emptyset\}$ (resp. $*\in\{\ge0\}$). For each $n\in\Z$ $D^{n}(P)$ is tiny, as is $S^{n}(P)$, by Proposition \ref{diagrampresented}. 
\end{proof}

\subsection{Generators in Monoidal Exact Categories}

Let us briefly mention a useful compatibility condition between generators and monoidal structure.

\begin{defn}
A monoidal exact category which has a collection of flat admissible generators is said to be \textbf{flatly generated}. 
\end{defn}

\begin{defn}\label{moneldef}
 A projectively monoidal exact category which is also $(\lambda;\mathcal{S})$-elementary is said to be \textbf{monoidal} $(\lambda;\mathcal{S})$\textbf{elementary}
\end{defn}

\begin{prop}\label{flatgenproj}
Suppose that $(\mathpzc{E},\otimes,k)$ is a flatly generated monoidal exact category in which direct sums are exact. Then every projective object is flat.
\end{prop}

\begin{proof}
In this case every projective will be a summand of a flat object, and therefore flat.
\end{proof}

In particular to check that a category is projectively monoidal, it suffices to find a collection of flat generators.

\subsection{Generators and Adjunctions}
We conclude this section with a note about passing generating collections through adjunctions. The specific application we have in mind is to categories of algebras over compatible monads. We have the following general setup.
$$\adj{F}{\mathpzc{E}}{\mathpzc{D}}{|-|}$$
is an adjunction where $\mathpzc{E}$ and $\mathpzc{D}$ are exact categories. We have the following result which is standard for abelian categories.
\begin{prop}\label{adproj}
Let $F\dashv |-|$ be an adjunction as above. Suppose that $|-|$ is an exact functor. If $P$ is a projective object of $\mathpzc{E}$ then $F(P)$ is a projective object of $\mathpzc{D}$.
\end{prop}
\begin{proof}
Let 
$$0\rightarrow X\rightarrow Y\rightarrow Z\rightarrow 0$$
be a short exact sequence in $\mathpzc{D}$, and let $P$ be projective in $\mathpzc{E}$. Then we have a diagram
\begin{displaymath}
\xymatrix{
0\ar[r] & \textrm{Hom}(F(P),X)\ar[r]\ar[d] & \textrm{Hom}(F(P),Y)\ar[r]\ar[d] & \textrm{Hom}(F(P),Z)\ar[r]\ar[d] & 0\\
0\ar[r] & \textrm{Hom}(P,|X|)\ar[r] & \textrm{Hom}(P,|Y|)\ar[r] & \textrm{Hom}(P,|Z|)\ar[r] & 0
}
\end{displaymath}
The vertical arrows are isomorphisms and the bottom row is exact since $|-|$ is exact and $P$ is projective. Hence the top row is short exact as well. 
\end{proof}

We know how adjunctions act on projectives. Let us now see what happens on generating collections.

\begin{prop}\label{epimonad}
Let $F\dashv |-|$ be an adjunction as above. Suppose that $|-|$ reflects admissible epimorphisms, and that $\mathpzc{E}$ has an admissible generating collection $\mathpzc{G}$. Let $F(\mathcal{G})$ denote the collection $\{F(G):G\in\mathcal{G}\}$ of objects of $\mathpzc{D}$. Then $F(\mathpzc{G})$ is an admissible generating collection in $\mathpzc{D}$.
\end{prop}

\begin{proof}
Let $X$ be an object of $\mathpzc{D}$. Suppose there is some object $Q$ of $\mathpzc{E}$ and an admissible epimorphism $p:Q\rightarrow|X|$. There is an induced morphism $\widetilde{p}:F(Q)\rightarrow X$. Then $p$ coincides with the composition $Q\rightarrow |F(Q)|\rightarrow |X|$. By Proposition \ref{obscure}, the map $|\widetilde{p}|$ is an admissible epimorphism. Since $|-|$ reflects admissible epimorphisms, $\widetilde{p}$ is an admissible epimorphism in $\mathpzc{D}$. 

Now let $\mathcal{G}$ be an admissible generating collection in $\mathpzc{E}$, and let $X$ be an object of $\mathpzc{D}$. Since $\mathcal{G}$ is an admissible generating collection, there is an object $G$ of $\bigoplus\mathcal{G}$ and an admissible epimorphism $G\twoheadrightarrow |X|$. The induced morphism $F(G)\rightarrow X$ is an admissible epimorphism by the above remarks. Since $F$ is a left adjoint it preserves colimits, so $F(G)$ is an element of $\bigoplus F(\mathcal{G})$.
\end{proof}

The following is clear.
\begin{prop}\label{prop:monadcompactness}
Let 
$$\adj{F}{\mathpzc{E}}{\mathpzc{D}}{|-|}$$ 
be an adjunction between any categories, $\kappa$ an ordinal, and $\mathcal{S}$ a class of morphisms in $\mathpzc{D}$. Let $|\mathcal{S}|$ denote the class of morphisms in $\mathpzc{E}$ which are images of morphisms in $\mathcal{S}$ under $|-|$.
\begin{enumerate}
\item
Suppose that $|-|$  commutes with $(\lambda,\mathcal{S})^{cocont}$-colimits whenever $\lambda\ge\kappa$ is regular. If $G\in\mathpzc{E}$ is $(\kappa,\mathcal{S})$-small then $F(G)$ is $(\kappa,|\mathcal{S}|)$-small.
\item
Suppose that $|-|$  commutes with $(\lambda,\mathcal{S})^{cocont}$-colimits whenever $\lambda\ge\kappa$. If $G\in\mathpzc{E}$ is $(\kappa,\mathcal{S})$-compact then $F(G)$ is $(\kappa,|\mathcal{S}|)$-compact.
\item
Suppose that $|-|$  commutes with $(\mathcal{I},\mathcal{S})^{cocont}$-colimits whenever $\mathcal{I}$ is a $\lambda$-filtered category for $\lambda\ge\kappa$ regular. If $G\in\mathpzc{E}$ is $(\kappa,|\mathcal{S}|)$-presented then $F(G)$ is $(\kappa,\mathcal{S})$-presented.
\end{enumerate}
\end{prop}
Using this proposition, Proposition \ref{adproj} and Proposition \ref{epimonad}, we then get the following.

\begin{prop}\label{genproj}
Let $F\dashv |-|$ be an adjunction as above, $\kappa$ an ordinal, and $\mathcal{S}$ a class of morphisms in $\mathpzc{D}$. Let $|\mathcal{S}|$ denote the class of morphisms in $\mathpzc{E}$ which are images of morphisms in $\mathcal{S}$ under $|-|$.
\begin{enumerate}
\item
Suppose that $|-|$ is exact, reflects exactness and commutes with $(\kappa,\mathcal{S})^{cocont}$-colimits. If $\mathpzc{E}$ is weakly $(\kappa,|\mathcal{S}|)$- elementary then $\mathpzc{D}$ is weakly $(\kappa,\mathcal{S})$-elementary.
\item
Suppose that $|-|$ is exact and reflects admissible epimorphisms. If $\mathpzc{G}$ is a projective generating collection in $\mathpzc{E}$ then $F(\mathpzc{G})$ is a projective generating collection in $\mathpzc{D}$. 
\item

Suppose that $|-|$ is exact, reflects admissible epimorphisms, and commutes  with $(\lambda,\mathcal{S})^{cocont}$-colimits for any ordinal $\lambda$. If $\mathpzc{E}$ is $|\mathcal{S|}$-elementary then $\mathpzc{D}$ is $\mathcal{S}$-elementary.

\end{enumerate}
\end{prop}

\begin{proof}
The first claim is obvious. The second and third assertions follow the previous Proposition, Proposition \ref{adproj}, and Proposition \ref{epimonad}.
\end{proof}
Note that by  Proposition \ref{prop:monadcompactness} the third claim in Proposition \ref{genproj} can be generalised to statements about $\mathpzc{D}$ being $(\kappa,\mathcal{S})^{\star}$-elementary for more general $\star$, by only requiring that $|-|$ commutes with certain colimits. We have omitted this level of generality so as not to have to write out all the cases. However the proof is identical.

\begin{example}
Let $T$ be a compatible monad on an exact category $\mathpzc{E}$. Then the forgetful functor $|-|:\mathpzc{E}^{T}\rightarrow\mathpzc{E}$ has a right adjoint $F:\mathpzc{E}\rightarrow\mathpzc{E}^{T}$ assigning to an object the free $T$-algebra on it. By construction of the exact structure on $\mathpzc{E}^{T}$ in Proposition \ref{monadexact}, the functor $|-|$ is admissibly exact and reflects exactness. Moreover it creates limits and colimits. By Lemma \ref{reflectlots}, Proposition \ref{genproj} is applicable in such categories.
\end{example}
\chapter{Examples}\label{Secex1}

In this chapter we give examples of interesting exact categories which satisfy very different set-theoretic properties but which are all weakly \textbf{AdMon}-elementary. In the next chapter we shall see that $\mathpzc{E}$ being weakly \textbf{AdMon}-elementary and having kernels is enough for the category $Ch(\mathpzc{E})$ to be equipped with the projective model structure. The moral of the story is that often difficult to check set-theoretic assumptions can be ignored to some extent when discussing such model structures.

\section{Categories of Topological Vector Spaces}

In this section we let $k$ be a Banach ring, that is, a unital commutative ring $k$ together with a map $|-|:k\rightarrow\R_{>0}$ such that for all $x,y\in k$ we have
\begin{enumerate}
\item
$|x|=0\Leftrightarrow x=0$
\item
$|x+y|\le|x| + |y|$
\item
$|xy|\le |x||y|$
\item
$k$ is complete with respect to the topology defined by $|-|$.
\end{enumerate}
$k$ is said to be \textbf{non-Archimedean} if $|x+y|\le\textrm{max}\{|x|,|y|\}$ and \textbf{Archimedean} otherwise. Over such rings we can consider categories of topological $k$-modules. For details of claims made in this section consult \cite{qacs}, \cite{orenbambozzi}, \cite{koren}, and \cite{bambozzi2015stein}.

\subsection{Categories of Normed and Banach Modules}

\begin{defn}
A \textbf{normed} $k$-\textbf{module} is a $k$-module $V$ together with a map $||-||:V\rightarrow\R_{>0}$ such that for all $\lambda\in k$ and for all $x,y\in V$ we have
\begin{enumerate}
\item
$||x||=0\Leftrightarrow x=0$
\item
$||x+y||\le||x|| + ||y||$
\item
$||\lambda x||\le |\lambda|||x||$
\end{enumerate}
If $V$ is complete with respect to the metric defined by $||-||$ then $V$ is said to be a \textbf{Banach} $k$-\textbf{module}. 
\end{defn}
If $k$ is non-Archimedean then $V$ is said to be non-Archimedean if $||x+y||\le max\{||x||,||y||\}$. We denote by $Norm_{k}$ the category whose objects are normed $k$-modules and whose morphisms are bounded $k$-linear maps. $Ban_{k}$ is the full subcategory of $Norm_{k}$ on Banach $k$-modules. For $k$ non-Archimedean we also consider the full subcategories $Norm^{nA}_{k}$ and $Ban^{nA}_{k}$ of non-Archimedean normed and Banach spaces respectively. All of these categories are additive, finitely complete, and finitely cocomplete. The inclusions
$$Ban_{k}\rightarrow Norm_{k},\;\; Ban_{k}^{nA}\rightarrow Norm_{k}^{nA}$$
have left-adjoint functors given by completion.

They are also symmetric monoidal. If $E$ and $F$ are objects in $Norm_{k}$ then we define $E\otimes_{\pi} F$ to be their usual module tensor product endowed with the cross-norm
$$||u||=\textrm{inf}\Bigr\{\sum_{i=1}^{n}||e_{i}|| ||f_{i}|| :u=\sum_{i=1}^{n} e_{i}\otimes f_{i}\Bigr\}$$ If $E$ and $F$ are objects in $Norm_{k}^{nA}$ we define $E\otimes_{\pi}^{nA}F$ to be their usual module tensor product endowed with the norm.
$$||x||_{\pi}=\textrm{inf}\Bigr\{\textrm{max}\{||a_{i}||||b_{i}||\}_{i=1}^{n}:\; x=\sum_{i=1}^{n}a_{i}\otimes b_{i}\Bigr\}$$
We refer to both of these constructions as the \textbf{projective tensor product}. If $E$ and $F$ are Banach spaces then $E\hat{\otimes}_{\pi} F$ is the completion of their projective tensor product as normed spaces. 
These constructions are functorial in each of the categories defined above and form part of symmetric monoidal structures on them with unit the ground ring $k$. These monoidal structures are in fact closed. The module $\textrm{Hom}_{k}(E,F)$ of bounded maps between $E$ and $F$ can be given the structure of a normed space. The norm of $T:E\rightarrow F$ is
$$||T||=\textrm{sup}_{e\in E\setminus \{0\}}\frac{||T(e)||_{F}}{||e||_{E}}$$
This gives an internal Hom functor, which we denote by $\underline{\textrm{Hom}}$. 
Thus $(Ban_{k},\hat{\otimes}_{\pi},\underline{\textrm{Hom}})$ is a monoidal quasi-abelian category. Details can be found in \cite{orenbambozzi} Proposition 3.17 and Proposition 3.19. Finally, the projective objects $l^{1}(I)$ are flat by \cite{orenbambozzi} Lemma 3.26. By Proposition \ref{flatgenproj} this category is projectively monoidal. There are unfortunately some problems with this category. Although it is finitely complete and cocomplete it does not even have countable colimits in general. The larger category $\hat{\mathcal{T}}_{c}$ of complete locally convex topological spaces is complete and cocomplete, but tragically it is not quasi-abelian (\cite{qacs} Proposition 3.1.14). Instead we pass to the formal completion $Ind(Ban_{k})$ of $Ban_{k}$ by filtered colimits.

\subsection{Ind and Pro Categories}\label{IndPro}

Recall that if $\mathpzc{C}$ is a $\mathbb{U}$-small category for some universe $\mathbb{U}$, and $\mathbb{V}$ is a universe, then the $\mathbb{V}$-ind-completion of $\mathpzc{C}$ is a category constructed as follows.  Objects are diagrams $E:\mathcal{I}\rightarrow\mathcal{\mathpzc{C}}$ where $\mathcal{I}$ is a $\mathbb{V}$-small filtered category. If $E:\mathcal{I}\rightarrow\mathpzc{C}$ and $F:\mathcal{J}\rightarrow\mathpzc{C}$ are objects in $Ind(\mathpzc{C})$ (where we suppress universes in the notation) then we write

$$\textrm{Hom}_{Ind(\mathpzc{C})}(E,F)=\textrm{lim}_{\leftarrow_{\mathcal{I}}}\textrm{lim}_{\rightarrow_{\mathcal{J}}}\textrm{Hom}_{\mathpzc{C}}(E_{i},F_{j})$$
Details of this can be found in \cite{kashiwara2005categories} Chapter 6.

\begin{prop}\label{indqac}
Let $\mathpzc{E}$ be a quasi-abelian category with enough projectives. Then $Ind(\mathpzc{E})$ is a cocomplete elementary quasi-abelian category. Moreover, if $\mathpzc{E}$ is a closed monoidal exact category, then its ind-completion has a canonical exact closed monoidal structure extending the one on $\mathpzc{E}$. Finally if $\mathpzc{E}$ is projectively monoidal then so is $Ind(\mathpzc{E})$.
\end{prop}

\begin{proof}
See \cite{qacs} Proposition 2.1.16 and Proposition 2.1.19.
\end{proof}

\begin{cor}
The category $Ind(Ban_{k})$ is a locally presentable, closed monoidal elementary quasi-abelian category.
\end{cor}
The category $Ind(Ban_{k})$ is not concrete. However it does have a natural concrete full subcategory $Ind^{m}(Ban_{k})$. An object of $Ind^{m}(Ban_{k})$ is a formal colimit $``lim_{\rightarrow}"E_{i}$ such that any map $E_{i}\rightarrow E_{j}$ is a monomorphism (not necessarily admissible!). It is shown in \cite{meyer2007local} Theorem 1.139 and Section 1.5.3 that this category is equivalent to the concrete category $CBorn_{k}$ of complete \textbf{bornological} $k$-modules, via the disection functor $diss:CBorn_{k}\rightarrow Ind^{m}(Ban_{k})$. These are spaces equipped with an appropriate notion of `bounded subsets'. To a (complete) locally convex space $E$ one can functorially assign both the von Neumann bornology $vN(E)$ and the compact bornology $Cpt(E)$. The von Neumann bornology is composed of the subsets of $E$ absorbed by all zero neighbourhoods. The compact bornology is composed of subsets with compact closure. There is a natural transformation of functors $Cpt\rightarrow vN$. For details see \cite{meyer2007local} Section 1.1.4. If $V$ is a nuclear locally convex space then the map $Cpt(V)\rightarrow vN(V)$ is an isomorphism by \cite{bambozzi2015stein} Lemma 3.67.

There is also the dual notion of the $\mathbb{V}$-pro-completion of $\mathpzc{C}$, which is defined to be
$$Pro(\mathpzc{C})=Ind(\mathpzc{C}^{op})^{op}$$
It is the formal completion of $\mathpzc{C}$ by projective limits. 

 For $k$ a Banach ring $Pro(Ban_{k})$ contains $\hat{\mathcal{T}}_{c,k}$ as a full subcategory. Indeed if $E$ is an object of $\hat{\mathcal{T}}_{c,k}$ defined by a family of seminorms $\mathcal{P}$ then define $PB(E)=``lim_{\leftarrow_{p\in\mathcal{P}}}"\hat{E}_{p}$ where $\hat{E}_{p}$ is the completion of $E$ with respect to the metric defined by the semi-norm $p$. This construction is functorial, lax monoidal, and $PB:\hat{\mathcal{T}}_{c,k}\rightarrow Pro(Ban_{k})$ is fully faithful.

 If $\mathpzc{E}$ is a quasi-abelian category enough projectives and injectives then by Proposition 2.1.15 in \cite{qacs} both $Ind(\mathpzc{E})$ and $Pro(\mathpzc{E})$ are complete and cocomplete. In particular by an obvious Kan extension there is a canonical functor
$$PI:Pro(\mathpzc{E})\rightarrow Ind(\mathpzc{E})$$
Again this is lax monoidal. 

Returning to the case $\mathpzc{E}=Ban$, there is a natural isomorphism of functors $PI\circ PB\cong diss\circ vN$ (see \cite{bambozzi2015stein}), and therefore a natural transformation $diss\circ Cpt\rightarrow PI\circ PB$. Let $E$ be a Banach space and $F$ a complete locally convex space. Then
\begin{align*}
Hom_{Ind(Ban)}(E,PI\circ PB(F))&\cong Hom_{Pro(Ban)}(E,PB(F))\\
&\cong Hom_{\hat{\mathcal{T}}_{c,k}}(E,F)\\
&\cong Hom_{CBorn}(E,vN(F))\\
&\cong Hom_{Ind(Ban)}(E,diss\circ vN(F))
\end{align*}
Where the isomorphisms just arises from the fact that $Hom(E,-)$ commutes with projective limits in both $\hat{\mathcal{T}_{c}}$ and $CBorn$ When restricted to the category of nuclear Fr\'{e}chet spaces the functor $diss\circ vN$ is fully faithful by Example 3.22 \cite{bambozzi2015stein}. Moreover since map $diss\circ Cpt\rightarrow PI\circ PB$ is a natural isomorphism this restriction of $PI\circ PB$ is also strong monoidal \cite{meyer2007local}) Theorem 1.87. In particular the category of nuclear Fr\'{e}chet algebras over $\C$ embeds fully faithfully in the category of commutative complete bornological algebras. Since the category $CBorn_{\C}$ has good categorical properties, in particular it is closed monoidal \textbf{AdMon}-elementary, this is evidence that it provides a convenient setting in which to study analytic algebra.

\subsection{The Non-Expanding Normed and Banach Categories}\label{sec7}
Each of the normed and Banach categories considered in the previous section has a corresponding `non-expanding' subcategory. If $E$ and $F$ are normed spaces and $s\in\R_{\ge0}$ then we denote by $Hom_{k}^{\le s}(E,F)\subset Hom_{k}(E,F)$
the set of maps of $k$-modules of norm at most $s$. Composition gives a map
$$Hom_{k}^{\le r}(E,F)\otimes Hom_{k}^{\le s}(F,G)\rightarrow Hom_{k}^{\le rs}(F,G)$$
In particular there are wide subcategories of $Norm_{k},Ban_{k},Norm^{nA}_{k},Ban^{nA}_{k}$ consisting of maps of norm at most $1$ which we denote by $Norm_{k}^{\le1},Ban_{k}^{\le1},Norm^{nA,\le1}_{k},Ban^{nA,\le1}_{k}$. They are equipped with closed symmetric monoidal structures by restricting the ones on the larger categories. If $k$ is non-Archimedean then these categories are also additive and in fact quasi-abelian. 

In both the Archimedean and non-Archimedean case these categories are complete and co-complete. Details of this can be found in \cite{koren} Appendix A. For convenience we recall how to construct arbitrary coproducts in $Ban_{k}^{\le1}$ and, for $k$ non-Archimedean, $Ban_{k}^{nA,\le1}$.  For $k$ Archimedean the coproduct $\coprod_{i\in I}^{\le 1}A_{i}$ of a collection $\{A_{i}\}_{i\in I}$ of Banach spaces in $Ban^{\le1}_{k}$ is 
$$\{(a_{i})_{i\in I}\subset\prod^{{}_{k}\mathpzc{Mod}}_{i\in I}A_{i}:\sum_{i\in I}||a_{i}||<\infty\}$$
with the norm $||(a_{i})||=\sum_{i\in I}||a_{i}||$. Here $\prod^{{}_{k}\mathpzc{Mod}}$ denotes the  product in the category of $k$-modules. For $k$ non-Archimedean the coproduct in both $\textit{Ban}^{\le1,nA}_{k}$  and $\textit{Ban}^{\le1}_{k}$ the coproduct $\coprod_{i\in I}^{\le1}A_{i}$ of a collection $\{A_{i}\}_{i\in I}$ of Banach spaces is the subspace
$$\{(a_{i})_{i\in I}\subset\prod^{{}_{k}\mathpzc{Mod}}_{i\in I}A_{i}:\textrm{lim}_{i\in I}||v_{i}||=0\}$$
endowed with the norm $||(a_{i})_{i\in I}||=\textrm{sup}_{i\in I}||a_{i}||$.
\subsubsection{Rescaling Functors}

For $r\in\R_{>0}$ we denote by $(-)_{r}:Norm_{k}\rightarrow Norm_{k}$ the endofunctor which sends a normed space $E$ to $E_{r}$ which has the same underlying $k$-vector space as $E$ but with norm rescaled by $r$. On morphisms it does nothing.  It is evidently an autoequivalence, and in fact an automorphism, with inverse given by $(-)_{\frac{1}{r}}$. Moreover it restricts to an auto-equivalence on all the normed and Banach categories defined above. These functors satisfy the following useful property.

\begin{prop}\label{rescalecontract}
Let $E$ and $F$ be Banach $k$-modules. Then
$\textrm{Hom}_{k}^{\le r}(E_{s},F_{t})=\textrm{Hom}_{k}^{\le\frac{sr}{t}}(E,F)$
\end{prop}

\begin{proof}
Let $f:E_{s}\rightarrow F_{t}$ have norm at most $r$, so that for any $e\in E$,

$$||f(e)||_{F}=\frac{1}{t}||f(e)||_{F_{t}}\le \frac{r}{t}||e||_{E_{s}}=\frac{sr}{t}||e||_{E}$$
Conversely suppose $f:E\rightarrow F$ has norm at most $\frac{sr}{t}$. Then we get the same inequality as above.
\end{proof}

\subsubsection{The Quasi-Abelian Exact Structure}
Although we will not go into the details here, it is not hard to see that the categories $\textit{Norm}^{\le1,nA}_{k}$ and $\textit{Ban}^{\le1,nA}_{k}$ are quasi-abelian. However let us note the following.

\begin{prop}\label{adcont}
In both $\textit{Norm}^{\le1,nA}_{k}$ and $\textit{Ban}^{\le1,nA}_{k}$ we have the following.
\begin{enumerate}
\item
A monomorphism $f:A\rightarrow B$ is admissible in the quasi-abelian exact structure if and only if it is an isometry with closed image. 
\item
An epimorphism $g:B\rightarrow C$ is admissible in the quasi-abelian exact structure if and only if it is a set-theoretic epimorphism and $||g(b)||=\textrm{inf}_{a\in\textrm{Ker}(g)}||b-a||$.
\end{enumerate}
\end{prop}

\begin{proof}
\begin{enumerate}
\item
Suppose that $f:A\rightarrow B$ is admissible in the quasi-abelian exact structure. Then it is the kernel of its cokernel $g:B\rightarrow C$. Therefore $f$ induces an isometric isomorphism with the normed subspace $K=\{b\in B:g(b)=0\}$. In particular $f$ is an isometry. Conversely suppose that $f$ is an isometry with closed image. Then $A\cong f(A)$ in $\textit{Norm}^{\le1,nA}_{k}$.  The cokernel of $f$ is isometrically isomorphic to the quotient space $B\big\slash f(A)$, and the kernel of $B\rightarrow B\big\slash f(A)$ is $\{b\in B: g(b)=0\}=f(A)\cong A$. 
\item
Suppose that $g:B\rightarrow C$ is an admissible epimorphism in the quasi-abelian exact structure. Then it is the cokernel of its kernel, which is the subspace $A=\{b\in B:g(b)=0\}$. In particular $g$ induces an isometric isomorphism $\overline{g}:B\big\slash A\cong C$. So 
$$||g(b)||=||[b]||=\textrm{inf}_{a\in A}||b-a||$$
Moreover $B\rightarrow B\big\slash A$ is a set-theoretic epimorphism, so $g$ is as well. Conversely suppose that $g$ is a set-theoretic epimorphism, and that  $||g(b)||=\textrm{inf}_{a\in\textrm{Ker}(g)}||b-a||$. Then $g$ clearly induces an isometric isomorphism. 
\end{enumerate}
\end{proof}

\begin{rem}
In the case of $\textit{Ban}^{\le1,nA}_{k}$ we may remove the assumption in Proposition \ref{adcont} 1) that $f$ has closed image, since an isometry of Banach spaces always has closed image.
\end{rem}

\subsubsection{The Strong Exact Structure}
We introduce a different exact structure on $\textit{Norm}^{\le1,nA}_{k}$ (resp. $\textit{Ban}^{\le1,nA}_{k}$). 

\begin{defn}
\begin{enumerate}
\item
We say that a morphism $f:A\rightarrow B$ in $\textit{Norm}^{\le1,nA}_{k}$ is a \textbf{strong monomorphism} if it is an isometry with closed image, and any $b\in B$ has a closest point $a_{b}\in f(A)$.
\item
We say that a morphism $g:B\rightarrow C$ in $\textit{Norm}^{\le1,nA}_{k}$ is a \textbf{strong epimorphism} if for any $c\in C$ there is a $b_{c}\in B$ with $g(b_{c})=c$ and $||b_{c}||=||c||$.
\item
We say that a morphism $f:A\rightarrow B$ in $\textit{Ban}^{\le1,nA}_{k}$ is a \textbf{strong monomorphism} (resp. \textbf{strong epimorphism}) if it is a strong monomorphism (resp. strong epimorphism) in $\textit{Norm}^{\le1,nA}_{k}$.
\end{enumerate}
\end{defn}

\begin{cor}
A strong monomorphism is an admissible monomorphism in the quasi-abelian exact structure. A strong epimorphism is an admissible epimorphism in the quasi-abelian exact structure.
\end{cor}

\begin{prop}
A map $f:A\rightarrow B$ is a strong monomorphism if and only if it is the kernel of a strong epimorphism. A map $g:B\rightarrow C$ is a strong epimorphism if and only if it is the cokernel of a strong monomorphism.
\end{prop}

\begin{proof}
Suppose that $f:A\rightarrow B$ is a strong monomorphism. Then in particular it is an admissible monomorphism in the quasi-abelian exact structure so it is the kernel of its cokernel $g:B\rightarrow C$. Let us show that $g$ is a strong epimorphism. Let $c=[b]\in C=B\big\slash f(A)$. Now $||[b]||=\textrm{inf}_{a\in A}||b-f(a)||$. By assumption there is some $a_{b}$ such that $||[b]||=||b-f(a_{b})||$. Moreover $g(b-f(a_{b}))=[b]$. So $g$ is a strong epimorphism. Conversely suppose that $f:A\rightarrow B$ is the kernel of a strong epimorphism $g:B\rightarrow C$. Let $b\in B$. There is $b'\in B$ such that $g(b)=g(b')$ and $||g(b)||=||b'||$. Now $b-b'\in A$. We claim that $b-b'$ is a closest point to $b$ in $A$. Indeed for any $a\in A$ 
$$||b-(b-b')||=||b'||=||g(b)||=||g(b-a)||\le ||b-a||$$
\end{proof}

So we get a class of kernel-cokernel pairs

\begin{displaymath}
\xymatrix{
0\ar[r] & A\ar[r]^{f} & B\ar[r]^{g} & C\ar[r] & 0
}
\end{displaymath}
where $f$ is a strong monomorphism and $g$ is a strong epimorphism. We denote this class by $\mathpzc{strong}$. We are going to prove the following.

\begin{thm}\label{strongexact}
The collection $\mathpzc{strong}$ of strong kernel-cokernel pairs is an exact structure on both $\textit{Norm}^{\le1,nA}_{k}$ and $\textit{Ban}^{\le1,nA}_{k}$.
\end{thm}

We do this in several steps. It is clear that isomorphisms are strong epimorphisms and strong monomorphisms. It is also clear that the projection $A\oplus B\rightarrow B$ is a strong epimorphism and the inclusion $A\rightarrow A\oplus B$ is a strong monomorphism.

\begin{prop}
Let 
\begin{displaymath}
\xymatrix{
A\ar[r]^{f}\ar[d]^{g} & B\ar[d]^{g'}\\
X\ar[r]^{f'} & Y
}
\end{displaymath}
be a pushout diagram in $\textit{Norm}^{\le1,nA}_{k}$ or $\textit{Ban}^{\le1,nA}_{k}$. If $f$ is a strong monomorphism then so is $f'$.
\end{prop}

\begin{proof}
We shall prove it for $\textit{Norm}^{\le1,nA}_{k}$. The case of $\textit{Ban}^{\le1,nA}_{k}$ is similar. The space $Y$ is isometrically isomorphic to the quotient normed space
$$X\oplus B\big\slash\{-g(a),f(a)\}$$
Let $[(x,b)]$ be an element of $X\oplus B\big\slash\{-g(a),f(a)\}$, and let $a_{b}\in A$ be such that $f(a_{b})$ is a closest point to $b$ in $f(A)$. Consider the element $[(x-g(a_{b}),0)]$ which is in the image of $f'$. We claim that this is a closest point to $[(x,b)]$ in $X\oplus B\big\slash\{-g(a),f(a)\}$. Let $[(x',0)]$ be an element in the image of $f'$. Then 
$$||[(x,b)]-[(x',0)]||=inf_{a\in A}max\{||x-x'+g(a)||,||b-f(a)||\}$$
Now $||b-f(a_{b})||\le ||b-f(a)||$ and $0\le ||x-x'+g(a)||$. 
$$||[(0,b-f(a_{b})]||=||[(x,b)]-[(x-g(a_{b}),0)]||\le ||(0,b-f(a_{b}))||\le ||[(x,b)]-[(x',0)]||$$
for any $x'\in X$. Therefore  $[(x-g(a_{b}),0)]$ is a closest point to $[(x,b)]$ in the image of $f'$.
\end{proof}

\begin{prop}
Let
\begin{displaymath}
\xymatrix{
A\ar[r]^{f'}\ar[d]^{g'} & B\ar[d]^{g}\\
X\ar[r]^{f} & Y
}
\end{displaymath}
be a pullback diagram in $\textit{Norm}^{\le1,nA}_{k}$ or $\textit{Ban}^{\le1,nA}_{k}$. If $f$ is a strong epimorphism then so is $f'$.
\end{prop}

\begin{proof}
$A$ is (isometrically isomorphic to) the subspace $\{(x,b):f(a)=g(b)\}$ of $X\oplus B$, with $f'$ being $(x,b)\mapsto b$. Let $b\in B$ and let $x\in X$ be such that $f(x)=g(b)$ and $||g(b)||=||x||$. Then $(x,b)\in A$ and $f'(x,b)=b$. Moreover $||(x,b)||=\textrm{max}\{||x||,||b||\}$. If $||x||\le ||b||$ then we are done. Suppose $||x||\ge ||b||$. Then $||b||\le ||x||=||g(b)||\le||b||$ so $||x||=||b||$. In either case $||(x,b)||=||b||$.
\end{proof}

It is clear that the composition of strong epimorphisms is a strong epimorphism. To conclude the proof of Theorem \ref{strongexact} let us next show that compositions of strong monomorphisms are strong. More generally we have the following.

\begin{prop}\label{compmonstrict}
Let $(C,d)$ be an ultrametric space and $A\subset B\subset C$ be subspaces. Let $c\in C$. Suppose that $c$ has a nearest point $b_{c}$ in $B$ and that $b_{c}$ has a nearest point $a_{c}\in A$. Then $c$ has a nearest point in $A$.
\end{prop}

\begin{proof}
Let $a\in A$. If $d(a,c)=d(b_{c},c)$ then $a$ is a nearest point to $C$ in $B$ and hence therefore in $A$. Hence we may assume that $d(a,c)>d(b_{c},c)$ for all $a\in A$. In particular $d(b_{c},a)=d(a,c)$. So $d(a_{c},c)=d(b_{c},a_{c})<d(b_{c},a)=d(a,c)$ for all $a\in c$ and $a_{c}$ is a closest point to $c$ in $A$.
\end{proof}

\begin{cor}
The composition of strong monomorphisms $f:A\rightarrow B$ and $g:B\rightarrow C$ in $\textit{Norm}^{\le1,nA}_{k}$, and hence in $\textit{Ban}^{\le1,nA}_{k}$ is a strong monomorphism.
\end{cor}

Now let us establish some properties of this exact structure. The following is clear.

\begin{prop}\label{normquotstrong}
Let $f:A\rightarrow B$ be a strong monomorphism. Then for $[b]\in B\big\slash f(A)$ we have $||[b]||=||b-f(a)||$ where $f(a)$ is a closest point to $b$ in $f(A)$.
\end{prop}

\begin{prop}
In both $\textit{Norm}^{\le1,nA}_{k}$ and $\textit{Ban}^{\le1,nA}_{k}$ products and coproducts preserve strong monomorphisms, strong epimorphisms and kernels. Coproducts preserve cokernels and products preserve cokernels of admissible monomorphisms. In particular they are exact for the strong exact structures.
\end{prop}

\begin{proof}
Let us first prove the claims about products. It suffices to show this for $\textit{Norm}^{\le1,nA}_{k}$. First note that products always commute with kernels. Now let
\begin{displaymath}
\xymatrix{
0\ar[r] & A_{i}\ar[r]^{f_{i}} & B_{i}\ar[r]^{g_{i}} & C_{i}\ar[r] & 0
}
\end{displaymath}
be a strong exact sequence. We write the product sequence
\begin{displaymath}
\xymatrix{
0\ar[r] & A\ar[r]^{f} & B\ar[r]^{g} & C\ar[r] & 0
}
\end{displaymath}
We need to show that this sequence is exact.

Let us show that the map $g$ is a strong epimorphism. Indeed by Proposition \ref{normquotstrong}  $||([b_{i}])||=\textrm{sup}_{i\in I}||b_{i}-f_{i}(a_{i})||$ where $a_{i}$ is such that $f_{i}(a_{i})$ is a closest point to $b_{i}$ in $f_{i}(A_{i})$. Now 
$$||a_{i}||=||f_{i}(a_{i})||=||(f_{i}(a_{i})-b_{i})+b_{i}||\le\textrm{max}\{||f_{i}(a_{i})-b_{i}||,||b_{i}||\}\le ||b_{i}||$$
So $(a_{i})\in A$. Moreover $||(b_{i}-f_{i}(a_{i}))||=||([b_{i}])||$ and $\pi(b_{i}-f_{i}(a_{i}))=([b_{i}])$. 
Now let us show that $f$ is a strong monomorphism. It is clearly an isometry. Let $c=(c_{i})\in\prod_{i\in I}C_{i}$. For each $i$ pick $b_{i}\in B_{i}$ with $g_{i}(b_{i})=c_{i}$ and $||c_{i}||=||b_{i}||$. Then clearly $\textrm{sup}_{i\in I}||b_{i}||=\textrm{sup}_{i\in I}||c_{i}||$. Set $b=(b_{i})\in\prod_{i\in I}B_{i}$. Then $g(b)=c$ and $||c||=||b||$. A sequence $(b_{i}^{n})$ converges to $(b_{i})$ in $\prod_{i\in I}B_{i}$ if and only if $b_{i}^{n}$ converges to $b_{i}$ uniformly $B_{i}$. In particular each $b_{i}^{n}$ converges to $b_{i}$ in $B_{i}$. It follows that the image of $f$ is closed in $B$. Finally let $(b_{i})\in B$ and for each $i$ pick a closest point $f_{i}(a_{i})$ to $b_{i}$ in $f_{i}(A_{i})$. Now
$$||([b_{i}])||=\textrm{sup}_{i\in I}\textrm{inf}_{a_{i}\in A_{i}}||b_{i}-f_{i}(a_{i})||$$
Pick $a_{i}$ such that $f_{i}(a_{i})$ is a closest point to $b_{i}$ in $f_{i}(A_{i})$. Then $||(b_{i})||=\textrm{sup}_{i\in I}||b_{i}-f_{i}(a_{i})||$.
By a computation similar to the previous part of the proof $\textrm{sup}_{i}||a_{i}||\le\textrm{sup}_{i}||b_{i}||<\infty$ and $(a_{i})\in A$. Moreover for any $(\widetilde{a}_{i})\in A$ we have
$$||(b_{i})-f((a_{i}))||=\textrm{sup}_{i\in I}||b_{i}-f_{i}(a_{i})||\le\textrm{sup}_{i\in I}||b_{i}-f_{i}(\widetilde{a}_{i})||=||(b_{i})-(a_{i})||$$
So $f((a_{i}))$ is a closest point to $(b_{i})$ in $f(A)$.\newline 
\\
Finally it is clear that $f$ is a kernel of $g$ and therefore the sequence is exact.\newline
\\
Coproducts always preserve cokernels. It is obvious that coproducts preserve strong epimorphisms in $\textit{Norm}^{\le1,nA}_{k}$ and for $\textit{Ban}^{\le1,nA}_{k}$ the proof is similar to the proof that products preserve strong epimorphisms. It is clear that coproducts preserve kernels. 
\end{proof}

\begin{cor}\label{prodcoprodna}
In $\textit{Norm}^{nA,\le1}_{k}$ and $\textit{Ban}^{nA,\le1}_{k}$ products are admissibly coexact and coproducts are strongly exact for the strong exact structure.
\end{cor}

\subsubsection{Completion Functors}

 There is a completion functor $\textit{Cpl}:\textit{Norm}^{nA}_{k}\rightarrow\textit{Ban}^{nA}_{k}$ which sends a normed space $A$ to its separated completion $\hat{A}$. It is left adjoint to the inclusion functor $\iota:\textit{Ban}^{nA}_{k}\rightarrow\textit{Norm}^{nA}_{k}$. 
 It restricts to a functor $\textit{Cpl}^{\le 1}:\textit{Norm}^{\le1,nA}_{k}\rightarrow\textit{Ban}^{\le1,nA}_{k}$. Again it is left adjoint to the inclusion functor $\iota^{\le1}:\textit{Ban}^{\le1,nA}_{k}\rightarrow\textit{Norm}^{\le1,nA}_{k}$. 

\begin{prop}\label{cplqas}
The functor $\textit{Cpl}$ is exact for the quasi-abelian exact structure.
\end{prop}

\begin{proof}
This is in \cite{dcfapp} 3.1.13 for $k=\C$, but the proof works for any Banach ring. 
\end{proof}

We are going to show the following.

\begin{prop}\label{compboundex}
The functor $\textit{Cpl}^{\le1}$ is exact for the strong exact structure.
\end{prop}

First we need two basic facts about Cauchy sequences in non-Archimedean fields.

\begin{prop}\label{cauchyna}
Let $(a_{n})$ be a sequence in $k$ such that $||a_{n+1}-a_{n}||\rightarrow 0$. Then $(a_{n})$ is a Cauchy sequence.
\end{prop}

\begin{proof}
For any pair $m> n$ we have $||a_{m}-a_{n}||\le\textrm{sup}_{n\le i\le m-1}\{||a_{i+1}-a_{i}||\}$. Let $\delta>0$ and let $N$ be such that $||a_{j+1}-a_{j}||<\delta$ for $j>N$. Then for $m>n>N$ we have $||a_{m}-a_{n}||<\delta$.
\end{proof}
By Lemma 2.19 in \cite{padic} we have

\begin{prop}\label{constantnorm}
Let $(a_{n})$ be a Cauchy sequence in $k$. If $(a_{n})$ does not converge to zero then the sequence $(|a_{n}|)$ is eventually constant.
\end{prop}

Combining these two propositions we  get the following.

\begin{prop}\label{completionstron}
Let 
\begin{displaymath}
\xymatrix{
0\ar[r] & A\ar[r]^{f} & B\ar[r]^{g} & C\ar[r] & 0
}
\end{displaymath}
be an strong exact sequence in $\textit{Norm}^{\le1,nA}_{k}$. Then
\begin{displaymath}
\xymatrix{
0\ar[r] & \hat{A}\ar[r]^{\hat{f}} & \hat{B}\ar[r]^{\hat{g}} & \hat{C}\ar[r] & 0
}
\end{displaymath}
is a strong exact sequence in $\textit{Ban}^{\le1,nA}_{k}$.
\end{prop}

\begin{proof}
By Proposition \ref{cplqas} the complex is a kernel-cokernel pair. Thus it remains to show that $\hat{g}$ is a strong epimorphism. Let $[(c_{n})]$ be a non-zero equivalence class of Cauchy sequences in $C$. By Proposition \ref{constantnorm} we may assume that $||c_{n}||$ is a constant $r$. Pick $\widetilde{b}_{0}$ such that $g(\widetilde{b}_{0})=c_{0}$ and $||b_{0}||=||c_{0}||$.  For each $n+1$ pick $\widetilde{b}_{n+1}$ such that $g(\widetilde{b}_{n+1})=c_{n+1}-c_{n}$ and $||\widetilde{b}_{n+1}||=||c_{n+1}-c_{n}||$. Write $b_{n}=\sum_{k=0}^{n}\widetilde{b}_{n}$. Then $g(b_{n})=c_{n}$. Moreover
$$r=||c_{n}||\le ||b_{n}||\le\textrm{max}_{k\le n}||\widetilde{b}_{n}||=\textrm{max}\{||c_{0}||,\textrm{max}_{1\le k\le n}||c_{k}-c_{k-1}||\}\le r$$
Hence $||b_{n}||=r$. Moreover $||b_{n+1}-b_{n}||=||\widetilde{b}_{n+1}||=||c_{n+1}-c_{n}||\rightarrow 0$, so by Proposition \ref{cauchyna}, $(b_{n})$ is a Cauchy sequence.
\end{proof}

\begin{prop}
For each $r\in\R_{>0}$ the object $k_{r}$ is projective in both $\textit{Norm}^{\le1,nA}_{k}$ and $\textit{Ban}^{\le1,nA}_{k}$. In particular the strong exact structures on both $\textit{Norm}^{\le1,nA}_{k}$ and $\textit{Ban}^{\le1,nA}_{k}$ have enough functorial projectives.
\end{prop}

\begin{proof}
Let us first prove the proposition for $\textit{Norm}^{\le1,nA}_{k}$. It suffices to show that the functor $\textrm{Hom}(k_{r},-):\textit{Norm}^{\le1,nA}_{k}\rightarrow\mathpzc{Ab}$ preserves cokernels. Let $f:A\rightarrow B$ be a strong monomorphism with cokernel $g:B\rightarrow C$. We need to show that the map
$$B_{B}\Bigr(0,\frac{1}{r}\Bigr)\rightarrow B_{C}\Bigr(0,\frac{1}{r}\Bigr)$$
on open balls is an epimorphism. This follows immediately from the definition of strong epimorphism. 

For the second assertion, let $E\in\textit{Norm}^{\le1,nA}_{k}$. Write $\mathcal{P}(E)\defeq\bigoplus_{e\in E}k_{||e||}$. There is a map $\mathcal{P}(E)\rightarrow E$ induced by the isometry 
$$k_{||e||}\rightarrow E,\; \lambda\mapsto\lambda e$$
This is clearly a strong epimorphism.

For $\textit{Ban}^{\le1,nA}_{k}$ we use the fact that $\textit{Cpl}^{\le1}$ is exact and preserves projectives since it is left adjoint to an exact functor.
\end{proof}

 \subsubsection{Smallness Conditions}
 Let $D:\mathcal{I}\rightarrow\textit{Norm}^{\le1,nA}_{k}$ be a diagram with $\mathcal{I}$ a directed category. Write $A_{i}=D(i)$ and $f_{ji}=D(i\le j)$. Suppose that the $f_{ji}$ are isometries. The direct limit $A\defeq\textrm{lim}_{\rightarrow}A_{i}$ is constructed as follows. The underlying vector space $\textrm{lim}_{\rightarrow}A_{i}$ is the direct limit of the underlying vector spaces of the $A_{i}$. Namely it is the disjoint union of the $A_{i}$ quotient by the relation $a_{i}\sim f_{ji}(a_{i})$ for any $j>i$. If $a_{i}\in A_{i}$ and $a_{j}\in A_{j}$ then $[a_{i}]+[a_{j}]\defeq [f_{Ki}(a_{i})+f_{Kj}(a_{j})]$ where $K$ is any upper bound of $i$ and $j$. If $\lambda\in k$ then $\lambda[a_{i}]\defeq[\lambda a_{i}]$. We define a norm on this vector space by $||[a_{i}]||\defeq ||a_{i}||$. This is well-defined because if $a_{i}\sim a_{j}$ then $a_{j}=f_{ji}(a_{i})$, so $||a_{j}||=||a_{i}||$. Clearly $||[a_{i}]||=0$ if and only if $a_{i}=0$ and if $\lambda\in k$ then $||\lambda[a_{i}]||=|\lambda|||a_{i}||$. Finally
\begin{align*}
||[a_{i}]+[a_{j}]||&=|| [f_{Ki}(a_{i})+f_{Kj}(a_{j})]||\\
&=||f_{Ki}(a_{i})+f_{Kj}(a_{j})||\\
&\le\textrm{max}\{||f_{Ki}(a_{i})||,||f_{Kj}(a_{j})||\}\\
&=\textrm{max}\{||a_{i}||,||a_{j}||\}\\
&=\textrm{max}\{||[a_{i}]||,||[a_{j}]||\}
\end{align*}
So this is a non-Archimedean norm. The map $f_{i}:A_{i}\rightarrow A$ sends $a_{i}$ to $[a_{i}]$.

\begin{prop}
The normed space described above is the direct limit in $\textit{Norm}^{\le1,nA}_{k}$.
\end{prop}

\begin{proof}
Let $g_{i}:A_{i}\rightarrow C$ be a cocone from $D$. There is a unique map of vector spaces $g:A\rightarrow C$ such that $g\circ f_{i}=g_{i}$. It remains to show that $g$ is bounded with $||g||\le 1$. Let $[a_{i}]\in A$ with $[a_{i}]=f_{i}(a_{i})$. Then $||g([a_{i}])||=||g_{i}(a_{i})||\le ||a_{i}||=||[a_{i}]||$.
\end{proof}

\begin{cor}\label{cor:colimadquasnex}
\begin{enumerate}
\item
Suppose that for each $j<k$, $f_{kj}:A_{j}\rightarrow A_{k}$ is an admissible monomorphism in the quasi-abelian exact structure on $\textit{Norm}^{\le1,nA}_{k}$. Then for each $i$ the map $A_{i}\rightarrow A$ is an admissible monomorphism in the quasi-abelian exact structure.
\item
Suppose that for each $j<k$, $f_{kj}:A_{j}\rightarrow A_{k}$ is an admissible monomorphism in the strong exact structure on $\textit{Norm}^{\le1,nA}_{k}$. Then for each $i$ the map $A_{i}\rightarrow A$ is an admissible monomorphism in the strong exact structure.
\end{enumerate}
\end{cor}

\begin{proof}
\begin{enumerate}
\item
It is clear from the definition of the norm on $A$ that $f_{i}$ is an isometry. Suppose that $(f_{i}(a_{i}^{n}))$ converges to $[a_{j}]$ with $a_{j}\in A_{j}$. Let $K$ be an upper bound of $i$ and $j$. Then $([f_{Ki}(a_{i}^{n})])$ converges to $[f_{Kj}(a_{j})]$. But by the definition of the norm on $A$ this clearly means that $f_{Ki}(a_{i}^{n})$ converges to $f_{Kj}(a_{j})$ in $A_{K}$. Since $f_{Ki}$ has closed image, $f_{Kj}(a_{j})=f_{Ki}(a_{i})$. Since $f_{Ki}$ is an isometry $(a_{i}^{n})$ converges to $a_{i}$, so $(f_{i}(a_{i}^{n}))$ converges to $f_{i}(a_{i})$, and $f_{i}$ has closed image.
\item
Let $[a_{j}]\in A$ with $a_{j}\in A_{j}$. Let $K$ be an upper bound of $i$ and $j$. The map $f_{Ki}:A_{i}\rightarrow A_{K}$ is a strong monomorphism. Therefore $f_{Kj}(a_{j})$ has a closest point $f_{Ki}(a_{i})$ in $f_{Ki}(A_{i})$. We claim that $[a_{i}]$ is a closest point to $[a_{j}]$ in $f_{i}(A_{i})$. Indeed let $[a'_{i}]\in f_{i}(A_{i})$ with $a'_{i}\in A_{i}$. Then
$$||[a_{j}]-[a'_{i}]||=||[f_{Kj}(a_{j})-f_{Ki}(a'_{i})]||=||f_{Kj}(a_{j})-f_{Ki}(a'_{i})||$$
$$\;\;\;\;\ge ||f_{Kj}(a_{j})-f_{Ki}(a_{i})||=||[a_{j}]-[a_{i}]||$$
\end{enumerate}
\end{proof}

\begin{cor}\label{strongadmonelem}
Let $\mathcal{I}$ be a filtered category. The functors $\varinjlim:\mathpzc{Fun}_{\textbf{AdMon}}(\mathcal{I},\textit{Norm}^{\le 1,nA}_{k})\rightarrow \textit{Norm}^{\le 1,nA}_{k}$ and $\varinjlim:\mathpzc{Fun}_{\textbf{AdMon}}(\mathcal{I},\textit{Ban}^{\le 1,nA}_{k})\rightarrow \textit{Ban}^{\le 1,nA}_{k}$ are exact.
\end{cor}
\begin{proof}
By Proposition \ref{completionstron} it suffices to prove the claim for $\textit{Norm}^{\le1,nA}_{k}$. Let 
\begin{displaymath}
\xymatrix{
0\ar[r] & X\ar[r]^{i}\ar[r] & Y\ar[r]^{p} & Z\ar[r] & 0
}
\end{displaymath}
 be a short exact sequence in $\mathpzc{Fun}(\mathcal{I},\textit{Norm}^{\le 1,nA}_{k})$ where $\mathcal{I}$ is filtered. It suffices to show that the map $lim_{\rightarrow}Y\rightarrow lim_{\rightarrow} Z$ is a strong epimorphism. Fix a class $[z_{i}]\in lim_{\rightarrow} Z$ where $z_{i}\in Z_{i}$. There is some $y_{i}\in Y_{i}$ with $p_{i}(y_{i})=z_{i}$ and $||z_{i}||=||y_{i}||$. But then $p([y_{i}])=[z_{i}]$ and $||[z_{i}]||=||z_{i}||=||y_{i}||=||[y_{i}]||$. 
\end{proof}
This immediately gives the following.
\begin{cor}
The strong exact structures on both $\textit{Norm}^{\le1,nA}_{k}$ and $\textit{Ban}^{\le1, nA}_{k}$ are weakly $\textbf{AdMon}$-elementary. 
\end{cor}

\begin{prop}
For $r\in\R_{>0}$ the objects $k_{r}$ are tiny with respect to the class of admissible monomorphisms in the strong exact structure on $\textit{Norm}^{\le1,nA}_{k}$.
\end{prop}

\begin{proof}
We need to show that for any $r\in\R_{>0}$ the map
$$\textrm{lim}_{\rightarrow}B_{A_{i}}(0,r)\rightarrow B_{\textrm{lim}_{\rightarrow}A_{i}}(0,r)$$
is an isomorphism. It suffices to prove that it is an epimorphism. Let $[a_{i}]\in\textrm{lim}_{\rightarrow}A_{i}$ be such that $||[a_{i}]||=||a_{i}||\le r$. Then $a_{i}\in B_{A_{i}}(0,r)\hookrightarrow \textrm{lim}_{\rightarrow}B_{A_{i}}(0,r)$ maps to $[a_{i}]$. 
\end{proof}

\begin{prop}\label{cocompno}
For any $r\in\R_{>0}$ the $k$-Banach space $k_{r}$ is not tiny with respect to the class of admissible monomorphisms in the strong exact structure (or even the split exact structure). However every object is $\aleph_{1}$-presented.
\end{prop}

\begin{proof}
Consider the sequence
with $X_{i}=k_{r}^{i}$ and the map $X_{i}\rightarrowtail X_{i+1}$ being the inclusion of the first $i$ copies of $k$. The group $\textrm{lim}_{\rightarrow}\textrm{Hom}(k_{r},X_{i})$ is the ascending union of the closed balls in $X_{i}$ of radius $r$, while $\textrm{Hom}(k_{r},\textrm{lim}_{\rightarrow} X_{i})$ is the closed ball of radius $r$ in $\textrm{lim}_{\rightarrow} X_{i}$. The map
$$\textrm{lim}_{\rightarrow}\textrm{Hom}(k_{r},X_{i})\rightarrow \textrm{Hom}(k_{r},\textrm{lim}_{\rightarrow} X_{i})$$
is the obvious inclusion. Consider the example with $X_{i}=k^{\oplus i}$ with $X_{i}\rightarrow X_{i+1}$ being the split injection $k^{i}\rightarrow k^{i+1}$ which is the inclusion of the first $i$ copies of $k$. Then $\textrm{lim}_{\rightarrow}X_{i}$ is the space of sequences in $k$ converging to $0$ with the supremum norm. The group $\textrm{lim}_{\rightarrow}\textrm{Hom}(k_{r},X_{i})$ is the group of finite sequences of norm at most $\frac{1}{r}$, while $\textrm{Hom}(k_{r},\textrm{lim}_{\rightarrow} X_{i})$ is the group of sequences converging to $0$ with norm at most $\frac{1}{r}$. It is clear that for a non-discrete field the map
$$\textrm{lim}_{\rightarrow}\textrm{Hom}(k_{r},X_{i})\rightarrow \textrm{Hom}(k_{r},\textrm{lim}_{\rightarrow} X_{i})$$
is not an epimorphism. The last claim is \cite{adamek1994locally} 1.48.
 \end{proof}
Recall that a Banach space $E$ is said to have the \textbf{Hahn-Banach extension property}  if for every subspace $D$ of $E$, every bounded functional $f:D\rightarrow k$ there is an extension $g:E\rightarrow k$ of $f$ with $||g||=||f||$. 
  
\begin{thm}[\cite{prolla2000topics} Theorem 4.12]
If $k$ is spherically complete then every Banach space over $k$ has the Hahn-Banach extension property.
\end{thm}

\begin{prop}
Let $E$ be a non-zero Banach space with the Hahn-Banach extension property and let $e\in E$. Then there is a Banach space $E'$ and an isometric isomorphism $E\cong E'\oplus k_{||e||}$. In particular if $k$ is spherically complete then there are no non-nonzero tiny objects in $\textit{Ban}^{\le1,nA}_{k}$.
\end{prop}
 
 \begin{proof}
 Let $\Bigr<e\Bigr>$ be the span of $e$ in $E$. The map $f:\Bigr<e\Bigr>\rightarrow k_{||e||}$ sending $e$ to $1$ is an isometric isomorphism with inverse $g$ sending $1$ to $e$. Therefore $f$ extends to a map $\overline{f}:E\rightarrow k_{||e||}$ with $||\overline{f}||=1$. Moreover $\overline{f}\circ g=\textrm{Id}_{k_{||e||}}$. Since $\textit{Ban}^{\le1}_{k_{||e||}}$ is quasi-abelian and in particular weakly-idempotent complete this gives a splitting. 
 \end{proof}

\subsubsection{The Monoidal Structure}

The following is straightforward using that both functors are left adjoints.

\begin{prop}\label{tenscomp}
Consider the functors
$$\textit{Cpl}\circ\otimes_{\pi}:\textit{Norm}^{nA}_{k}\otimes\textit{Norm}^{nA}_{k}\rightarrow\textit{Ban}^{nA}_{k}$$
and
$$\hat{\otimes}_{\pi}\circ\textit{Cpl}\times\textit{Cpl}:\textit{Norm}^{nA}_{k}\otimes\textit{Norm}^{nA}_{k}\rightarrow\textit{Ban}^{nA}_{k}$$
There is a natural isometric isomorphism
$$\phi:\textit{Cpl}\circ\otimes_{\pi}\rightarrow\hat{\otimes}_{\pi}\circ\textit{Cpl}\times\textit{Cpl}$$
In particular we get a natural isomorphism
$$\phi^{\le1}:\textit{Cpl}^{\le1}\circ\otimes_{\pi}\rightarrow\hat{\otimes}_{\pi}\circ\textit{Cpl}^{\le1}\times\textit{Cpl}^{\le1}$$
\end{prop}

\begin{prop}\label{rescalecomp}
Let $r\in\R_{>0}$. Consider the functors
$$(-)_{r}\circ\textit{Cpl}:\textit{Norm}^{nA}_{k}\rightarrow\textit{Ban}^{nA}_{k}$$
and
$$\textit{Cpl}\circ (-)_{r}:\textit{Norm}^{nA}_{k}\rightarrow\textit{Ban}^{nA}_{k}$$
Then there is a natural isometric isomorphism 
$$\zeta:(-)_{r}\circ\textit{Cpl}\rightarrow \textit{Cpl}\circ (-)_{r}$$
 In particular this induces a natural isomorphism of functors
$$\zeta^{\le1}:(-)_{r}\circ\textit{Cpl}^{\le1}\cong\textit{Cpl}^{\le1}\circ(-)_{r}$$
\end{prop}

\begin{prop}\label{rescaletens}
Let $s,r\in\R_{>0}$ and consider the functors
$$(-)_{rs}\circ\otimes_{\pi}:\textit{Norm}^{nA}_{k}\times\textit{Norm}^{nA}_{k}\rightarrow\textit{Norm}^{nA}_{k}$$
and
$$\otimes_{\pi}\circ(-)_{r}\times(-)_{s}:\textit{Norm}^{nA}_{k}\times\textit{Norm}^{nA}_{k}\rightarrow\textit{Norm}^{nA}_{k}$$
Then there is an natural isometric isomorphism
$$\eta:\otimes_{\pi}\circ(-)_{r}\times(-)_{s}\rightarrow (-)_{rs}\circ\otimes_{\pi}$$
\end{prop}
At this point let us make the following remark
\begin{rem}
The rescaling functors are exact for both the quasi-abelian and strong exact structures.
\end{rem} 
By Proposition \ref{rescalecomp} and Proposition \ref{tenscomp} we get

\begin{cor}\label{rescaletensban}
Let $s,r\in\R_{>0}$ and consider the functors
$$(-)_{rs}\circ\otimes_{\pi}:\textit{Ban}^{nA}_{k}\times\textit{Ban}^{nA}_{k}\rightarrow\textit{Ban}^{nA}_{k}$$
and
$$\otimes_{\pi}\circ(-)_{r}\times(-)_{s}:\textit{Ban}^{nA}_{k}\times\textit{Ban}^{nA}_{k}\rightarrow\textit{Ban}^{nA}_{k}$$
Then there is an natural isometric isomorphism
$$\eta:\otimes_{\pi}\circ(-)_{r}\times(-)_{s}\rightarrow (-)_{rs}\circ\otimes_{\pi}$$
\end{cor}

\begin{cor}
Projective objects in $\textit{Norm}^{\le1,nA}_{k}$ an $\textit{Ban}^{\le1,nA}_{k}$ are flat in both the quasi-abelian and strong exact structures.
\end{cor}

\begin{proof}
By Proposition \ref{rescaletens} and Corollary \ref{rescaletensban}, we only need to note that tensoring with $k$ is the identity functor and hence is exact.
\end{proof}

\begin{cor}
The tensor product of projective objects in $\textit{Norm}^{\le1,nA}_{k}$ an $\textit{Ban}^{\le1,nA}_{k}$ is a projective object.
\end{cor}

\begin{proof}
It suffices to prove this in $\textit{Norm}^{\le1,nA}_{k}$ for objects of the form $k_{r}$ with $r\in\R_{>0}$. But $k_{r}\otimes k_{s}\cong k_{rs}$ which is projective.
\end{proof}
We summarise this section with the following result.
\begin{thm}
$Ban_{k}^{\le1}$ is a projectively monoidal weakly $\textbf{AdMon}$-elementary exact category which is $\aleph_{1}$-presentable but not $\aleph_{0}$-presentable. $Norm_{k}^{\le1}$ is a monoidal $\textbf{AdMon}$-elementary exact category.
\end{thm}

\section{The Exact Category with Respect to a Generator}
\subsection{The Gabriel-Popescu Theorem}
Let $\mathpzc{G}$ be a Grothendieck abelian category with a set $\{G_{i}\}_{i\in\mathcal{I}}$ of generators. By Proposition 3.3 in \cite{gillespie2016derived} there is the structure of an exact category on $\mathpzc{G}$, called the $G$-exact structure, whereby a null sequence
$$0\rightarrow X\rightarrow Y\rightarrow Z\rightarrow 0$$
is exact if 
$$0\rightarrow Hom(G_{i},X)\rightarrow Hom(G_{i},Y)\rightarrow Hom(G_{i},Z)\rightarrow 0$$
is exact for all $i\in\mathcal{I}$. By Theorem A in \cite{gillespie2016derived} projective resolutions always exist in the $G$-exact structure, and as a consequence the projective model structure exists on $Ch(\mathpzc{G})$, where $\mathpzc{G}$ is equipped with the $G$-exact structure. Gillespie further studies the case in which $G_{i}$ are tiny. In this case the $G$-exact structure has exact coproducts. Moreover since the $G$-exact category has enough projectives, it has exact products, and Gillespie prove the injective model also exists
Let us prove a generalisation of this for what we shall call strongly efficient exact categories. The terminology here is following the notion of an efficient exact category in \cite{vst2012exact}, which we will discuss in the next section. 
An exact category $\mathpzc{E}$ is said to be \textbf{strongly efficient} if
\begin{defn}
\begin{enumerate}
\item
It is locally presentable.
\item
It has an admissible generator $G$.
\item
It has kernels.
\item
Kernels commute with filtered colimits.
\end{enumerate}
\end{defn}
The proof of the following is a minor generalisation of \cite{stacks-project} Lemma 19.14.1.
\begin{prop}
Let $G$ be an admissible generator of a weakly idempotent complete exact category $\mathpzc{E}$, and let $X$ be an object of $\mathpzc{E}$. The natural map
$$\pi:\bigoplus_{\alpha\in Hom(G,X)}G\rightarrow X$$
is an admissible epimorphism.
\end{prop}
\begin{proof}
There is some set $I$ and an admissible epimorphism $\gamma:\bigoplus_{i\in\mathcal{I}}G\rightarrow X$. For each $i$ let $\gamma_{i}$ be the restriction of $\gamma$ to the copy of $G$ in the coproduct indexed by $i$. We may regard $\bigoplus_{i\in\mathcal{I}}G$ as a summand of $\bigoplus_{\alpha\in Hom(G,X)}G$ where the copy of $G$ in $\bigoplus_{i\in\mathcal{I}}G$ indexed by $i$ corresponds to the copy of $G$ in $\bigoplus_{\alpha\in Hom(G,X)}G$ indexed by $\gamma_{i}$. Moreover the composition
$$\bigoplus_{i\in\mathcal{I}}G\rightarrow\bigoplus_{\alpha\in Hom(G,X)}G\rightarrow X$$
is $\gamma$. This is an admissible epimorphism, so by the obscure axiom $\pi$ is an admissible epimorphism.
\end{proof}
There is also an easy adaptation of \cite{stacks-project} Lemma 19.14.1.
\begin{lem}
Let $\mathpzc{E}$ be a strongly efficient exact category, and let $G$ be an admissible generator. Then the functor $R:Hom(G,-):\mathpzc{E}\rightarrow {}_{End(G)}\mathpzc{Mod}$ has a left adjoint such that the counit $LR(G)\rightarrow G$ is an isomorphism.
\end{lem}
\begin{proof}
The fact that the functor has a left adjoint $L$ follows immediately from the adjoint functor theorem. However it will be helpful to give an explicit construction. Fix a functorial projective resolution functor $Q:\mathpzc{Mod}_{End(G,G)}\rightarrow Ch_{\ge0}(\mathpzc{Mod}_{End(G,G)})$ such that for each module $M$ and each $n\ge0$, $Q(M)_{n}$ is a free $End(G,G)$-module. For $M$ an $End(G,G)$-module define $L(M)\defeq coker(Q(M)_{1}\rightarrow Q(M)_{0})$. Let $Q(M)_{0}=\bigoplus_{i\in I}End(G,G)$ and $Q(M)_{1}=\bigoplus_{j\in J}End(G,G)$. This determines a map $\bigoplus_{i\in I}G\rightarrow \bigoplus_{j\in J}G$. Define $L(M)$ to be the cokernel of this map. For $X\in\mathpzc{E}$ we have an exact sequence
$$0\rightarrow Hom(L(M),X)\rightarrow\prod_{i\in\mathcal{I}}R(X)\rightarrow\prod_{j\in\mathcal{J}}R(X)$$
which is isomorphic to the exact sequence
$$0\rightarrow K\rightarrow Hom(\bigoplus_{i\in I}End(G,G),R(X))\rightarrow Hom(\bigoplus_{j\in J}End(G,G),R(X))$$
But $K$ is isomorphic to $Hom(M,R(X))$ due to the exact sequence
$$\bigoplus_{j\in J}End(G,G)\rightarrow\bigoplus_{i\in I}End(G,G)\rightarrow M\rightarrow0$$
Hence $L$ is left adjoint to $R$. Moreover this construction in fact shows that the counit $LR(G)\rightarrow G$ is an isomorphism.
\end{proof}
Using the previous two results, the following Lemma can be proven exactly as in \cite{stacks-project} Lemma 19.14.2.
\begin{lem}
Let $f:M\rightarrow R(X)$ be injective in $\mathpzc{Mod}_{End(G,G)}$. Then the adjoint map $f':L(M)\rightarrow X$ is injective.
\end{lem}

This allows us to prove a version of the Gabriel-Popescu Theorem, mimicking the proof of \cite{stacks-project} Theorem 19.14.3.
\begin{thm}[Gabriel-Popescu]
Let $\mathpzc{E}$ be a strongly efficient exact category, and let $G$ be an admissible generator. Then the functor $R:Hom(G,-):\mathpzc{E}\rightarrow {}_{End(G)}\mathpzc{Mod}$ is fully faithful with a left adjoint.
\end{thm}
\begin{proof}
We need to show that the counit $v:L\circ R(X)\rightarrow X$  is an ismorphism. It is adjoint to the identity map $R(X)\rightarrow R(X)$ so it is an monomorphism. It remains to show that it is an admissible epimorphism for any $X$. Let $\pi:\bigoplus_{\alpha\in Hom(G,X)}G\rightarrow X$ be the canonical map, which is an admissible epimorphism. Each $\alpha\in Hom(G,X)$ is an element of $R(X)$ and therefore determines a map $\alpha':G\cong LR(G)\rightarrow LR(X)$ such that $v\circ\alpha'=\alpha$. This gives a map $\pi':\bigoplus_{\alpha\in Hom(G,X)}G\rightarrow LR(X)$ such that $v\circ\pi'=\pi$. Since $\pi$ is an admissible epimorphism, $v$ must be as well.
\end{proof}
Using this Gabriel-Popescu theorem, the proofs of Proposition 3.3 and Corollary 3.4 in \cite{gillespie2016derived} also work for this more general setting.
\begin{cor}
Let $\mathpzc{E}$ be a strongly efficient exact category with generator $G$. Then the collection of $G$-exact sequences defines an exact structure on $\mathpzc{E}$. Moreover $G$ is a projective generator for the $G$-exact structure on $\mathpzc{E}$.
\end{cor}
\begin{example}
This is almost how the strong exact structures on $Norm_{k}^{nA,\le1}$ and $Ban_{k}^{nA,\le1}$ arise. Indeed the category $sNorm_{k}^{nA,\le1}$ of \textbf{semi-normed} non-Archimedean $k$-modules with non-expanding bounded morphisms is strongly efficient as an exact category with the quasi-abelian exact structure. The strong exact structure is the one determined by the generator $\bigoplus^{\le1}_{r>0}k_{r}$. Both $Norm_{k}^{nA,\le1}$ and $Ban_{k}^{nA,\le1}$ are extension-closed subcategories of $sNorm_{k}^{nA,\le1}$, and therefore inherit the strong exact structure from $sNorm_{k}^{nA,\le1}$.
\end{example}
 \section{Exact Categories of Grothendieck Type}
  In \cite{vst2012exact} \v{S}t'ov\'{i}\v{c}ek introduces efficient exact categories, and exact categories of Grothendieck type. As suggested by the name, the latter are supposed to generalise Grothendieck abelian categories. In particular, \v{S}t'ov\'{i}\v{c}ek shows that such categories have enough injectives, and in fact that the injective model structure (explained in the next chapter) exists on unbounded complexes in such categories under a mild extra assumption. Let us recall the definitions.
\begin{defn}[ \cite{vst2012exact}, Definition 3.11]
A exact category $\mathpzc{E}$ is said to be \textbf{efficient} if
\begin{enumerate}
\item
$\mathpzc{E}$ is weakly idempotent complete.
\item
Transfinite compositions of admissible monics exist and are admissible monics.
\item
Every object of $\mathpzc{E}$ is compact relative to the class of all admissible monics.
\item
$\mathpzc{E}$ has an admissible generator.
\end{enumerate}
\end{defn}
 \begin{defn}[\cite{vst2012exact} 3.2]
Let $\mathpzc{E}$ be an exact category. $\mathpzc{E}$ is said to be \textbf{deconstructible in itself} if there is a set $\mathcal{S}$ of objects of $\mathpzc{E}$, such that every object $X$ of $\mathpzc{E}$ has a presentation
$$X\cong\textrm{colim}_{\alpha<\lambda}S_{\alpha}$$
where $\lambda$ is a regular ordinal, $S_{\alpha}\in\mathcal{S}$, and for $\alpha+1<\lambda$, $S_{\alpha}\rightarrow S_{\alpha+1}$ is an admissible monomorphism.
 \end{defn}
 \begin{defn}[\cite{vst2012exact} Definition 3.11]
 An exact category $\mathpzc{E}$ is said to be \textbf{of Grothendieck type} if it is efficient and deconstructible in itself.
 \end{defn}
 \begin{lem}[\cite{vst2012exact} Corollary 5.9.]\label{lem:grexenoughinj}
 Let $\mathpzc{E}$ be an exact category of Grothendieck type. Then $\mathpzc{E}$ has enough functorial injectives.
 \end{lem}
 \subsection{The $\lambda$-Pure Exact Structure}
 Following an argument of \cite{estrada2017pure} in the case of the $\otimes$-pure exact structure, we show that  exact structures on locally presentable additive categories in which filtered colimits of split exact sequences are exact are of Grothendieck type. This is closely related to work of \cite{krause2012approximations} where the additive category is assumed to be locally finitely presentable, and of \cite{gillespie2016derived}, for the the case $\mathpzc{E}$ is a Grothendieck abelian category. 
 Let $\mathpzc{E}$ be a locally $\lambda$-presentable additive category for some regular cardinal $\lambda$. Following \cite{adamek1994locally}, say that a null sequence
 $$0\rightarrow X\rightarrow Y\rightarrow Z\rightarrow 0$$
 in $\mathpzc{E}$ is $\lambda$-\textbf{pure exact} if for any $\lambda$-presentable object $H$, the sequence of abelian groups
  $$0\rightarrow Hom(H,X)\rightarrow Hom(H,Y)\rightarrow Hom(H,Y)\rightarrow 0$$
  is exact. 
  \begin{lem}[\cite{adamek1994locally}, 2.30/ \cite{gillespie2016derived} Proposition A.1]
  A map $f:X\rightarrow Y$ in $\mathpzc{E}$ is a $\lambda$-pure monomorphism if and only if it is a $\lambda$-directed colimit of split monomorphisms. 
  \end{lem}
  \begin{prop}[\cite{estrada2017pure}, Proposition 2.9]\label{prop:colimitslambdapure}
  Let $P:\mathcal{I}\rightarrow\mathpzc{E}$ be a $\lambda$-directed system. Then the map $\bigoplus_{i\in\mathcal{I}}P_{i}\rightarrow\textrm{colim}_{\mathcal{I}}P_{i}$ is a $\lambda$-pure epimorphism.
  \end{prop}
  \begin{proof}
Proposition 2.9 in \cite{estrada2017pure} is stated and proven in the context that $\mathpzc{E}$ is a Grothendieck abelian category, however the proof also works for any additive category which is cocomplete and finitely complete. Let us give a short independent proof. The direct sum  $\bigoplus_{i\in\mathcal{I}}P_{i}$ is also a $\lambda$-directed system, and it follows that for $H$ $\lambda$-presentable, $\bigoplus_{i\in\mathcal{I}}Hom(H,P_{i})\cong Hom(H,\bigoplus_{i\in\mathcal{I}}P_{i})$, and $\textrm{colim}_{\mathcal{I}}Hom(H,P_{i})\cong Hom(H,\textrm{colim}_{\mathcal{I}}P_{i})$. The map $\bigoplus_{i\in\mathcal{I}}Hom(H,P_{i})\rightarrow\textrm{colim}_{\mathcal{I}}Hom(H,P_{i})$ is an epimorphism of abelian groups, so $\bigoplus_{i\in\mathcal{I}}P_{i}\rightarrow\textrm{colim}_{\mathcal{I}}P_{i}$ is a $\lambda$-pure epimorphism.
  \end{proof}
The following is proven in \cite{krause2012approximations} when the additive category is assumed to be locally finitely presentable, and of \cite{gillespie2016derived}, for the the case $\mathpzc{E}$ is a Grothendieck abelian category. 
  \begin{prop}
  The collection of all $\lambda$-pure exact sequences defines an exact structure on $\mathpzc{E}$ with enough projectives.
  \end{prop}
  \begin{proof}
Clearly split exact sequences are $\lambda$-pure exact, and $\lambda$-pure epimorphisms are stable by composition and pullback. Let $f:X\rightarrow Y$ be a $\lambda$-pure monomorphism. Then $f$ is $\lambda$-directed colimit of split monomorphisms. Since pushouts commute with colimits, and pushouts of split monomorphisms are split monomorphisms, any pushout of $f$ is a $\lambda$-directed colimit of split monomorphisms, and therefore a $\lambda$-pure monomorphism. Moreover, a composition of $\lambda$-directed colimits of pure monomorphisms, is again a $\lambda$-directed colimit of split monomorphisms. 
To show that the $\lambda$-pure exact structure has enough projectives, let $X$ be an object of $\mathpzc{E}$. We may write $X$ as a $\lambda$-directed colimit $X\cong\textrm{colim}_{\mathcal{I}}H_{i}$ where $H_{i}$ is $\lambda$-presentable. But then by Proposition \ref{prop:colimitslambdapure} the map $\bigoplus_{i\in\mathcal{I}}H_{i}\rightarrow X$ is a $\lambda$-pure epimorphism. Moreover $\bigoplus_{i\in\mathcal{I}}H_{i}$ is projective.
  \end{proof}
  \begin{lem}\label{lem:Grothendieckinj}
  Let $\mathpzc{E}$ be a locally $\lambda$-presentable additive category. Suppose that $(\mathpzc{E},\mathcal{Q})$ is an exact category and that for any filtered category $\mathcal{I}$, the functor $\textrm{lim}_{\rightarrow}:\mathpzc{Fun}(\mathcal{I},\mathpzc{E})$ is exact (in particular if $\mathpzc{E}$ is elementary). Then $(\mathpzc{E},\mathcal{Q})$ is an exact category of Grothendieck type.
  \end{lem}
  \begin{proof}
  Let $(\mathpzc{E},\mathcal{Q}_{\lambda})$ denote the $\lambda$-pure exact structure on $\mathpzc{E}$. Since $\mathpzc{E}$ is weakly elementary, and split monomorphisms are exact in any exact category, the identity functor $(\mathpzc{E},\mathcal{Q}_{\lambda})\rightarrow(\mathpzc{E},\mathcal{Q})$ is exact. Thus a generator for $(\mathpzc{E},\mathcal{Q}_{\lambda})$ is still a generator for $(\mathpzc{E},\mathcal{Q})$. Similarly, the fact that  $(\mathpzc{E},\mathcal{Q}_{\lambda})$ is deconstructible in intself implies that $(\mathpzc{E},\mathcal{Q})$ is. The only missing axiom is that transfinite compositions of admissible monomorphisms are admissible monomorphisms, but we have in fact assumed something stronger.
  \end{proof}
  As we will explain later, the next corollary is essentially a generalisation of \cite{gillespie2016derived} Corollary 5.12 (and the part of Proposition 5.6 which says that the $G$-exact structure has enough injectives).
    \begin{cor}\label{cor:GexactGRothendieck}
  Let $\mathpzc{E}$ be a strongly efficient exact category, with $\mathpzc{E}$ locally $\lambda$-presentable. Suppose that $\mathpzc{E}$ has a set of generators $\{G_{i}\}_{i\in\mathcal{I}}$ with each $G_{i}$ tiny. Then the $G$-exact structure on $\mathpzc{E}$ is of Grothendieck type, where $G=\bigoplus_{i\in\mathcal{I}}G_{i}$.
  \end{cor}
  \begin{proof}
  Since the $G_{i}$ are tiny, they are in particular $\lambda$-presentable. It follows that $\lambda$-pure exact sequences are exact in the $G$-exact structure. Also since the $G_{i}$ are tiny, the $G$-exact structure is elementary.
  \end{proof}
  This next result then generalises \cite{estrada2017pure} Lemma 3.6. 
  \begin{cor}
  Let $\mathpzc{E}$ be a locally $\lambda$-presentable additive category equipped with a closed monoidal structure $(\otimes,\underline{Hom})$. Suppose that $(\mathpzc{E},\mathcal{Q})$ is an exact category and that for any filtered category $\mathcal{I}$, the functor $\textrm{lim}_{\rightarrow}:\mathpzc{Fun}(\mathcal{I},\mathpzc{E})$ is exact. Then the $\otimes$-pure exact structure on $(\mathpzc{E},\mathcal{Q}_{\otimes})$ is of Grothendieck type.  
  \end{cor}
  \begin{proof}
  The functor $X\otimes(-)$ commutes with colimits, and sends  split exact sequences to split exact sequences. Thus it sends $\lambda$-pure exact sequences to $\lambda$-pure exact sequences.
  \end{proof}
  
\section{The Split Exact Structure}

We conclude with an example of a category which has no small generating set whatsoever but is still weakly $\textbf{AdMon}$-elementary.
Let $\mathpzc{E}$ be an additive category and endow it with the split exact structure. Let us prove the following useful lemma. Its proof, as well as the proof of Corollary \ref{trans} later, is a minor generalisation of \cite{vst2012exact} Proposition 5.7 (the proof is essentially the same, and is also similar to the corresponding statement for abelian categories - Lemma 6.2 in \cite{hovey}). 
\begin{lem}\label{transgeneral}
Let $\mathpzc{E}$ be an exact category. Let $\lambda$ be an ordinal and 
\begin{displaymath}
\xymatrix{
0\ar[r]& Y\ar[r]^{f} &Z\ar[r]^{p} &X\ar[r] & 0
}
\end{displaymath}
be a short exact sequence in $Fun(\lambda,\mathpzc{E})$ with $X$ in $Fun^{cocont}_{adm}(\lambda,\mathpzc{E})$. Suppose that $$Ext^{1}(\textrm{Coker}(X_{\alpha}\rightarrow X_{\alpha+1}),Y_{\alpha+1})=0$$
 for any $\alpha\le\lambda$ such that $\alpha+1\le\lambda$. Then
$$0\rightarrow lim_{\rightarrow}Y\rightarrow lim_{\rightarrow}Z\rightarrow lim_{\rightarrow} X\rightarrow 0$$
is a split exact sequence.
\end{lem}

\begin{proof}
For $\alpha\le\beta$ in $\lambda$, denote the corresponding maps in the diagrams by $x_{\alpha,\beta}:X_{\alpha}\rightarrow X_{\beta}$, $z_{\alpha,\beta}:Z_{\alpha}\rightarrow Z_{\beta}$, and $y_{\alpha,\beta}:Y_{\alpha}\rightarrow Y_{\beta}$. The proof is by transfinite induction. If $\lambda=0$ or $\lambda$ is a successor ordinal then the claim is clear. Now let $\lambda=lim_{\rightarrow_{\alpha<\lambda}}\alpha$ be a limit ordinal. Suppose the claim has been proven for all $\alpha<\lambda$. 
Let $\alpha<\lambda$. Since $\textrm{Ext}^{1}(X_{\alpha},Y_{\alpha})=0$ there is some splitting $t_{\alpha}:X_{\alpha}\rightarrow Z_{\alpha}$ of $p_{\alpha}$. We are going to modify the $t_{\alpha}$ to $s_{\alpha}$ so that they are compatible, i.e. $z_{\alpha,\gamma}s_{\alpha}=s_{\gamma}x_{\alpha,\gamma}$ for all $\alpha\le\gamma$. We will do this by transfinite induction.

Set $s_{0}=t_{0}$. If $\gamma$ is a limit ordinal let $s_{\gamma}:\textrm{colim}_{\alpha<\gamma}X_{\alpha}\rightarrow Z_{\gamma}$ be the map whose restriction to $Z_{\alpha}$ is $z_{\alpha,\gamma}s_{\alpha}$, where $z_{\alpha,\gamma}:Z_{\alpha}\rightarrow Z_{\gamma}$ is the transfinite composition of the continuous functor $\gamma\rightarrow\mathpzc{E}$, $\beta\mapsto Z_{\beta}$. Then by construction $z_{\alpha,\gamma}s_{\alpha}=s_{\gamma}x_{\alpha,\gamma}$. Now for the successor case $\gamma=\alpha+1$. Suppose we have constructed $s_{\alpha}$. Let us construct $s_{\alpha+1}$. We have
\begin{align*}
p_{\alpha+1}(z_{\alpha,\alpha+1}s_{\alpha}-t_{\alpha+1}x_{\alpha,\alpha+1})&=x_{\alpha,\alpha+1}\circ p_{\alpha}\circ s_{\alpha}-p_{\alpha+1}\circ t_{\alpha+1}\circ x_{\alpha,\alpha+1}\\
&=x_{\alpha,\alpha+1}-x_{\alpha,\alpha+1}\\
&=0
\end{align*}
Therefore there is a map $h:X_{\alpha}\rightarrow Y_{\alpha+1}$ such that $f_{\alpha+1}h=z_{\alpha,\alpha+1}s_{\alpha}-t_{\alpha+1}x_{\alpha,\alpha+1}$. Since $x_{\alpha,\alpha+1}:X_{\alpha}\rightarrow X_{\alpha+1}$ is an admissible monic and $\textrm{Ext}^{1}(\textrm{Coker}(x_{\alpha,\alpha+1}),Y_{\alpha+1})=0$, the long exact $\textrm{Ext}$ sequence implies that there is a map $g:X_{\alpha+1}\rightarrow Y_{\alpha+1}$ such that $gx_{\alpha,\alpha+1}=h$. Let $s_{\alpha+1}=t_{\alpha+1}+f_{\alpha+1}g$. Then clearly $s_{\alpha+1}$ is a section of $p_{\alpha+1}$. Moreover
\begin{align*}
s_{\alpha+1}x_{\alpha,\alpha+1} &=t_{\alpha+1}\circ x_{\alpha,\alpha+1}+f_{\alpha+1} g\circ x_{\alpha,\alpha+1}\\
&=z_{\alpha,\alpha+1}s_{\alpha}-f_{\alpha+1}\circ h+ f_{\alpha+1}\circ h\\
&=z_{\alpha,\alpha+1}s_{\alpha}
\end{align*}
as required.
\end{proof}

\begin{prop}
If $\mathpzc{E}$ is an additive category with kernels and countable coproducts then the split exact structure is weakly $(\aleph_{0};\textbf{AdMon})$-elementary.
\end{prop}

\begin{proof}
It has kernels by assumption. It trivially has enough projectives since every object is projective. The fact $(\aleph_{0};\textbf{AdMon})$-colimits exist and are exact follows from Lemma \ref{transgeneral}.
\end{proof}

For $\mathpzc{E}=\mathpzc{Ab}$ this category has no small generating set by \cite{christensen2002quillen} Section 5.4. This can be generalised to other exact structures defined by projective classes as discussed in the same paper.

\chapter{Model Structures on Exact Categories}\label{chmodelexact} 

In this chapter we discuss model structures on categories of chain complexes in exact categories. We give very general conditions under which unbounded complexes are equipped with the projective model structure. We also investigate when such a model structure is monoidal and satisfies the monoid axiom, which will be crucial for studying homotopical algebra in exact categories in the next section. Finally we generalise the Dold-Kan correspondence. 
\section{Cotorsion Pairs}
In \cite{hovey}, Hovey introduced the notion of a \textbf{compatible model structure} on an abelian category. He showed that there is a 1-1 correspondence between such model structures and purely homological data now known as \textbf{Hovey triples}. Gillespie noticed that this correspondence generalises to weakly idempotent complete exact categories, and explains in \cite{gillespie} how to adapt Hovey's proofs. Moreover in \cite{vst2012exact}, \v{S}t'ov\'{i}\v{c}ek  explained the relationship between cotorsion pairs and compatible weak factorisation systems. In the next two subsections we will recall some of Hovey's/ Gillespie's/\v{S}t'ov\'{i}\v{c}ek's  results both for the reader's convenience and because we will need many of the individual propositions later anyway. For basic facts about weak factorisation systems and model structures in general see Appendix \ref{modelapp}.

Let $\mathcal{S}$ be a class of objects in an exact category $\mathpzc{E}$. We shall denote by $\;^{\perp}\mathcal{S}$ the class of all objects $X$ such that $\textrm{Ext}^{1}(X,S)=0$ for all $S\in\mathcal{S}$, and by $\mathcal{S}^{\perp}$ the class of all objects $X$ such that $\textrm{Ext}^{1}(S,X)=0$ for all $S\in\mathcal{S}$. The class $\mathcal{S}^{\perp}$ is called the class of $\mathcal{S}$-\textbf{injectives}, and the class $\;^{\perp}\mathcal{S}$ is called the class of $\mathcal{S}$-\textbf{projectives}. For a class of objects $\mathcal{P}$ in an exact category $\mathpzc{E}$, let $\textbf{AdMon}_{\mathcal{P}}$ denote the class of admissible monomorphisms whose cokernels is in $\mathcal{P}$.

\begin{cor}\label{trans}
Let $\mathpzc{E}$ be an exact category. Let $\mathcal{S}$ be a class of objects in $\mathpzc{E}$, and let $\mathfrak{L}=\;^{\perp}\mathcal{S}$. Then $\mathfrak{L}$ is closed under retracts and finite extensions. If $\mathpzc{E}$ is cocomplete and weakly $(\lambda,\textbf{AdMon}_{\mathfrak{L}})$-elementary then $\mathfrak{L}$ is closed under $\lambda$-transfinite extensions.
\end{cor}

\begin{proof}
First we show  that $\mathfrak{L}$ is closed under retracts. Note that it is sufficient to show that for a given $Y\in\mathpzc{E}$, the collection of objects $X$ such that $\textrm{Ext}^{1}(X,Y)=0$ is closed under retracts. Let $X$ be such that $\textrm{Ext}^{1}(X,Y)=0$ and let $X'$ be a retract of $X$. Then $X'$ is a summand of $X$, and so $\textrm{Ext}^{1}(X',Y)=0$.

Let us show that $\mathfrak{L}$ is closed under transfinite extensions. Again it is sufficient to show that for any object $Y\in\mathpzc{E}$ the collection of all $X$ with $\textrm{Ext}^{1}(X,Y)=0$ is closed under transfinite extensions and retracts. Let $\lambda$ be an ordinal $X:\lambda\rightarrow\mathpzc{E}$ an object of $Fun^{cocont}_{adm}(\lambda,\mathpzc{E})$. Let
\begin{displaymath}
\xymatrix{
0\ar[r] & Y\ar[r]^{f} & N\ar[r]^{p} & lim_{\rightarrow}X\ar[r] & 0
}
\end{displaymath}
represent an element of $\textrm{Ext}^{1}(lim_{\rightarrow}X,Y)$. For each $\beta\in\lambda$, pull this short exact sequence back through the map $x_{\beta}:X_{\beta}\rightarrow lim_{\rightarrow}X$. For $\alpha\le\gamma$ in $\lambda$ we get a commutative diagram.
\begin{displaymath}
\xymatrix{
0\ar[r] & Y\ar[r]^{f} & N\ar[r]^{p} & X_{\beta}\ar[r] & 0\\
0\ar[r] & Y\ar[r]^{f_{\gamma}}\ar@{=}[u] & N_{\gamma}\ar[r]^{p_{\gamma}}\ar[u]^{k_{\gamma}} & X_{\gamma}\ar[u]^{j_{\gamma}}\ar[r] & 0\\
0\ar[r] & Y\ar@{=}[u]\ar[r]^{f_{\alpha}} & N_{\alpha}\ar[u]^{k_{\alpha,\gamma}}\ar[r]^{p_{\alpha}} & X_{\alpha}\ar[r]\ar[u]^{j_{\alpha,\gamma}} & 0
}
\end{displaymath}
Each of the sequences
$$0\rightarrow Y\rightarrow N_{\gamma}\rightarrow X_{\gamma}\rightarrow 0$$
is exact. Moreover for each $\alpha+1\le\lambda$ the map $k_{\alpha,\alpha+1}$ has the same cokernel as $j_{\alpha,\alpha+1}$. In particular it is in $\textbf{AdMon}_{\mathfrak{L}}$. It suffices to show that the map $lim_{\rightarrow}N_{\alpha}\rightarrow N$ is an isomorphism. Then we may apply Lemma \ref{transgeneral}. Indeed by assumption we have a commutative diagram of exact sequences
\begin{displaymath}
\xymatrix{
0\ar[r] & Y\ar[r]\ar[d] & lim_{\rightarrow}N_{\alpha}\ar[r]\ar[d]& lim_{\rightarrow}X_{\alpha}\ar[r]\ar[d] & 0\\
0\ar[r] & Y\ar[r] &  N\ar[r] & X\ar[r] & 0
}
\end{displaymath}
The first and last vertical maps are isomorphisms, so the middle one is as well. 
\end{proof}
Let us now define cotorsion pairs, and discuss their relation with weak factorisation systems. We shall largely follow the notation of \cite{vst2012exact}.

\begin{defn}
Let $\mathpzc{E}$ be an exact category. A \textbf{cotorsion pair} on $\mathpzc{E}$ is a pair of families of objects $(\mathfrak{L},\mathfrak{R})$ of $\mathpzc{E}$ such that $\mathfrak{L}=\;^{\perp}\mathfrak{R}$ and $\mathfrak{R}=\mathfrak{L}^{\perp}$.
\end{defn}

\begin{defn}
A cotorsion pair $(\mathfrak{L},\mathfrak{R})$ is said to have \textbf{enough (functorial) projectives} if for every $X\in\mathpzc{E}$ there is an admissible epic $p:Y\rightarrow X$, (functorial in $X$), such that $Y\in\mathfrak{L}$ and $\textrm{Ker}(p)\in\mathfrak{R}$. It is said to have \textbf{enough (functorial) injectives} if, for every $X$, there is an admissible monic $i:X\rightarrow Z$, (functorial in $X$), such that  $Z\in\mathfrak{R}$ and $\textrm{Coker}(i)\in\mathfrak{L}$. A cotorsion pair is said to be \textbf{(functorially) complete} if it has enough (functorial) projectives and enough (functorial) injectives.
\end{defn}

\begin{example}\label{cotorsionproj}
Our main example is the projective cotorsion pair. Let $\mathpzc{E}$ be an exact category. Let $\textbf{Proj}(\mathpzc{E})$ denote the collection of projective objects of $\mathpzc{E}$. Then $(\textbf{Proj}(\mathpzc{E}),\textbf{Ob}(\mathpzc{E}))$ is clearly a cotorsion pair. Suppose that $\mathpzc{E}$ has enough (functorial) projectives. Then the cotorsion pair $(\textbf{Proj}(\mathpzc{E}),\textbf{Ob}(\mathpzc{E}))$ is trivially (functorially) complete.
\end{example}

\begin{notation}
Let $\mathpzc{E}$ be an exact category and $(\mathcal{L},\mathcal{R})$ a weak factorisation system on $\mathpzc{E}$. Denote by $
\textrm{Coker}\mathcal{L}$ the collection of objects $L$ such $L$ is a cokernel of some map in, $\mathcal{L}$ and by $\textrm{Ker}\mathcal{R}$ the collection of objects $R$ such that $R$ is the kernel of some map in $\mathcal{R}$.

Given classes of objects $\mathfrak{A},\mathfrak{B}$ in $\mathpzc{E}$, we denote by $\textrm{Infl}(\mathfrak{A})$ the class of admissible monics with cokernel in $\mathfrak{A}$ and by $\textrm{Defl}(\mathfrak{B})$ the class of admissible epics with kernel in $\mathfrak{B}$.
\end{notation}

\begin{defn}
Let $\mathpzc{E}$ be an exact category. A  weak factorisation system $(\mathcal{L},\mathcal{R})$ on $\mathpzc{E}$ is said to be \textbf{compatible} if
\begin{enumerate}
\item
$f\in\mathcal{L}$ if and only if $f$ is an admissible monic and $0\rightarrow\textrm{Coker}(f)$ belongs to $\mathcal{L}$.
\item
$f\in\mathcal{R}$ if and only if $f$ is an admissible epic and $\textrm{Ker}(f)\rightarrow 0$ belongs to $\mathcal{R}$.
\end{enumerate}
\end{defn}

The following result is Theorem 5.13 in \cite{vst2012exact}.

\begin{thm}
Let $\mathpzc{E}$ be an exact category. Then
$$(\mathcal{L},\mathcal{R})\mapsto(\textrm{Coker}\mathcal{L},\textrm{Ker}\mathpzc{R})\textrm{ and }(\mathfrak{A},\mathfrak{B})\mapsto(\textrm{Infl}(\mathfrak{A}),\textrm{Defl}(\mathfrak{B}))$$
define mutually inverse bijective mappings between compatible weak factorisation systems and complete cotorsion pairs. The bijections restrict to mutually inverse mappings between compatible functorial weak factorisation systems and functorially complete cotorsion pairs.
\end{thm}

\subsection{Compatible Model Structures}

Having described the bijection between cotorsion pairs and compatible weak factorisation systems, we now introduce compatible model structures, and explain how they too correspond to purely homological data. Remember that we do not assume our model categories are complete or cocomplete.

Let $(\mathcal{C},\mathcal{F},\mathcal{W})$ be a model structure on an additive category $\mathpzc{E}$.

\begin{defn}
Let $\mathpzc{E}$ be an exact category. Let $(\mathcal{C},\mathcal{F},\mathcal{W})$ be a model structure on $\mathpzc{E}$. The model structure is said to be \textbf{compatible} if both $(\mathcal{C}\cap\mathcal{W},\mathcal{F})$ and $(\mathcal{C},\mathcal{F}\cap\mathcal{W})$ are compatible weak factorisation systems.
\end{defn}

Let us now define the corresponding homological data. As for abelian categories, we will call a subcategory $\mathpzc{D}$ of an exact category $\mathpzc{E}$ \textbf{thick} if whenever
$$0\rightarrow A\rightarrow B\rightarrow C\rightarrow 0$$
is a short exact sequence and two of the objects are in $\mathpzc{D}$, then so is the third.

\begin{defn}
A \textbf{Hovey triple} on an exact category $\mathpzc{E}$ is a triple $(\mathfrak{C},\mathfrak{W},\mathfrak{F})$ of collections of objects of $\mathpzc{E}$ such that the full subcategory on $\mathfrak{W}$ is closed under retracts and thick, and that both $(\mathfrak{C},\mathfrak{F}\cap\mathfrak{W})$ and $(\mathfrak{C}\cap\mathfrak{W},\mathfrak{F})$ are complete cotorsion pairs. 
\end{defn}

We then have the following theorem (Theorem 6.9 in \cite{vst2012exact}). It is originally due to \cite{hovey} in the abelian case and \cite{gillespie} Theorem 3.3 in the more general exact case.

\begin{thm}\label{weakmon}
Let $\mathpzc{E}$ be a weakly idempotent complete exact category. Then there is a bijection between Hovey triples and compatible model structures. The correspondence assigns to a Hovey triple $(\mathfrak{C},\mathfrak{W},\mathfrak{F})$ the model structure $(\mathcal{C},\mathcal{W},\mathcal{F})$ such that
\begin{enumerate}
\item
$\mathcal{C}=\textrm{Infl}(\mathfrak{C})$
\item
$\mathcal{F}=\textrm{Defl}(\mathfrak{F})$
\item
$\mathcal{W}$ consists of morphisms of the form $p\circ i$ where $i\in\textrm{Infl}(\mathfrak{W})$ and $p\in\textrm{Defl}(\mathfrak{W})$.
\end{enumerate}
\end{thm}

Before we move on let us mention a more general notion than compatible model structures. We will need it when we consider the projective model structure on $ Ch_{\ge0}(\mathpzc{E})$.

\begin{defn}
Let $\mathpzc{E}$ be an exact category. A model structure $(\mathcal{C},\mathcal{F},\mathcal{W})$ on $\mathpzc{E}$ is said to be \textbf{left pseudo-compatible} if there are classes of objects $\mathfrak{C}$ and $\mathfrak{W}$  such that
\begin{enumerate}
\item
The full subcategory on $\mathfrak{W}$ is thick.
\item
 A map $f$ is in $\mathcal{C}$ (resp. $\mathcal{C}\cap\mathcal{W}$) if and only if it is an admissible monic with cokernel in $\mathfrak{C}$ (resp. $\mathfrak{C}\cap\mathfrak{W}$). 
 \item
 An admissible monic is in $\mathcal{W}$ if and only if its cokernel is in $\mathfrak{W}$.
 \end{enumerate}
 As before $\mathfrak{C}$/ $\mathfrak{W}$ /$\mathfrak{C}\cap\mathfrak{W}$ are called the \textbf{cofibrant }/\textbf{trivial}/  \textbf{trivially cofibrant} objects. The pair $(\mathfrak{C},\mathfrak{W})$ will be called the \textbf{left homological Waldhausen pair} of the model structure. Dually one defines \textbf{right pesudo-compatible} model structures and  \textbf{right homological Waldhausen pairs}
 \end{defn}

The terminology comes from the notion of a Waldhausen category, in which classes of weak equivalences and cofibrations are specified. Clearly any compatible model structure is left pseudo-compatible.

\subsection{Small Cotorsion Pairs and Cofibrant Generation}

When working with model categories, it is computationally convenient that they be generated by suitably compact objects (see Appendix \ref{modelapp} for exactly what we mean here). In this section, we study what conditions on the cotorsion pairs defining a compatible model structure guarantee that the model structure is cofibrantly small. The material here is adapted from \cite{hovey} \S 6 to exact categories. Gillespie has also done this in \cite{gillespie2016derived} Section 5.2 for the $G$-exact structure on a Grothendieck abelian category.

\begin{defn}
Let $\mathpzc{E}$ be an exact category. A cotorsion pair $(\mathfrak{L},\mathfrak{R})$ on $\mathpzc{E}$ is said to be \textbf{cogenerated by a set} if there is a set of objects $\mathcal{G}$ in $\mathfrak{L}$ such that $X\in\mathfrak{R}$ if and only if $\textrm{Ext}^{1}(G,X)=0$ for all $G\in\mathcal{G}$.
\end{defn}

\begin{defn}
Suppose $\mathpzc{E}$ is an exact category. A cotorsion pair $(\mathfrak{L},\mathfrak{R})$ is said to be \textbf{small} if the following conditions hold
\begin{enumerate}
\item
$\mathfrak{L}$ contains a set of admissible generators of $\mathpzc{E}$.
\item
$(\mathfrak{L},\mathfrak{R})$ is cogenerated by a set $\mathcal{G}$.
\item
For each $G\in\mathcal{G}$ there is an admissible monic $i_{G}$ with cokernel $G$ such that, if $\textrm{Hom}_{\mathpzc{E}}(i_{G},X)$ is surjective for all $G\in\mathcal{G}$, then $X\in\mathfrak{R}$.
\end{enumerate}
The set of $i_{G}$ together with the maps $0\rightarrow U_{i}$ for some generating set $\{U_{i}\}$ contained in $\mathfrak{L}$ is called a set of \textbf{generating morphisms} of $(\mathfrak{L},\mathfrak{R})$.
\end{defn}

There is an easy example. 

\begin{example}\label{projsmall}
Recall the projective cotorsion pair $(\textbf{Proj}(\mathpzc{E}),\textbf{Ob}(\mathpzc{E}))$. Suppose that the category $\mathpzc{E}$ is projectively generated, with $\mathcal{P}$ a generating set of projectives. We claim that in this case the projective cotorsion pair is small. Indeed by assumption $\textbf{Proj}(\mathpzc{E})$ contains a set of generators $\mathcal{P}$. This set trivially cogenerates the cotorsion pair as well. The third condition is also trivial.
\end{example}

We now come to the connection between cofibrantly small model structures and cotorsion pairs. The proof of the following is a straightforward modification of \cite{hovey} Lemma 6.7 and Proposition 5.4 in \cite{gillespie2016derived}.

\begin{lem}\label{cofibgen}
Let $\mathpzc{E}$ be a weakly idempotent complete exact category together with a compatible weak factorisation system $(\mathcal{L},\mathcal{R})$ with corresponding cotorsion pair $(\mathfrak{L},\mathfrak{R})$. If the cotorsion pair is small with generating morphisms $I=\{0\rightarrow U_{i}\}\cup\{i_{G}\}$, then this weak factorisation system is cofibrantly small. If in addition the generating morphisms have $cell(I)$-presented domain, the weak factorisation system is cofibrantly generated.
\end{lem}

\subsection{Cotorsion Pairs on Monoidal Exact Categories}

In this section $(\mathpzc{E},\otimes,k)$ is a monoidal exact category.

We will now study sufficient conditions on cotorsion pairs defining a model category structure so that the resulting structure is monoidal. We generalise  the work of \cite{hovey} \S 7 to exact categories. In fact this has mostly been done in  \cite{vst2012exact}, apart from the monoid axiom.  The theorem below is proven for compatible model structures in \cite{vst2012exact}. However the proof for left pseudo-compatible model structures is formally identical. Note that \cite{vst2012exact} assumes that the monoidal unit $k$ is in fact already cofibrant, which makes the final assumption below automatic. This does not affect the proof.
\begin{thm}\label{exactmonoidal}
Let $\mathpzc{E}$ be a closed symmetric monoidal exact category. Suppose that $\mathpzc{E}$ has a left pseudo-compatible model structure with Waldhausen pair $(\mathfrak{C},\mathfrak{W})$. Suppose the following conditions are satisfied.
\begin{enumerate}
\item
Every cofibration is pure.
\item
If $X,Y\in\mathfrak{C}$ then $X\otimes Y\in\mathfrak{C}$.
\item
If $X,Y\in\mathfrak{C}$ and one of them is in $\mathfrak{W}$, then $X\otimes Y\in\mathfrak{C}\cap\mathfrak{W}$.
\item
 If $C\rightarrow k$ is an acyclic fibration with $C$ in $\mathfrak{C}$, then for any object $X$ of $\mathpzc{E}$, $C\otimes X \rightarrow X$ is a weak equivalence
\end{enumerate}
Then $\mathpzc{E}$ is a monoidal model category.
\end{thm}

Note that Lemma \ref{flatcofib} says that if cofibrant objects are flat then condition 1 in Theorem \ref{exactmonoidal} is automatically satisfied. 

\begin{thm}\label{monoid}
Let $\mathpzc{E}$ be a complete and cocomplete, monoidal exact category. Suppose that $\mathpzc{E}$ has a  left pseudo-compatible model structure satisfying the hypotheses of Theorem \ref{exactmonoidal}. In addition, suppose that the following conditions hold
\begin{enumerate}
\item
If $X\in\mathfrak{C}\cap\mathfrak{W}$ and $Y$ is arbitrary, then $X\otimes Y$ is in $\mathfrak{W}$.
\item
Transfinite compositions of weak equivalences which are also pure monics with cokernel in $\mathfrak{C}\cap\mathfrak{W}\otimes Ob(\mathpzc{E})$ are still weak equivalences.
\end{enumerate}
Then the model structure satisfies the monoid axiom.
\end{thm}

\begin{proof}
The first condition implies that if $i$ is an acyclic cofibration, then $i\otimes Y$ is a weak equivalence. By Propositions \ref{pushoutpure} and the fact that pushouts commute with cokernels any pushout of $i\otimes Y$ is a weak equivalence as well as a pure monic. By the second condition, any transfinite composition of such maps is a weak equivalence. 
\end{proof}

If in $\mathpzc{E}$ transfinite compositions of admissible monics are admissible monics (e.g. if $\mathpzc{E}$ is weakly $\textbf{AdMon}$-elementary) then one can replace the second condition by requiring that the class $\mathfrak{W}$ is closed under transfinite compositions of pure monomorphisms. By this we mean that if $\lambda$ is some ordinal, and $X:\lambda\rightarrow\mathpzc{E}$ a continuous functor such that $0\rightarrow X_{0}$ is a weak equivalence, and for each $i<j$ in $\lambda$ the map $X_{i}\rightarrow X_{j}$ is a pure monic which is also a weak equivalence, then $X_{\lambda}$ is in $\mathfrak{W}$. (This is the condition used in \cite{hovey} Theorem 7.4). Since
$\mathfrak{W}$ forms a thick subcategory and $X_{0}\rightarrow X_{\lambda}$ is an admissible monic, this is equivalent to the cokernel of the map $X_{0}\rightarrow X_{\lambda}$ being in $\mathfrak{W}$ which in turn is equivalent to $X_{0}\rightarrow X_{\lambda}$ being a weak equivalence.

\section{Model Structures on Chain Complexes}\label{mscc}

Generalising results of \cite{Gillespie2} and and \cite{gillespie2016derived}, in this section we describe a method for constructing compatible model structures on categories of chain complexes $ Ch_{*}(\mathpzc{E})$ from cotorsion pairs on $\mathpzc{E}$. Note that what we describe below will not always produce a model structure. However we will show in the next chapter that it does in the case that $\mathpzc{E}$ has enough projectives, and the cotorsion pair is the projective one (Example \ref{cotorsionproj}).   First we define the collections of objects which will be candidates for the (trivially) fibrant and (trivially) cofibrant objects.

\begin{defn}
Let $(\mathfrak{L},\mathfrak{R})$ be a cotorsion pair on an exact category $\mathpzc{E}$. Let $X\in Ch(\mathpzc{E})$ be a chain complex.
\begin{enumerate}
\item
$X$ is called an $\mathfrak{L}$ complex if it is acyclic and $Z_{n}X\in\mathfrak{L}$ for all $n$. The collection of all $\mathfrak{L}$ complexes is denoted $\widetilde{\mathfrak{L}}$.
\item
$X$ is called an $\mathfrak{R}$ complex if it is acyclic and $Z_{n}X\in\mathfrak{R}$ for all $n$. The collection of all $\mathfrak{R}$ complexes is denoted $\widetilde{\mathfrak{R}}$.
\item
$X$ is called a $K$-$\mathfrak{L}$ complex if $\textbf{Hom}(X,B)$ is exact whenever $B$ is an $\mathfrak{R}$ complex. The class of all $K$-$\mathfrak{L}$ complexes is denoted $K\mathcal{L}$.
\item
$X$ is called a $K$-$\mathcal{R}$ complex if $\textbf{Hom}(A,X)$ is exact whenever $A$ is an $\mathfrak{L}$ complex. The class of all $K$-$\mathcal{L}$ complexes is denoted $K\mathcal{R}$.
\item
$X$ is called a $dg\mathfrak{L}$ complex if $X$ is a $K$-$\mathfrak{L}$ complex and $X_{n}\in\mathfrak{L}$ for each $n\in\mathbb{Z}$. The collection of all $dg\mathfrak{L}$ complexes is denoted $\widetilde{dg\mathfrak{L}}$.
\item
$X$ is called a $dg\mathfrak{R}$ complex if $X$ is a $K$-$\mathfrak{R}$ complex and $X_{n}\in\mathfrak{R}$ for each $n\in\mathbb{Z}$. The collection of all $dg\mathfrak{R}$ complexes is denoted $\widetilde{dg\mathfrak{R}}$.
\end{enumerate}
\end{defn}

\begin{notation}
We define the collections $\widetilde{\mathfrak{L}},\widetilde{\mathfrak{R}},\widetilde{dg\mathfrak{L}},\widetilde{dg\mathfrak{R}},K\mathfrak{L},K\mathfrak{R}$ similarly in the categories $ Ch_{*}(\mathpzc{E})$ for  $*\in\{\ge0,\le0,+,-,b\}$. We will use the same notation for these collections irrespective of which category of chain complexes we are working in.
\end{notation}

\begin{rem}
In $ Ch_{*}(\mathpzc{E})$ for $*\in\{+,-,\ge0,b,\emptyset\}$ all of the above classes are closed under shifts $[n]$ for $n\le0$. For $*\in\{+,-,\le0,b,\emptyset\}$ they are closed under shifts $[n]$ for $n\ge0$.
\end{rem}

Let us start to populate these collections. We first make the following easy observation.

\begin{prop}
Let $X$ be an $\mathfrak{R}$-complex. Then $X_{n}\in\mathfrak{R}$ for each $n$.
\end{prop}

\begin{proof}
For each $n$ we have a short exact sequence
$$0\rightarrow Z_{n}X\rightarrow X_{n}\rightarrow Z_{n-1}X\rightarrow0$$
and $Z_{n}X,Z_{n-1}X\in\mathfrak{R}$. By Corollary \ref{trans} $\mathfrak{R}$ is closed under extensions.
\end{proof}

With this in hand the result belows generalises immediately from \cite{Gillespie2} Lemma 3.4.

\begin{lem}\label{dgbounded}
\begin{enumerate}
\item
Bounded below complexes with entries in $\mathfrak{L}$ are $dg\mathfrak{L}$ complexes.
\item
Bounded above complex with entries in $\mathfrak{R}$ are $dg\mathfrak{R}$ complexes.
\end{enumerate}
\end{lem}

Gillespie's crucial Proposition 3.6 in \cite{Gillespie2} does not hold in arbitrary exact categories. However some of it can be salvaged to give the following two results.

\begin{prop}\label{almostcotor}
Let $(\mathfrak{L},\mathfrak{R})$ be a cotorsion pair in an exact category $\mathpzc{E}$. Then in $ Ch_{*}(\mathpzc{E})$ for $*\in\{+,-,b,\emptyset\}$ we have
\begin{enumerate}
\item
$\widetilde{dg\mathfrak{L}}=\;^{\perp}\widetilde{\mathfrak{R}}$.
\item
$\widetilde{dg\mathfrak{R}}=\widetilde{\mathfrak{L}}^{\perp}$
\item
$\widetilde{\mathfrak{R}}\subseteq(\widetilde{dg\mathfrak{L}})^{\perp}$
\item
$\widetilde{\mathfrak{L}}\subseteq\;^{\perp}(\widetilde{dg\mathfrak{R}})$
\item
Suppose $\mathpzc{E}$ has enough $\mathfrak{L}$-objects. Let $X\in (\widetilde{dg\mathfrak{L}})^{\perp}$ be good. Then $X$ is an $\mathfrak{R}$-complex. 
\item
Suppose $\mathpzc{E}$ has enough $\mathfrak{R}$-objects. Let $X\in \;^{\perp}dg(\widetilde{\mathfrak{R}})$ be cogood. Then $X$ is an $\mathfrak{L}$-complex.
\end{enumerate}
\end{prop}

\begin{proof}
Parts 1) and 3) are easily seen to generalise to the exact case from the Gillespie's proof.
\begin{enumerate}
\item
Let $X\in\;^{\perp}\widetilde{\mathfrak{R}}$. Then $\textrm{Ext}^{1}(X,B)=0$ whenever $B$ is an $\mathfrak{R}$ complex. In particular $\textrm{Ext}^{1}_{dw}(X,B)=0$. Hence $\textbf{Hom}(X,B)$ is exact whenever $B$ is an $\mathfrak{R}$ complex by Corollary \ref{dw}. It remains to show $X_{n}\in\mathfrak{L}$. Let $B\in\mathfrak{R}$. By Lemma \ref{extstuff} we have
$$\textrm{Ext}^{1}(X_{n},B)=\textrm{Ext}^{1}(X,D^{n+1}(B))=0$$
since $D^{n+1}(B)\in\widetilde{\mathfrak{R}}$. So $X_{n}\in\mathfrak{L}$, and $\;^{\perp}\widetilde{\mathfrak{R}}\subset \widetilde{dg\mathfrak{L}}$. 
Now let $X\in \widetilde{dg\mathfrak{L}}$. Since the entries of $X$ are in $\mathfrak{L}$, for any $Y\in\widetilde{\mathfrak{R}}$, any short exact sequence
\begin{displaymath}
\xymatrix{
0\ar[r] & Y\ar[r] & Z\ar[r] & X\ar[r] & 0
}
\end{displaymath}
is split exact in each degree. But also $\textrm{Ext}^{1}_{dw}(X,Y)=0$. Hence, any sequence as above must be split exact, i.e. $\textrm{Ext}^{1}(X,Y)=0$.
\item
This is dual to the previous part.
\item
Let $X\in\widetilde{\mathfrak{R}}$ and $A\in \widetilde{dg\mathfrak{L}}$. Note that since $X_{n}\in\mathfrak{R}$,  $\textrm{Ext}^{1}(X,A)=\textrm{Ext}^{1}_{dw}(X,A)$. Now since $\textbf{Hom}(A,X)$ is exact, $\textrm{Ext}_{dw}^{1}(X,A)=0$.
\item
This is dual to the previous part.
\item
Let us show that $X$ is acyclic. We will again use Proposition \ref{goodtrick}. Let $n$ be such that $d_{n}$ has a kernel. Since we have enough $\mathfrak{L}$-objects, we may choose an admissible epic $f'_{n}:A'\rightarrow Z_{n}X$ for some $A'\in\mathfrak{L}$. By Lemma \ref{extstuff} this induces a map $f:S^{n}(A')\rightarrow X$. Now $\textrm{Ext}^{1}_{dw}(S^{n}(A')[-1],X)\subset\textrm{Ext}^{1}(S^{n}(A')[-1],X)=0$ by assumption. Hence $f$ is homotopic to $0$. Applying Proposition \ref{homotopygood} the map $d'_{n+1}:X_{n+1}\rightarrow Z_{n}X$ is an admissible epic. By Proposition \ref{goodtrick} $X$ is acyclic. To see that $Z_{n}X\in\mathfrak{R}$, we note that since $X$ is acyclic, we have for any $A\in\mathfrak{L}$,
$$\textrm{Ext}^{1}_{\mathpzc{E}}(A,Z_{n}X)\cong\textrm{Ext}^{1}(S^{n}(A),X)=0$$
Since $(\mathfrak{L},\mathfrak{R})$ is a cotorsion pair, $Z_{n}X\in\mathfrak{R}$. 
Hence $X\in\widetilde{\mathfrak{R}}$ and so $(\widetilde{dg\mathfrak{L}})^{\perp}\subseteq\widetilde{\mathfrak{R}}$.
\item
The proof for the second part is dual.
\end{enumerate}
\end{proof}

We also have the following.

\begin{prop}
Let $*\in\{\ge0\}$, and let $(\mathfrak{L},\mathfrak{R})$ be a cotorsion pair in $\mathpzc{E}$ with enough $\mathfrak{L}$-objects. Then $\widetilde{dg\mathfrak{L}}=\;^{\perp}\widetilde{\mathfrak{R}}$ and $\widetilde{\mathfrak{R}}=(\widetilde{dg\mathfrak{L}})^{\perp}$. Dually, if the cotorsion pair has enough $\mathfrak{R}$-objects, then for $*\in\{\le0\}$ $\widetilde{dg\mathfrak{R}}=\widetilde{\mathfrak{L}}^{\perp}$ and $\widetilde{\mathfrak{L}}=\;^{\perp}\widetilde{dg\mathfrak{R}}$.
\end{prop}

\begin{proof}
The proofs of parts (3) and (5) in the previous proposition go through here, as does the proof that $\widetilde{dg\mathfrak{L}}\subset\;^{\perp}\widetilde{\mathfrak{R}} $. Now let $X\in\;^{\perp}\widetilde{\mathfrak{R}}$. The same proof as in part (1) of the previous proposition shows that each $X_{n}$ must be an object in $\mathfrak{L}$. Thus $X$ is a bounded below complex of objects in $\mathfrak{L}$ and hence a $\widetilde{dg\mathfrak{L}}$ complex. 
\end{proof}

We get an immediate corollary.

\begin{cor}\label{bootcotor}
Let $(\mathfrak{L},\mathfrak{R})$ be a cotorsion pair on an exact category $\mathpzc{E}$ with enough $\mathfrak{L}$-objects and enough $\mathfrak{R}$-objects.
\begin{enumerate}
\item
$(\widetilde{dg\mathfrak{L}},\widetilde{\mathfrak{R}})$ is a cotorsion pair on $ Ch_{\ge0}(\mathpzc{E})$ and $ Ch_{+}(\mathpzc{E})$. If $\mathpzc{E}$ has all kernels  then it is a cotorsion pair on $ Ch(\mathpzc{E})$.
\item
$(\widetilde{\mathfrak{L}},\widetilde{dg\mathfrak{R}})$ is a cotorsion pair on $ Ch_{\le0}(\mathpzc{E})$ and $ Ch_{-}(\mathpzc{E})$. If $\mathpzc{E}$ has all cokernels then it is a cotorsion pair in $ Ch(\mathpzc{E})$.
\item
 $(\widetilde{\mathfrak{L}},\widetilde{dg\mathfrak{R}})$ and $(\widetilde{dg\mathfrak{L}},\widetilde{\mathfrak{R}})$ are cotorsion pairs in $ Ch_{b}(\mathpzc{E})$.
\item
If $\mathpzc{E}$ has all kernels and cokernels, in particular if $\mathpzc{E}$ is quasi-abelian, then $(\widetilde{\mathfrak{L}},\widetilde{dg\mathfrak{R}})$ and $(\widetilde{dg\mathfrak{L}},\widetilde{\mathfrak{R}})$ are cotorsion pairs in $ Ch(\mathpzc{E})$.
\end{enumerate}
\end{cor}

\subsection{Existence of dg-Model Structures}
The hope now is that the class $\mathfrak{W}$ of acyclic complexes satisfies
$$\widetilde{\mathfrak{L}}=\widetilde{dg\mathfrak{L}}\cap\mathfrak{W},\;\; \widetilde{\mathfrak{R}}=\widetilde{dg\mathfrak{R}}\cap\mathfrak{W}$$
and that the cotorsion pairs $(\widetilde{dg\mathfrak{L}},\widetilde{\mathfrak{R}})$ and $(\widetilde{\mathfrak{L}},\widetilde{dg\mathfrak{R}})$ are functorially complete. It is not at all clear that this will be the case. In \cite{yang2014question} it is shown that for a complete and cocomplete abelian category it is always the case. We suspect this result can be easily adapted for complete or cocomplete exact categories satisfying a similar condition. In fact in private communication \cite{moreauprivate} Timoth\'{e}e Moreau has shown that it holds whenever $\mathpzc{E}$ has exact products and exact coproducts. For the projective model structure we will only require that complexes of projectives satisfy the AB4-k axiom for some $k$ (though we suspect Moreau's work can be easily generalised to arbitrary complete cotorsion pairs on $\mathpzc{E}$ satisfying some version of the $AB4-k$ axiom.) In this section we will give some partial results about the induced cotorsion pairs on $Ch(\mathpzc{E})$ without any exactness assumptions on products or coproducts. First we need acyclic complexes to form a thick subcategory.

\begin{prop}\label{acycthick}
Let $\mathpzc{E}$ be a weakly idempotent complete exact category. Then for $*\in\{\ge0,\le0,+,-,b\}$ the full subcategory on $\mathfrak{W}$ is a thick subcategory of $ Ch_{*}(\mathpzc{E})$. If $\mathpzc{E}$ has all kernels then this is also true for $*=\{\emptyset\}$.
\end{prop}

\begin{proof}
One may assume that $\mathpzc{E}$ is abelian by passing to a left abelianisation for $*\in\{\ge0,+,b\}$, (or a right abelianisation for $*\in\{\le0,-\}$). The result in this case follows from the long exact sequence on homology.
\end{proof}

\begin{rem}
Timoth\'{e}e Moreau has pointed out a nice proof to us which does not use abelianisations.
\end{rem}

It turns out  that we always have the inclusions $\widetilde{\mathfrak{L}}\subset \widetilde{dg\mathfrak{L}}\cap\mathfrak{W}$, and $\widetilde{\mathfrak{R}}\subset \widetilde{dg\mathfrak{R}}\cap\mathfrak{W}$. This follows from the next result, which is an easy modification of the proof of \cite{Gillespie2} Lemma 3.9.

\begin{lem}\label{LRhomotop}
Every chain map from an $\mathfrak{L}$ complex to an $\mathfrak{R}$ complex is homotopic to $0$.
\end{lem}

\begin{cor}
Let $(\mathfrak{L},\mathfrak{R})$ be a cotorsion pair in an exact category. Then $\widetilde{\mathfrak{L}}\subset \widetilde{dg\mathfrak{L}}\cap\mathfrak{W}$, and $\widetilde{\mathfrak{R}}\subset \widetilde{dg\mathfrak{R}}\cap\mathfrak{W}$.
\end{cor}

In order to have any chance of getting the reverse inclusion, we'll need the cotorsion pair on $\mathpzc{E}$ to be hereditary. The following definition and the subsequent proposition are immediate generalisations of \cite{rozas} \S 1.2.3 from abelian categories to exact categories.

\begin{defn}
A cotorsion pair $(\mathfrak{L},\mathfrak{R})$ is said to be \textbf{hereditary} if
$$\textrm{Ext}^{i}(A,B)=0$$
for any $A\in\mathfrak{L},B\in\mathfrak{R}$ and $i\ge 1$.
\end{defn}

\begin{example}
Clearly the projective cotorsion pair is hereditary.
\end{example}

\begin{prop}
Let $(\mathfrak{L},\mathfrak{R})$ be a hereditary cotorsion pair on a weakly idempotent complete  exact category $\mathpzc{E}$. Then
\begin{enumerate}
\item
$\mathfrak{L}$ is resolving. That is $\mathfrak{L}$ is closed under taking kernels of admissible epis.
\item
$\mathfrak{R}$ is coresolving. That is $\mathfrak{R}$ is closed under taking cokernels of admissible monics.
\end{enumerate}
If $\mathpzc{E}$ has enough $\mathfrak{R}$-projectives then $(\mathfrak{L},\mathfrak{R})$ is hereditary if and only if $\mathfrak{L}$ is resolving. Dually if $\mathpzc{E}$ has enough $\mathfrak{L}$-injectives then $(\mathfrak{L},\mathfrak{R})$ is hereditary if and only if $\mathfrak{R}$ is coresolving.
\end{prop}

With this result in hand  \cite{Gillespie2} Theorem 3.12 generalises immediately to the exact setting.

\begin{thm}
Let $(\mathfrak{L},\mathfrak{R})$ be a hereditary cotorsion pair in an exact category $\mathpzc{E}$. If $\mathpzc{E}$ has enough projectives then in $ Ch_{*}(\mathpzc{E})$ for $*\in\{\ge0,+,\emptyset\}$, $\widetilde{dg\mathfrak{R}}\cap\mathfrak{W}=\widetilde{\mathfrak{R}}$. If $\mathpzc{E}$ has enough injectives then in $ Ch_{*}(\mathpzc{E})$ for $*\in\{\le0,-,\emptyset\}$ $\widetilde{dg\mathfrak{L}}\cap\mathfrak{W}=\widetilde{\mathfrak{L}}$. In particular, if $\mathpzc{E}$ has enough projectives and injectives, then the induced cotorsion pairs on $\mathpzc{E}$ are compatible.
\end{thm}

Lemma 3.14 in \cite{Gillespie2}, which partially handles the case in which we may not have enough injectives or projectives also passes essentially unaffected to exact categories.

\begin{lem}\label{dgcompat}
Let $\mathpzc{E}$ be an exact category and $(\mathfrak{L},\mathfrak{R})$ a cotorsion pair on $\mathpzc{E}$. Consider the categories $ Ch_{*}(\mathpzc{E})$ for any $*\in\{\ge0,\le0,+,-,b,\emptyset\}$.
\begin{enumerate}
\item
If $(\widetilde{\mathfrak{L}},\widetilde{dg\mathfrak{R}})$ is a cotorsion pair with enough projectives and $\widetilde{dg\mathfrak{R}}\cap\mathfrak{W}=\widetilde{\mathfrak{R}}$ then $\widetilde{dg\mathfrak{L}}\cap\mathfrak{W}=\widetilde{\mathfrak{L}}$.
\item
If $(\widetilde{dg\mathfrak{L}},\widetilde{\mathfrak{R}})$ is a cotorsion pair with enough injectives and $\widetilde{dg\mathfrak{L}}\cap\mathfrak{W}=\widetilde{\mathfrak{L}}$ then $\widetilde{dg\mathfrak{R}}\cap\mathfrak{W}=\widetilde{\mathfrak{R}}$.
\end{enumerate}
\end{lem}

These next two results partially deal with the issue of completeness.

\begin{lem}\label{acker}
Let $\mathpzc{E}$ be an exact category.
Suppose
\begin{displaymath}
\xymatrix{
0\ar[r] & B\ar[r] & A\ar[r]^{f} & X\ar[r] & 0
}
\end{displaymath}
is a short exact sequence of complexes in the degree-wise exact structure with both $B$ and $\textrm{cone}(f)$ either good or cogood. Then $B$ is acyclic if and only if $f$ is a quasi-isomorphism.
\end{lem}

\begin{proof}
Let $I:\mathpzc{E}\rightarrow\mathpzc{A}$ a suitable abelianisation . Then by \cite{Weibel} Exercise 1.59 there is a long exact sequence 
\begin{displaymath}
\xymatrix{
\ldots\ar[r] & H_{n+1}(\textrm{Ker}(I(f_{\bullet})))\ar[r] & H_{n}(\textrm{cone}(I(f_{\bullet})))\ar[r] & H_{n}(\textrm{Coker}(I(f_{\bullet})))\ar[r] &   \ldots\\
&\ar[r] & H_{n-1}(\textrm{Ker}(I(f)))\ar[r]&\ldots
}
\end{displaymath}
If $f_{\bullet}$ is a quasi-isomorphism, then $\textrm{cone}(I(f_{\bullet}))$ is acyclic. It is also an admissible epimorphism, so $\textrm{Coker}(I(f_{\bullet}))=0$. Hence $\textrm{Ker}(I(f_{\bullet}))=I(B)$ is acyclic.

If $B$ is acyclic then again since $\textrm{Coker}(I(f_{\bullet}))=0$, $H_{n}(\textrm{cone}(I(f_{\bullet})))=0$ as well. Thus $I(f)$ is a quasi-isomorphism, so $f$ is as well.

\end{proof}

\begin{prop}\label{dgcomplete}
Let $(\mathfrak{L},\mathfrak{R})$ be a functorially complete cotorsion pair on a weakly idempotent complete exact  category $\mathpzc{E}$. Then any complex $X_{\bullet}$ of $ Ch_{*}(\mathpzc{E})$ where $*\in\{\ge0,+\}$ admits a resolution by an object $L_{\bullet}$ of $\widetilde{dg\mathfrak{L}}$ whose kernel is an acyclic complex $R_{\bullet}$ with $R_{n}\in\mathfrak{R}$. In particular the cotorsion pair $(dg\widetilde{\textbf{Proj}},\widetilde{\textbf{Ob}})$ on both $ Ch_{\ge0}(\mathpzc{E})$ and $ Ch_{+}(\mathpzc{E})$  has enough functorial projectives. 
\end{prop}

\begin{proof}
Let $X_{\bullet}$ be an object of $ Ch_{*}(\mathpzc{E})$ where $*\in\{\ge0,+\}$. By an easy adaptation of the proof of Lemma \ref{enoughres}, one can find a (functorial) quasi-isomorphism $f_{\bullet}:L_{\bullet}\rightarrow X_{\bullet}$ with each $L_{n}$ an object of $\mathfrak{L}$, which is an admissible epimorphism, and whose kernel is a complex $R_{\bullet}$ with $R_{n}\in\mathfrak{R}$. Now $L_{\bullet}$ is a $dg\mathfrak{L}$ complex by Lemma \ref{dgbounded}. By Lemma \ref{acker} $R_{\bullet}$ is acyclic.
\end{proof}

\subsection{Properties of dg-Model Structures}

\begin{defn}
Let $\mathpzc{E}$ be a weakly idempotent complete exact category and $(\mathfrak{L},\mathfrak{R})$ a cotorsion pair on $\mathpzc{E}$. 
\begin{enumerate}
\item
We say that $(\mathfrak{L},\mathfrak{R})$ is $dg_{\ge0}$-compatible if $(\widetilde{dg\mathfrak{L}},\widetilde{\mathfrak{R}})$ is a functorially complete cotorsion pair on $Ch_{\ge0}(\mathpzc{E})$, $\mathfrak{W}\cap \widetilde{dg\mathfrak{L}}=\widetilde{\mathfrak{L}}$ and the model structure whose cofibrations are $\textrm{Infl}(\widetilde{dg\mathfrak{L}})$, and whose acyclic cofibrations are $\textrm{Infl}(\widetilde{\mathfrak{L}})$ exists on $Ch_{\ge0}(\mathpzc{E})$.
\item
We say that $(\mathfrak{L},\mathfrak{R})$ is $dg_{\le0}$-compatible if $(\widetilde{\mathfrak{L}},\widetilde{dg\mathfrak{R}})$ is a functorially complete cotorsion pair on $Ch_{\le0}(\mathpzc{E})$, $\mathfrak{W}\cap \widetilde{dg\mathfrak{R}}=\widetilde{\mathfrak{R}}$ and the model structure whose fibrations are $\textrm{Defl}(\widetilde{dg\mathfrak{R}})$, and whose acyclic fibrations are $\textrm{Defl}(\widetilde{\mathfrak{R}})$ exists on $Ch_{\le0}(\mathpzc{E})$.
\item
For $*\in\{b,+,-\emptyset\}$ we say that $(\mathfrak{L},\mathfrak{R})$ is $dg_{*}$-compatible if $(\widetilde{\mathfrak{L}},\widetilde{dg\mathfrak{R}})$ and $(\widetilde{dg\mathfrak{L}},\widetilde{\mathfrak{R}})$ are (functorially) complete cotorsion pairs on $ Ch_{*}(\mathpzc{E})$, $dg\mathfrak{L}\cap\mathfrak{W}=\widetilde{\mathfrak{L}}$, and $dg\mathfrak{R}\cap\mathfrak{W}=\widetilde{\mathfrak{R}}$
\end{enumerate}
\end{defn}

Let us establish some properties of model structures arising from $dg_{*}$-compatible cotorsion pairs. First note the following easy observation.
\begin{prop}\label{truncquillen}
Let $(\mathfrak{L},\mathfrak{R})$ be a cotorsion pair on $\mathpzc{E}$ which has kernels, and which is both $dg_{\ge0}$-compatible and $dg$-compatible. Then with the induced model structures on $Ch_{\ge0}(\mathpzc{E})$ and $Ch(\mathpzc{E})$ the adjunction
$$\adj{i}{Ch_{\ge0}(\mathpzc{E})}{Ch(\mathpzc{E})}{\tau_{\ge0}}$$
is a Quillen adjunction, where $i$ is the natural inclusion functor.
\end{prop}
\begin{prop}\label{adepitrunc}
Let $\mathpzc{E}$ be weakly idempotent complete.
\begin{enumerate}
\item
If $(\mathfrak{L},\mathfrak{R})$ is $dg_{\ge0}$-compatible and $f:X\rightarrow Y$ is a fibration, then $f$ is an admissible epimorphism in each strictly positive degree. If $f$ is an acyclic fibration then $f$ is a quasi-isomorphism and an admissible epimorphism in each degree.
\item
If $(\mathfrak{L},\mathfrak{R})$ is $dg_{\le0}$-compatible and $f:X\rightarrow Y$ is a cofibration, then $f$ is an admissible monomorphism in each strictly negative degree. If $f$ is an acyclic cofibration then $f$ is a quasi-isomorphism and an admissible monomorphism in each degree.
\end{enumerate}
\end{prop}

\begin{proof}
We prove the first claim, the second being dual. First assume that $f$ is a fibration. Then $f$ has the right lifting property with respect to acyclic cofibrations. In particular it has the right lifting property with respect to maps of the form $0\rightarrow D^{n}(L)$ where $n\ge 1$ and $L\in\mathfrak{L}$. Since $(\mathfrak{L},\mathfrak{R})$ is complete, this implies that $f_{n}$ is an admissible epimorphism in each strictly positive degree. Now suppose that $f$ is an acyclic fibration. By the first part $f_{n}$ is an admissible epimorphism in each strictly positive degree. But $f$ also has the right lifting property with respect to the map $0\rightarrow S^{0}(L)$ for any $L$ in $\mathfrak{L}$, implying that $f_{0}$ is an admissible epimorphism. Now for any $n $ and any $L\in\mathfrak{L}$ one can always find a lift in the diagram
\begin{displaymath}
\xymatrix{
S^{n}(L)\ar[d]\ar[r] & X\ar[d]^{f}\\
D^{n+1}(L)\ar[r] & Y
}
\end{displaymath}
Indeed for $n<-1$ this is trivial, and for $n\ge0$ this follows from the fact that $f$ is an acyclic fibration. For $n=-1$ the finding a lift amounts to finding a lift in the diagram
\begin{displaymath}
\xymatrix{
0\ar[r]\ar[d] & X_{0}\ar[d]^{f_{0}}\\
L\ar[r] & Y_{0}
}
\end{displaymath}
i.e. a lift in the diagram
\begin{displaymath}
\xymatrix{
0\ar[r]\ar[d] & X\ar[d]\\
S^{0}(L)\ar[r] & Y
}
\end{displaymath}
Now the claim follows from Corollary \ref{liftquasi}.
\end{proof}

\begin{prop}\label{cofibrantshift}
Let $(\mathfrak{L},\mathfrak{R})$ be a $dg_{*}$-compatible cotorsion pair for $*\in\{\emptyset,\ge0,b,+\}$. If $f$ is a cofibration (resp. trivially cofibrant) in $Ch_{*}(\mathpzc{E})$, then so is $f[-1]$. 
\end{prop}
\begin{proof}
It suffices to observe that the classes $\widetilde{dg\mathfrak{L}}$ and $\widetilde{\mathfrak{L}}$ are closed under applying $[-1]$. 
\end{proof}
\begin{prop}\label{conecofibrantiscofibrant}
Let $(\mathfrak{L},\mathfrak{R})$ be a $dg_{*}$-compatible cotorsion pair for $*\in\{\emptyset,\ge0,b,+\}$. If $f:X\rightarrow Y$ is a map between cofibrant objects then $cone(f)$ is cofibrant. It is trivially cofibrant whenever $f$ is a quasi-isomorphism.
\end{prop}
\begin{proof}
There is an exact sequence 
$$0\rightarrow Y\rightarrow cone(f)\rightarrow X[-1]\rightarrow 0$$
Thus the result follows from Proposition \ref{cofibrantshift} and Corollary \ref{trans}.
\end{proof}
If $(\mathfrak{L},\mathfrak{R})$ is $dg_{*}$-compatible, then there is an induced  model structure on $ Ch_{*}(\mathpzc{E})$. The resulting model structure will have quasi-isomorphisms as its weak equivalences.

\begin{prop}\label{weakquasmod}
Suppose that  $*\in\{\ge0,\le0,+,-,b\}$ and that $(\mathfrak{L},\mathfrak{R})$ is a $dg_{*}$-compatible cotorsion pair on an exact category $\mathpzc{E}$. The weak equivalences in the induced model structure are precisely the quasi-isomorphisms. If $\mathpzc{E}$ has all kernels then this is also true for $*\in\{\emptyset\}$.
\end{prop}

\begin{proof}
First we show that admissible monics and admissible epics which are weak equivalences are quasi-ismorphisms. By duality it suffices to show it for epics and $*\in\{\ge0,b,+,-,\emptyset,\le0\}$.  Let $f:B\rightarrow C$ be an an admissible epic which is a weak equivalence. It is sufficient to show that $I(f)$ is a quasi-isomorphism, where $I:\mathpzc{E}\rightarrow\mathpzc{A}(\mathpzc{E})$ is a suitable abelianisation. Now we have an exact sequence
\begin{displaymath}
\xymatrix{
0\ar[r] & A\ar[r]^{g} & B\ar[r]^{f} & C\ar[r] & 0
}
\end{displaymath}
Let us argue that $A$ is acyclic. We can factor $f$ as $f=p\circ i$ where $p$ is an acyclic fibration and $i$ is an acyclic cofibration. By Proposition \ref{adepitrunc} for $*\in\{\ge0,\le0\}$ and Theorem \ref{weakmon} for $*\in\{b,+,-,\emptyset\}$, both $i$ is an admissible monomorphism with acyclic cokernel, and $p$ is an admissible epimorphism with acyclic kernel. By the snake lemma there is an exact sequence $0\rightarrow ker(f)\rightarrow ker(p)\rightarrow coker(i)\rightarrow 0$. In particular $ker(f)$ is acyclic by Lemma \ref{acker}. Again by Lemma \ref{acker}, $f$ is acyclic. 
Now let $f$ be a morphism of $ Ch_{*}(\mathpzc{E})$. Factor it as $p\circ i$ where $i$ is a fibration, $p$ is a cofibration and either $p$ or $i$ is trivial, and therefore a quasi-isomorphism. By the exact triangle (after passing to an abelianisation)
$$\textrm{cone}(i)\rightarrow\textrm{cone}(f)\rightarrow\textrm{cone}(p)\rightarrow^{+1}$$
and the fact that acyclic complexes form a thick subcategory, we find that $f$ is a quasi-isomorphism if and only if the other factor is trivial. 
\end{proof}

\begin{prop}\label{transquas}
Let $\mathpzc{E}$ be an exact category and $\mathpzc{S}$ a class of morphisms in $\mathpzc{E}$ closed under finite direct sums. Suppose that $\mathpzc{E}$ is weakly $\mathpzc{S}$-elementary. Then transfinite compositions of quasi-isomorphisms in $ Ch(\mathpzc{E})$ which are also maps in $\mathpzc{S}$ are quasi-isomorphisms. 
\end{prop}

\begin{proof}
The proof is by transfinite induction. Since a finite composition of quasi-isomorphisms is a quasi-isomorphism, the successor part of the induction is finished. Now let $\lambda$ be a limit ordinal and $F:\lambda\rightarrow Ch(\mathpzc{E})$ a continuous functor with $F(\alpha\le\beta)_{n}\in\mathpzc{S}$ for any morphism $\alpha\le\beta$ in $\lambda$ and $n\in\mathbb{Z}$. For $\alpha\le\beta\le\lambda$ denote by $f_{\alpha,\beta}$ the map $F_{\alpha}\rightarrow F_{\beta}$. For $\beta\le\lambda$ write $f_{\beta}=f_{0,\beta}$. It is clear that
$$\textrm{cone}(f_{\lambda})\cong\textrm{lim}_{\rightarrow_{\beta<\lambda}}\textrm{cone}(f_{\beta})$$
Since each $f_{\beta}$ is a quasi-isomorphism, $\textrm{cone}(f_{\beta})$ is acyclic. Since $\mathpzc{E}$ is weakly $\mathpzc{S}$-elementary, this implies $\textrm{lim}_{\rightarrow_{\beta<\lambda}}\textrm{cone}(f_{\beta})$ is acyclic, which means that $\textrm{cone}(f_{\lambda})$ is acyclic and hence that $f_{\lambda}$ is a quasi-isomorphism.
\end{proof}

\begin{cor}
Let $\mathpzc{E}$ be a weakly idempotent complete exact category and $\mathcal{S}$ a class of morphisms in $\mathpzc{E}$ such that $\mathpzc{E}$ is weakly $\mathpzc{S}$-elementary as an exact category. Suppose that there is a model category structure on $Ch(\mathpzc{E})$ in which the weak equivalences are the quasi-isomorphisms, and the class of cofibrations is contained in $Ch(\mathcal{S})$. Then $Ch(\mathpzc{E})$ is weakly $Ch(\mathcal{S})$-elementary as a model category (Definition \ref{defn:weaklyelmodel}).
\end{cor}
\begin{proof}
Let $\lambda$ be a limit ordinal and $F:\lambda\rightarrow Ch(\mathpzc{E})$ a continuous functor with $F(\alpha\le\beta)_{n}\in\mathpzc{S}$. Let $\widetilde{F}\rightarrow F$ be a cofibrant replacement for $F$. Then $\widetilde{F}(\alpha\le\beta)$ is a cofibration so is in $Ch(\mathcal{S})$. Thus the map $\textrm{lim}_{\rightarrow_{\alpha<\lambda}}\widetilde{F}\rightarrow\textrm{lim}_{\rightarrow_{\alpha<\lambda}}F$ is an equivalence, as required. 
\end{proof}

Such model structures are also both left and right proper. More generally, we have the following.

\begin{prop}\label{lrprop}
Let $\mathpzc{E}$ be a weakly idempotent complete exact category. Let  $*\in\{\ge0,\le0,+,-,b\}$. Suppose there is a model structure on $ Ch_{*}(\mathpzc{E})$ whose weak equivalences are the quasi-isomorphisms and such that any cofibration is an admissible monomorphism in each degree. Then the model structure is left proper. If $\mathpzc{E}$ has all kernels then this is also true for $ Ch(\mathpzc{E})$. Dually, if any fibration is an admissible epimorphism in each degree then the model structure is right proper.

\end{prop}

\begin{proof}
The dual case is slightly easier to write down, so we will prove that. We need to check that, given a pull-back diagram
\begin{displaymath}
\xymatrix{
A_{\bullet}\ar[r]^{p'}\ar[d]^{q'} & B_{\bullet}\ar[d]^{q}\\
X_{\bullet}\ar[r]^{p} & Y_{\bullet}
}
\end{displaymath}
where $p$ is an admissible epic, and $q$ is a quasi-isomorphism, then $q'$ is a quasi-isomorphism. By Lemma \ref{bicart} without loss of generality, we may assume that the category $\mathpzc{E}$ is actually abelian. We argue by elements. $A_{\bullet}$ is isomorphic to 
$$\{(x,b)\in X_{\bullet}\times B_{\bullet}:p(x)=q(b)\}$$
with $q'$ and $p'$ being the restrictions of the projections. Suppose $(x,b)\in\textrm{Ker}(d_{n}^{A})$ is such that $q'(x,b)=x=0$. But then $q(b)=p(x)=0$. So $b=d_{n+1}^{B}(\widetilde{b})$ for some $b$, and $(x,b)=d_{n+1}^{A}((0,\widetilde{b}))$. Now suppose $x\in\textrm{Ker}(d_{n}^{X})$. Then $p(x)\in\textrm{Ker}(d_{n}^{Y})$. Thus there is a $b\in \textrm{Ker}(d_{n}^{B})$ and a $\widetilde{y}\in Y_{n+1}$ such that $q(b)=p(x)+d_{n+1}^{Y}(\widetilde{y})$. Now, $p$ is an epic, so there is $\widetilde{x}\in  X_{n+1}$ such that $\widetilde{y}=p(\widetilde{x})$. Write $a=(x+d_{n+1}^{X}(\widetilde{x}),b)$. Then $a\in A_{\bullet}$ and $q'(a)=x+d_{n+1}^{X}(\widetilde{x})$. This shows that $q'$ is a quasi-isomorphism.
\end{proof}

\begin{cor}\label{coradmhcof}
Let $\mathpzc{E}$ be a weakly idempotent complete exact category. Let  $*\in\{\ge0,\le0,+,-,b\}$. Suppose there is a model structure on $ Ch_{*}(\mathpzc{E})$ whose weak equivalences are the quasi-isomorphisms and such that any cofibration is an admissible monomorphism in each degree. Then degree-wise admissible monomorphisms are $h$-cofibrations (Definition \ref{defn:hcof}). If $\mathpzc{E}$ has all kernels then this is also true for $ Ch(\mathpzc{E})$. Dually, if any fibration is an admissible epimorphism in each degree then degree-wise admissible epimorphisms are $h$-fibrations.
\end{cor}
\begin{proof}
We prove the left properness claim, the right properness claim being dual. The model structure is left proper by Proposition \ref{lrprop}. By the proof of Proposition \ref{lrprop} pushouts of weak equivalences along degree-wise admissible monomorphisms are weak equivalences. Moreover degree-wise admissible monomorphisms are pushout-stable. The result now follows from Poposition \ref{prop:pushoutstableftprop}.
\end{proof}
We also have the following useful result.
\begin{prop}
Suppose there is a model structure on $Ch(\mathpzc{E})$, where $\mathpzc{E}$ is weakly idempotent complete, whose weak equivalences are the quasi-isomorphisms. Then the model category is stable (\ref{defn:stable}).
\end{prop}
\begin{proof}
We need to compute the suspension functor $0\coprod_{X}^{\mathbb{L}}0$. Now $X\rightarrow cone(Id_{X})$ is a $h$-cofibration. Thus $0\coprod_{X}^{\mathbb{L}}0$ can be computed as the pushout $cone(Id_{X})\coprod_{X}0\cong X[-1]$. Thus as one should expect, the suspension functor is the shift functor, which is an auto-equivalence.
\end{proof}

\subsection{Small dg-Cotorsion Pairs}
Let us now examine when the cotorsion pair $(\widetilde{dg\mathfrak{L}},\widetilde{\mathfrak{R}})$ is small. Again this is done in \cite{gillespie2016derived} Section 4.4 for the $G$-exact structure on a Grothendieck abelian category.

\begin{prop}\label{dgsmall}
Let $(\mathfrak{L},\mathfrak{R})$ be a cotorsion pair in an exact category $\mathpzc{E}$ which has a set of admissible generators $\mathcal{G}$. Suppose that $(\mathfrak{L},\mathfrak{R})$ is cogenerated by a set $\{A_{i}\}_{i\in I}$. Then  $(\widetilde{dg\mathfrak{L}},\widetilde{\mathfrak{R}})$ is cogenerated by the set 
$$S=\{S^{n}(G):G\in\mathcal{G},n\in\Z\}\cup\{S^{n}(A_{i}):n\in\Z,i\in I\}$$
for $*\in\{+\}$ (and. $*\in\{\emptyset\}$ if $\mathpzc{E}$ has kernels) and 
$$S=\{S^{n}(G):G\in\mathcal{G},n\ge0\}\cup\{S^{n}(A_{i}):n\ge0,i\in I\}$$
for $*\in\{\ge0\}$. 

Furthermore, suppose $(\mathfrak{L},\mathfrak{R})$ is small with generating morphisms the maps $\{0\rightarrow G:G\in\mathcal{G}\}$ together with monics $k_{i}$ as below (one for each $i\in I$).
\begin{displaymath}
\xymatrix{
0\ar[r] & Y_{i}\ar[r]^{k_{i}} & Z_{i}\ar[r] & A_{i}\ar[r] & 0
}
\end{displaymath}
Then $(\widetilde{dg\mathfrak{L}},\widetilde{\mathfrak{R}})$ is small with generating morphisms the set
$$\widetilde{I}=\{0\rightarrow D^{n}(G)\}\cup\{S^{n-1}(G)\rightarrow D^{n}(G)\}\cup\{S^{n}(k_{i}):S^{n}(Y_{i})\rightarrow S^{n}(Z_{i})\}$$
for $*\in\{+\}$ (and. $*\in\{\emptyset\}$ if $\mathpzc{E}$ has kernels)
and
$$\widetilde{I}=\{0\rightarrow S^{0}(G)\}\cup\{0\rightarrow D^{n}(G):n>0\}\cup\{S^{n-1}(G)\rightarrow D^{n}(G):n>0\}$$
$$\;\;\;\;\;\;\;\;\;\;\;\;\;\;\;\;\;\;\;\;\;\;\;\;\;\;\;\;\;\;\;\;\;\;\;\;\;\;\cup\{S^{n}(k_{i}):S^{n}(Y_{i})\rightarrow S^{n}(Z_{i}):n\ge0\}$$
for $*\in\{\ge0\}$.
\end{prop}

\begin{proof}
For $*\in\{+,\emptyset\}$ the proof of \cite{kaplansky} Proposition 3.8 generalises immediately to exact categories. Now consider the case $*\in\{\ge0\}$. The only difference in the proof is that now the generating set for $ Ch_{\ge0}(\mathpzc{E})$ is $\{D^{n}(G):G\in\mathcal{G}:n>0\}\cup\{S^{0}(G):G\in\mathcal{G}\}$. This is also a subset of $\widetilde{dg\mathfrak{L}}$.
\end{proof}

\begin{rem}\label{circlesenough}
The proof of the above proposition in fact shows that if $\textbf{Hom}(S^{n}(F),X_{\bullet})$ is acyclic for any generating collection consisting of objects in $\mathfrak{L}$ (i.e. not necessarily a set), then $X_{\bullet}$ is a $\widetilde{\mathfrak{R}}$-complex.
\end{rem}

\begin{rem}\label{dgcellular}
In the situation of the previous proposition, if the domains of the generating morphisms for the cotorsion pair $(\mathfrak{L},\mathfrak{R})$ are $(\kappa,\mathcal{S})-^{\star}$ for $\star\in\{small, compact, presented, tiny\}$ then the domains of the maps in $I$ are  $Ch(\mathcal{S})-{\star}$ by Proposition \ref{diagrampresented}.
\end{rem}

\subsection{Monoidal Model Structures on Chain Complexes}
In this section we investigate when cotorsion pairs on monoidal exact categories induce monoidal model structures on the category of chain complexes. First we have the following easy results, which says when a complex is flat.

\begin{prop}\label{flatcomplex}
Let $(\mathpzc{E},\otimes,k)$ be an additive symmetric monoidal category with $\mathpzc{E}$ an exact category.  For $*\in\{\ge0,\le0,b,+,-\}$ the flat objects in $( Ch_{*}(\mathpzc{E}),\otimes,S^{0}(k))$ are precisely the complexes $F_{\bullet}$ in $ Ch_{*}(\mathpzc{E})$ such that for each $n\in\Z$, $F_{n}$ is flat. If in addition countable direct sums exist and are exact, then the flat objects in $( Ch(\mathpzc{E}),\otimes,S^{0}(k))$ are also the complexes $F_{\bullet}$ that for each $n\in\Z$, $F_{n}$ is flat
\end{prop}

\begin{proof}
Let
\begin{displaymath}
\xymatrix{
0\ar[r] & X_{\bullet}\ar[r] & Y_{\bullet}\ar[r] & Z_{\bullet}\ar[r] & 0
}
\end{displaymath}
be a short exact sequence in $ Ch_{*}(\mathpzc{E})$. Let $F_{\bullet}$ be a complex. Then the $n$th row of
\begin{displaymath}
\xymatrix{
0\ar[r]&X_{\bullet}\otimes F_{\bullet}\ar[r] & Y_{\bullet}\otimes F_{\bullet}\ar[r] & Z_{\bullet}\otimes F_{\bullet}\ar[r] & 0
}
\end{displaymath}
is
\begin{displaymath}
\xymatrix{
0\ar[r] & \bigoplus_{i+j=n}X_{i}\otimes F_{j}\ar[r] & \bigoplus_{i+j=n}Y_{i}\otimes F_{j}\ar[r] & \bigoplus_{i+j=n}Z_{i}\otimes F_{j}\ar[r] & 0
}
\end{displaymath}
Since the direct sums involved are exact, this sequence is short exact if  for each $i,j$,
\begin{displaymath}
\xymatrix{
0\ar[r] & X_{i}\otimes F_{j}\ar[r] & Y_{i}\otimes F_{j}\ar[r] &Z_{i}\otimes F_{j}\ar[r] & 0
}
\end{displaymath}
is short exact. It follows immediately that a complex whose entries are flat in $\mathpzc{E}$ is itself a flat object in $ Ch_{*}(\mathpzc{E})$. To see that a flat complex must have flat entries, simply take a short exact sequence in $\mathpzc{E}$, and regard it as a short exact sequence in $ Ch_{*}(\mathpzc{E})$ concentrated in degree $0$.
\end{proof}

\begin{defn}\label{tensormodel}
Let $(\mathpzc{E},\otimes)$ be a monoidal exact category. A cotorsion pair $(\mathfrak{L},\mathfrak{R})$ on $\mathpzc{E}$ is said to be \textbf{monoidally }$dg_{*}$-\textbf{compatible} for $*\in\{\ge0,+,\emptyset\}$ if
\begin{enumerate}
\item
 $(\mathfrak{L},\mathfrak{R})$ is $dg_{*}$ compatible.
 \item
For $*\in\emptyset$ countable product functors are admissibly exact and and countable coproduct functors are admissibly coexact 
\item
$\mathfrak{L}$ contains $k$, and is closed under $\otimes$.
\item
short exact sequences in $\mathfrak{L}$ are $\mathfrak{L}$-pure.
\end{enumerate}
Now let $*\in\{\ge0,\emptyset\}$. If in addition objects of $\mathfrak{L}$ are flat, $\mathpzc{E}$ is weakly $(\lambda,\textbf{PureMon})$-elementary for any ordinal $\lambda$, where $\textbf{PureMon}$ is the class of pure monomorphisms, and every (trivially) cofibrant complex is an $(\aleph_{0};\textbf{PureMon})$-extension of bounded below (trivially) cofibrant complexes, then the cotorsion pair is said to be \textbf{ strongly monoidally }$dg_{*}$-\textbf{compatible}.
\end{defn}

\begin{prop}\label{homdg}
Let $(\mathfrak{L},\mathfrak{R})$ be a hereditary monoidally $dg_{\ge0}$-compatible cotorsion pair on $\mathpzc{E}$, with $\mathpzc{E}$ weakly idempotent complete.. 
\begin{enumerate}
\item
If $L_{\bullet}$ is a $dg{\widetilde{\mathfrak{L}}}$-complex and $R_{\bullet}$ is a $\widetilde{\mathfrak{R}}$-complex then $\underline{Hom}(L_{\bullet},R_{\bullet})$ is a $\widetilde{\mathfrak{R}}$-complex. 
\item
If $L_{\bullet}$ is a $\widetilde{\mathfrak{L}}$-complex and $R_{\bullet}$ is a $\widetilde{dg\mathfrak{R}}$ complex, then $\underline{Hom}(L_{\bullet},R_{\bullet})$ is a $\widetilde{\mathfrak{R}}$-complex.
\end{enumerate}
\end{prop}

\begin{proof}
By Remark \ref{circlesenough} for both statements it suffices to show that $\textbf{Hom}(S^{n}(F),\underline{Hom}(L_{\bullet},R_{\bullet}))$ is acyclic for any $F\in\mathfrak{L}$, and all appropriate $n$. But $\textbf{Hom}(S^{n}(F),\underline{Hom}(L_{\bullet},R_{\bullet}))\cong\textbf{Hom}(S^{n}(F)\otimes L_{\bullet},R_{\bullet})$.

Thus for the first statement we need to show that $S^{n}(F)\otimes L_{\bullet}$ is a $\widetilde{dg\mathfrak{L}}$-complex when $L_{\bullet}$ is, and for the second we need to show that it is a $\widetilde{\mathfrak{L}}$-complex when $L_{\bullet}$ is. Since objects of $F$ are flat and $\mathfrak{L}$ is closed under $\otimes$, if we can show that first statement then the second follows immediately. 

Let $X_{\bullet}$ be a $\widetilde{\mathfrak{R}}$-complex. Then
$$\textbf{Hom}(S^{n}(F)\otimes L_{\bullet},X_{\bullet})\cong\textbf{Hom}(L_{\bullet},\underline{Hom}(S^{n}(F),X_{\bullet}))$$
To show that this is acyclic it now suffices to show that $\underline{Hom}(S^{n}(F),X_{\bullet})$ is a $\widetilde{\mathfrak{R}}$-complex. By shifting we may assume that $n=0$, and then this complex is just the internal hom functor in $\mathpzc{E}$ taken degree-wise, $\underline{Hom}(F,X_{\bullet})$. Since $X_{\bullet}$ is an acyclic complex with $X_{n}\in\mathfrak{R}$ and $Z_{n}X\in\mathfrak{R}$ for each $n$, this complex is clearly exact, and $Z_{n}\underline{Hom}(F,X_{\bullet})\cong \underline{Hom}(F,Z_{n}X)$. Finally we reduce to showing that for $F\in\mathfrak{L}$ and $G\in\mathfrak{R}$, $\underline{Hom}(F,G)\in\mathfrak{R}$. As the cotorsion pair $(\mathfrak{L},\mathfrak{R})$ is complete, it suffices to show that $\textbf{Hom}(Z_{\bullet},S^{0}(\underline{Hom}(F,G)))$ is acyclic whenever $Z_{\bullet}$ is a bounded below $\widetilde{\mathfrak{L}}$ complex. Let $Z_{\bullet}$ be such a complex. Now $\textbf{Hom}(Z_{\bullet},S^{0}(\underline{Hom}(F,G)))\cong\textbf{Hom}(Z_{\bullet}\otimes S^{0}(F),S^{0}(G))$. $Z_{\bullet}\otimes S^{0}(F)$ is the complex obtained from $Z_{\bullet}$ by tensoring with $F$ degree-wise. By the assumptions on $\mathfrak{L}$ this is clearly a $\widetilde{\mathfrak{L}}$-complex. Moreover $S^{0}(G)$ is a $\widetilde{dg\mathfrak{R}}$-complex. Thus $\textbf{Hom}(Z_{\bullet},\underline{Hom}(S^{0}(F),S^{0}(G)))$ is acyclic, and $\underline{Hom}(S^{0}(F),S^{0}(G))$ is a $\widetilde{dg\mathfrak{R}}$-complex.
\end{proof}

\begin{prop}\label{monoidaldgcompat}
Let $(\mathfrak{L},\mathfrak{R})$ be a monoidally $dg_{*}$-compatible cotorsion pair for $*\in\{\ge,\emptyset\}$ on weakly idempotent complete $\mathpzc{E}$. The model category structure induced by $(\mathfrak{L},\mathfrak{R})$ on $Ch_{*}(\mathpzc{E})$ is monoidal.
\end{prop}

\begin{proof}
For the first part we use Theorem \ref{exactmonoidal}. First suppose that the cotorsion pair is monoidally $dg_{*}$-compatible. Clearly $S^{0}(k)$ is cofibrant. Let $L$ and $L'$ be $\widetilde{dg\mathfrak{L}}$ complexes and let $R$ be a $\widetilde{\mathfrak{R}}$ complex. Then
$$\textbf{Hom}(L\otimes L',\mathfrak{R})\cong\textbf{Hom}(L,\underline{Hom}(L',\mathfrak{R}))$$
By Proposition \ref{homdg} $\underline{Hom}(L',\mathfrak{R})$ is a $\widetilde{\mathfrak{R}}$-complex. Therefore $\textbf{Hom}(L,\underline{Hom}(L',\mathfrak{R}))$ is acyclic. Hence $L\otimes L'$ is a $\widetilde{dg\mathfrak{L}}$-complex. In particular the class of cofibrant objects is closed under $\otimes$. If one of them is acyclic then again using Proposition \ref{homdg} $L\otimes L'$ is also acyclic. 
\end{proof}

\begin{prop}
Let $\mathfrak{L}$ be a class of objects in a monoidal weakly idempotent complete exact category $\mathpzc{E}$, and suppose that $\mathpzc{E}$ is weakly $\textbf{PureMon}$-elementary
\begin{enumerate}
\item
Suppose that any admissible monomorphism with cokernel in $\mathfrak{L}$ is pure. If $X$ is any complex then for any complex $L$ in $\widetilde{\mathfrak{L}}$, $L\otimes X$ is acyclic.
\item
Suppose that objects in $\mathfrak{L}$ are flat. If $L$ is an $(\aleph_{0};\textbf{PureMon})$-extension of bounded below complexes of objects in $\mathfrak{L}$ then for any acyclic complex $X$, $X\otimes L$ is acyclic.
\end{enumerate}
In particular if there is a  strongly monoidally $dg_{*}$-compatible cotorsion pair $(\mathfrak{L},\mathfrak{R})$ for $*\in\{\ge0,\emptyset\}$, then the induced model structure satisfies the monoid axiom. Moreover in this case if $C$ is cofibrant and $X$ is acyclic then $C\otimes X$ is acyclic. 
\end{prop}

\begin{proof}
In either case we can write it as $L=\textrm{lim}_{\rightarrow}L^{n}$ where the maps $L^{n}\rightarrow L^{n+1}$ are pure monomorphisms, and each $L^{n}$ is bounded below. If $L$ is in $\widetilde{\mathfrak{L}}$ we may assume that each $L^{n}$ is acyclic. Let $X$ be any complex. Then $X\otimes L\cong lim_{\rightarrow_{n}}(X\otimes L^{n})$. Once again the maps $X\otimes L^{n}\rightarrow X\otimes L^{n+1}$ are pure monomorphisms.

Suppose that either $X$ or $L$ is acyclic. We want to show that $X\otimes L$ is acyclic.  By Proposition \ref{transquas} it suffices to show that $L^{n}\otimes X$ is acyclic. Without loss of generality let us assume that $L^{n}$ is concentrated in degrees $\ge0$. 
First suppose condition 1) is satisfied. We need to show that $L^{n}\otimes X$ is acyclic whenever $L^{n}$ is a bounded below complex in $\widetilde{\mathfrak{L}}$. Such a complex can be written as an $(\aleph_{0};\textbf{PureMon})$-extension of bounded complexes in $\widetilde{\mathfrak{L}}$. So we may in fact assume that $L^{n}$ is bounded. We induct on the length of $L^{n}$. Suppose $L$ is a complex in $\widetilde{\mathfrak{L}}$ of length $k+1$. There is an exact sequence
$$0\rightarrow D^{k+1}(L_{k+1}) \rightarrow L\rightarrow \tau_{\le k} L\rightarrow 0$$
This is pure exact since $\tau_{\le k} L^{n}$ is a complex of objects in $\mathfrak{L}$. Therefore tensoring with $X$ gives an exact sequence
$$0\rightarrow X\otimes D^{k+1}(L_{k+1}) \rightarrow X\otimes L\rightarrow X\otimes \tau^{\ge k} L\rightarrow 0$$
By assumption $X\otimes \tau_{\le k} L$ is acyclic. Up to a shift, $X\otimes D^{k+1}(L_{k+1})$ is $X\otimes cone(Id_{L_{k+1}})$ and is therefore acyclic. Again by thickness, $X\otimes L$ is acyclic.

Suppose now that conidition 2) is satisfied. Now $L^{n}$ may be obtained from $S^{0}(L^{n}_{0})$ by a transfinite composition of pushouts of the form
\begin{displaymath}
\xymatrix{
S^{k}(F)\ar[d]\ar[r] & A\ar[d]\\
D^{k+1}(F)\ar[r] & B
}
\end{displaymath}
Tensoring with $X$ gives a pushout diagram
\begin{displaymath}
\xymatrix{
S^{k}(F)\otimes X\ar[d]\ar[r] & A\otimes X\ar[d]\\
D^{k+1}(F)\otimes X\ar[r] & B\otimes X
}
\end{displaymath}
$S^{k}(F)\rightarrow D^{k+1}(F)$ is a pure monomorphism. Therefore there is an exact sequence $$0\rightarrow S^{k}(F)\otimes X\rightarrow D^{k+1}(F)\otimes X\rightarrow S^{k+1}(F)\otimes X\rightarrow 0$$ By pushout there is an exact sequence $$0\rightarrow A\otimes X\rightarrow B\otimes X\rightarrow S^{k+1}(F)\otimes X\rightarrow 0$$ $S^{k+1}(F)\otimes X$ is acyclic since, up to shift, it is tensoring with a flat object. By induction on the length of the complex we may assume that $A\otimes X$ is acyclic. Since acyclic objects are thick, $B\otimes X$ is acyclic. Moreover the map $A\otimes X\rightarrow B\otimes X$ is a pure monomorphism. Therefore $L^{n}\otimes X$ is an $(\aleph_{0};\textbf{PureMon})$-extension of acyclic objects, and so is acylic.
\end{proof}

\section{The Projective Model Structure}\label{sec5}
In this section we specialise to the cotorsion pair $(\textbf{Proj}(\mathpzc{E}),\textbf{Ob}(\mathpzc{E}))$.  $\mathpzc{E}$ will be an exact category with enough functorial projectives. We denote the collection of all projective objects in $\mathpzc{E}$ by $\textbf{Proj}(\mathpzc{E})$

\begin{defn}
Let $\mathpzc{E}$ be an exact category. If it exists, the \textbf{projective model structure} on $ Ch_{*}(\mathpzc{E})$, for $*\in\{+,\emptyset,\le0\}$ is the model structure in which
\begin{itemize}
\item
Weak equivalences are quasi-isomorphisms.
\item
Fibrations are degree-wise admissible epics.
\item
Cofibrations are maps which have the left lifting property with respect to acyclic fibrations.
\end{itemize}
\end{defn}

\begin{prop}\label{funcinproj}
Let $\mathpzc{E}$ be an exact category. Suppose that the cotorsion pair $(dg\widetilde{\textbf{Proj}(\mathpzc{E})},\widetilde{\textbf{Ob}(\mathpzc{E}}))$ on $ Ch_{*}(\mathpzc{E})$ for $*\in\{+,\ge0,\emptyset\} $ has enough functorial projectives. Then it has enough functorial injectives.
\end{prop}

\begin{proof}
Let $X_{\bullet}$ be an object of $ Ch_{*}(\mathpzc{E})$, and let $f_{\bullet}:L_{\bullet}\rightarrow X_{\bullet}$ be a quasi-isomorphism and admissible epimorphism with acyclic kernel,  and $L_{\bullet}\in dg\widetilde{\textbf{Proj}(\mathpzc{E})}$.

We have a short exact sequence
$$0\rightarrow X_{\bullet}\rightarrow\textrm{cone}(f_{\bullet})\rightarrow L_{\bullet}[-1]\rightarrow 0$$
$\textrm{cone}(f_{\bullet})$ is an acyclic complex, so it is in $\widetilde{\textbf{Ob}(\mathpzc{E})}$. Clearly $L_{\bullet}[-1]\in dg\widetilde{\textbf{Proj}(\mathpzc{E})}$.
\end{proof}

We are now ready to prove the following theorem.

\begin{thm}\label{projmod}
Let $\mathpzc{E}$ be a weakly idempotent complete exact category with enough projectives. Then the projective model structure exists on $ Ch_{+}(\mathpzc{E})$ and is compatible. It is functorial if $\mathpzc{E}$ has enough functorial projectives. It is cofibrantly generated if $\mathpzc{E}$ is elementary, and combinatorial if $\mathpzc{E}$ is locally presentable. If $DG$-projective resolutions exist, in particular if $\mathpzc{E}$ has all kernels and projectives satisfy condition $AB4-k$ for some $k$, then this is all true for $ Ch(\mathpzc{E})$ as well.
\end{thm}
\begin{proof}

 Consider the projective cotorsion pair $(\textbf{Proj}(\mathpzc{E}),\textbf{Ob}(\mathpzc{E}))$ on $\mathpzc{E}$. By Corollary \ref{bootcotor},\newline
 $(dg\widetilde{\textbf{Proj}(\mathpzc{E})},\widetilde{\textbf{Ob}(\mathpzc{E}}))$ is a cotorsion pair on $ Ch_{+}(\mathpzc{E})$. It is functorially complete by Proposition \ref{dgcomplete} and Proposition \ref{funcinproj}.
 
 We claim that $(\widetilde{\textbf{Proj}(\mathpzc{E})},dg\widetilde{\textbf{Ob}(\mathpzc{E}}))$ is also a cotorsion pair on $ Ch_{+}(\textbf{Ob}(\mathpzc{E}))$.  First note that $\widetilde{\textbf{Proj}(\mathpzc{E})}$ consists of split exact complexes of projectives. By Proposition \ref{chainproj} this is precisely the class of projective objects in $ Ch_{+}(\mathpzc{E})$. Then by Proposition \ref{homdg} $dg\widetilde{\textbf{Ob}(\mathpzc{E})}= Ch_{+}(\textbf{Ob}(\mathpzc{E}))$. Hence $(\widetilde{\textbf{Proj}(\mathpzc{E})},dg\widetilde{\textbf{Ob}(\mathpzc{E}}))$  is just the projective cotorsion pair.  Now $\widetilde{\textbf{Ob}(\mathpzc{E})}$ is the class of all acyclic complexes, $\mathfrak{W}$.  Thus $dg\widetilde{\textbf{Ob}(\mathpzc{E})}\cap\mathfrak{W}= Ch_{+}(\mathpzc{E})\cap\mathfrak{W}=\mathfrak{W}=\widetilde{\textbf{Ob}(\mathpzc{E})}$. Moreover $ Ch_{+}(\mathpzc{E})$ has enough projectives by Corollary \ref{enoughprojchain}. By Lemma \ref{dgcompat}  it remains to prove that $(\widetilde{\textbf{Proj}(\mathpzc{E})},dg\widetilde{\textbf{Ob}(\mathpzc{E}}))$ is (functorially) complete. But in a category with enough (functorial) projectives the projective cotorsion pair is always (functorially) complete by Example \ref{cotorsionproj}.
 
 Assume further that  $\mathpzc{E}$ is elementary. Then by Example \ref{projsmall}, the cotorsion pair $(\widetilde{\textbf{Proj}(\mathpzc{E})},dg\widetilde{\textbf{Ob}(\mathpzc{E})})$ is small and by Proposition \ref{dgsmall} the cotorsion pair 
 $(dg\widetilde{\textbf{Proj}(\mathpzc{E})},\widetilde{\textbf{Ob}(\mathpzc{E})})$ is small. By Lemma \ref{cofibgen}, the model structure is cofibrantly generated. The fact about combinatoriality is clear.
 
The proof for unbounded complexes works in almost exactly the same way. All that needs to be verified in this case is that $(dg\widetilde{\textbf{Proj}(\mathpzc{E})},\widetilde{\textbf{Ob}(\mathpzc{E}}))$ is complete.  Now the class of projectives is closed under $(\aleph_{0};\textbf{AdMon})$-extensions by Corollary \ref{trans}. Completeness therefore follows from from Corollary \ref{Kproj}, Proposition \ref{klift} and Proposition \ref{funcinproj}.
\end{proof}

There is another common situation in which we have $DG$-projective resolutions, which was established in the relative homological algebra setting in  \cite{christensen2002quillen} Proposition 4.2. This result can be considered as a version of `Case B' of \cite{christensen2002quillen} Theorem 2.2, while the $AB4-k$ critestion is a generalisation of `Case A'.
\begin{prop}\label{prop:christenssenenoughproj}
Let $\mathpzc{E}$ be an exact category with kernels and coproducts. Suppose that $\mathpzc{E}$ has a collection $\mathcal{P}$ of (functorially) generating projectives such that there exists a cardinal $\kappa$ for which each $P\in\mathcal{P}$ is $(\kappa;\textbf{AdMon}_{\textbf{Proj}})$-small. Then $Ch(\mathpzc{E})$ has (functorial) $DG$-projective resolutions.
\end{prop}
\begin{proof}
The proof for exact categories works essentially identically to \cite{christensen2002quillen} Proposition 4.2, mutatis-mutandis. The only subtlety is that since we have assumed $\mathpzc{E}$ has kernels, it holds that a map $f:X\rightarrow Y$ in $Ch(\mathpzc{E})$ is a weak equivalence if and only if for any $P\in\mathcal{P}$, $Hom(P,f)$ is a weak equivalence in $Ch(\mathpzc{Ab})$. 
\end{proof}

\begin{rem}
The projective model structure on $Ch(\mathpzc{E})$ (but not on $Ch_{\ge0}(\mathpzc{E})$) is a projective model structure on an exact category in the sense of \cite{gillespie} Definition 4.4.
\end{rem}

\begin{rem}
The existence of the projective model structure on bounded below chain complexes on a quasi-abelian category with enough projectives was already known to B\"{u}hler \cite{buhler2011algebraic} (see Appendix C). The proof there is more direct.  In fact the proof works for any idempotent complete exact category in which the class of all kernel-cokernel pairs forms the exact structure (all kernels and cokernels need not exist).
\end{rem}

Recall that if $\mathpzc{E}$ is (quasi)-elementary quasi-abelian, then Proposition \ref{qacel} says that $\textit{LH}(\mathpzc{E})$ is as well. Thus the projective model structure exists on $ Ch(\textit{LH}(\mathpzc{E}))$. Moreover the induced functor $I: Ch(\mathpzc{E})\rightarrow Ch(\textit{LH}(\mathpzc{E}))$ is then right Quillen. Indeed it is left adjoint to the induced functor $C: Ch(\textit{LH})(\mathpzc{E})\rightarrow Ch(\mathpzc{E})$. It preserves fibrations since $I:\mathpzc{E}\rightarrow\textit{LH}(\mathpzc{E})$ is a left abelianisation, and it preserves quasi-isomorphisms by Corollary \ref{quasboundtest}. Moreover by Theorem \ref{leftheart}, Proposition \ref{lhproj} and Proposition \ref{weakquasmod} it induces an equivalence between the homotopy categories. We therefore have

\begin{prop}
Let $\mathpzc{E}$ be an elementary quasi-abelian category. Then the adjunction
\begin{displaymath}
\xymatrix{
 Ch(\textit{LH}(\mathpzc{E}))\ar@/^1.0pc/@[black][r]^{C} &  Ch(\mathpzc{E})\ar@/^1.0pc/@[black][l]^{I} 
}
\end{displaymath}
is a Quillen equivalence between the projective model structures.
\end{prop}

We claim that the projective model structure exists also on $ Ch_{\ge0}(\mathpzc{E})$ for $\mathpzc{E}$ an  exact category with kernels. It will be strongly left pseudo-compatible, but not compatible.

\begin{defn}\label{projmodbound}
Let $\mathpzc{E}$ be an exact category. If it exists, the \textbf{projective model structure} on $ Ch_{\ge0}(\mathpzc{E})$, is the model structure in which
\begin{itemize}
\item
Weak equivalences are quasi-isomorphisms.
\item
Fibrations are degree-wise admissible epics in each strictly positive degree.
\item
Cofibrations are maps which have the left lifting property with respect to acyclic fibrations.
\end{itemize}
\end{defn}

\begin{thm}\label{projmodker}
Let $\mathpzc{E}$ be an exact category with enough projectives and which has all kernels. Then the projective model structure exists on $ Ch_{\ge0}(\mathpzc{E})$. Moreover it is a strong left pseudo-compatible model structure with Waldhausen pair $(\widetilde{\textrm{dg}\textbf{Proj}}(\mathpzc{E}),\mathfrak{W})$. In particular the acyclic cofibrations are the degree-wise admissible monics whose cokernels are split exact complexes of projectives. If $\mathpzc{E}$ is \textbf{SplitMon}-elementary then it is cofbrantly generated. In particular if $\mathpzc{E}$ is locally presentable and \textbf{SplitMon}-elementary then the projective model structure is combinatorial.
\end{thm}

\begin{proof}
The class of weak equivalences satisfies the $2$-out-of-$6$ property since it does so in $ Ch_{+}(\mathpzc{E})$.
Denote the class of fibrations by $\mathcal{F}$ and of weak equivalences by $\mathcal{W}$. Also denote the class of admissible monomorphisms with degree-wise projective cokernel by $\mathcal{C}$. By Proposition \ref{liftquasi} $\mathcal{F}\cap\mathcal{W}$ consists of quasi-isomorphisms which are admissible epimorphisms in each degree. 

By Proposition \ref{dgcomplete} and Proposition \ref{funcinproj}, it follows $(\mathcal{C},\mathcal{F}\cap\mathcal{W})$ is a (compatible) weak factorisation system with corresponding cotorsion pair $(\widetilde{dg\textbf{Proj}(\mathpzc{E})},\mathfrak{W})$. In particular the cofibrations in the sense of Definition \ref{projmodbound} coincide with the class $\mathcal{C}$. It therefore remains to check that $(\mathcal{C}\cap\mathcal{W},\mathcal{F})$ is a weak factorisation system. 

Let us first check the lifting conditions. First suppose a map $A_{\bullet}\rightarrow B_{\bullet}$ in $ Ch_{\ge0}(\mathpzc{E})$ has the left lifting property with respect to maps $X_{\bullet}\rightarrow Y_{\bullet}$ in $ Ch_{\ge0}(\mathpzc{E})$ which are admissible epimorphisms in each strictly positive degree. Let $E_{\bullet}\rightarrow F_{\bullet}$ be a map between any complexes in $ Ch(\mathpzc{E})$ which is an admissible epimorphism in all degrees. Consider a diagram

\begin{displaymath}
\xymatrix{
A_{\bullet}\ar[d]\ar[r] & E_{\bullet}\ar[d]\\
B_{\bullet}\ar[r] & F_{\bullet}
}
\end{displaymath}
Since $A_{\bullet}$ and $B_{\bullet}$ are in $ Ch_{\ge0}$ we can factor the above diagram as
\begin{displaymath}
\xymatrix{
A_{\bullet}\ar[d]\ar[r] & \tau_{\ge0}E_{\bullet}\ar[d]\ar[r]& E_{\bullet}\ar[d]\\
B_{\bullet}\ar[r] & \tau_{\ge0}F_{\bullet}\ar[r] & F_{\bullet}
}
\end{displaymath}
Now the map $\tau_{\ge0}E_{\bullet}\rightarrow\tau_{\ge0}F_{\bullet}$ is an epimorphism in each strictly positive degree.
By assumption we can find a lift as follows.
\begin{displaymath}
\xymatrix{
A_{\bullet}\ar[d]\ar[r]& \tau_{\ge0} E_{\bullet}\ar[d]\ar[r]& E_{\bullet}\ar[d]\\
B_{\bullet}\ar[r]\ar@{-->}[ur] & \tau_{\ge0}F_{\bullet}\ar[r] & F_{\bullet}
}
\end{displaymath}
Thus the map $A_{\bullet}\rightarrow B_{\bullet}$ has the left lifting property with respect to all degree-wise epimorphisms in $ Ch_{+}(\mathpzc{E})$. By Theorem \ref{projmod} $A_{\bullet}\rightarrow B_{\bullet}$ is an admissible monic whose cokernel is a split exact complex of projectives. 
Now, any acyclic cofibration is of the form $A_{\bullet}\rightarrow A_{\bullet}\oplus\Bigr(\oplus_{n>0}D^{n}(P_{n})\Bigr)$ where each $P_{n}$ is a projective object in $\mathpzc{E}$, and the map is the inclusion into the first factor of the direct sum. Clearly then it is enough to show that the collection of maps $\{0\rightarrow D^{n}(P):n>0,P\textrm{ is projective }\}$ has the left lifting property with respect to $\mathcal{F}$, and that a map is in $\mathcal{F}$ if and only if it has the right lifting property with respect to these maps. However this follows from Lemma \ref{extstuff} and Proposition \ref{reflectsurj}.

It remains to find a (functorial) factorisation. Let $f_{\bullet}:X_{\bullet}\rightarrow Y_{\bullet}$ be a map in $ Ch_{\ge0}(\mathpzc{E})$. We can factor it in $ Ch_{+}(\mathpzc{E})$ as
$$X_{\bullet}\rightarrow X_{\bullet}\oplus\Bigr(\oplus_{n\ge0}D^{n}(P_{n})\Bigr)\rightarrow Y_{\bullet}$$
where $X_{\bullet}\rightarrow X_{\bullet}\oplus\Bigr(\oplus_{n\ge0}D^{n}(P_{n})\Bigr)$ is the inclusion into the first factor, and $X_{\bullet}\oplus\Bigr(\oplus_{n\ge0}D^{n}(P_{n})\Bigr)\rightarrow Y_{\bullet}$ is an admissible epimorphism in each degree. Then
$$X_{\bullet}\rightarrow X_{\bullet}\oplus\Bigr(\oplus_{n>0}D^{n}(P_{n})\Bigr)\rightarrow Y_{\bullet}$$
is also a factorisation of $f_{\bullet}$, $X_{\bullet}\rightarrow X_{\bullet}\oplus\Bigr(\oplus_{n>0}D^{n}(P_{n})\Bigr)$ is an acyclic cofibration in $ Ch_{\ge0}(\mathpzc{E})$, and $X_{\bullet}\oplus\Bigr(\oplus_{n>0}D^{n}(P_{n})\Bigr)\rightarrow Y_{\bullet}$ is an admissible epimorphism in each strictly positive degree.

We prove the statement about cofibrant generation. Suppose that $\mathcal{P}$ is a projective generating set consisting of \textbf{SplitMon}-tiny objects. It follows from Proposition \ref{dgsmall} that the weak factorsiation system $(\mathcal{C},\mathcal{F}\cap\mathcal{W})$ is cofibrantly generated. From our proof above that $(\mathcal{C}\cap\mathcal{W},\mathcal{F})$ is a weak factorisation system, it follows that $\{0\rightarrow D^{n}(P):n>0,P\in\mathcal{P}\}$ is a set of generating morphisms for $(\mathcal{C}\cap\mathcal{W},\mathcal{F})$, so it is also a cofibrantly generated weak factorisation system. The claim about combinatoriality is clear.
\end{proof}

\begin{rem}
The existence of the projective model structure on $ Ch_{\ge0}(\mathpzc{E})$ in the case that $\mathpzc{E}$ is quasi-abelian was also known. This is mentioned in a math.stackexchange.com exchange, \cite{846202}, as an adaptation of the proof for $ Ch_{+}(\mathpzc{E})$ in \cite{buhler2011algebraic}.
\end{rem}

We finish with the non-positively graded case, which appears for the category of abelian groups in \cite{castiglioni2004cosimplicial}.
\begin{thm}
Let $\mathpzc{E}$ be an exact category which has cokernels and $DG$-projective resolutions. Then the projective model structure exists on $Ch_{\le0}(\mathpzc{E})$. Moreover, it is the right transferred model structure along the adjunction
$$\adj{\tau_{\le0}}{Ch(\mathpzc{E})}{Ch_{\le0}(\mathpzc{E})}{i_{\le0}}$$
where $i_{\le0}$ is the inclusion functor. In particular it is cofibrantly generated if $\mathpzc{E}$ is elementary, and combinatorial if $\mathpzc{E}$ is locally presentable.
\end{thm}
\begin{proof}
Let $g:X\rightarrow Y$ be a map of complexes in $Ch_{\le0}(\mathpzc{E})$. We factor it as $c:X\rightarrow\widetilde{Y},f:\widetilde{Y}\rightarrow Y$, where $c$ is a cofibration in $Ch(\mathpzc{E})$, $f$ is a fibration in $Ch(\mathpzc{E})$, and one of them is acyclic. We therefore get a factorisation of $f$ as $\tau_{\le0}c:X\rightarrow\tau_{\le0}\widetilde{Y}$, $\tau_{\le0}f:\tau_{\le0}\widetilde{Y}\rightarrow Y$. If $c$ was an equivalence then so is $\tau_{\le0}c$, and likewise for $f$ and $\tau_{\le0}f$. It remains to show that $\tau_{\le0}c$ is a cofibration, and $\tau_{\le0}f$ a fibration.  Let
\begin{displaymath}
\xymatrix{
X\ar[d]^{\tau_{\le0}c}\ar[r] & A\ar[d]^{g}\\
\tau_{\le0}\widetilde{Y}\ar[r] & B
}
\end{displaymath}
be a commutative diagram with $g$ an acyclic fibration. Then $g$ is an acyclic fibration in $Ch(\mathpzc{E})$. Thus there is a lift in the diagram
\begin{displaymath}
\xymatrix{
X\ar[d]^{c}\ar[r] & A\ar[d]^{g}\\
\widetilde{Y}\ar[r] & B
}
\end{displaymath}
Since $A$ is concentrated in non-positive degrees the map $\widetilde{Y}\rightarrow A$ factors through $\tau_{\le0}\widetilde{Y}$, and gives a lift in $Ch_{\le0}(\mathpzc{E})$. Thus $\tau_{\le0}c$ is a cofibration.
Now the composition $\widetilde{Y}_{n}\rightarrow(\tau_{\le0}\widetilde{Y})_{n}\rightarrow Y_{n}$ is an admissible epimorphism for each $n\le 0$, so $(\tau_{\le0}\widetilde{Y})_{n}\rightarrow Y_{n}$ is an admssible epimorphism for each $n\le0$. Hence $\tau_{\le0}f$ is a fibration. The fact that it is the right-transferred model structure is clear. 
\end{proof}
\begin{rem}
Note that since the model category structure on $Ch_{\le0}(\mathpzc{E})$ is transferred from the one on $Ch(\mathpzc{E})$, if $\mathcal{P}$ is a projective generating set for $\mathpzc{E}$, then $\{0\rightarrow D^{n}(P):P\in\mathcal{P},n<0\}$ is a generating set of acyclic cofibrations for $Ch_{\le0}(\mathpzc{E})$, and 
$$\{0\rightarrow D^{n}(P):P\in\mathcal{P},n<0\}\cup\{S^{0}(P)\rightarrow0:P\in\mathcal{P}\}\cup\{S^{n}(P)\rightarrow D^{n+1}(P):n<0,P\in\mathcal{P}\}$$
is a set of generating cofibrations for $Ch_{\le0}(\mathpzc{E})$.
\end{rem}

\subsection{The Projective Model Structure on Monoidal Exact Categories}

We now turn our attention to monoidal model structures on categories of chain complexes.  Using the technology developed earlier, it is reasonably easy to prove the following.

\begin{thm}\label{chainmonoidal}
Let $\mathpzc{E}$ be a projectively monoidal weakly idempotent complete  exact category with enough projectives. Then the projective model structure on $ Ch_{+}(\mathpzc{E})$ is monoidal. If $\mathpzc{E}$ also has kernels, then the projective model structure on $ Ch_{*}(\mathpzc{E})$ is monoidal. If in addition $\mathpzc{E}$ is weakly elementary then $ Ch_{*}(\mathpzc{E})$ satisfies the monoid axiom, and moreover all of this is also true for $Ch(\mathpzc{E})$
\end{thm}

\begin{proof}
By Theorem \ref{projmod}, Theorem \ref{projmodker}, and Proposition \ref{monoidaldgcompat} all that remains to notice is that $k$ is projective, projective objects are flat, and the class of projective objects is closed under $\otimes$. 
\end{proof}

The model category $Ch_{\le0}(\mathpzc{E})$ \textit{is not} a monoidal model category in general. Indeed consider the generating cofibration $S^{0}(P)\rightarrow 0$, and let $C$ be cofibrant. In particular $0\rightarrow C$ is a cofibration. The pushout-product axiom would imply that $S^{0}(P)\otimes C\rightarrow 0$ is a cofibration. Consider $C=S^{1}(Q)$ for some projective $Q$, and let $P$ be the monoidal unit. Then $S^{0}(P)\otimes C\cong S^{1}(P\otimes Q)$. The map $S^{1}(P\otimes Q)\rightarrow 0$ is very rarely a cofibration. 

\subsection*{Duality}

In any closed monoidal category $(\mathpzc{E},\otimes,k,\underline{Hom})$ one can consider the functor 
$$(-)^{\vee}:\mathpzc{E}\rightarrow\mathpzc{E}^{op},\; E\mapsto\underline{Hom}(E,k)$$
This functor is contravariantly self-adjoint.
\begin{prop}
Let $\mathpzc{E}$ be a monoidal elementary exact category. The functor $(-)^{\vee}:Ch_{*}(\mathpzc{E})\rightarrow Ch_{*}(\mathpzc{E})^{op}$ is left Quillen for the projective model structure on the left and its opposite model structure on the right.
\end{prop}
\begin{proof}
Since any object of $Ch_{*}(\mathpzc{E})^{op}$ is cofibrant and $\underline{Hom}(-,k)$ clearly preserves degree-wise split exact sequences all that remains to prove is that it sends trivially cofibrant objects to acyclic objects. Indeed if $P_{\bullet}$ is trivially cofibrant then $P_{\bullet}\rightarrow0$ is a homotopy equivalence. Hence $0\rightarrow(P_{\bullet})^{\vee}$ is a homotopy equivalence and we are done.
\end{proof}

\section{The Dold-Kan Correspondence}

In this section we generalise the Dold-Kan correspondence for abelian groups to exact categories exact categories with enough projectives. If $\mathpzc{C}$ is a category, we denote by $\textbf{s}\mathpzc{C}$ the functor category $[\Delta^{op},\mathpzc{C}]$, where $\Delta$ is the usual simplicial category.  We use this to show that when $\mathpzc{E}$ is elementary the projective model structure on $Ch(\mathpzc{E})$ and $Ch_{\ge0}(\mathpzc{E})$ are Kan complex-enriched.

Let us recall the Dold-Kan correspondence for abelian categories. The exposition here follows \cite{Weibel} 8.4.
For an abelian category $\mathpzc{A}$, there are functors 
$$\Gamma: Ch_{\ge0}(\mathpzc{A})\rightarrow\textbf{s}\mathpzc{A},\;N:\textbf{s}\mathpzc{A}\rightarrow Ch_{\ge0}(\mathpzc{A})$$
constructed as follows.

To construct $N$, we first define an `unnormalised', simpler functor $C:\textbf{s}\mathpzc{A}\rightarrow Ch_{\ge0}(\mathpzc{A})$. 
Given an object $A\in\textbf{s}\mathpzc{A}$ let $CA$ be the complex with $(CA)_{n}=A_{n}$, and differential $\partial_{n}:(CA)_{n}\rightarrow (CA)_{n-1}$ given by
$$\sum_{i=0}^{n}(-1)^{i}d_{i}$$
where the $d_{i}$ are the degeneracy maps. One checks straightforwardly that this is well-defined as a chain complex. Now let
$$(NA)_{n}=\bigcap_{i=0}^{n-1}\textrm{Ker}(d_{i})$$
$NA\subset CA$ is in fact a subcomplex, and the restricted differential is  $\delta_{n}=(-1)^{n}d_{n}:NA_{n}\rightarrow NA_{n-1}$. Since by definition a map of simplicial objects commutes with the face maps,  the constructions $CA$ and $NA$ are functorial. Moreover we have the following
\begin{prop}[\cite{Goerss-Jardine}, Theorem III 2.1, 2.4]
The natural inclusion $N\rightarrow C$ is a natural homotopy equivalence, and a split monomorphism.
\end{prop}
The construction of $\Gamma$ is more involved. For a chain complex $C\in Ch_{\ge0}(\mathpzc{A})$, one sets
$$\Gamma(C)_{n}=\bigoplus_{\eta:[n]\twoheadrightarrow [p],p\le n}C_{\eta}$$
where for $\eta:[n]\twoheadrightarrow [p]$, $C_{\eta}=C_{p}$. Given a morphism $\alpha:[n]\rightarrow[m]$ in $\Delta$, define a morphism $\Gamma(C)(\alpha):\Gamma_{m}(C)\rightarrow \Gamma_{n}(C)$ by its restriction $\Gamma(\alpha,\eta):C_{\eta}\rightarrow\Gamma(C)$ to each summand $C_{\eta}$ as follows. For each surjection $\eta:[n]\rightarrow[p]$ we consider its epi-mono factorisation $\epsilon\eta'$ of $\eta\alpha$.
\begin{displaymath}
\xymatrix{
[m]\ar[r]^{\alpha}\ar[d]^{\eta'} & [n]\ar[d]^{\eta}\\
[q]\ar[r]^{\epsilon} & [p]
}
\end{displaymath}
If $p=q$ so that $\eta\alpha=\eta'$ then we take $\Gamma(\alpha,\eta)$ to be the natural identification of $C_{\eta}$ with the summand $C_{\eta'}$ of $\Gamma_{m}$. If $p=q+1$ and $\epsilon=\epsilon_{p}$, so that the image of $\eta\alpha$ is $\{0,\ldots,p-1\}$, then we take $\Gamma(\alpha,\eta)$ to be the composition
\begin{displaymath}
\xymatrix{
C_{\eta}=C_{p}\ar[r]^{d} & C_{p-1}=C_{\eta'}\ar[r] & \Gamma_{m}(C)
}
\end{displaymath}
Otherwise we take $\Gamma(\alpha,\eta)$ to be $0$.
 The Dold-Kan Correspondence says the following
\begin{thm}[Dold-Kan for Abelian Categories]\label{dkab}
Let $\mathpzc{A}$ be an abelian category. Then the functors
$$\Gamma: Ch_{\ge0}(\mathpzc{A})\rightarrow\textbf{s}\mathpzc{A},\;N:\textbf{s}\mathpzc{A}\rightarrow Ch_{\ge0}(\mathpzc{A})$$
form an equivalence of categories.
\end{thm}

\begin{proof}
See \cite{Weibel} \S 8.4.
\end{proof}

The constructions of $\Gamma$ and $C$ make sense in any additive category, and  $N$ makes sense in any additive category which has kernels. Thus for an additive category $\mathpzc{E}$ with kernels we get functors
$$\Gamma: Ch_{\ge0}(\mathpzc{E})\rightarrow\textbf{s}\mathpzc{E},\;\;N:\textbf{s}\mathpzc{E}\rightarrow Ch_{\ge0}(\mathpzc{E})$$
constructed mutatis mutandis as above. We also get a functor $C:\textbf{s}\mathpzc{E}\rightarrow Ch_{\ge0}(\mathpzc{E})$, and a natural split inclusion and homotopy equivalence $i:N\rightarrow C$.

\begin{cor}[Dold-Kan for Exact Categories]
Let $\mathpzc{E}$ be an elementary exact category. The functors
$$\Gamma: Ch_{\ge0}(\mathpzc{E})\rightarrow\textbf{s}\mathpzc{E},\;\;N:\textbf{s}\mathpzc{E}\rightarrow Ch_{\ge0}(\mathpzc{E})$$
defined above are weakly inverse to each other. In particular they give equivalences of categories.
\end{cor}

\begin{proof}
Pick a left abelianisation $I:\mathpzc{E}\rightarrow\mathpzc{A}$. Then $I$ extends to functors $\textbf{s}\mathpzc{E}\rightarrow\textbf{s}\mathpzc{A}$ and $ Ch_{\ge0}(\mathpzc{E})\rightarrow Ch_{\ge0}(\mathpzc{A})$, which we will also denote by $I$. Since $I$ preserves kernels we get a commutative diagram.
\begin{displaymath}
\xymatrix{
\textbf{s}\mathpzc{A}\ar[r]^{N} &  Ch_{\ge0}(\mathpzc{A})\\
\textbf{s}\mathpzc{E}\ar[u]^{I}\ar[r]^{N} &  Ch_{\ge0}(\mathpzc{E})\ar[u]^{I}
}
\end{displaymath}
It is also clear from the construction of $\Gamma$ that the following diagram commutes
\begin{displaymath}
\xymatrix{
\textbf{s}\mathpzc{A} &  Ch_{\ge0}(\mathpzc{A})\ar[l]^{\Gamma}\\
\textbf{s}\mathpzc{E}\ar[u]^{I} &  Ch_{\ge0}(\mathpzc{E})\ar[u]^{I}\ar[l]^{\Gamma}
}
\end{displaymath}
Since the functor $I$ is fully faithful, Theorem \ref{dkab} implies the result.
\end{proof}

\begin{rem}
This result is actually overkill. It has been pointed out to us by Theo B\"{u}hler that the Dold-Kan equivalence is valid for any weakly idempotent complete additive category. A proof (which in fact works on the level of quasi-categories) can be found in \cite{joyal2008notes} Section 35.
\end{rem}

If $\mathpzc{A}=\mathpzc{Ab}$ is just the category of abelian groups, then there is a well-known model structure on the category $\textbf{s}\mathpzc{Ab}$. The weak equivalences (resp.  fibrations) are those maps of simplicial abelian groups which are weak equivalences (resp. fibrations) on the underlying simplicial set. As usual, the cofibrations are maps of simplicial abelian groups which have the left lifting property with respect to the trivial fibrations. Moreover, the category $\mathpzc{Ab}$ is an elementary abelian category. As a set of compact projective generators we can take $\mathcal{P}=\{\Z\}$. Thus there is a projective model structure on $ Ch_{\ge0}(\mathpzc{Ab})$. In this case the functors $N$ and $\Gamma$  also form a Quillen equivalence between these model categories. For a proof see \cite{Goerss-Jardine} Chapter 3 Section 2. 
The model structure on $\textbf{s}\mathpzc{Ab}$ is a special case of a much more general model structure.
\begin{notation}
\begin{enumerate}
\item
Let $Z$ be an object in a category $\mathpzc{C}$. We denote by $\textbf{s}Z$ the constant simplicial object in $s\mathpzc{C}$ which is $Z$ in each degree, and such that the face and degeneracy maps are all $\textrm{id}_{Z}$.
\item
If $\mathpzc{C}$ is additive, then the category $\textbf{s}\mathpzc{C}$ is enriched over $\textbf{s}\mathpzc{Ab}$ in an obvious way. We denote the enriched hom functor by $\textbf{Hom}_{\textbf{s}\mathpzc{C}}$
\end{enumerate}
\end{notation}

The result below is a special case of Theorem 6.3 in \cite{christensen2002quillen}.
\begin{thm}
Let $\mathpzc{E}$ be a complete and cocomplete exact category with enough projectives. There is a model structure on $\textbf{s}\mathpzc{E}$ in which a map $f:X\rightarrow Y$ is a weak equivalence (respectively fibration)
 if and only if the induced map 
$$\textbf{Hom}_{\textbf{s}\mathpzc{E}}(\textbf{s}P,A)\rightarrow\textbf{Hom}_{\textbf{s}\mathpzc{E}}(\textbf{s}P,B)$$
is a weak equivalence (respectively fibration) for all projectives $P$.
\end{thm}

\begin{thm}[Model Dold-Kan for Elementary Exact Categories]\label{dkex}
Let $\mathpzc{E}$ be a  complete and cocomplete exact category with enough projectives. Endow $ Ch_{\ge0}(\mathpzc{E})$ and $\textbf{s}\mathpzc{E}$ with their projective model structures. Then the functors
$$\Gamma: Ch_{\ge0}(\mathpzc{E})\rightarrow\textbf{s}\mathpzc{E},\; N:\textbf{s}\mathpzc{E}\rightarrow Ch_{\ge0}(\mathpzc{E})$$
form a Quillen equivalence.
\end{thm}
We use the following notion.

\begin{defn}
Let $\mathpzc{M},\mathpzc{N}$ be model categories. $\mathpzc{M}$ is said to be \textbf{generated} by a collection of functors $\{F_{i}:\mathpzc{M}\rightarrow\mathpzc{N}\}_{i\in I}$ if  a map $f:X\rightarrow Y$ in $\mathpzc{M}$ is a fibration (resp. weak equivalence) if and only if $F_{i}(f)$ is a fibration (resp. weak equivalence) for each $i\in I$.
\end{defn}

By construction the model structure on $\textbf{s}\mathpzc{E}$ is generated by the functors 
$$\{\textbf{Hom}_{\textbf{s}\mathpzc{E}}(\textbf{s}P,-):\textbf{s}\mathpzc{E}\rightarrow\textbf{s}\mathpzc{Ab}\}_{P\in\mathcal{P}}$$
where we endow $\textbf{s}\mathpzc{Ab}$ with its projective model structure.

The model structure on $ Ch_{\ge0}(\mathpzc{E})$ is generated by a similar set of functors.

\begin{prop}
Let $\mathpzc{E}$ be a weakly idempotent complete exact category with  class of projectives $\mathcal{P}$. The projective model structure on $ Ch_{\ge0}(\mathpzc{E})$ is generated by the functors 
$$\{\textbf{Hom}(S^{0}(P),-): Ch_{\ge0}(\mathpzc{E})\rightarrow Ch_{\ge0}(\mathpzc{Ab}):P\in\mathcal{P}\}$$
where we endow $ Ch_{\ge0}(\mathpzc{Ab})$ with its projective model structure.
\end{prop}

\begin{proof}
The fibrations in $ Ch_{\ge0}(\mathpzc{E})$ are the degree-wise admissible epics in positive degree, and the fibrations in $ Ch_{\ge0}(\mathpzc{Ab})$ are the degree-wise epics in positive degree. Let $f_{\bullet}:X_{\bullet}\rightarrow Y_{\bullet}$ be a morphism in $ Ch_{\ge0}(\mathpzc{E})$. Then the components of $\textbf{Hom}(S^{0}(P),f_{\bullet})$ are $\textrm{Hom}_{\mathpzc{E}}(P,f_{n})$. Now $f_{\bullet}$ is a fibration if and only if each $f_{n}$ is an admissible epimorphism for $n>0$. This is true if and only if $\textrm{Hom}_{\mathpzc{E}}(P,f_{n})$ is an epic for each $n>0$ and each $P\in\mathcal{P}$, i.e. if and only if $\textbf{Hom}(S^{0}(P),f_{\bullet})$ is a fibration for each $P\in\mathcal{P}$. 

It is clear that $\textbf{Hom}(S^{0}(P),\textrm{cone}(f_{\bullet}))\cong\textrm{cone}(\textbf{Hom}(S^{0}(P),f_{\bullet}))$. Now by Corollary \ref{reflshexact}, $\textrm{cone}(f_{\bullet})$ is acyclic if and only if $\textbf{Hom}(S^{0}(P),\textrm{cone}(f_{\bullet}))$ is acyclic for all $P\in\mathcal{P}$. Equivalently, $f_{\bullet}$ is a weak equivalence if and only if $\textbf{Hom}(S^{0}(P),f_{\bullet})$ is a weak equivalence for each $P\in\mathcal{P}$.
\end{proof}

With these structures in hand, we will use the following result in order to prove the theorem.

\begin{prop}\label{fiberquillen}
Let $\mathpzc{M},\mathpzc{N},\mathpzc{M}',\mathpzc{N}'$ be model categories. Suppose $\mathpzc{M}$ is generated by functors $\{F_{i}:\mathpzc{M}\rightarrow\mathpzc{N}\}_{i\in I}$, and $\mathpzc{M}'$ is generated by functors $\{F'_{i}:\mathpzc{M}'\rightarrow\mathpzc{N}'\}_{i\in I}$. Let $G:\mathpzc{M}\rightarrow\mathpzc{M}'$ and $H:\mathpzc{M}'\rightarrow\mathpzc{M}$ be adjoint functors
$$G\dashv H$$
Suppose also that there is a Quillen adjunction $P\dashv Q$, with
$P:\mathpzc{N}\rightarrow\mathpzc{N}'$ and $Q:\mathpzc{N}'\rightarrow\mathpzc{N}$ such that for each $i\in I$ the diagram
\begin{displaymath}
\xymatrix{
\mathpzc{M}\ar[d]^{F_{i}} & \mathpzc{M}'\ar[d]^{F'_{i}}\ar[l]^{H}\\
\mathpzc{N} & \mathpzc{N}'\ar[l]^{Q}
}
\end{displaymath}
commutes.
Then $G\dashv H$ is a Quillen adjunction. 
\end{prop}

\begin{proof}
We need to show that $H$ preserves (acyclic) fibrations. Let $f$ be an (acyclic) fibration in $\mathpzc{M}'$. By assumption, for each $i$, $F'_{i}(f)$ is an (acyclic) fibration in $\mathpzc{N}'$. Since $Q$ is right Quillen, $Q\circ F'_{i}(f)$ is an (acyclic) fibration. By commutativity of the diagram
\begin{displaymath}
\xymatrix{
\mathpzc{M}\ar[d]^{F_{i}} & \mathpzc{M}'\ar[d]^{F'_{i}}\ar[l]^{H}\\
\mathpzc{N} & \mathpzc{N}'\ar[l]^{Q}
}
\end{displaymath}
$F_{i}\circ H(f)$ is an (acyclic) fibration for each $i\in I$. Again by assumption, $H(f)$ is an (acyclic) fibration.
\end{proof}

Before proving the theorem, we shall make the following easy observation. 

\begin{prop}\label{Quillequiv}
Let $\mathpzc{M}$ and $\mathpzc{M'}$ be model categories, and $G:\mathpzc{M}\rightarrow\mathpzc{M}'$ and $H:\mathpzc{M}'\rightarrow\mathpzc{M}$ be Quillen adjoint functors
$$G\dashv H$$
Suppose further that
\begin{enumerate}
\item
The unit and counit maps of the adjunction are weak equivalences.
\item
$G$ preserves weak equivalences of the form $X\rightarrow HY$ where $X$ is cofibrant and $Y$ is fibrant.
\item
$H$ preserves weak equivalences of the form $GX\rightarrow Y$ where $X$ is cofibrant and $Y$ is fibrant.
\end{enumerate}
Then $G\dashv H$ is a Quillen equivalence.

\end{prop}

\begin{proof}
Let $X$ be a cofibrant object of $\mathpzc{M}$ and $Y$ a fibrant object of $\mathpzc{M}'$. Suppose that $f:GX\rightarrow Y$ is a weak equivalence. Then by assumption $HGX\rightarrow HY$ is a weak equivalence. Also by assumption $X\rightarrow HGX$ is a weak equivalence. Hence $X\rightarrow HY$ is a weak equivalence.

Conversely suppose that $X\rightarrow HY$ is a weak equivalence. Then $GX\rightarrow GHY$ is a weak equivalence by assumption. Also by assumption $GHY\rightarrow Y$ is a weak equivalence. Thus $GX\rightarrow Y$ is a weak equivalence. 
\end{proof}

\begin{proof}[Proof of Theorem \ref{dkex}]
We first note that the following diagrams commute (up to natural isomorphism).
\begin{displaymath}
\xymatrix{
\textbf{s}\mathpzc{E}\ar[d]_{\textbf{Hom}_{\textbf{s}\mathpzc{E}}(\textbf{s}P,-)} &  Ch_{\ge0}(\mathpzc{E})\ar[l]^{\Gamma}\ar[d]^{\textbf{Hom}(S^{0}(P),-)}\\
\textbf{s}\mathpzc{Ab} &  Ch_{\ge0}(\mathpzc{Ab})\ar[l]^{\Gamma}
}
\end{displaymath}

\begin{displaymath}
\xymatrix{
\textbf{s}\mathpzc{E}\ar[d]_{\textbf{Hom}_{\textbf{s}\mathpzc{E}}(\textbf{s}P,-)}\ar[r]^{N} &  Ch_{\ge0}(\mathpzc{E})\ar[d]^{\textbf{Hom}(S^{0}(P),-)}\\
\textbf{s}\mathpzc{Ab}\ar[r]^{N} &  Ch_{\ge0}(\mathpzc{Ab})
}
\end{displaymath}
The second diagram follows from the fact that $\textrm{Hom}(P,-):\mathpzc{E}\rightarrow\mathpzc{Ab}$ preserves kernels (and therefore intersections). The first diagram follows from the fact that $\textrm{Hom}(P,-):\mathpzc{E}\rightarrow\mathpzc{Ab}$ preserves finite direct sums. By Proposition \ref{fiberquillen} the adjunction is a Quillen adjunction. Let us now check the hypotheses of Proposition \ref{Quillequiv}. The unit and counit maps are isomorphisms. In particular they are weak equivalences. In the Dold-Kan correspondence for abelian groups, it can be shown that the functors $N:\textbf{s}\mathpzc{Ab}\rightarrow Ch_{\ge0}(\mathpzc{Ab})$ and $\Gamma: Ch_{\ge0}(\mathpzc{Ab})\rightarrow\textbf{s}\mathpzc{Ab}$ both preserve all weak equivalences. By the commutativity of the above diagrams, this also implies that the functors $N:\textbf{s}\mathpzc{E}\rightarrow Ch_{\ge0}(\mathpzc{E})$ and $\Gamma: Ch_{\ge0}(\mathpzc{E})\rightarrow\textbf{s}\mathpzc{E}$ also preserve all weak equivalences.
\end{proof}

Of course if $\mathpzc{E}$ is a weakly idempotent complete additive category with kernels, and there is a model structure on $Ch_{\ge0}(\mathpzc{E})$, then via the categorical equivalence $Ch_{\ge0}(\mathpzc{E})\rightarrow\textbf{s}\mathpzc{E}$ there is a transferred model category structure on $\textbf{s}\mathpzc{E}$ which is Quillen equivalent to the one on $Ch_{\ge0}(\mathpzc{E})$. The point of the above is that for the projective model structure on $Ch_{\ge0}(\mathpzc{E})$, the transferred model structure coincides with the one established in Theorem 6.3 of \cite{christensen2002quillen}.

\begin{rem}\label{adjunbounded}
Let $\mathpzc{E}$ be a small complete and cocomplete exact category equipped with a small $dg_{\ge0}$-compatible cotorsion pair $(\mathfrak{L},\mathfrak{R})$. Equip $Ch_{\ge0}(\mathpzc{E})$ with the resulting model category structure, and $\textbf{s}\mathpzc{E}$ with the model category structure transferred along the Dold-Kan equivalence Suppose that the cotorsion pair is also $dg$-compatible. By Proposition \ref{truncquillen} we also have an adjunction
$$\adj{i\circ N}{\mathpzc{s\mathpzc{E}}}{Ch(\mathpzc{E})}{\Gamma\circ\tau_{\ge0}}$$
\end{rem}

\subsection{Exactness of the Dold-Kan Functors and Left Properness}
As a diagram category, $\textbf{s}\mathpzc{E}$ can be equipped with the structure of an exact category, where exact sequences are determined term-wise. With this exact structure we have the following.
\begin{prop}
Both of the functors $N:\textbf{s}\mathpzc{E}\rightarrow Ch_{\ge0}(\mathpzc{E})$ and 
$\Gamma:Ch_{\ge0}(\mathpzc{E})\rightarrow\textbf{s}\mathpzc{E}$ are exact.
\end{prop}
\begin{proof}
$\Gamma$ is clearly exact, and $N$ is a retract of the functor $C$ which is also clearly exact. 
\end{proof}
The following is tautological.
\begin{prop}
Let $\mathpzc{C}$ be a model category, and $F:\mathpzc{C}\rightarrow\mathpzc{D}$ an equivalence. Equip $\mathpzc{D}$ with the transferred model structure. Let $f:X\rightarrow Y$ be a left proper map in $\mathpzc{C}$. Then $F(f)$ is left-proper in $\textbf{s}\mathpzc{E}$.
\end{prop}
\begin{cor}
Admissible monomorphisms in $\textbf{s}\mathpzc{E}$ are left-proper.
\end{cor}
\subsection{The Cosimplicial Dold-Kan Correspondence}
In this section we discuss a generalisation of the cosimplicial Dold-Kan correspondence of \cite{castiglioni2004cosimplicial} to exact categories. If $\mathpzc{E}$ is a weakly idempotent complete exact category then so is $\mathpzc{E}^{op}$. Thus there is an equivalence 
$$\adj{\Gamma_{c}}{Ch_{\le0}(\mathpzc{E})}{\textbf{cs}\mathpzc{E}}{N_{c}}$$
between the categories $Ch_{\le0}(\mathpzc{E})$ and the category $\textbf{cs}\mathpzc{E}$ of cosimplicial objects. Therefore a model structure on $Ch_{\le0}(\mathpzc{E})$ induces a Quillen equivalent model structure on $\textbf{cs}\mathpzc{E}$. In this case, it is not immediately clear that a map $f:X\rightarrow Y$ in $\textbf{cs}\mathpzc{E}$ is a weak equivalence if and only if the map $Hom(P,f)$ of cosimplicial abelian groups is an equivalence for each projective $P$, where $Hom(P,-)$ is applied level-wise. This is because cokernels appear in the construction of $N_{c}$, and $Hom(P,-)$ does not necessarily commute with these. Fortunately we still have the following.
\begin{prop}\label{prop:cosimpgenproj}
A map $f:X\rightarrow Y$ in $\textbf{cs}(\mathpzc{E})$ is a weak equivalence if and only if for any projective $P$, $\underline{Hom}(\textbf{cs}P,f)$ is a weak equivalence  in $\textbf{cs}\mathpzc{Ab}$. 
\end{prop}
\begin{proof}
Let $C_{c}:\textbf{cs}\mathpzc{E}\rightarrow Ch_{\le0}(\mathpzc{E})$ be cosimplicial unnormalised Moore complex functor (i.e. the unnormalised Moore complex functor computed in the opposite category). There is a natural homotopy equivalence and split epimorphism $C_{c}\rightarrow N_{c}$. Thus a map $f:X\rightarrow Y$ in $\textbf{cs}\mathpzc{E}$ is an equivalence if and only if $C_{c}(f)$ is an equivalence. Now there is a commutative diagram
\begin{displaymath}
\xymatrix{
Ch_{\le0}(\mathpzc{E})\ar[d]^{\textbf{Hom}(S^{0}(P),-)} & \textbf{cs}\mathpzc{E}\ar[l]^{C_{c}}\ar[d]^{\underline{Hom}(\textbf{cs}P,-)}\\
Ch_{\le0}(\mathpzc{Ab}) & \textbf{cs}\mathpzc{Ab}\ar[l]^{C_{c}}
}
\end{displaymath}
and the result follows.
\end{proof}

\subsubsection{An Alternative Cosimplicial Dold-Kan Correspondence}
In \cite{castiglioni2004cosimplicial} Castiglioni and Corti{\~n}a construct an alternative pair of functors between cosimplicial and non-positively graded complexes. Their construction also works for more general additive categories. Let $\mathpzc{E}$ be a cocomplete additive category. Then $\mathpzc{E}$ is tensored over the category $\mathpzc{Ab}$ of abelian groups. For $n\ge 0$, write $V^{n}$ for the abelian group given by the kernel of the canonical map
$$\bigoplus_{i=0}^{n}\mathbb{Z}\rightarrow\mathbb{Z}$$
Let $\{e_{i}:0\le i\le n\}$ be the standard basis of $\bigoplus_{i=0}^{n}\mathbb{Z}$. Then $\{v_{i}=e_{i}-e_{0}:0\le i\le n\}$ is a basis of $V^{n}$. $V^{\bullet}$ may be regarded as an object of $\textbf{cs}\mathpzc{Ab}$ as follows. For $\alpha:[n]\rightarrow [m]$ a map, set 
$$\alpha v_{i}=v_{\alpha(i)}-v_{\alpha(0)}$$
and extend to $V^{n}$ by linearity. Let $T(V)$ denote the free associative monoid in $\textbf{cs}\mathpzc{Ab}$ on $V$. We shall denote the multiplication on this algebra by $\mu$. 
Write 
$$Q:Ch_{\le0}(\mathpzc{E})\rightarrow\textbf{cs}\mathpzc{E}$$
for the following functor. Let $(A_{\bullet},d)\in Ch_{\le0}(\mathpzc{E})$. Write
$$(QA)_{n}=\bigoplus_{i=0}^{\infty}A_{-n}\otimes T^{i}(\mathbb{Z}^{n})$$
For any map $\alpha:[n]\rightarrow[m]$ of finite sets, there is a map
$$(QA)_{n}\rightarrow(QA)_{m}$$
defined by
$$Id_{A}\otimes \alpha+d\otimes\mu(v_{\alpha_{0}}\otimes\alpha)$$
It is immediately verified as in \cite{castiglioni2004cosimplicial} Section 3 that this is a well-defined object of $\textbf{cs}\mathpzc{E}$. 
\begin{prop}[Proposition 7.4/ Remark 7.5 \cite{castiglioni2004cosimplicial}]
The functor
$$Q:Ch_{\le0}(\mathpzc{E})\rightarrow\textbf{cs}\mathpzc{E}$$
has a right adjoint $H$.
\end{prop}
\begin{proof}
Of course with some mild assumptions one can prove this using the adjoint functor theorem. As in \cite{castiglioni2004cosimplicial} we give an explicit construction.
For $B\in\textbf{cs}\mathpzc{E}$ write
$$DB=\bigoplus_{n\le0}\bigoplus_{P\in\mathcal{P}}\bigoplus_{Hom(Q(D^{n}(P)),B)}D^{n}(P)$$
For $s\in Hom(Q(D^{n}(P)),B)$ let $j_{s}:D^{n}(P)\rightarrow B$ be the inclusion. Denote by $\alpha:QDB\rightarrow B$ the map defined on the summand $Q(j_{s}):QD^{n}(P)\rightarrow B$ by $s$.  Let $K=\textrm{colim}_{f:I\rightarrow DB:\alpha\circ Q(f)=0}I$, and set $HB=coker(K\rightarrow I)$. The functor $Q$ commutes with colimits, so $Q(HB)\cong\textrm{coker}(Q(K)\rightarrow QDB)$. In particular $\alpha:QDB\rightarrow B$ factors through a map $\hat{\alpha}:DHB\rightarrow B$. Thus $(HB,\hat{\alpha})$ is an object of the arrow category $Q\uparrow B$. To prove that $H$ is a right adjoint is equivalent to proving that it is a final object of this category. Now the proof proceeds formally the same way as \cite{castiglioni2004cosimplicial} Proposition 7.4.
\end{proof}
A formally identical proof to \cite{castiglioni2004cosimplicial} Theorem 4.2 i) shows that there is a natural homotopy equivalence $\hat{p}:Q\rightarrow\Gamma_{c}$
Now let $\mathpzc{E}$ be a weakly idempotent complete exact category such that the projective model structure exists on $Ch_{\le0}(\mathpzc{E})$, and equip $Ch_{\le0}(\mathpzc{E})$ and $\textbf{cs}\mathpzc{E}$ with the projective model structures. 
\begin{lem}
The adjunction
$$\adj{Q}{Ch_{\le0}(\mathpzc{E})}{\textbf{cs}\mathpzc{E}}{H}$$
is a Quillen equivalence.
\end{lem}
\begin{proof}
It suffices to prove that the adjunction is a Quillen adjunction. Indeed if this is the case then we get an adjunction of homotopy categories
$$\adj{\mathbb{L}Q}{\textrm{Ho}(Ch_{\le0}(\mathpzc{E}))}{\textrm{Ho}(\textbf{cs}\mathpzc{E})}{\mathbb{R}H}$$
The functors $Q$ and $\Gamma_{c}$ preserve all equivalences  (since $\Gamma_{c}$ does, and $Q$ is homotopy equivalent to it), so we get equivalences of functors of homotopy categories 
$$\mathbb{L}Q\cong Q\cong \Gamma_{c}\cong\mathbb{L}\Gamma_{c}$$
Since 
$$\adj{\Gamma_{c}}{Ch_{\le0}(\mathpzc{E})}{\textbf{cs}\mathpzc{E}}{N_{c}}$$
is a Quillen equivalence, $\mathbb{L}\Gamma_{c}\cong\mathbb{L}Q$ is an equivalence of categories. It remains to show that adjunction between $Q$ and $H$ is Quillen. 
Now we have shown that the functor $Q$ preserves all equivalences. It therefore suffices to prove that it preserves cofibrations. We will show that in fact $N_{c}Q(f)$ is a cofibration for any generating cofibration $f:X\rightarrow Y$. By the proof of Theorem 4.2 in \cite{castiglioni2004cosimplicial}, for $X$ an object of $Ch_{\le0}(\mathpzc{E})$, $(N_{c}QX)_{-n}\cong\bigoplus_{-r=n}^{\infty}X_{r}\otimes\mathbb{Z}[Sur_{r,n}]$, where $Sur_{r,n}$ is the set of surjections from the set with $r$ elements to the set with $n$ elements. In particular, if $X$ is bounded below then so is $N_{c}QA$. First consider a cofibration of the form $S^{0}(P)\rightarrow 0$ with $P$ projective. $(N_{c}QS^{0}(P))_{m}=0$ for $m<0$, so $N_{c}QS^{0}(P)\cong S^{0}(P)$, and $N_{c}Q(S^{0}(P)\rightarrow 0)\cong S^{0}(P)\rightarrow 0$ is a cofibration. Next consider a cofibration of the form $S^{-n}(P)\rightarrow D^{-n+1}(P)$ for $n\ge 1$ and $P$ projective. $N_{c}Q(S^{-n}(P)\rightarrow D^{-n+1}(P))$ is degree-wise split, and the cokernel $C$ satisfies $C_{-m}\cong (N_{c}QS^{-n+1}(P))_{-m}$. This is a bounded below complex of projectives, and is therefore $DG$-projective. 
\end{proof}

\subsection{The Simplicial Model Structure}

In this section we determine conditions under which the model structures on $Ch_{\ge0}(\mathpzc{E})$ and $Ch(\mathpzc{E})$ are simplicial (refer to Section \ref {twovarquill} for the terminology in this section). First note that we have the following result, using Proposition \ref{twovarcont}, Theorem \ref{dkex}, and Remark \ref{adjunbounded}.
\begin{prop}
Let $\mathpzc{E}$ be a small complete and cocomplete exact category equipped with a small $dg_{\ge0}$-compatible cotorsion pair $(\mathfrak{L},\mathfrak{R})$. Endow $ Ch_{\ge0}(\mathpzc{E})$ and $\mathpzc{sE}$ with their induced model structures. Let $\mathpzc{M}$ be a $Ch_{\ge0}(\mathpzc{E})$-model category. Then $\mathpzc{M}$ is a $\mathpzc{sE}$-model category. If $(\mathfrak{L},\mathfrak{R})$ is also $dg$-compatible, then any $Ch(\mathpzc{E})$-model category is also a $\mathpzc{sE}$-model category.
\end{prop}
We claim that if $\mathpzc{E}$ is any complete and cocomplete exact category equipped with a $dg_{*}$-compatible cotorsion pair for $*\in\{\ge0,\emptyset\}$, then with the induced exact structure $Ch_{*}(\mathpzc{E})$ is a $Ch_{*}(\mathpzc{Ab})$-model category. In particular by the above proposition it is a $\mathpzc{sAb}$-model category. Since there is a Quillen adjunction
$$\adj{\mathbb{Z}\otimes(-)}{\mathpzc{sSet}}{\mathpzc{sAb}}{|-|}$$ 
this in turn implies by Proposition \ref{twovarcont} that $Ch_{*}(\mathpzc{E})$ is a simplicial model category. 
Now by Prop. 3.46 in \cite{kelly1982basic} we have the following.
\begin{prop}
Let $\mathpzc{E}$ be a complete and cocomplete additive category. Then $Ch_{*}(\mathpzc{E})$ is tensored, cotensored, and enriched over $Ch_{*}(\mathpzc{Ab})$ for $*\in\{\ge0,\emptyset\}$
\end{prop}
Let us denote the tensoring, enrichment, and cotensoring by
$$\otimes:Ch_{*}(\mathpzc{Ab})\times Ch_{*}(\mathpzc{E})\rightarrow Ch_{*}(\mathpzc{E})$$
$$\textbf{Hom}:Ch_{*}(\mathpzc{E})\times Ch_{*}(\mathpzc{E})^{op}\rightarrow Ch_{*}(\mathpzc{Ab})$$
$$(-)^{(-)}:Ch_{*}(\mathpzc{E})\times Ch_{*}(\mathpzc{E})^{op}\rightarrow Ch_{*}(\mathpzc{E})$$
\begin{prop}
Let $\mathpzc{E}$ be a complete and cocomplete exact category equipped with a $dg_{*}$-compatible cotorsion pair $(\mathfrak{L},\mathfrak{R})$ for $*\in\{\ge0,\emptyset\}$. Then the two-variable adjunction above is a two-variable Quillen adjunction.
\end{prop}
\begin{proof}
It suffices to prove that $\otimes:Ch_{*}(\mathpzc{Ab})\times Ch_{*}(\mathpzc{E})\rightarrow Ch_{*}(\mathpzc{E})$ is a left Quillen bifunctor. Let $f:A\rightarrow B$ be a cofibration in $Ch_{*}(\mathpzc{Ab})$ and $g:X\rightarrow Y$ be a cofibration in $Ch_{*}(\mathpzc{Ab})$. We may assume that $f$ is a generating cofibration. First suppose it is a generating acyclic cofibration of the form $0\rightarrow D^{n}(\mathbb{Z})$. By shifting we may assume that it is of the form $0\rightarrow D^{1}(\mathbb{Z})$. Observe that for any complex $Z$, $D^{1}(\mathbb{Z})\otimes Z\cong cone(Id_{Z})$. The pushout of the diagram
\begin{displaymath}
\xymatrix{
0\ar[d]\ar[r] & cone(Id_{X})\\
0
}
\end{displaymath}
is of course just $cone(Id_{X})$. So we need to show that the map $cone(Id_{X})\rightarrow cone(Id_{Y})$ is an acyclic cofibration. Let $C$ be the cokernel of $g:X\rightarrow Y$. Then the cokernel of $cone(Id_{X})\rightarrow cone(Id_{Y})$ is $cone(Id_{C})$. However this follows from Proposition \ref{conecofibrantiscofibrant}.
Now suppose $f$ is a generating cofibration of the form $S^{n}(\mathbb{Z})\rightarrow D^{n+1}(\mathbb{Z})$. Again by shifting we may assume that it is of the form  $S^{0}(\mathbb{Z})\rightarrow D^{1}(\mathbb{Z})$. Note that $S^{0}(\mathbb{Z})\otimes Z=Z$ for any complex $Z$. The pushout of the diagram
\begin{displaymath}
\xymatrix{
X\ar[d]^{g}\ar[r] & cone(Id_{X})\\
Y
}
\end{displaymath}
is $cone(g)$. Thus we need to show that the map $cone(g)\rightarrow cone(Id_{Y})$ is a cofibration, and an acyclic cofibration when $g$ is acyclic. The map is clearly an admissible monomorphism. Moreover its cokernel is isomorphic to $C$, which is cofibrant, and trivially cofibrant when $g$ is acyclic.
Finally, for the $\ge0$ case, we need cofibrations of the form $0\rightarrow S^{n}(\mathbb{Z})$. Again we may assume that $n=0$, in which case $S^{0}(\mathbb{Z})\otimes(-)$ is isomorphic to the identity functor. Therefore in this case the pushout-product axiom is obvious.
\end{proof}
\begin{cor}\label{Kanriched}
Let $\mathpzc{E}$ be a complete and cocomplete exact category equipped with a $dg_{\ge0}$-compatible cotorsion pair $(\mathfrak{L},\mathfrak{R})$. Then $Ch_{\ge0}(\mathpzc{E})$ is a simplicial model category. In fact it is Kan-complex enriched. If $(\mathfrak{L},\mathfrak{R})$ is also $dg$-compatible then this is also true of $Ch(\mathpzc{E})$.
\end{cor}

\section{The Injective Model Structure, the $K$-Projective Model Structure, and Examples}

\subsection{The Injective Model Structure}

\begin{defn}
Let $\mathpzc{E}$ be an exact category. If it exists, the \textbf{injective model structure} on $ Ch_{*}(\mathpzc{E})$, for $*\in\{-,b,\emptyset,\le0,\ge0\} $ is the model structure in which
\begin{itemize}
\item
Weak equivalences are quasi-isomorphisms.
\item
Cofibrations are degree-wise admissible monomorphisms for $*\in\{-,b,\emptyset,\ge0\}$, and degree-wise admissible monomorphisms in strictly negative degree for $*\in\{\le0\}$.
\end{itemize}
\end{defn}

\begin{prop}
Let $\mathpzc{E}$ be an exact category. Suppose that both the projective and injective model structures exist on $Ch_{*}(\mathpzc{E})$ for $*\in\{\emptyset,\ge0,\le0\}$. Then the adjunction
$$\adj{Id}{Ch_{*}(\mathpzc{E})}{Ch_{*}(\mathpzc{E})}{Id}$$
is a Quillen equivalence, where the left hand side is equipped with the projective model structure, and the right-hand side with the injective model structure.
\end{prop}
\begin{proof}
 Equivalences in both model structures are the same. Moreover for $*\in\{\emptyset,\ge0\}$, cofibrations on the left-hand side are in particular degree-wise admissible monomorphisms. For $*\in\{\le0\}$, cofibrations on the left-hand side are degree-wise admissible in strictly negative degree. Thus the adjunction is Quillen. The induced adjunction on homotopy categories is just the identity adjunction, so is clearly an equivalence. \end{proof}

 By duality we have the following.

 \begin{prop}
 Let $\mathpzc{E}$ be a weakly idempotent complete exact category with enough injectives. Then the injective model structure exists on $Ch_{\le0}(\mathpzc{E})$. If $\mathpzc{E}$ has cokernels, countable products of injectives, and injectives satisfy axiom $AB4*-k$ (the dual of axiom $AB4-k$) for some $k\in\mathbb{Z}$, then this is also true for $Ch(\mathpzc{E})$.
 \end{prop}
  Thanks to Lemma \ref{lem:Grothendieckinj}, we also have the following.
 \begin{cor}
 If a countably complete and countably cocomplete exact category $\mathpzc{E}$ has both enough injectives and enough projectives, then both the injective and projective model structures exist on $Ch(\mathpzc{E})$.
 \end{cor}
 
 In particular if a quasi-abelian category $\mathpzc{E}$ has enough injectives then $Ch(Pro(\mathpzc{E}))$ is equipped with the injective model structure. This was proven by Pridham for $\mathpzc{E}=Ban_{\C}$ in \cite{pridham2017k}.
 
 \subsubsection{Exact categories of Grothendieck type}
 
  In \cite{vst2012exact}, \v{S}t'ov\'{i}\v{c}ek proves the following
  \begin{thm}[\cite{vst2012exact} Theorem 7.11]
  Let $\mathpzc{E}$ be an exact category of Grothendieck type such that the class $\mathpzc{W}$ of acyclic complexes is deconstructible in $Ch(\mathpzc{E})$. Then the injective model structure exists on $Ch(\mathpzc{E})$.
  \end{thm}
When $\mathpzc{E}$ has countable products, and such products are exact, we also get the following existence result.
  \begin{thm}
  Let $\mathpzc{E}$ be an exact category of Grothendieck type. If $\mathpzc{E}$ has cokernels, countable products of injectives, and injectives satisfy axiom $AB4*-k$  for some $k\in\mathbb{Z}$. Then the injective model structure exists on $Ch(\mathpzc{E})$. 
  \end{thm}
  \begin{proof}
  The category $\mathpzc{E}^{op}$ has enough projectives by Lemma \ref{lem:grexenoughinj}, and countable coproducts of projectives satisfy $AB4-k$ for some $k$. Thus the projective model structure exists on $Ch(\mathpzc{E}^{op})$, i.e. the injective model structure exists on $Ch(\mathpzc{E})$. 
  \end{proof}
  \begin{cor}
  Let $\mathpzc{E}$ be a complete and cocomplete elementary exact category which is locally presentable. Then both the projective and injective model structures exist on $Ch(\mathpzc{E})$, and they are Quillen equivalent. 
  \end{cor}

  \subsection{The $K$-Projective Model Structure}
  Throughout this section $\mathpzc{E}$ is a weakly idempotent complete exact category. In \cite{gillespie2016exact}  Gillespie proves a general theorem concerning the existence of exact model structures in which the cofibrant objects are of the form $K\mathfrak{L}$ (rather than $\widetilde{dg\mathfrak{L}}$).  %
  \begin{thm}[\cite{gillespie2016exact} Theorem 6.3]
  Let $\mathpzc{E}$ be a weakly idempotent complete exact category with enough projectives. Let $(\mathcal{U},\mathcal{F})$ be a cotorsion pair in $Ch(\mathpzc{E})$ such that
  \begin{enumerate}
  \item
$(\mathcal{U},\mathcal{F})$ is complete.
\item
$\mathcal{U}$ is thick.
\item
$\mathcal{U}\cap\mathcal{F}$ is the class of projectives in $Ch(\mathpzc{E})$.
\end{enumerate}
Then there is a model structure on $Ch(\mathpzc{E})$ in which
\begin{enumerate}
\item
the cofibrations are the degree-wise split monomorphisms with cokernel in $K\mathcal{U}$
\item
 the trivial cofibrations are the degree-wise split monomorphisms with contractible cokernel.
 \item
the fibrations are the degree-wise split epimorphisms.
\item
the trivial fibrations are the degree-wise split epimorphisms  with kernel in $\mathcal{F}$.
\end{enumerate}
  \end{thm}
  By definition, the $K$-projective model structure is the model structure of the theorem for the cotorsion pair $(dg\widetilde{\textbf{Proj}}(\mathpzc{E}),\mathfrak{W})$. The following is proven for Grothendieck abelian categories in \cite{gillespie2016exact}  Section 6 but the proof is the same.
  \begin{prop}
  If $X\in K\textbf{Proj}$ is acyclic then it is homotopy equivalent to $0$.
  \end{prop}
  \begin{proof}
  $\textbf{Hom}(X,X)$ is acyclic. Thus $[Id_{X}]\in Hom(X,X)\big\slash\sim\cong H_{0}\textbf{Hom}(X,X)=0$ is homotopy equivalent to the zero map.
  \end{proof}
  Suppose the projective model structure exists on $Ch(\mathpzc{E})$. In particular the cotorsion pair $(dg\widetilde{\textbf{Proj}}(\mathpzc{E}),\mathfrak{W})$ is complete. Thus, as mentioned in \cite{gillespie2016exact} Corollary 6.6 in the case that $\mathpzc{E}$ is a Grothendieck abelian category with a projective generator, the $K$-projective model structure also exists on $Ch(\mathpzc{E})$. Moreover, the identity functor is clearly left Quillen from the projective model structure to the $K$-projective model structure. We claim it is a Quillen equivalence. Let $f:X\rightarrow Y$ be an equivalence in the $K$-projective model structure. Then $f$ factors as $p\circ i$ where $i$ is a degree-wise split monomorphism with contractible cokernel, and $p$ is a degree-wise split epimorphism with acyclic kernel. Thus both $p$ and $i$, and therefore $f$, are equivalences in the projective model structure. On the other hand, let $f:X\rightarrow Y$ be an equivalence in the projective model structure. We factor $f$ as $p\circ i$ where $i$ is a cofibration in the $K$-projective model structure, $p$ is a fibration in the $K$-projective model structure, and $p$ is trivial in the $K$-projective model structure. In particular it is also trivial in the projective model structure. Thus $i$ is also trivial in the projective model structure. $i$ is a degree-wise split monomorphism with $K$-projective cokernel $C$. Since it is trivial in the projective model structure $C$ is acyclic and hence contractible, as required. 
  \subsection{Examples}
All the examples of Section \ref{Secex1} satisfy the assumptions of Theorem \ref{projmod} such that their categories of unbounded complexes have projective model structures. The model structures for $Ch(Ind(Ban_{k}))$, $Ch(CBorn_{k})$, and for unbounded complexes in the contracting normed and Banach categories are monoidal and satisfy the monoid axiom. In fact if any quasi-abelian category $\mathpzc{E}$ has enough projectives then the projective model structure exists on $Ch(Ind(\mathpzc{E}))$.
\subsubsection{The Derived Category With Respect to a Generator}
  Let $\mathpzc{E}$ be a strongly efficient exact category with a generator $G=\bigoplus_{i\in\mathcal{I}}G_{i}$. Consider the $G$-exact structure on $\mathpzc{E}$. Then  $G$ is a projective generator of $\mathpzc{E}$. Moreover since $\mathpzc{E}$ is locally presentable, $G$ is $\kappa$-presented for some ordinal $\kappa$. By Proposition \ref{prop:christenssenenoughproj} $DG$-projective resolutions exist. Thus we get the following, which for the case that $\mathpzc{E}$ is a Grothendieck abelian category is Corollary 4.7 in \cite{gillespie2016derived}.
  \begin{cor}
  Equip $\mathpzc{E}$ with the $G$-exact structure. Then the projective model structure exists on $Ch(\mathpzc{E})$. Moreover it is combinatorial.
  \end{cor}
Suppose now that each $G_{i}$ is tiny. Consider the $G$-exact structure on $\mathpzc{E}$. By construction this category has enough projectives. By Corollary \ref{cor:GexactGRothendieck} it is of Grothendieck type, and therefore by Lemma \ref{lem:Grothendieckinj} it has enough injectives. Therefore we get the following, which generalises \cite{gillespie2016derived} Corollary 5.12.  
  \begin{cor}
  Both the projective and injective model structures exist on $Ch(\mathpzc{E})$ where $\mathpzc{E}$ is equipped with the $G$-exact structure.
  \end{cor}

\chapter{Filtered and Graded Objects in Exact Categories}\label{sec6}
In \cite{qacs} Schneiders shows that the category of exhaustively $\mathbb{Z}$-filtered abelian groups is an elementary quasi-abelian category. In particular the category of chain complexes of filtered abelian groups or, equivalently, filtered chain complexes of abelian groups, is equipped with a projective model structure. We generalise this to quasi-abelian and exact categories. ( In \cite{schapira2016derived} Schapira and Schneiders study derived categories of $\Lambda$-filtered objects in abelian categories, where $\Lambda$ is a more general filtered category. Similar results to theirs should hold for exact categories, though we shall not deal with that here). We will also relate our results to the work of \cite{gwilliam2018enhancing} on filtered objects in more general model categories and $(\infty,1)$-categories, and to the work of \cite{calaque2021lie}.We will also study monoidal structures on categories of filtered objects. This is crucial for the Koszul duality result Theorem 4.23\cite{kelly2019koszul} where one needs to understand associated graded objects of filtered cooperads. Our considerations of homotopically complete filtered objects will be important for joint work with Kobi Kremnizer and Devarshi Mukherjee on a bornological version of the HKR theorem, following \cite{raksit2020hochschild}.

\section{Preliminaries}
Before continuing let us fix some notation.
\begin{notation}
\begin{enumerate}
\item
If $\mathcal{I}$ is a category then we denote by $\mathcal{I}_{top}$ the category obtained by freely adjoining a terminal object $top$ to $\mathcal{I}$, and by $\mathcal{I}_{\bullet}$ the category obtained by freely adjoining a discrete element $\bullet$.
\item
Throughout this section $\mathbb{Z}$ will denote the set of integers regarded as a discrete category, and $\mathbb{Z}^{<}$ the poset of integers with increasing ordering.
\end{enumerate}
\end{notation}
\subsection{Restricted Diagram Categories}
Let $\mathpzc{E}$ be a category, $\mathcal{I}$ be a small category, and $\mathcal{S}$ a class of morphisms in $\mathpzc{E}$. We denote by $\mathpzc{Fun}_{\mathcal{S}}(\mathcal{I},\mathpzc{E})$ the full subcategory of $\mathpzc{Fun}(\mathcal{I},\mathpzc{E})$ consisting of functors $F$ such that for any $\alpha:i\rightarrow j$ in $\mathcal{I}$, $F(\alpha)$ is in $\mathcal{S}$. It will be convenient to have a description of limits and colimits in such categories. Let $D:\mathcal{J}\rightarrow \mathpzc{Fun}_{\mathcal{S}}(\mathcal{I},\mathpzc{E})$ be a diagram. For each object $j\in\mathcal{J}$ and $i\in\mathcal{I}$ write $D_{i}(j)$ for the functor $D(j)$ evaluated at $i$, and for $\alpha:i\rightarrow i'$ in $\mathcal{I}$ write $D_{\alpha}(j)$ for the corresponding map. For each $i\in\mathcal{I}$ this gives a functor $D_{i}:\mathcal{J}\rightarrow\mathpzc{E}$.
\begin{prop}\label{filtdiag}
Suppose that for each morphism $\alpha$ in $\mathcal{I}$ the map $\textrm{(co)lim}_{\mathcal{J}}D_{\alpha}:\textrm{(co)lim}_{\mathcal{J}}D_{i}\rightarrow \textrm{(co)lim}_{\mathcal{J}}D_{i'}$ is in $\mathcal{S}$. Then the functor $i\mapsto\textrm{(co)lim}_{\mathcal{J}}D_{i}(j)$ is a (co)limit of the diagram $D$ in $\mathpzc{Fun}_{\mathcal{S}}(\mathcal{I},\mathpzc{E})$.
\end{prop}
\begin{proof}
The universal property is checked directly.
\end{proof}

\section{Graded Objects in General Categories}
Before specialising to exact categories, we will discuss graded and filtered objects in arbitrary categories. Let us fix a complete and cocomplete category $\mathpzc{E}$. 
\subsection{Graded Objects}
\begin{defn}
The category of $\mathbb{Z}$-\textbf{graded objects} in $\mathpzc{E}$, denoted $\mathpzc{Gr}(\mathpzc{E})$, is the diagram category $\mathpzc{Fun}(\mathbb{Z},\mathpzc{E})$.
\end{defn}
We will write an object of $\mathpzc{Gr}(\mathpzc{E})$ as $\coprod_{n\in\mathbb{Z}}A_{n}$. The following is clear.
\begin{prop}
The category $\mathpzc{Gr}(\mathpzc{E})$ is complete and cocomplete. 
\end{prop}
For $n\in\mathbb{Z}$ we denote by $S^{n}:\mathpzc{E}\rightarrow\mathpzc{Gr}(\mathpzc{E})$ the functor sending an object $X$ to the graded object with $(S^{n}(X))_{n}=X$ and $(S^{n}(X))_{m}=0$ for $m\neq n$.

\subsection{Pointed Graded Objects}
It will also be convenient to consider the category of pointed graded objects. 
\begin{defn}
The category of \textbf{pointed graded objects} in $\mathpzc{E}$, denoted $\mathpzc{Gr}_{\bullet}(\mathpzc{E})$, is the category of functors $\mathpzc{Fun}(\mathbb{Z}_{\bullet},\mathpzc{E})$.
\end{defn}
We will write an object of $\mathpzc{Gr}_{\bullet}(\mathpzc{E})$ as $(\coprod_{n\in\mathbb{Z}}A_{n},A_{\bullet})$. Again $\mathpzc{Gr}_{\bullet}(\mathpzc{E})$ is complete and cocomplete. Functors $S^{n}:\mathpzc{E}\rightarrow\mathpzc{Gr}_{\bullet}(\mathpzc{E})$ for $n\in\mathbb{Z}$ or $n=\bullet$ are defined in the obvious way.
\subsection{Monoidal Structures on Graded Objects}
Suppose that $(\mathpzc{E},\otimes, k)$ is a (symmetric) monoidal category. The induced (symmetric) monoidal structure on $\mathpzc{Gr}(\mathpzc{E})$ is easy to define. If $\coprod_{i}A_{i}$ and $\coprod_{j}B_{j}$ are graded objects then their monoidal product is 
$$(\coprod_{i}A_{i})\otimes(\coprod_{j}B_{j})\defeq\coprod_{n}\coprod_{i+j=n}A_{i}\otimes B_{j}$$
It is functorial in the obvious way. $\mathpzc{Gr}_{\bullet}(\mathpzc{E})$ can also naturally be given the structure of a (symmetric) monoidal category by defining 
$$(\coprod_{n\in\mathbb{Z}}A_{n},A_{\bullet})\otimes (\coprod_{n\in\mathbb{Z}}B_{n},B_{\bullet})\defeq(\coprod_{n}\coprod_{i+j=n}A_{i}\otimes B_{j},A_{\bullet}\otimes B_{\bullet})$$
\subsection{Model Structures on Graded Objects}\label{modelgraded}
Let $\mathpzc{E}$ be a combinatorial model category. Say that a map in $f=(f_{n})$ in $\mathpzc{Gr}(\mathpzc{E})$ (resp. $\mathpzc{Gr}_{\bullet}(\mathpzc{E})$) is a fibration/ cofibration/ weak equivalence if for each $n\in\mathbb{Z}$ (resp. $n\in\mathbb{Z}_{\bullet}$), $f_{n}$ is a fibration/ cofibration/ weak equivalence in $\mathpzc{E}$. Since everything is computed component-wise in categories of graded objects it is clear that this determines a combinatorial model structure on $\mathpzc{Gr}(\mathpzc{E})$ (resp. $\mathpzc{Gr}_{\bullet}(\mathpzc{E})$). Note that if $I$ (resp. $J$) is a set of generating cofibrations (resp. generating acyclic cofibrations) for $\mathpzc{E}$ then $\{S^{n}(f):f\in I,n\in\mathbb{Z}\}$ (resp. $\{S^{n}(f):f\in J,n\in\mathbb{Z}\}$) is a set of generating cofibrations (resp. generating acyclic cofibrations) for $\mathpzc{Gr}(\mathpzc{E})$, and similarly for $\mathpzc{Gr}_{\bullet}(\mathpzc{E})$ by also taking $n=\bullet$. If $\mathpzc{E}$ is a monoidal model category then, again since everything is computed component-wise, both $\mathpzc{Gr}(\mathpzc{E})$ and $\mathpzc{Gr}_{\bullet}(\mathpzc{E})$ are monoidal model categories. 
\section{Filtered Objects in General Categories}
In this section we again fix a complete and cocomplete \textit{pointed} category $\mathpzc{E}$. The initial/ terminal object will be denoted $0$. 

\begin{defn}
\begin{enumerate}
Let $\mathcal{S}$ be a class of maps in $\mathpzc{E}$. 
\item
The category of $\mathcal{S}$-\textbf{filtered objects}, denoted $\mathpzc{Filt}_{\mathcal{S}}(\mathpzc{E})$, is the functor category $\mathpzc{Fun}_{\mathcal{S}}(\mathbb{Z}^{<}_{top},\mathpzc{E})$.
\item
The category of \textbf{exhaustively} $\mathcal{S}$-\textbf{filtered objects}, denoted $\mathpzc{Filt}_{\mathcal{S}}(\mathpzc{E})$, is the functor category $\mathpzc{Fun}_{\mathcal{S}}(\mathbb{Z}^{<},\mathpzc{E})$.
\end{enumerate}
\end{defn}
We shall write a filtered object as a tuple $(A_{top},\alpha_{i},a_{i})$, where $A_{top}$ is the value of the diagram at the terminal object, and $\alpha_{i}:A_{i}\rightarrow A_{top}$, $a_{i}:A_{i}\rightarrow A_{i+1}$ are maps in $\mathcal{S}$. A map $(A_{top},\alpha_{i},a_{i})\rightarrow (B_{top},\beta_{i},b_{i})$ will be written as $(g_{top},g_{i})$. We shall write the data of an exhaustively filtered object as $(a_{i}:A_{i}\rightarrow A_{i+1})$, or $(a_{i})$ when the $A_{i}$ are understood, and a map of exhaustively filtered objects as $(g_{i}:A_{i}\rightarrow B_{i})$.
\begin{rem}
The category  $\overline{\mathpzc{Filt}}_{\mathcal{S}}(\mathpzc{E})$ is equivalent to the full subcategory of $\mathpzc{Filt}_{\mathcal{S}}(\mathpzc{E})$ consisting of objects $(A_{top},\alpha_{i},a_{i})$ such that $A_{top}$ together with the maps $\alpha_{i}:A_{i}\rightarrow A_{top}$ is a direct limit of the diagram
\begin{displaymath}
\xymatrix{
&\ldots\ar[r] & A_{-1}\ar[r]^{a_{-1}} & A_{0}\ar[r]^{a_{0}} & A_{1}\ar[r]^{A_{1}} & A_{2}\ar[r] &\ldots
}
\end{displaymath}
We will freely identify $\overline{\mathpzc{Filt}}_{\mathcal{S}}(\mathpzc{E})$ with this subcategory. 
\end{rem}
We will typically only be interested in classes $\mathcal{S}$ satisfying the properties below.
\begin{defn}
Let $\mathpzc{E}$ be a pointed category. A subclass $\mathcal{S}\subset\mathpzc{Mor}(\mathpzc{E})$ is said to be a  filtering class if
\begin{enumerate}
\item
$\mathcal{S}$ contains all identity morphisms.
\item
$\mathcal{S}$ is closed under coproducts.
\item
$\mathcal{S}$ contains all  monomorphisms of the form $X\rightarrow X\coprod Y$ for $X,Y\in\mathpzc{E}$.
\end{enumerate}
If $\lambda$ is an ordinal then $\mathcal{S}$ is said to be $\lambda$-\textbf{closed} if $\lambda$-transfinite compositions of maps in $\mathcal{S}$ are in $\mathcal{S}$, and is said to be \textbf{left-cancellable} if whenever $g\circ f$ is in $\mathcal{S}$ and $g\in\mathcal{S}$, then $f\in\mathcal{S}$.
\end{defn}
We will mainly be interested in the cases $\mathcal{S}=\mathpzc{Mor}(\mathpzc{E})$, or $\mathcal{S}$ is the class of regular monomorphisms when $\mathpzc{E}$ is additive, which we denote by $\textbf{RegMon}$. Recall that in an additive category $\mathpzc{E}$ a morphism is said to be a \textbf{regular monomorphism} if it is the kernel of a morphism. When $\mathpzc{E}$ is exact, we will be interested in $\mathcal{S}=\textbf{AdMon}$.

It is useful to define the $\mathcal{S}$-image of a map.
\begin{defn}
Let $\mathcal{S}$ be a class of maps, and let $f:X\rightarrow Y$ be a map in $\mathpzc{E}$. The $\mathcal{S}$-\textbf{image} of $f$, if it exists, is a factorisation of $f$
\begin{displaymath}
\xymatrix{
X\ar[r]^{i} & \widetilde{X}\ar[r]^{\widetilde{f}} & Y
}
\end{displaymath}
with $\widetilde{f}\in\mathcal{S}$, such that whenever 
\begin{displaymath}
\xymatrix{
X\ar[r]^{i'} & \widetilde{X'}\ar[r]^{\widetilde{f}'} & Y
}
\end{displaymath}
is a factorisation of $f$ with $\widetilde{f}'$ in $\mathcal{S}$, there is a unique map $g:\widetilde{X}\rightarrow\widetilde{X}'$ such that the diagram below commutes
\begin{displaymath}
\xymatrix{
X\ar[r]^{i}\ar[d]^{Id_{X}} & \widetilde{X}\ar[d]^{g}\ar[r]^{\widetilde{f}} & Y\ar[d]^{Id_{Y}}\\
X\ar[r]^{i'} & \widetilde{X'}\ar[r]^{\widetilde{f}'} & Y
}
\end{displaymath}
$\mathpzc{E}$ is said to have \textbf{functorial }$\mathcal{S}$-\textbf{images} if there is a functor $Im_{\mathcal{S}}:\mathpzc{Mor}(\mathpzc{E})\rightarrow\mathpzc{E}$ and natural transformations $\textrm{domain}\rightarrow Im_{\mathcal{S}}\rightarrow\textrm{codomain}$ such that for any map $f:X\rightarrow Y$, $X\rightarrow Im_{\mathcal{S}}\rightarrow Y$ is an $\mathcal{S}$-image of $f$.
\end{defn}
\begin{defn}
A map $f:X\rightarrow Y$ is said to be $\mathcal{S}$-\textbf{epic} if for any $g:Y\rightarrow Z$, all maps in the diagram
$$Im_{\mathcal{S}}(A\rightarrow X\rightarrow Y)\rightarrow Im_{\mathcal{S}}(Im_{\mathcal{S}}(A\rightarrow X)\rightarrow Y)\rightarrow Im_{\mathcal{S}}(X\rightarrow Y)$$
.are isomorphisms
  \end{defn}
%
\begin{example}
\begin{enumerate}
\item
If $\mathcal{S}=\mathpzc{Mor}(\mathpzc{E})$ and $f:X\rightarrow Y$ is a map, then the $\mathcal{S}$-image is the factorisation
\begin{displaymath}
\xymatrix{
X\ar[r]^{Id_{X}}& X\ar[r]^{f} & Y
}
\end{displaymath}
$\mathcal{S}$-epic maps are isomorphisms.
\item
If $\mathpzc{E}$ is additive, $\mathcal{S}=\textbf{RegMon}$ and $f:X\rightarrow Y$ is a map, then the $\mathcal{S}$-image is the factorisation
\begin{displaymath}
\xymatrix{
X\ar[r]^{i}& Im(f)\ar[r]^{\widetilde{f}} & Y
}
\end{displaymath}
where here $Im(f)\rightarrow Y$ is the usual image, namely $Ker(Coker(f))\rightarrow Y$. 
$\mathcal{S}$-epic maps are categorical epimorphisms.
\end{enumerate}
\end{example}

We have a shift functor for filtered objects.
\begin{notation}[\cite{calaque2021lie}, Section 2.1]
For $r\in\mathbb{Z}$ denote by $<r>:\mathpzc{Filt}_{\mathcal{S}}(\mathpzc{E})\rightarrow \mathpzc{Filt}_{\mathcal{S}}(\mathpzc{E})$ the functor which sends $(A_{top},\alpha_{i},a_{i})$ to the object $A<r>=(A_{top},\alpha_{i+r},a_{i+r})$. Note that this restricts to a functor $<r>:\overline{\mathpzc{Filt}}_{\mathcal{S}}(\mathpzc{E})\rightarrow\overline{\mathpzc{Filt}}_{\mathcal{S}}(\mathpzc{E})$.
\end{notation}

\begin{notation}
Following \cite{gwilliam2018enhancing} we denote by $\mathpzc{Seq}(\mathpzc{E})$ the category $\overline{\mathpzc{Filt}}_{\mathpzc{Mor}(\mathpzc{E})}(\mathpzc{E})$, and by $\mathpzc{Seq}_{top}(\mathpzc{E})$ the category $\mathpzc{Filt}_{\mathpzc{Mor}(\mathpzc{E})}(\mathpzc{E})$.
\end{notation}
There is an obvious forgetful functor $R:\mathpzc{Filt}_{\mathcal{S}}(\mathpzc{E})\rightarrow\overline{\mathpzc{Filt}}_{\mathcal{S}}(\mathpzc{E})$ induced by the restriction functor $\mathpzc{Fun}(\mathbb{Z}_{top}^{<},\mathpzc{E})\rightarrow \mathpzc{Fun}(\mathbb{Z}^{<},\mathpzc{E})$. In nice circumstances this functor has a left adjoint which realises $\overline{\mathpzc{Filt}}_{\mathcal{S}}(\mathpzc{E})$ as a coreflective subcategory of $\mathpzc{Filt}_{\mathcal{S}}(\mathpzc{E})$.  Let $A$ be an exhaustively filtered object, and consider the filtered object $(\overline{A}_{top},\overline{\alpha}_{i},\overline{a}_{i})$ defined as follows. $\overline{A}_{top}=lim_{\rightarrow_{n}}A_{n}$. $\overline{a}_{i}=A_{i}\rightarrow A_{i+1}$ and $\overline{\alpha}_{i}$ is the canonical map $A_{i}\rightarrow lim_{\rightarrow_{n}}A_{n}$. This construction is naturally functorial, and we easily get the following result. 
\begin{prop}\label{coreflectivefilt}
Suppose that  $\mathcal{S}$ is $\aleph_{0}$-closed. Then the functor $C:\overline{\mathpzc{Filt}}_{\mathcal{S}}(\mathpzc{E})\rightarrow\mathpzc{Filt}_{\mathcal{S}}(\mathpzc{E})$ is right adjoint to the forgetful functor $R:\mathpzc{Filt}_{\mathcal{S}}(\mathpzc{E})\rightarrow\overline{\mathpzc{Filt}}_{\mathcal{S}}(\mathpzc{E})$ . Moreover the unit $Id\rightarrow R\circ C$ is a natural isomorphism.
\end{prop}
Let $\mathcal{S}$ and $\mathcal{R}$ be filtering classes with $\mathcal{S}\subset\mathcal{R}$, such that $\mathcal{S}$ is left-cancellable and functorial $\mathcal{S}$-images exist. Denote by $I_{\mathcal{S};\mathcal{R}}:\mathpzc{Filt}_{\mathcal{R}}(\mathpzc{E})\rightarrow\mathpzc{Filt}_{\mathcal{S}}(\mathpzc{E})$ the functor defined as follows. For $(A_{top},\alpha_{i},a_{i})$ an object in $\mathpzc{Filt}_{\mathcal{R}}(\mathpzc{E})$, let $\overline{\alpha}_{i}$ denote the map $\textrm{Im}_{\mathcal{S}}(\alpha_{i})\rightarrow A_{top}$, and let $\overline{a}_{i}$ denote the induced map $\textrm{Im}_{\mathcal{S}}(\alpha_{i})\rightarrow \textrm{Im}_{\mathcal{S}}(\alpha_{i+1})$. This is a well-defined object of $\mathpzc{Filt}_{\mathcal{S}}(\mathpzc{E})$. Moreover if $R_{\mathcal{S};\mathcal{R}}:\mathpzc{Filt}_{\mathcal{S}}(\mathpzc{E})\rightarrow\mathpzc{Filt}_{\mathcal{R}}(\mathpzc{E})$ is the obvious inclusion functor, then there is a natural isomorphism $I_{\mathcal{S};\mathcal{R}}\circ R_{\mathcal{S};\mathcal{R}}\rightarrow Id$. Moreover if $\mathcal{R}$ is $\aleph_{0}$-closed, this restricts to a well defined functor $I_{\mathcal{S};\mathcal{R}}:\overline{\mathpzc{Filt}}_{\mathcal{R}}(\mathpzc{E})\rightarrow\overline{\mathpzc{Filt}}_{\mathcal{S}}(\mathpzc{E})$.

 In fact these functors often form an adjunction.
 \begin{prop}\label{prop:filtadj}
Let $\mathcal{S}$ and $\mathcal{R}$ be filtering classes with $\mathcal{S}\subset\mathcal{R}$. Suppose that functorial $\mathcal{S}$-images exist, and $\mathcal{S}$ is left-cancellable. Then there is an adjunction
$$\adj{I_{\mathcal{S};\mathcal{R}}}{\mathpzc{Filt}_{\mathcal{R}}(\mathpzc{E})}{\mathpzc{Filt}_{\mathcal{S}}(\mathpzc{E})}{R_{\mathcal{S};\mathcal{R}}}$$
where $R_{\mathcal{S};\mathcal{R}}$ is the obvious inclusion functor. Moreover if $\mathcal{R}$ and $\mathcal{S}$ are $\aleph_{0}$-closed this restricts to an adjunction
$$\adj{I_{\mathcal{S};\mathcal{R}}}{\overline{\mathpzc{Filt}}_{\mathcal{R}}(\mathpzc{E})}{\overline{\mathpzc{Filt}}_{\mathcal{S}}(\mathpzc{E})}{R_{\mathcal{S};\mathcal{R}}}$$
\end{prop}

%
%

\subsection{Complete and Bounded Below Objects}
Let $A=(A_{top},\alpha_{i},a_{i})$  be a filtered object. We define $\widetilde{A}$ to be the projective limit of the diagram.
\begin{displaymath}
\xymatrix{
&\ldots\ar[r] & A_{top}\big\slash A_{-1}\ar[r] & A_{top}\big\slash  A_{0}\ar[r] & A_{top}\big\slash  A_{1}\ar[r] & A_{top}\big\slash  A_{2}\ar[r] &\ldots
}
\end{displaymath}
Define $\widetilde{A}_{n}\defeq Ker(\widetilde{A}_{top}\rightarrow A_{top}\big\slash A_{n})$. This construction is functorial, and there is a natural transformation $A\rightarrow\widetilde{A}$. As in, e.g. \cite{calaque2021lie} 2.1, we define complete filtered objects.
\begin{defn}
A filtered object $(A_{top},\alpha_{i},a_{i})$  is said to be \textbf{complete} if $A\rightarrow\widetilde{A}$ is an isomorphism of filtered objects. The full subcategory of $\mathpzc{Filt}_{\mathcal{S}}(\mathpzc{E})$ on objects equipped with a complete filtration will be denoted by $\reallywidehat{\mathpzc{Filt}}_{\mathcal{S}}(\mathpzc{E})$. 
The full subcategory of exhaustive and complete objects will be denoted $\reallywidehat{\overline{\mathpzc{Filt}}}_{\mathcal{S}}(\mathpzc{E})$
\end{defn}
We generally will not be interested in complete objects which are not exhaustively filtered. The inclusion functor $J:\reallywidehat{\mathpzc{Filt}}_{\mathcal{S}}(\mathpzc{E})\rightarrow\mathpzc{Filt}_{\mathcal{S}}(\mathpzc{E})$ does not in general have a left adjoint (problems arise because of the appearance of both limits and colimits in the definition of complete objects). When dealing with model categories, this is one of the reasons why it is better to work with (homotopically) complete filtered objects, as studied in detail in \cite{gwilliam2018enhancing}. For certain exact categories we shall show that one can work with complete objects.
\begin{defn}
A filtered object $(A_{top},\alpha_{i},a_{i})$ is said to be \textbf{bounded below} if for sufficiently large $i>0$ $A_{-i}=0$. The full subcategory of $\mathpzc{Filt}_{\mathcal{S}}(\mathpzc{E})$ (resp. $\overline{\mathpzc{Filt}}_{\mathcal{S}}(\mathpzc{E})$) consisting of bounded below objects is denoted $\mathpzc{Filt}^{+}_{\mathcal{S}}(\mathpzc{E})$ (resp. $\overline{\mathpzc{Filt}}^{+}_{\mathcal{S}}(\mathpzc{E})$). The full subcategory of $\mathpzc{Filt}^{+}_{\mathcal{S}}(\mathpzc{E})$ (resp. $\overline{\mathpzc{Filt}}^{+}_{\mathcal{S}}(\mathpzc{E})$) consisting of objects $(A_{top},\alpha_{i},a_{i})$ such that $A_{i}=0$ for $i<0$ is denoted $\mathpzc{Filt}^{\mathbb{N}_{0}}_{\mathcal{S}}(\mathpzc{E})$ (resp. $\overline{\mathpzc{Filt}}^{\mathbb{N}_{0}}_{\mathcal{S}}(\mathpzc{E})$).
\end{defn}
Note that $\mathpzc{Filt}^{+}_{\mathcal{S}}(\mathpzc{E})$ (resp. $\overline{\mathpzc{Filt}}^{+}_{\mathcal{S}}(\mathpzc{E})$) is a full subcategory of $\reallywidehat{\mathpzc{Filt}}_{\mathcal{S}}(\mathpzc{E})$ (resp. $\reallywidehat{\overline{\mathpzc{Filt}}}_{\mathcal{S}}(\mathpzc{E})$). Moreover the category $\mathpzc{Filt}^{\mathbb{N}_{0}}_{\mathcal{S}}(\mathpzc{E})$ (resp. $\overline{\mathpzc{Filt}}^{\mathbb{N}_{0}}_{\mathcal{S}}(\mathpzc{E})$) is in fact a coreflective subcategory of $\mathpzc{Filt}_{\mathcal{S}}(\mathpzc{E})$ (resp. $\overline{\mathpzc{Filt}}_{\mathcal{S}}(\mathpzc{E})$). The right adjoint to the inclusion sends an object $A=(A_{top},\alpha_{i},a_{i})$ to the object $A^{\mathbb{N}_{0}}$ with $A^{\mathbb{N}_{0}}_{i}=0$ for $i<0$, and $A^{\mathbb{N}_{0}}_{i}=A_{i}$ for $i\ge0$ and $i=top$. 
\subsection{Natural Functors}
For each $l\in\mathbb{Z}^{<}_{top}$ we denote by $(-)_{l}$ the functor
$$\mathpzc{Filt}_{\mathcal{S}}(\mathpzc{E})\rightarrow\mathpzc{E}$$
which sends a filtered object
$$(A_{top},\alpha_{i},a_{i})$$
to $A_{l}$. It sends a morphism $(f_{top},f_{i}):(A_{top},\alpha_{i},a_{i})\rightarrow(B_{top},\beta_{i},b_{i})$ to $f_{l}$. For $i\in\Z$ we denote by $Q_{i}:\mathpzc{Filt}_{\mathcal{S}}(\mathpzc{E})\rightarrow\mathpzc{E}$ the functor defined on objects by
$$Q_{i}(A_{top},\alpha_{i},a_{i})=coker(\alpha_{i}:A_{i}\rightarrow A_{top})$$
It is defined on morphisms in the obvious way. Finally we denote by $F_{i}:\mathpzc{E}\rightarrow\mathpzc{Filt}_{\mathcal{S}}(\mathpzc{E})$ the functor which sends an object $A$ of $\mathpzc{E}$ to the following filtered object. $(F_{i}(A))_{j}$ is $0$ for $j<i$, $(F_{i}(A))_{top}=A$, and $(F_{i}(A))_{j}=A$ for $i\le j<\infty$, with the structure maps being the obvious ones. Again it is defined on morphisms in the obvious way.
\begin{prop}\label{filteredadj}
There are adjunctions
$$Q_{i}\dashv F_{i+1}\dashv(-)_{i+1}.$$
\end{prop}
\begin{proof}
Let us first prove the second adjunction. Fix an object $A$ of $\mathpzc{E}$, and a filtered object $(B_{top},\beta_{i},b_{i})$. Let $f:A\rightarrow (B)_{i+1}$ be a map in $\mathpzc{E}$. There is an induced map $\widetilde{f}:F_{i+1}A\rightarrow B$ defined as follows. $\widetilde{f}_{j}=0$ for $j<i+1$ and $\widetilde{f}_{j}$ is the composition $A\rightarrow B_{i+1}\rightarrow B_{j}$ for $i+1\le j<\infty$. $\widetilde{f}_{top}$ is given by the composition $\beta_{i+1}\circ f$. This gives a map
$$\textrm{Hom}_{\mathpzc{E}}(A,(B)_{i+1})\rightarrow\textrm{Hom}_{\mathpzc{Filt}(\mathpzc{E})}(F_{i+1}A,B)$$
It is straightforward to verify that it is natural in both $A$ and $B$. It is clearly an isomorphism of abelian groups.
Let us now show the first adjunction. Let $(B_{top},\beta_{i},b_{i})$ be a filtered object, and let $f:coker(\beta_{i}:B_{i}\rightarrow B_{top})\rightarrow A$ be a morphism in $\mathpzc{E}$. There is an induced map $\widetilde{f}:(B_{top},\beta_{i},b_{i})\rightarrow F_{i+1}A$ defined as follows. $\widetilde{f}_{j}$ is $0$ for $j<i+1$, and for $i+1\le j<\infty$ or $j=top$, $\widetilde{f}_{j}$ is given by the composition
$$B_{j}\rightarrow B_{top}\rightarrow B_{top}\big\slash B_{i+1}\rightarrow A$$
This gives a homomorphism of abelian groups
$$\textrm{Hom}_{\mathpzc{E}}(Q_{i}(B_{top},\beta_{i},b_{i}),A)\rightarrow\textrm{Hom}_{\mathpzc{Filt}(\mathpzc{E})}((B_{top},\beta_{i},b_{i}),F_{i+1}A)$$
which is clearly natural in $(B_{top},\beta_{i},b_{i})$ and $A$. It is also clearly an isomorphism.
\end{proof}
Note that these functors are all also well-defined at the level of exhaustively filtered objects, at least for $i\neq top$.

There is a functor $\textrm{filt}_{top}:\mathpzc{Gr}_{\bullet}(\mathpzc{E})\rightarrow\mathpzc{Filt}_{\mathcal{S}}(\mathpzc{E})$. It sends a pointed graded object $(\coprod_{j\in\mathbb{Z}} E_{j},E_{top})$ to the filtered object $(\coprod_{j\in\mathbb{Z}_{\bullet}} E_{j},\coprod_{k \le i}E_{k}\rightarrow\coprod_{j\in\mathbb{Z}_{\bullet}} E_{j},\coprod_{k \le i}E_{k}\rightarrow\coprod_{l \le i+1}E_{l})$. It acts on morphisms in the obvious way.  There is also a functor $\textrm{gr}_{top}:\mathpzc{Filt}_{\mathcal{S}}(\mathpzc{E})\rightarrow\mathpzc{Gr}_{\bullet}(\mathpzc{E})$, called the \textbf{associated graded functor} defined as follows. To a filtered object $A=(A_{top},\alpha_{i},a_{i})$ it assigns the pointed object $\textrm{gr}_{top}(A)$ with $\textrm{gr}_{top}(A)_{i}=coker(a_{i-1}:A_{i-1}\rightarrow A_{i})$ for $i\in\mathbb{Z}$ and $\textrm{gr}_{top}(A)_{\bullet}=A_{top}$.  Again it acts on morphisms in the obvious way. 
The functor $\textrm{gr}_{top}$ is neither left nor right adjoint adjoint to $\textrm{filt}_{top}$. However we have the following.
\begin{prop}
The functor $\textrm{filt}_{_{top}}:\mathpzc{Gr}_{_{top}}(\mathpzc{E})\rightarrow\mathpzc{Filt}_{\mathcal{S}}(\mathpzc{E})$ is left-adjoint to the functor $\sum_{n\in\mathbb{Z}_{\bullet}}(-)_{n}$ given by the composition
\begin{displaymath}
\xymatrix{
\mathpzc{Filt}_{\mathcal{S}}(\mathpzc{E})\ar[r]^{\Delta} & \coprod_{n\in\mathbb{Z}_{\bullet}}\mathpzc{Filt}_{\mathcal{S}}(\mathpzc{E})\ar[rr]^{\coprod_{n\in\mathbb{Z}_{\bullet}}(-)_{n}}& & \coprod_{n\in\mathbb{Z}_{\bullet}}\mathpzc{E} \cong \mathpzc{Gr}_{\bullet}(\mathpzc{E})
}
\end{displaymath}
Here $\Delta$ is the diagonal morphism.
\end{prop}
\begin{proof}
The functor $\sum_{n\in\mathbb{Z}_{\bullet}}(-)_{n}$ is a composition of right-adjoints. Computing the composition of the corresponding left adjoints gives $\textrm{filt}_{top}$. 
\end{proof}
We denote by $\textrm{filt}:\mathpzc{Gr}(\mathpzc{E})\rightarrow\overline{\mathpzc{Filt}}_{\mathcal{S}}(\mathpzc{E})$ the functor which sends a graded object $\coprod_{j\in\mathbb{Z}} E_{j}$ to the exhaustively filtered object $(\coprod_{k \le i}E_{k}\rightarrow\coprod_{l \le i+1}E_{l})$. This functor also has a right adjoint $\sum_{n\in\mathbb{Z}}(-)_{n}$ defined by the composition 
\begin{displaymath}
\xymatrix{
\overline{\mathpzc{Filt}}_{\mathcal{S}}(\mathpzc{E})\ar[r]^{\Delta} & \coprod_{n\in\mathbb{Z}}\overline{\mathpzc{Filt}}_{\mathcal{S}}(\mathpzc{E})\ar[rr]^{\coprod_{n\in\mathbb{Z}}(-)_{n}} & & \coprod_{n\in\mathbb{Z}}\mathpzc{E} \cong \mathpzc{Gr}(\mathpzc{E})
}
\end{displaymath}
Finally we define $\textrm{gr}:\overline{\mathpzc{Filt}}_{\mathcal{S}}(\mathpzc{E})\rightarrow\mathpzc{Gr}(\mathpzc{E})$ to be the functor sending $A=(A_{top},\alpha_{i},a_{i})$ to the graded object $\textrm{gr}(A)_{i}\defeq coker(a_{i-1}:A_{i-1}\rightarrow A_{i})$. Note that $\textrm{gr}\circ\textrm{filt}$ is equivalent to the identity functor. 
\subsection{Limits and Colimits of Filtered Objects}

We can use Proposition \ref{filtdiag}  to analyse limits and colimits of filtered objects. The first easy result to note is the following.
\begin{prop}\label{filteredlimitsseq}
Let $\mathpzc{E}$ be a complete and cocomplete category. Then $\mathpzc{Seq}_{top}(\mathpzc{E})$ and $\mathpzc{Seq}(\mathpzc{E})$ are complete and cocomplete.
\end{prop}
We also have the following important result.

\begin{prop}\label{filtcomp}
Let $A=(A_{\infty},\alpha_{i},a_{i})$ be an object of $\mathpzc{Filt}_{\mathcal{S}}(\mathpzc{E})$. Suppose that each $A_{i}$ and $A_{top}$ satisfy one of the smallness conditions of Definition \ref{defsmallnesscond}, $\mathcal{S}$ is closed under the corresponding colimits, and for sufficiently large $|i|$, $a_{i}$ is an isomorphism. Then $A$ satisfies the same smallness condition in $\mathpzc{Filt}_{\mathcal{S}}(\mathpzc{E})$. This also holds for the exhaustively filtered category.
\end{prop}
\begin{proof}
We prove the claim for non-exhaustively filtered objects. Let $D:\mathcal{I}\rightarrow\mathpzc{Filt}_{\mathcal{S}}(\mathpzc{E})$ be a relevant filtered diagram. By Proposition \ref{filtdiag} the colimit is computed by taking the colimit in each degree of the filtration. For each $k\in\mathbb{Z}$, there is an $i_{k}\in\mathcal{I}$ such that $A_{k}\rightarrow colim(-)_{k}\circ D$ factors through $(-)_{k}(i_{k})$. Let $n\in\mathbb{N}_{0}$ be such that $A_{n}\rightarrow A_{n+i}$, and $A_{-n-i}\rightarrow A_{-n}$ is an isomorphism for any $i\in\mathbb{N}_{0}$. Let $i=max_{-n\le k\le n}i_{i_{k}}$. Then the map $A\rightarrow colim D$ factors through $D(i)$.
\end{proof}
Note that if $\mathcal{S}$ is $\aleph_{0}$-closed then, since $\overline{\mathpzc{Filt}}_{\mathcal{S}}(\mathpzc{E})$ is a coreflective subcategory of $\mathpzc{Filt}_{\mathcal{S}}(\mathpzc{E})$, the former will be complete and cocomplete as long as the latter is. Moreover in this case, colimits computed in $\overline{\mathpzc{Filt}}_{\mathcal{S}}(\mathpzc{E})$ coincide with colimits computed in $\mathpzc{Filt}_{\mathcal{S}}(\mathpzc{E})$. For certain limits the computations in the exhaustive and non-exhaustive categories also coincide. Indeed the following is essentially tautological.
\begin{prop}
Let $D:\mathcal{J}\rightarrow\mathpzc{Filt}(\mathpzc{E})_{\mathcal{S}}$. If each $D(j)$ is exhaustively filtered, and  (co)limits of diagrams of shape $\mathcal{J}$ commute with $\aleph_{0}$-transfinite compositions of morphisms in $\mathcal{S}$, then the formula in Proposition \ref{filtcomp} is also a (co)limit in $\overline{\mathpzc{Filt}}_{\mathcal{S}}(\mathpzc{E})$.
\end{prop}
In the situation that we get and adjunction $\adj{I_{\mathcal{S};\mathpzc{Mor}(\mathpzc{E})}}{\mathpzc{Seq}_{top}(\mathpzc{E})}{\mathpzc{Filt}_{\mathcal{S}}(\mathpzc{E})}{R_{\mathcal{S};\mathpzc{Mor}(\mathpzc{E})}}$, then the right hand side is a coreflective subcategory of $\mathpzc{Seq}_{top}$. In particular we get the following.%
\begin{cor}\label{cor:cocompfilt}
Let $\mathcal{S}$ be a filtering class. Suppose that functorial $\mathcal{S}$-images exist , and $\mathcal{S}$ is left-cancellable. Then $\mathpzc{Filt}_{\mathcal{S}}(\mathpzc{E})$ is complete and complete. If $\mathcal{S}$ is $\aleph_{0}$-closed this is also true for $\overline{\mathpzc{Filt}}_{\mathcal{S}}(\mathpzc{E})$
\end{cor}
\subsection{Monoidal Structures on Filtered Objects}
For this subsection let $(\mathpzc{E},\otimes,k)$ be a monoidal category such that $\otimes$ commutes with countable colimits in each variable, and let $\mathcal{S}$ be left-cancellable. Suppose further that functorial $\mathcal{S}$-images exist. For $\mathcal{S}$-filtered objects $A=(A_{top},\alpha_{i},a_{i})$ and $B=(B_{top},\beta_{i},b_{i})$, define a $\mathcal{S}$-filtered object $A\otimes B$. Define $(A\otimes B)_{top}:=A_{top}\otimes B_{top}$. For each $n\in\mathbb{Z}$ consider the category $\{(i,j)\in\mathbb{Z}^{<}:i+j\le n\}$, and the diagram  on this category sending $(i,j)$ to $A_{i}\otimes B_{j}$. Write $(A\tilde{\otimes} B)_{n}\defeq\textrm{colim}_{(i,j):i+j\le n}A_{i}\otimes B_{j}$. This gives a well defined object of $\mathpzc{Seq}(\mathpzc{E})$, $(A_{top}\otimes B_{top},(A\tilde{\otimes} B)_{n}\rightarrow (A\tilde{\otimes} B)_{n+1},(A\tilde{\otimes} B)_{n}\rightarrow A_{top}\otimes B_{top})$. Notice that this is just Day convolution, as in \cite{gwilliam2018enhancing} Section 2.23. Define $A\otimes B\defeq I_{\mathcal{S};\mathpzc{Mor}(\mathpzc{E})}(A\tilde{\otimes}B)$. 
\begin{prop}
If $\mathcal{S}$ is left-cancellable then $(A_{top}\otimes B_{top},(\alpha\otimes \beta)_{n},(a\otimes b)_{n})$ is a  $\mathcal{S}$-filtered object. Suppose that
\begin{enumerate}
\item
For any sequence
$$A_{0}\rightarrow A_{1}\rightarrow\ldots$$
together with a map $\textrm{lim}_{\rightarrow_{n}}A_{n}\rightarrow B$, the map
$$\textrm{lim}_{\rightarrow_{n}}\textrm{Im}_{\mathcal{S}}(A_{n}\rightarrow B)\rightarrow\textrm{Im}_{\mathcal{S}}(\textrm{lim}_{\rightarrow_{n}}A_{n}\rightarrow B)$$
is an isomorphism.
\item
 $\otimes$ preserves colimits in each variable.
 \end{enumerate}
 If  $A$ and $B$ are exhaustive, then $A\otimes B$ is exhaustive.
\end{prop}
\begin{proof}
The first claim is clear. The maps
$(A\otimes B)_{n}\rightarrow\textrm{Im}_{\mathcal{S}}(A_{n}\otimes B_{n}\rightarrow A_{top}\otimes B_{top})$ and $Im_{\mathcal{S}}(A_{i}\otimes B_{n-i}\rightarrow A_{top}\otimes B_{top})\rightarrow(A\otimes B)_{n}$ furnish an isomorphism (and its inverse)
 $$\textrm{lim}_{\rightarrow_{i}}\textrm{lim}_{\rightarrow_{j}}\textrm{Im}_{\mathcal{S}}(A_{i}\otimes B_{j}\rightarrow A_{top}\otimes B_{top})\cong(A\otimes B)_{n}$$
 We then have.
 \begin{align*}
\textrm{lim}_{\rightarrow_{i}}\textrm{lim}_{\rightarrow_{j}}\textrm{Im}_{\mathcal{S}}(A_{i}\otimes B_{j}\rightarrow A_{top}\otimes B_{top})&\cong\textrm{Im}_{\mathcal{S}}(\textrm{lim}_{\rightarrow_{i}}\textrm{lim}_{\rightarrow_{j}}A_{i}\otimes B_{j}\rightarrow A_{top}\otimes B_{top})\\
&\cong\textrm{Im}_{\mathcal{S}}(A_{top}\otimes B_{top}\rightarrow A_{top}\otimes B_{top})\\
&=A_{top}\otimes B_{top}
\end{align*}
\end{proof} 
The monoidal functor $\otimes:\mathpzc{Filt}_{\mathcal{S}}(\mathpzc{E})\otimes\mathpzc{Filt}_{\mathcal{S}}(\mathpzc{E})\rightarrow\mathpzc{Filt}_{\mathcal{S}}(\mathpzc{E})$ has a unit, namely $F_{0}(k)$. Moreover it is naturally symmetric. However it is not associative in general. It will be useful to consider monoidal subcategories of $\mathpzc{Filt}_{\mathcal{S}}(\mathpzc{E})$ wherein the restriction of the monoidal product is naturally associative. One such category is the essential image of 
$\textrm{filt}_{top}:\mathpzc{Gr}_{\bullet}(\mathpzc{E})\rightarrow\mathpzc{Filt}_{\mathcal{S}}(\mathpzc{E})$. Indeed this follows from the fact that this functor is strong monoidal.
\begin{prop}\label{filtmonass}
 Let $\mathpzc{A}$ be a full subcategory of $\mathpzc{Filt}_{\mathcal{S}}(\mathpzc{E})$ which is closed under $\otimes$, contains $k$,and such that for any $A,B,C\in\mathpzc{A}$, the maps $I_{\mathcal{S};\mathpzc{Mor}(\mathpzc{E})}((A\tilde{\otimes} B)\tilde{\otimes}C)\rightarrow (A\otimes B)\otimes C$ and $I_{\mathcal{S};\mathpzc{Mor}(\mathpzc{E})}(A\tilde{\otimes} (B\tilde{\otimes }C))\rightarrow A\otimes (B\otimes C)$ are isomorphisms. Then the restriction of $\otimes$ to $\mathpzc{A}$ is naturally associative.
\end{prop}
\begin{proof}
The conditions on $\mathpzc{A}$ ensure that $((A\otimes B)\otimes C)_{n}$ can be computed as $Im_{\mathcal{S}}(\textrm{colim}_{i+j+k\le n}(A_{i}\otimes B_{j})\otimes C_{k}\rightarrow (A_{top}\otimes B_{top})\otimes C_{top})$. This is clearly associative.
\end{proof}
\begin{prop}\label{filtmonass2}
Let $\mathpzc{S}$ be a filtering class such that 
\begin{enumerate}
\item
for any map $A\rightarrow B$, $A\rightarrow Im_{\mathcal{S}}(A\rightarrow B)$ is a $\mathcal{S}$-epimorphism
\item
$\otimes$ preserves $\mathcal{S}$-epimorphisms and countable colimits in each variable
\end{enumerate}
Then $\mathpzc{Filt}_{\mathcal{S}}(\mathpzc{E})$ is associative.
\end{prop}
\begin{proof}
Again the assumptions imply that $((A\otimes B)\otimes C)_{n}$ can be computed as $Im_{\mathcal{S}}(\textrm{colim}_{i+j+k\le n}(A_{i}\otimes B_{j})\otimes C_{k}\rightarrow (A_{top}\otimes B_{top})\otimes C_{top})$. This is clearly associative.
\end{proof}
\subsubsection{Closed Monoidal Structures}
Suppose that $\mathpzc{E}$ is closed monoidal, with internal hom $\underline{Hom}$, and let $\mathcal{S}$ be a filtering class. Then $\mathpzc{Filt}_{\mathcal{S}}(\mathpzc{E})$ is enriched over $\mathpzc{E}$, with $\underline{Hom}_{tot}(A,B)$ defined to be the equaliser of the two obvious maps
\begin{displaymath}
\xymatrix{
\prod_{n\in\mathbb{Z}_{\bullet}}\underline{Hom}(A_{n},B_{n})\ar@/^1.0pc/[rr]\ar@/_1.0pc/[rr]& & \prod_{n\in\mathbb{Z}_{\bullet}}\underline{Hom}(A_{n},B_{n+1})
}
\end{displaymath}
As in \cite{calaque2021lie} Section 2 (which considers filtered vector spaces ), it is also enriched over $\mathpzc{Seq}_{top}(\mathpzc{E})$ by defining $\underline{Hom}_{filt}(A,B)$ to be the filtered object
$$(\underline{Hom}(A_{top},B_{top}),\underline{Hom}_{tot}(A,B<i>)\rightarrow\underline{Hom}(A_{top},B_{top}),\underline{Hom}_{tot}(A,B<i>)\rightarrow\underline{Hom}_{tot}(A,B<i+1>))$$
If $\mathcal{S}=\mathpzc{Mor}(\mathpzc{E})$, i.e. $\mathpzc{Filt}_{\mathcal{S}}(\mathpzc{E})=\mathpzc{Seq}_{top}(\mathpzc{E})$ then 
$$(\mathpzc{Seq}_{top}(\mathpzc{E}),\otimes,\underline{Hom}_{filt})$$
 is a closed monoidal category, as in \cite{gwilliam2018enhancing} 2.23. Moreover 
 $$(\mathpzc{Seq}(\mathpzc{E}),R(C(-)\otimes C(-)),R\underline{Hom}_{filt}(C(-),C(-)))$$
  is also a closed monoidal category. \newline
\\
Let $\mathcal{S}$ be left-cancellable, and suppose that functorial $\mathcal{S}$-images exist. The inclusion $R_{\mathcal{S};\mathpzc{Mor}(\mathpzc{E})}:\mathpzc{Filt}_{\mathcal{S}}(\mathpzc{E})\rightarrow\mathpzc{Seq}_{top}(\mathpzc{E}$) has a left adjoint $I_{\mathcal{S};\mathpzc{Mor}(\mathpzc{E})}$. If it happens that $\underline{Hom}_{filt}(A,B)$ is an object of $\mathpzc{Filt}_{\mathcal{S}}(\mathpzc{E})$ whenever $A\in\mathpzc{Seq}_{top}(\mathpzc{E})$ and $B\in\mathpzc{Filt}_{\mathcal{S}}(\mathpzc{E})$, then by abstract nonsense 
$$(\mathpzc{Filt}_{\mathcal{S}}(\mathpzc{E}),\otimes,F_{0}(k),\underline{Hom}_{filt})$$ is a closed monoidal category. Finally if $A\otimes B$ is exhaustive whenever $A$ and $B$ are, and $\mathcal{S}$ is $\aleph_{0}$-closed, then $(\overline{\mathpzc{Filt}}_{\mathcal{S}}(\mathpzc{E}),R(-\otimes-),F_{0}(k),R(\underline{Hom}_{filt}))$ is also a closed monoidal category. 

%

\subsection{Model Structures on Filtered Objects}
Let $\mathpzc{E}$ be a pointed combinatorial model category.
\begin{defn}
The \textbf{filtered model structure} on $\mathpzc{Filt}_{\mathcal{S}}(\mathpzc{E})$ (resp. $\overline{\mathpzc{Filt}}_{\mathcal{S}}(\mathpzc{E})$), if it exists, is the model structure transferred along the adjunction
$$\adj{filt_{top}}{\mathpzc{Gr}_{\bullet}(\mathpzc{E})}{\mathpzc{Filt}_{\mathcal{S}}(\mathpzc{E})}{\sum_{n\in\mathbb{Z}_{\bullet}}(-)_{n}}\;\;\textrm{(resp. }
\adj{filt}{\mathpzc{Gr}(\mathpzc{E})}{\overline{\mathpzc{Filt}}_{\mathcal{S}}(\mathpzc{E})}{\sum_{n\in\mathbb{Z}}(-)_{n}}\textrm{)}$$
\end{defn}
Note that if $I$ (resp. $J$) is a generating set of cofibrations (resp. trivial cofibrations) then $\{F_{i}g:g\in I,i\in\mathbb{Z}\}$ (resp. $\{F_{i}g:g\in J,i\in\mathbb{Z}\}$) is a generating set of cofibrations (resp. trivial cofibrations) for the filtered model structure on $\overline{\mathpzc{Filt}}_{\mathcal{S}}(\mathpzc{E})$. For $\mathpzc{Filt}_{\mathcal{S}}(\mathpzc{E})$ we have to add $F_{top}g$ for $g\in I$ (resp. $g\in J$).
As in \cite{gwilliam2018enhancing} Section 3, we have the following.
\begin{lem}
The filtered model structure exists on $\mathpzc{Seq}_{top}(\mathpzc{E})$, and $\mathpzc{Seq}(\mathpzc{E})$
\end{lem}
\begin{proof}
The filtered model structure on $\mathpzc{Seq}_{top}(\mathpzc{E})$ (resp. $\mathpzc{Seq}(\mathpzc{E}))$ is just the projective model structure for functors ( \cite{Riehl} Theorem 12.3.2).
\end{proof}
\begin{rem}
Let $\mathcal{S}$ be a filtering class such that $\mathpzc{Filt}_{\mathcal{S}}(\mathpzc{E})$ is a (functorial) model category with the filtered model structure. Suppose that all cofibrations in $\mathpzc{E}$ are contained in  $\mathcal{S}$. Because $\mathcal{S}$ contains cofibrations, the cofibrant objects in $\mathpzc{Filt}_{\mathcal{S}}(\mathpzc{E})$ coincide with the cofibrant objects in $\mathpzc{Seq}_{top}(\mathpzc{E})$ (as a consequence of the description of cofibrant objects in the projective model structure for functors). By \cite{mazel2015quillen} Lemma 2.8 the $(\infty,1)$-categories presented by $\mathpzc{Seq}_{top}(\mathpzc{E})$ and $\mathpzc{Filt}_{\mathcal{S}}(\mathpzc{E})$ are equivalent. This is also true for the exhaustively filtered categories. Note that in fact since an object $A$ of $\mathpzc{Filt}_{\mathcal{S}}(\mathpzc{E})$ is fibrant if and only if each $A_{i}$ is fibrant for each $i\in\mathbb{Z}^{top}$, $\mathpzc{Seq}_{top}(\mathpzc{E})$ and $\mathpzc{Filt}_{\mathcal{S}}(\mathpzc{E})$ in fact have the same fibrant-cofibrant objects.
\end{rem}
Often we get a Quillen equivalence.
\begin{prop}
Let $\mathcal{S}$ and $\mathcal{R}$ be filtering monomorphic classes with $\mathcal{S}\subset\mathcal{R}$ such that the filtering model structure exists on both $\mathpzc{Filt}_{\mathcal{R}}(\mathpzc{E})$ and $\mathpzc{Filt}_{\mathcal{S}}(\mathpzc{E})$. Suppose that functorial $\mathcal{S}$-images exist, that $\mathcal{S}$ is left-cancellable, that all cofibrations are maps in $\mathcal{S}$, and that both $\mathcal{S}$ and $\mathcal{R}$ are $\aleph_{0}$-closed.. Then the adjunctions
    $$\adj{I_{\mathcal{S};\mathcal{R}}}{\mathpzc{Filt}_{\mathcal{R}}(\mathpzc{E})}{\mathpzc{Filt}_{\mathcal{S}}(\mathpzc{E})}{R_{\mathcal{S};\mathcal{R}}}$$
$$\adj{I_{\mathcal{S};\mathcal{R}}}{\overline{\mathpzc{Filt}}_{R}(\mathpzc{E})}{\overline{\mathpzc{Filt}}_{\mathcal{S}}(\mathpzc{E})}{R_{\mathcal{S};\mathcal{R}}}$$
are Quillen equivalences.
\end{prop}
\begin{proof}
We prove this for the exhaustive case, the case for non-exhaustive filtered objects being similar. Denote by $\Sigma^{\mathcal{S}}_{n\in\mathbb{Z}}(-)_{n}$ (resp. $\Sigma^{\mathcal{R}}_{n\in\mathbb{Z}}(-)_{n}$) the functor $\Sigma_{n\in\mathbb{Z}}(-)_{n}:\overline{\mathpzc{Filt}}_{\mathcal{S}}(\mathpzc{E})\rightarrow\mathpzc{Gr}(\mathpzc{E})$ resp. ($\Sigma_{n\in\mathbb{Z}}(-)_{n}:\overline{\mathpzc{Filt}}_{\mathcal{R}}(\mathpzc{E})\rightarrow\mathpzc{Gr}(\mathpzc{E})$). The fact that the adjunction is Quillen follows from the fact $\Sigma^{\mathcal{R}}_{n\in\mathbb{Z}}(-)_{n}\circ R_{\mathcal{S};\mathcal{R}}=\Sigma^{\mathcal{S}}_{n\in\mathbb{Z}}(-)_{n}$. Clearly a map $f$ in $Ch(\overline{\mathpzc{Filt}}_{\mathcal{S}}(\mathpzc{E}))$ is a weak equivalence if and only if $R_{\mathcal{S};\mathcal{R}}(f)$ is a weak equivalence. Moreover any cofibrant $X$ in  $Ch(\overline{\mathpzc{Filt}}_{\mathcal{R}}(\mathpzc{E})$ is in the image of $R_{\mathcal{S};\mathcal{R}}(f)$. 
Let $X\rightarrow R_{\mathcal{S};\mathcal{R}}(Y)$ be a map with $X$ cofibrant and $Y$ fibrant. Since $X$ is cofibrant, $X\cong R_{\mathcal{S};\mathcal{R}}I_{\mathcal{S};\mathcal{R}}(X)$. Hence $X\rightarrow R_{\mathcal{S};\mathcal{R}}(Y)$ is an equivalence if and only if $I_{\mathcal{S};\mathcal{R}}(X)\rightarrow Y$ is.
\end{proof}

\subsubsection{Homotopically Complete Objects}
Let $\mathcal{S}$ be a filtering monomorphic class and $\mathpzc{E}$ a combinatorial left proper model category. 
\begin{defn}
A cofibrant object $(A_{top},\alpha_{i},a_{i})$ in either $\mathpzc{Filt}_{\mathcal{S}}(\mathpzc{E})$ or $\overline{\mathpzc{Filt}}_{\mathcal{S}}(\mathpzc{E})$ is said to be \textbf{homotopically complete} if the map $A_{top}\rightarrow\textrm{holim}_{\leftarrow_{n}}(A_{top}\big\slash A_{n})$ is an equivalence, where here $\textrm{holim}_{\leftarrow_{n}}$ denotes the homotopy limit.
\end{defn}
Note that in general homotopically complete objects need not be complete, and complete objects need not be homotopically complete. Any bounded below filtered object is both complete and homotopically complete. 
\begin{defn}[\cite{gwilliam2018enhancing} Definition 3.6]
Suppose the filtering model structure exists on $\overline{\mathpzc{Filt}}_{\mathcal{S}}(\mathpzc{E})$. A map $f:X\rightarrow Y$ is said to be a \textbf{derived graded equivalence} if $gr(Q(f))$ is a graded equivalence,  where $Q$ is the cofibrant replacement functo
\end{defn}
\begin{defn}[\cite{gwilliam2018enhancing} Definition 3.5]
The homotopically complete filtered model structure on $\overline{\mathpzc{Filt}}_{\mathcal{S}}(\mathpzc{E})$, if it exists, is the left Bousfield localisation of the filtered model structure at the derived graded equivalences.
\end{defn}
\begin{rem}
Let $\mathcal{S}$ be a filtering class containing cofibrations such that the filtering model structure exists on $\overline{\mathpzc{Filt}}_{\mathcal{S}}(\mathpzc{E})$. Suppose the homotopically complete model structure exists on both $\overline{\mathpzc{Filt}}_{\mathcal{S}}(\mathpzc{E})$ and $\mathpzc{Seq}(\mathpzc{E})$. As left Bousfield localisation does not change the cofibrant objects, with the homotopically complete model structures, $\overline{\mathpzc{Filt}}_{\mathcal{S}}(\mathpzc{E})$ and $\mathpzc{Seq}(\mathpzc{E})$ still have the same cofibrant objects. Moreover a map $f:X\rightarrow Y$ between such cofibrant objects is an equivalence in either $\overline{\mathpzc{Filt}}_{\mathcal{S}}(\mathpzc{E})$ or $\mathpzc{Seq}(\mathpzc{E})$ precisely if $gr(f)$ is an equivalence. Hence, once again, they present the same $(\infty,1)$-categories.
\end{rem}

Again we often have a Quillen equivalence.
\begin{prop}
Let $\mathcal{S}$ and $\mathcal{R}$ be filtering monomorphic classes with $\mathcal{S}\subset\mathcal{R}$ such that the filtering  and homotopically complete model structures exists on both $\overline{\mathpzc{Filt}}_{\mathcal{R}}(\mathpzc{E})$ and $\overline{\mathpzc{Filt}}_{\mathcal{S}}(\mathpzc{E})$. Suppose further that cofibrations are contained in $\mathcal{S}$, that $\mathcal{S}$ is left-cancellable, and that functorial $\mathcal{S}$-images exist.. Then the adjunction
$$\adj{I_{\mathcal{S};\mathcal{R}}}{\overline{\mathpzc{Filt}}_{R}(\mathpzc{E})}{\overline{\mathpzc{Filt}}_{\mathcal{S}}(\mathpzc{E})}{R_{\mathcal{S};\mathcal{R}}}$$
is a Quillen equivalence where both sides are equipped with the homotopically complete model structure.
\end{prop}
\begin{proof}
The fact that cofibrations are contained in $\mathcal{S}$ means that $\overline{\mathpzc{Filt}}_{\mathcal{R}}(\mathpzc{E})$ and $\overline{\mathpzc{Filt}}_{\mathcal{S}}(\mathpzc{E})$ have the same cofibrant objects. In particular for $X$ cofibrant in $\overline{\mathpzc{Filt}}_{\mathcal{R}}(\mathpzc{E})$, we have $gr(X)\cong gr(I_{\mathcal{S};\mathcal{R}}(f))$. Thus a map $f:X\rightarrow Y$ between cofibrant objects in $\overline{\mathpzc{Filt}}_{\mathcal{R}}(\mathpzc{E})$ is a graded equivalence if and only if $I_{\mathcal{S};\mathcal{R}}(f)$ is a graded equivalence. Thus the adjunction on the Bousfield localisations at derived graded equivalences is a Quillen equivalence.
\end{proof}
\begin{thm}[\cite{gwilliam2018enhancing}, Proposition 3.31]\label{thm:htpcompletegwilliam}
Let $\mathpzc{E}$ be a stable, combinatorial, left proper model category. The homotopically complete filtered model structure exists on $\mathpzc{Seq}(\mathpzc{E})$. Moreover in this case the model structure is the left Bousfield localisation at maps of the form $const(0)\rightarrow const(K)$, where $K$ is a cofibrant generator of $\mathpzc{E}$, and $const(0)$ is the constant filtered object with $const(K)_{top}=const(K)_{i}=K$ for all $i\in\mathbb{Z}$, with the maps $a_{i}$ and $\alpha_{i}$ being the identity.
\end{thm}
\subsubsection{Monoidal Model Structures}
Suppose that $(\mathpzc{E},\otimes,k)$ is a monoidal model category. Let $\mathcal{S}$ be a filtering class such that the filtering model structure exists on $\mathpzc{Filt}_{\mathcal{S}}(\mathpzc{E})$ (resp. $\overline{\mathpzc{Filt}}_{\mathcal{S}}(\mathpzc{E})$) and is combinatorial. Suppose further that the model structure on $\mathpzc{E}$ is monoidal, and that the filtered monoidal structure on $\mathpzc{Filt}_{\mathcal{S}}(\mathpzc{E})$ (resp. $\overline{\mathpzc{Filt}}_{\mathcal{S}}(\mathpzc{E})$) is associative. It follows from the fact that the functor $filt_{top}$ (resp. $filt$) is strong monoidal, and $\mathpzc{Gr}_{\bullet}(\mathpzc{E})$ (resp. $\mathpzc{Gr}(\mathpzc{E})$) is a monoidal model category, that $\mathpzc{Filt}_{\mathcal{S}}(\mathpzc{E})$ (resp. $\overline{\mathpzc{Filt}}_{\mathcal{S}}(\mathpzc{E})$) is a monoidal model category.
\section{Filtered and Graded Objects in Exact Categories}
In this section we fix a complete and cocomplete exact category $\mathpzc{E}$.
\subsection{Graded and Pointed Graded Objects}
As diagram categories, the categories $\mathpzc{Gr}(\mathpzc{E})$ and $\mathpzc{Gr}_{\bullet}(\mathpzc{E})$ have natural exact structures, in which a map $f:\coprod_{n}X\rightarrow \coprod_{n}Y$ is an admissible epimorphism/ admissible monomorphism if for each $n\in\mathbb{Z}$ (resp. $n\in\mathbb{Z}_{\bullet}$ is an admissible epimorphism/ admissible monomorphism. \newline
\\
\subsubsection{Compatible Model Structures on Graded Objects}
For a class of objects $\mathfrak{O}$ in $\mathpzc{E}$ we define by $\mathpzc{Gr}(\mathfrak{O})$ the class of objects in $\mathpzc{Gr}(\mathpzc{E})$ of the form $\coprod_{i\in I}A_{i}$ where each $A_{i}\in\mathfrak{O}$. Since exactness is degree-wise the following is clear.
\begin{prop}
Let $\mathpzc{E}$ be an exact category and let $(\mathfrak{L},\mathfrak{R})$ be a cotorsion pair on $\mathpzc{E}$. Then $(\mathpzc{Gr}(\mathfrak{L}),\mathpzc{Gr}(\mathfrak{R}))$ is a cotorsion pair on $\mathpzc{Gr}(\mathpzc{E})$. Moreover if $(\mathfrak{L},\mathfrak{R})$ is $dg_{*}$-compatible then so is $(\mathpzc{Gr}(\mathfrak{L}),\mathpzc{Gr}(\mathfrak{R}))$. Finally if $(\mathfrak{L},\mathfrak{R})$ is monoidally $dg_{*}$-compatible then so is $(\mathpzc{Gr}(\mathfrak{L}),\mathpzc{Gr}(\mathfrak{R}))$.
\end{prop}
The model structure induced on $Ch_{*}(\mathpzc{Gr}(\mathpzc{E}))$ is the component-wise one of Section \ref{modelgraded}. One can of course repeat all this for $\mathpzc{Gr}_{\bullet}(\mathpzc{E})$.
\subsection{Filtered Objects in Exact Categories}
In this section we consider categories of the form $\mathpzc{Filt}_{\mathcal{S}}(\mathpzc{E})$ where $\mathpzc{E}$ is an exact category and $\mathcal{S}$ a filtering monomorphic class.  We will typically be interested in the cases $\mathcal{S}=\textbf{RegMon}$ and $\mathcal{S}=\textbf{AdMon}$.
%
%
Since regular monomorphisms in a quasi-abelian category are left-cancellable, we have the following useful result.
\begin{prop}\label{filteredlimits}
Let $\mathpzc{E}$ be a complete and cocomplete  weakly $(\aleph_{0};\textbf{RegMon})$-elementary exact category whose underlying additive category is quasi-abelian, and such that $\textbf{RegMon}$ is $\aleph_{0}$-closed. The pairs of functors below are adjunction
$$\adj{I_{\mathpzc{Mor}(\mathpzc{E});\textbf{RegMon}}}{\mathpzc{Seq}_{top}(\mathpzc{E})}{\mathpzc{Filt}_{\textbf{RegMon}}(\mathpzc{E})}{R_{\mathpzc{Mor}(\mathpzc{E});\textbf{RegMon}}}$$
$$\adj{I_{\mathpzc{Mor}(\mathpzc{E});\textbf{RegMon}}}{\mathpzc{Seq}(\mathpzc{E})}{\overline{\mathpzc{Filt}}_{\textbf{RegMon}}(\mathpzc{E})}{R_{\mathpzc{Mor}(\mathpzc{E});\textbf{RegMon}}}$$
is an adjunction. In particular, in this case $\mathpzc{Filt}_{\textbf{RegMon}}(\mathpzc{E})$ and $\overline{\mathpzc{Filt}}_{\textbf{RegMon}}(\mathpzc{E})$ are complete and cocomplete.
\end{prop}

\subsubsection{Exactness Properties of Categories of Filtered Objects}
We shall see that in general categories of filtered objects do not have exact structures. However we can still define the notion of an exact sequence of filtered objects.
\begin{prop}\label{filtkercoker}
Let 
\begin{displaymath}
\xymatrix{
0\ar[r] &(A_{top},\alpha_{i},a_{i})\ar[r]^{f} &(B_{top},\beta_{i},b_{i})\ar[r]^{g} & (C_{top},\gamma_{i},c_{i})\ar[r] & 0
}
\end{displaymath}
be a null sequence in $\mathpzc{Filt}_{\mathcal{S}}(\mathpzc{E})$. It is a kernel-cokernel pair if for each $i\in\mathbb{Z}$ and $i=top$.
\begin{displaymath}
\xymatrix{
0\ar[r] & A_{i}\ar[r]^{f_{i}} & B_{i}\ar[r]^{g_{i}} & C_{i}\ar[r] & 0
}
\end{displaymath}
is a kernel-cokernel pair. If $\mathpzc{E}$ is a abelian then the converse is true.
\end{prop}
\begin{proof}
By Proposition \ref{filtdiag}  
if for each $i\in\mathbb{Z}$ and $i=top$.
\begin{displaymath}
\xymatrix{
0\ar[r] & A_{i}\ar[r]^{f_{i}} & B_{i}\ar[r]^{g_{i}} & C_{i}\ar[r] & 0
}
\end{displaymath}
is a kernel-cokernel pair then 
\begin{displaymath}
\xymatrix{
0\ar[r] &(A_{top},\alpha_{i},a_{i})\ar[r]^{f} &(B_{top},\beta_{i},b_{i})\ar[r]^{g} & (C_{top},\gamma_{i},c_{i})\ar[r] & 0
}
\end{displaymath}
is a kernel-cokernel pair. 
Now suppose that $\mathpzc{E}$ is abelian, and that
\begin{displaymath}
\xymatrix{
0\ar[r] &(A_{top},\alpha_{i},a_{i})\ar[r]^{f} &(B_{top},\beta_{i},b_{i})\ar[r]^{g} & (C_{top},\gamma_{i},c_{i})\ar[r] & 0
}
\end{displaymath}
is a kernel-cokernel pair. Then $A_{i}\rightarrow B_{i}$ is a kernel of $B_{i}\rightarrow C_{i}$. Moreover $C_{i}=Im(B_{i}\rightarrow B_{top}\big\slash A_{top})$. Therefore $B_{i}\rightarrow C_{i}$ is an epimorphism, and we are done.
\end{proof}
\begin{defn}\label{filtexact}
We say that a null sequence 
\begin{displaymath}
\xymatrix{
0\ar[r] &(A_{top},\alpha_{i},a_{i})\ar[r] &(B_{top},\beta_{i},b_{i})\ar[r] & (C_{top},\gamma_{i},c_{i})\ar[r] & 0
}
\end{displaymath}
in $\mathpzc{Filt}_{\mathcal{S}}(\mathpzc{E})$ is exact if for each $i\in\mathbb{Z}$ and $i=top$ the null sequence
\begin{displaymath}
\xymatrix{
0\ar[r] & A_{i}\ar[r] & B_{i}\ar[r] & C_{i}\ar[r] & 0
}
\end{displaymath}
is an exact sequence in $\mathpzc{E}$. A null-sequence as above in $\overline{\mathpzc{Filt}}_{\mathcal{S}}(\mathpzc{E})$ is said to be exact if each sequence $0\rightarrow A_{i}\rightarrow B_{i}\rightarrow C_{i}\rightarrow 0$ is exact for all $i\in\mathbb{Z}$.
\end{defn}
\begin{defn}
\begin{enumerate}
\item
A map $f:A\rightarrow B$ in $\mathpzc{Filt}_{\mathcal{S}}(\mathpzc{E})$ or  is said to be an \textbf{admissible monomorphism} if there is an exact sequence
\begin{displaymath}
\xymatrix{
0\ar[r] & A\ar[r]^{f} & B\ar[r] & C\ar[r] & 0
}
\end{displaymath}
in $\mathpzc{Filt}_{\mathcal{S}}(\mathpzc{E})$.
\item
A map $g:B\rightarrow C$ in $\mathpzc{Filt}_{\mathcal{S}}(\mathpzc{E})$ is said to be an \textbf{admissible epimorphism} if there is an exact sequence
\begin{displaymath}
\xymatrix{
0\ar[r] & A\ar[r] & B\ar[r]^{g} & C\ar[r] & 0
}
\end{displaymath}
in $\mathpzc{Filt}_{\mathcal{S}}(\mathpzc{E})$.
\end{enumerate}
Likewise one defines admissible monomorphisms and admissible epimorphisms in $\overline{\mathpzc{Filt}_{\mathcal{S}}}(\mathpzc{E})$.
\end{defn}

\begin{prop}
Let $\mathpzc{E}$ be a complete and cocomplete exact category. Then $\mathpzc{Seq}_{top}(\mathpzc{E})$ and $\mathpzc{Seq}(\mathpzc{E})$ are complete and cocomplete exact categories with the classes of admissible monomorphisms and admissible epimorphisms defined above.
\end{prop}
The classes of admissible monomorphisms and epimorphisms are clearly stable under composition, and contain isomorphisms. We claim that admissible epimorphisms are stable under pullback.
\begin{prop}\label{pullbackpresmon}
Let
\begin{displaymath}
\xymatrix{
Y\ar[r]\ar[d]^{f} & X\ar[d]^{g}\\
\widetilde{Y}\ar[r] & \widetilde{X}
}
\end{displaymath}
be a commutative diagram where the vertical maps are admissible monomorphisms and the horizontal maps are admissible epimorphisms. Let
\begin{displaymath}
\xymatrix{
A\ar[r]\ar[d]^{h} & X\ar[d]^{g}\\
\widetilde{A}\ar[r] & \widetilde{X}
}
\end{displaymath}
be a diagram where $h$ is an admissible monomorphism. Then the map $A\times_{X}Y\rightarrow \widetilde{A}\times_{\widetilde{X}}\widetilde{Y}$ is an admissible monomorphism.
\end{prop}
\begin{proof}
There is a diagram of exact sequences
\begin{displaymath}
\xymatrix{
0\ar[r] & A\times_{X}Y\ar[r]\ar[d] & A\oplus Y\ar[r]\ar[d] & X\ar[r]\ar[d] & 0\\
0\ar[r] & \widetilde{A}\times_{\widetilde{X}}\widetilde{Y}\ar[r] & \widetilde{A}\oplus \widetilde{Y}\ar[r] & \widetilde{X}\ar[r] & 0
}
\end{displaymath}
The middle vertical map is an admissible monomorphism, so the first is as well.
\end{proof}
\begin{cor}\label{filtpullbackclosed}
Let $g:B\rightarrow C$ be an admissible epimorphism in $\mathpzc{Filt}_{\textbf{AdMon}}(\mathpzc{E})$ and let $f:A\rightarrow C$ be any morphism. Then the pullback $A\times_{C}B$ exists and the map $A\times_{C}B\rightarrow A$ is an admissible epimorphism.
\end{cor}
\begin{proof}
By Proposition \ref{filtdiag} and Proposition \ref{pullbackpresmon} the pullback is given by $(A\times_{C}B)_{i}=A_{i}\times_{C_{i}}B_{i}$ for $i\in\mathbb{Z}$ and $i=top$. with the maps $A_{i}\times_{C_{i}}B_{i}\rightarrow A_{j}\times_{C_{j}}B_{j}$ for $j>i$ being the obvious ones.
\end{proof}
The situation for pushouts of admissible monomorphisms is less satisfying. In general it only holds when $\mathpzc{E}$ is quasi-abelian.
\begin{prop}
Suppose that $\mathpzc{E}$ is a quasi-abelian category. Then $\mathpzc{Filt}_{\textbf{AdMon}}(\mathpzc{E})$ is an exact category. If $\mathpzc{E}$ is abelian then $\mathpzc{Filt}_{\textbf{AdMon}}(\mathpzc{E})$ is a quasi-abelian category. Finally, if $\mathpzc{E}$ is weakly $(\aleph_{0};\textbf{Admon})$-elementary, then all of this is true for $\overline{\mathpzc{Filt}}_{\textbf{AdMon}}(\mathpzc{E})$. 
 \end{prop}
\begin{proof}
It remains to show that pushouts of admissible monomorphisms are admissible monomorphisms. Let
\begin{displaymath}
\xymatrix{
(A_{top},\alpha_{i},a_{i})\ar[d]\ar[r]^{(f_{top},f_{i})} & (B_{top},\beta_{i},b_{i})\ar[d]\\
(X_{top},\chi_{i},x_{i})\ar[r]^{(f'_{top},f_{i})} & (Y_{top},\xi_{i},y_{i})
}
\end{displaymath}
be a pushout diagram with $(f_{top},f_{i})$ an admissible monic. By Proposition \ref{filteredadj}, the diagram
\begin{displaymath}
\xymatrix{
A_{top}\ar[r]^{f_{top}}\ar[d] & B_{top}\ar[d]\\
X_{top}\ar[r]^{f'_{top}} & Y_{top}
}
\end{displaymath}
is a pushout diagram. The filtration on $Y$ is given by $Y_{i}=Im(B_{i}\oplus X_{i}\rightarrow Y_{top})$. It remains to see that $X_{i}\rightarrow Im(B_{i}\oplus X_{i})$ is an admissible monomorphism. But $X_{i}\rightarrow Im(B_{i}\oplus X_{i}\rightarrow Y_{top})$ coincides with the composition $X_{i}\rightarrow X_{top}\rightarrow Y_{top}$.  $X_{top}\rightarrow Y_{top}$ is an admissible monomorphism as the pushout of an admissible monomorphism. The claim when $\mathpzc{E}$ is abelian follows immediately from Proposition \ref{filtkercoker}. The claim for exhaustively filtered objects also follows from the fact that in weakly $(\aleph_{0};\textbf{AdMon})$-elementary quasi-abelian categories we get isomorphisms 
$$lim_{\rightarrow} Im(B_{i}\oplus X_{i}\rightarrow Y_{top})\cong Im(lim_{\rightarrow} B_{i}\oplus X_{i}\rightarrow Y_{top})\cong Im(B_{top}\oplus X_{top}\rightarrow Y_{top})\cong Y_{top}$$
\end{proof}
 The categories $\reallywidehat{\mathpzc{Filt}}_{\mathcal{S}}(\mathpzc{E})$ and $\overline{\mathpzc{Filt}}_{\mathcal{S}}(\mathpzc{E})$ may considered as subcategories of $\mathpzc{Filt}_{\mathcal{S}}(\mathpzc{E})$ in the following sense.
\begin{prop}
Let
\begin{displaymath}
\xymatrix{
0\ar[r] & (A_{top},\alpha_{i},a_{i})\ar[r]^{(f_{top},f_{i})} & (B_{top},\beta_{i},b_{i})\ar[r]^{(g_{top},g_{i})} & (C_{top},\gamma_{i},c_{i})\ar[r] &0
}
\end{displaymath}
be a short exact sequence of in $\mathpzc{Filt}_{\mathcal{S}}(\mathpzc{E})$. 
\begin{enumerate}
\item
Suppose that $\mathpzc{E}$ is weakly $(\aleph_{0};\mathcal{S})$-elementary. If $(A_{top},\alpha_{i},a_{i})$ and $(C_{top},\gamma_{i},c_{i})$ are exhaustive then so is $(B_{top},\beta_{i},b_{i})$
\item
If $\mathcal{S}\subset\textbf{AdMon}$ $(A_{top},\alpha_{i},a_{i})$ and $(C_{top},\gamma_{i},c_{i})$ are complete then so is $(B_{top},\beta_{i},b_{i})$.
\end{enumerate}
\end{prop}
\begin{proof}
\begin{enumerate}
\item
Consider the diagram of short exact sequences
\begin{displaymath}
\xymatrix{
0\ar[r] & \textrm{lim}_{\rightarrow} A_{i}\ar[d]\ar[r] & \textrm{lim}_{\rightarrow} B_{i}\ar[d]\ar[r] &\textrm{lim}_{\rightarrow} C_{i}\ar[d]\ar[r] & 0\\
0\ar[r] & A_{top}\ar[r] & B_{top}\ar[r] & C_{top}\ar[r] & 0
}
\end{displaymath}
The two outer vertical maps are isomorphisms so the middle one is as well.
\item
By passing to a right abelianisation we may assume that $\mathpzc{E}$ is abelian. By the $3\times 3$ lemma we get for each $n\in\mathbb{Z}$ an exact sequence
$$0\rightarrow A_{top}\big\slash A_{n}\rightarrow B_{top}\big\slash B_{n}\rightarrow C_{top}\big\slash C_{n}\rightarrow0$$
Taking projective limits commutes with kernels, so we get a commutative diagram
\begin{displaymath}
\xymatrix{
0\ar[r] & A_{top}\ar[d]\ar[r] & B_{top}\ar[r]\ar[d] & C_{top}\ar[d]\ar[r] & 0\\
0\ar[r] & \lim_{\leftarrow_{n}} A_{top}\big\slash A_{n}\ar[r] & \lim_{\leftarrow_{n}} B_{top}\big\slash B_{n}\ar[r] &  \lim_{\leftarrow_{n}} C_{top}\big\slash C_{n} & 
}
\end{displaymath}
in which the top and bottom rows are exact, and the first and last maps are isomorphisms. The Snake Lemma implies that the middle map is also an isomorphism, as required.
\end{enumerate}
\end{proof}
\begin{cor}
If $\mathpzc{E}$ is weakly $(\aleph_{0};\mathcal{S})$-elementary and $\mathcal{S}$ is $\aleph_{0}$-closed, then $\overline{\mathpzc{Filt}}_{\mathcal{S}}(\mathpzc{E})$ is an extension-closed, coreflective, exact subcategory of $\mathpzc{Filt}_{\mathcal{S}}(\mathpzc{E})$.
\end{cor}
In general $\mathpzc{Filt}_{\mathcal{S}}(\mathpzc{E})$ and $\overline{\mathpzc{Filt}}_{\mathcal{S}}(\mathpzc{E})$ are not exact. However when $\mathpzc{D}$ is an exact category we can and will say that a functor $F:\mathpzc{Filt}_{\mathcal{S}}(\mathpzc{E})\rightarrow\mathpzc{D}$ or $F:\overline{\mathpzc{Filt}}_{\mathcal{S}}(\mathpzc{E})\rightarrow\mathpzc{D}$ is exact if it sends an exact sequence as defined in Definition \ref{filtexact} to an exact sequence in $\mathpzc{D}$. Likewise one defines exact functors $F:\mathpzc{D}\rightarrow\mathpzc{Filt}_{\mathcal{S}}(\mathpzc{E})$ or $F:\mathpzc{D}\rightarrow\overline{\mathpzc{Filt}}_{\mathcal{S}}(\mathpzc{E})$.

\begin{rem}
Although the categories $\mathpzc{Filt}_{\mathcal{S}}(\mathpzc{E})$ and $\overline{\mathpzc{Filt}}_{\mathcal{S}}(\mathpzc{E})$ may not be exact, they are strongly left exact in the sense of \cite{bazzoni2013one} Definition 3.2.
\end{rem}
\begin{example}
It is clear that the functors $(-)_{l},F_{i}$, and $\textrm{filt}$ are exact functors.
\end{example}
\subsection{Exactness of \textbf{AdMon}-Filtered Objects}
In this subsection we consider the special case that $\mathcal{S}\subseteq\textbf{AdMon}$. As we shall see, in many instances this is better behaved than the general case.
\subsubsection{Complete Objects}
We begin by analysing complete objects.
\begin{lem}
Let $\mathpzc{E}$ be an exact category with enough projectives. Then the inclusion functors 
$\mathpzc{Filt}_{\textbf{AdMon}}(\mathpzc{E})\rightarrow\reallywidehat{\mathpzc{Filt}}_{\textbf{AdMon}}(\mathpzc{E})$ and $\overline{\mathpzc{Filt}}_{\textbf{AdMon}}(\mathpzc{E})\rightarrow\reallywidehat{\overline{\mathpzc{Filt}}}_{\textbf{AdMon}}(\mathpzc{E})$ have left adjoints. 
\end{lem}
\begin{proof}
For each $n$ the map $A_{top}\big\slash A_{n}\rightarrow A_{top}\big\slash A_{n+1}$ is an admissible epimorphism. Let $P$ be projective. Then $Hom(P,A_{top}\big\slash A_{n})$ is a Mittag-Leffler system of abelian groups. The maps $A_{top}\rightarrow A_{top}\big\slash A_{n}$ are also admissible epimorphisms. Thus $Hom(P,A_{top})\rightarrow Hom(P,\lim_{\leftarrow_{n}}A_{top}\big\slash A_{n})$ is a epimorphism, so  $A_{top}\rightarrow\lim_{\leftarrow_{n}}A_{top}\big\slash A_{n}$ is an admissible epimorphism. In particular for each $n$ the map $\widetilde{A}_{n}\rightarrow\widetilde{A}_{top}$ is an admissible monomorphism, and by the Obscure Lemma each $\widetilde{A}_{n}\rightarrow \widetilde{A}_{n+1}$ is an admissible monomorphism. Thus $\widetilde{A}$ is a well-defined object of $\mathpzc{Filt}_{\textbf{AdMon}}(\mathpzc{E})$. Moreover for any projective $P$ we have $Hom(P,\widetilde{A})\cong\widetilde{Hom(P,A)}$. Thus $\tilde{A}$ is complete, since completion is well-defined for abelian groups. 
\end{proof}
\begin{prop}
Any object of $\reallywidehat{\mathpzc{Filt}}_{\textbf{AdMon}}(\mathpzc{E}))$ (resp. $\reallywidehat{\overline{\mathpzc{Filt}}}_{\textbf{AdMon}}(\mathpzc{E})$) is a projective limit of objects in $\mathpzc{Filt}^{+}_{\textbf{AdMon}}(\mathpzc{E})$ (resp. $\overline{\mathpzc{Filt}}^{+}_{\textbf{AdMon}}(\mathpzc{E})$).
\end{prop}
\begin{proof}
Let $A=(A_{top},\alpha_{i},a_{i})$ be a filtered object. Denote by $A^{\ge n}$ the filtered object with $A^{\ge n}_{top}=A_{top}\big\slash A_{n-1}$, $A^{\ge n}_{m}=0$ for $m<n$, and $A^{\ge n}_{m}=A_{m}\big\slash A_{n-1}$ for $m\ge n$. The maps $\alpha_{i}^{\ge n}$ and $a_{i}^{\ge n}$ are defined in the obvious way. There is a natural map $A^{\ge n-1}\rightarrow A^{\ge n}$. Moreover there is a natural isomorphism $\widetilde{A}\cong\textrm{lim}_{\leftarrow_{n}}A^{\ge n}$. In particular if $A$ is complete then $A\cong\textrm{lim}_{\leftarrow_{n}}A^{\ge n}$. 
\end{proof}
We introduce a concept which will be useful for dealing with graded equivalences of complete filtered objects later.
\begin{defn}
A filtered object $(A_{top},\alpha_{i},a_{i})$ in $\mathpzc{Seq}_{top}(\mathpzc{E})$ is said to be \textbf{completion acyclic} if for any $m\in\mathbb{Z}$, the projective sequence $A_{m}\big\slash A_{m-n}$ is $\lim_{\leftarrow}$-acyclic. 
\end{defn}
In particular any bounded-below filtered object is completion acyclic. 
\begin{prop}
Suppose that $\mathpzc{E}$ has enough projectives. Then any filtered object $(A_{top},\alpha_{i},a_{i})$ in $\mathpzc{Filt}_{\textbf{AdMon}}(\mathpzc{E})$ is completion acyclic.
\end{prop}
\begin{proof}
This follows immeditaley from Proposition \ref{prop:injadel}, and the fact that $A_{m}\big\slash A_{-+n}\rightarrow A_{m}\big\slash A_{m-n-1}$ is an admssible epimorphism.
\end{proof}
\subsubsection{Exactness and the Associated Graded Functor}
In this section we show that often one can determine when a map is an admissible epimorphism/ admissible monomorphism by looking at the associated graded map.
\begin{prop}\label{assgradmod}
Let $\mathcal{S}$ be a class contained in $\textbf{AdMon}$, and let $f:A\rightarrow B$ be a map in $\mathpzc{Filt}_{\mathcal{S}}(\mathpzc{E})$. Then
\begin{enumerate}
\item
 If $f$ is an admissible monomorphism  then $\textrm{gr}_{top}(f)$ is an admissible monomorphism. The converse is true if the filtrations on $A$ and $B$ are complete, $C=Coker(f)$ exists in $\mathpzc{Filt}_{\mathcal{S}}(\mathpzc{E})$, and $A$ is completion acyclic.
\item
 If $f$ is an admissible epimorphism  then $\textrm{gr}_{top}(f)$ is an admissible epimorphism. The converse is true if the filtrations on $A$, and $B$ are complete, and for each $m\in\mathbb{Z}$ and $m=top$ the sequence $K$ given by the kernel of $A_{m}\big\slash A_{m-n}\rightarrow B_{m}\big\slash B_{m-n}$ is $\lim_{\leftarrow}$-acyclic.
 \item
 If a null sequence
 $$0\rightarrow A\rightarrow B\rightarrow C\rightarrow 0$$
 in $\mathpzc{Filt}_{\mathcal{S}}(\mathpzc{E})$ is exact  then
 $$0\rightarrow\textrm{gr}_{top}(A)\rightarrow\textrm{gr}_{top}(B)\rightarrow\textrm{gr}_{top}(B)\rightarrow0$$
is exact. The converse is true if the filtration on $A$ is completion acyclic and all of $A,B,C$ are complete. 
\end{enumerate}
\end{prop}
\begin{proof}
The argument (particularly for the second claim) is a generalisation of \cite{calaque2021lie} Lemma 2.16 (2). Let us prove the first claim, the others are similar. The statement that $f$ being an admissible monomorphism implies that $gr_{top}(f)$ is follows  by definition and by the $3\times 3$ lemma. Let us prove the (partial) converse statement.  Suppose that $gr(f)_{top}$ is an admissible monomorphism. Let $C$ be the cokernel of $f$. For each $n,m\in\mathbb{Z}$ (or $m=top$) we have an exact sequence
\begin{displaymath}
\xymatrix{
0\ar[r] & A_{m}\big\slash A_{m-n}\ar[r] & B_{m}\big\slash B_{m-n}\ar[r] & C_{m}\big\slash C_{m-n}\ar[r] & 0
}
\end{displaymath}
Taking the projective limit over $n$, and noting that the filtrations are complete, with $A$ completion acyclic gives that
\begin{displaymath}
\xymatrix{
0\ar[r] & A_{m}\ar[r] & B_{m}\ar[r] & C_{m}\ar[r] & 0
}
\end{displaymath}
is exact for each $m\in\mathbb{Z}$, and $m=top$.
\end{proof}

\begin{rem}
If $A$ and $B$ are bounded below then the statements in the previous Proposition all become if and only if statements. For the second claim, if we let $K_{m,n}$ denote the kernel of $A_{m}\big\slash A_{m-n}\rightarrow B_{m}\big\slash B_{m-n}$. The Snake Lemma implies that $K_{m,n}\rightarrow K_{m,n+1}$ is an admissible epimorphism. Thus if projective limits of admissible epimorphic systems are exact, e.g. if $\mathpzc{E}$ has enough projectives, then the second statement of the previous proposition is also if and only if whenever $A$ and $B$ are complete. 
\end{rem}
\begin{cor}
Let $\mathcal{S}$ be a class of morphisms contained in $\textbf{AdMon}$. A map $f$ in $\overline{\mathpzc{Filt}}_{\mathcal{S}}(\mathpzc{E})$ between objects with bounded below filtrations is an admissible monomorphism (resp. epimorphism) if and only if $\textrm{gr}(f)$ is. Moreover a null sequence  
 $$0\rightarrow A\rightarrow B\rightarrow C\rightarrow 0$$
 in $\overline{\mathpzc{Filt}}_{\mathcal{S}}(\mathpzc{E})$ in which all objects  have bounded below filtrations is exact if and only if 
 $$0\rightarrow\textrm{gr}(A)\rightarrow\textrm{gr}(B)\rightarrow\textrm{gr}(B)\rightarrow0$$
is exact.
\end{cor}

\subsubsection{Generators and Projectives}
In this subsection we discuss what it means for a collection of objects in $\mathpzc{Filt}_{\mathcal{S}}(\mathpzc{E})$ to be a (projective) generating collection.
\begin{prop}\label{filtgen}
Let $\mathcal{G}$ be an admissible generating set in $\mathpzc{E}$. Then for any object $A$ of $\mathpzc{Filt}_{\mathcal{S}}(\mathpzc{E})$ (resp. $\overline{\mathpzc{Filt}}_{\mathcal{S}}(\mathpzc{E})$) there is an object $X$ of $\bigoplus_{i\in\mathbb{Z}^{<}_{top}}F_{i}(\mathcal{G})$ (resp. $\bigoplus_{i\in\mathbb{Z}}F_{i}(\mathcal{G})$) and an admissible epimorphism $X\rightarrow A$.  
\end{prop}
\begin{proof}
Let $(A_{top},\alpha_{i},a_{i})$ be a filtered object. For each $i\in\mathbb{Z}$ or $i=top$ pick some $G_{i}\in\mathcal{G}$ and an admissible epimorphism $G_{i}\twoheadrightarrow A_{i}$. Then $(\bigoplus_{i\in\mathbb{Z}}F_{i})G_{i}\oplus F_{top}G_{top}\rightarrow A$ is an admissible epimorphism. The claim for exhaustively filtered objects is similar.
\end{proof}
Although $\mathpzc{Filt}_{\mathcal{S}}(\mathpzc{E})$ is not an exact category, we still have the notion of a projective object.
\begin{defn}
An object $P$ of $\mathpzc{Filt}_{\mathcal{S}}(\mathpzc{E})$ (resp. $\overline{\mathpzc{Filt}}_{\mathcal{S}}(\mathpzc{E})$) is said to be \textbf{projective} if the functor $Hom(P,-)$ is exact. 
\end{defn}

The following is clear.
\begin{prop}\label{fiproj}
Let $P$ be a projective object in $\mathpzc{E}$, and $i\in\mathbb{Z}$. Then $F_{i}P$ is projective in both $\mathpzc{Filt}_{\mathcal{S}}(\mathpzc{E})$ and $\overline{\mathpzc{Filt}}_{\mathcal{S}}(\mathpzc{E})$ for any filtering class of morphisms $\mathcal{S}$. $F_{top}P$ is projective in $\mathpzc{Filt}_{\mathcal{S}}(\mathpzc{E})$.
\end{prop}
Using the fact that Proposition \ref{adproj} clearly works in this more general context, we can quite easily classify projective objects when $\mathcal{S}\subseteq\textbf{AdMon}$.
\begin{prop}\label{filtproj}
Let $\mathcal{S}\subseteq\textbf{AdMon}$ and assume that $\mathpzc{E}$ is a weakly $(\aleph_{0};\mathcal{S})$-elementary exact category. If a filtered object $(A_{top},\alpha_{i},a_{i})$ is projective in $\mathpzc{Filt}_{\mathcal{S}}(\mathpzc{E})$ then $A_{top}\big\slash A_{i}$  is projective in $\mathpzc{E}$ for all $i\in\Z$. In  $\overline{\mathpzc{Filt}}^{+}_{\mathcal{S}}(\mathpzc{E})$ $A$ is projective if and only if  $A_{top}\big\slash A_{i}$ is projective for each $i\in\mathbb{Z}$.
\end{prop}
\begin{proof}
The first assertion is a consequence of the fact that the functor $Q_{i}:\mathpzc{Filt}_{\mathcal{S}}(\mathpzc{E})\rightarrow\mathpzc{E}$ is left adjoint to the exact functor $F_{i+1}:\mathpzc{E}\rightarrow\mathpzc{Filt}_{\mathcal{S}}(\mathpzc{E})$. Thus $Q_{i}(A_{top},\alpha_{i},a_{i})=A_{top}\big\slash A_{i}$ is projective. The second assertion is a consequence of Corollary \ref{trans}, and the fact that the functor $\overline{(-)}$ is left adjoint to an exact functor, and so preserves projectives.
\end{proof}
Using Propositions \ref{filtcomp}, \ref{filtgen}, and \ref{fiproj} we get the following corollary.
\begin{cor}
Let $\mathpzc{E}$ be an elementary quasi-abelian category. Then $\overline{\mathpzc{Filt}}_{\textbf{AdMon}}(\mathpzc{E})$ is an elementary exact category. If $\mathpzc{E}$ is abelian then it is an elementary quasi-abelian category.
\end{cor}
In particular for the case $\mathpzc{E}=\mathpzc{Ab}$ is the category of abelian groups, this recovers Proposition 3.1.5 of \cite{qacs}.
\subsubsection{Filtered Objects in Monoidal Exact Categories}
Recall that the functor $\textrm{filt}_{top}:\mathpzc{Gr}_{\bullet}(\mathpzc{E})\rightarrow\mathpzc{Filt}_{\mathcal{S}}(\mathpzc{E})$ is strong monoidal. On the other hand the associated graded functor $gr_{top}:\mathpzc{Filt}_{\mathcal{S}}(\mathpzc{E})\rightarrow\mathpzc{Gr}(\mathpzc{E})$ is unfortunately only lax monoidal. However we have the following.

\begin{prop}\label{gradestronmon}
For $1\le j\le k$ let $A^{j}=(A_{top}^{j},\alpha^{j},a^{j})$ be a filtered object.
\begin{enumerate}
\item
Suppose that for each $n$ the map
$$\bigoplus_{i_{1}+\ldots+i_{k}=n}A^{1}_{i_{1}}\otimes\ldots\otimes A^{k}_{i_{k}}\rightarrow A^{1}_{top}\otimes\ldots\otimes A^{k}_{top}$$
is admissible. Then the map $\bigotimes_{j=1}^{k}\textrm{gr}(A^{j})\rightarrow\textrm{gr}\Bigr(\bigotimes_{j=1}^{k}A^{j}\Bigr)$ is an admissible epimorphism.\\
\item
If for each $1\le j\le k$ and each $i\in\mathbb{Z}$ the map each map $A^{j}_{i}\rightarrow A^{j}_{i+1}$ is a pure monomorphism, then the map $\bigotimes_{j=1}^{k}\textrm{gr}(A^{j})\rightarrow\textrm{gr}\Bigr(\bigotimes_{j=1}^{k}A^{j}\Bigr)$ is an isomorphism. 
\end{enumerate}
\end{prop}
\begin{proof}
\begin{enumerate}
\item
Let $I_{n}$ denote the image of the map $\bigoplus_{i_{1}+\ldots+i_{k}=n}A^{1}_{i_{1}}\otimes\ldots\otimes A^{k}_{i_{k}}\rightarrow A^{1}_{top}\otimes\ldots\otimes A^{k}_{top}$. By the obscure lemma the map $I_{n}\rightarrow I_{n+1}$ is an admissible monomorphism. Moreover the map $\bigoplus_{i_{1}+\ldots+i_{k}=n}A^{1}_{i_{1}}\otimes\ldots\otimes A^{k}_{i_{k}}\rightarrow I_{n}$ is an admissible epimorphism. Hence the map  $\bigoplus_{i_{1}+\ldots+i_{k}=n+1}A^{1}_{i_{1}}\otimes\ldots\otimes A^{k}_{i_{k}}\rightarrow I_{n+1}\big\slash I_{n}$ is an admissible epimorphism. The obscure lemma then implies the result.
\item
Suppose now that for each $1\le k\le n$ and each $i\in\mathbb{Z}$ the map $A^{k}_{i}\rightarrow A^{k}_{i+1}$  is a pure monomorphism. Equivalently $0\rightarrow A^{k}_{i}\rightarrow A^{k}_{i+1}\rightarrow\textrm{gr}_{i+1}(A^{k})$ is a pure exact sequence. By tensoring there is an induced $n$-dimensional chain complex which is exact along each axis. There is an acyclic sequence
$$\bigoplus_{l=1}^{k}A_{i_{1}+1}\otimes\ldots\otimes A_{i_{l}}\otimes\ldots\otimes A_{i_{k}+1}\rightarrow A_{i_{1}+1}\otimes\ldots\otimes A_{i_{k}+1}\rightarrow\textrm{gr}_{i_{1}+1}(A)\otimes\ldots\otimes\textrm{gr}_{i_{k}+1}(A)\rightarrow 0$$
Moreover this is a pure exact sequence. Hence there is a pure exact sequence.
$$0\rightarrow\sum_{l=1}^{k}A_{i_{1}+1}\otimes\ldots\otimes A_{i_{l}}\otimes\ldots\otimes A_{i_{k}+1}\rightarrow A_{i_{1}+1}\otimes\ldots\otimes A_{i_{k}+1}\rightarrow\textrm{gr}_{i_{1}+1}(A)\otimes\ldots\otimes\textrm{gr}_{i_{k}+1}(A)\rightarrow 0$$
This completes the proof.
\end{enumerate}
\end{proof}
\begin{defn}
A filtered object $H$ in $\mathpzc{Filt}_{\mathcal{S}}(\mathpzc{E})$ is said to be \textbf{flat} if the functor $H\otimes(-):\mathpzc{Filt}_{\mathcal{S}}(\mathpzc{E})\rightarrow\mathpzc{Filt}_{\mathcal{S}}(\mathpzc{E})$ is exact. 
\end{defn}
\begin{prop}\label{filtflat}
Let $X$ be a flat object of $\mathpzc{E}$ such that $X\otimes f\in\mathcal{S}$ whenever $f\in\mathcal{S}$. Then for any $i\ge0$, $F_{i}X$ is a flat object of $\mathpzc{Filt}_{\mathcal{S}}(\mathpzc{E})$.
\end{prop}
\begin{proof}
If $X$ is flat and $A=(A_{top},\alpha_{i},a_{i})$ is an object of $\mathpzc{Filt}_{\mathcal{S}}(\mathpzc{E})$, then $(F_{i}X\otimes A)_{j}$ is $0$ for $j<i$ and $X\otimes A_{j}$ for $j\ge i$. From the definition of exact sequences in $\mathpzc{Filt}_{\mathcal{S}}\mathpzc{E}$ it is clear that $F_{i}X\otimes A$ is exact.
\end{proof}
We cannot precisely classify flat objects. However we have the following.
\begin{prop}
Let $\mathcal{S}\subset\textbf{AdMon}$. Suppose that for any object $A$ of $\mathpzc{E}$, and any exact sequence $0\rightarrow X\rightarrow Y\rightarrow Z\rightarrow 0$ with $X\rightarrow Y$ in $\mathcal{S}$, the sequence
$$0\rightarrow Im_{\mathcal{S}}(X\otimes A\rightarrow Y\otimes A)\rightarrow Y\otimes A\rightarrow Z\otimes A\rightarrow 0$$
is exact. If a filtered object $(H_{top},t_{i},h_{i})$ is flat then $H_{top}$, and $H_{top}\big\slash H_{i}$ are flat. 
\end{prop}
\begin{proof}
Suppose that $(H_{top},t_{i},h_{i})$ is flat. Let
\begin{displaymath}
\xymatrix{
0\ar[r] & A\ar[r] & B\ar[r] & C\ar[r] & 0
}
\end{displaymath}
be an exact sequence in $\mathpzc{E}$.
Then
\begin{displaymath}
\xymatrix{
0\ar[r] & F_{0}A\ar[r] & F_{0}B\ar[r] & F_{0}C\ar[r] & 0
}
\end{displaymath}
is exact in $\mathpzc{Filt}_{\mathcal{S}}(\mathpzc{E})$. Therefore
\begin{displaymath}
\xymatrix{
0\ar[r] & H\otimes  F_{0}A\ar[r] & H\otimes  F_{0}B\ar[r] & H\otimes F_{0}C\ar[r] & 0
}
\end{displaymath}
is exact. In particular
\begin{displaymath}
\xymatrix{
0\ar[r] & H_{top}\otimes A\ar[r] & H_{top}\otimes B\ar[r] & H_{top}\otimes C\ar[r] & 0
}
\end{displaymath}
is exact. Hence $H_{top}$ is flat. Moreover, by assumption we have the following diagram.
\begin{displaymath}
\xymatrix{
& 0\ar[d] & 0\ar[d] & 0\ar[d] &\\
0\ar[r] & \textrm{Im}_{\mathcal{S}}(H_{i}\otimes A\rightarrow H_{top}\otimes A)\ar[r]\ar[d] & \textrm{Im}_{\mathcal{S}}(H_{i}\otimes B\rightarrow H_{top}\otimes B)\ar[r]\ar[d] & \textrm{Im}_{\mathcal{S}}(H_{i}\otimes C\rightarrow H_{top}\otimes C)\ar[r]\ar[d] & 0\\
0\ar[r] & H_{top}\otimes A\ar[r]\ar[d] & H_{top}\otimes B\ar[r]\ar[d] & H_{top}\otimes C \ar[r]\ar[d] & 0\\
0\ar[r] & \Bigr(H_{top}\big\slash H_{i}\Bigr)\otimes A\ar[r]\ar[d] & \Bigr(H_{top}\big\slash H_{i}\Bigr)\otimes B\ar[r]\ar[d] & \Bigr(H_{top}\big\slash H{i}\Bigr)\otimes C \ar[r]\ar[d] & 0\\
& 0 & 0 & 0 &
}
\end{displaymath}
The columns are exact, and the top two rows are exact by assumption. Therefore the third row is exact, so $H_{top}\big\slash H_{i}$ is flat.
\end{proof}
The assumptions of the propositions are satisfied, for example, if $\mathpzc{E}$ is quasi-abelian and $\mathcal{S}=\textbf{AdMon}$. The assumptions of Proposition \ref{filtflat} are satisfied for $\mathcal{S}=\mathpzc{Mor}(\mathpzc{E})$
We then get the following.
\begin{cor}
\begin{enumerate}
\item
If $\mathpzc{E}$ is a monoidal elementary exact category then $\mathpzc{Seq}_{top}(\mathpzc{E})$ and $\mathpzc{Seq}(\mathpzc{E})$ are monoidal elementary exact categories.
\item
If $\mathpzc{E}$ is a monoidal elementary quasi-abelian category then  $\overline{\mathpzc{Filt}}_{\textbf{AdMon}}(\mathpzc{E})$ is a monoidal elementary exact category. 
\end{enumerate}
\end{cor}
\subsubsection{Model Structures on Categories of Filtered Objects}
As we have seen, the categories $\mathpzc{Filt}_{\mathcal{S}}(\mathpzc{E})$ and $\overline{\mathpzc{Filt}}_{\mathcal{S}}(\mathpzc{E})$ in general only have exact structures when $\mathpzc{E}$ is quasi-abelian and $\mathcal{S}=\textbf{AdMon}$. However one is often still able to equip these categories with natural model structures, even when $\mathpzc{E}$ is a more general exact category. We shall assume that $\mathpzc{E}$ is weakly $(\aleph_{0};\mathcal{S})$-elementary and that $\mathcal{S}$ is $\aleph_{0}$-closed throughout this section. Denote by $Ch_{*}(\mathcal{S})$ the class of morphisms in $Ch_{*}(\mathpzc{E})$ consisting of morphisms $f_{\bullet}:X_{\bullet}\rightarrow Y_{\bullet}$ such that $f_{n}$ is in $\mathcal{S}$ for each $n\in\Z$. We consider the categories $\overline{\mathpzc{Filt}}_{Ch_{*}(\mathcal{S})}(Ch_{*}(\mathpzc{E}))\cong Ch(\overline{\mathpzc{Filt}}_{\mathcal{S}}(\mathpzc{E}))$ and $\mathpzc{Filt}_{Ch_{*}(\mathcal{S})}(Ch_{*}(\mathpzc{E}))\cong Ch(mathpzc{Filt}_{\mathcal{S}}(\mathpzc{E}))$. 
Let $\mathfrak{O}$ be a class of objects in $\mathpzc{E}$. Denote by $\mathpzc{Filt}_{\mathcal{S}}(\mathfrak{O})$ the class of objects in $\mathpzc{Filt}_{\mathcal{S}}(\mathpzc{E})$ of the form $(A_{top},\alpha_{i},a_{i})$ such that $A_{top}\in\mathfrak{O}$, $A_{i}\in\mathfrak{O}$ for each $i\in\mathbb{Z}$ and $\textrm{gr}(A_{top},\alpha_{i},a_{i})$ is in $\mathpzc{Gr}(\mathfrak{O})$. 
\begin{prop}\label{filteredthick}
Suppose that $\mathfrak{O}$ is extension closed, $\mathcal{S}\subset\textbf{AdMon}$, $(A_{top},\alpha_{i},a_{i})\in\mathpzc{Filt}^{+}_{\mathcal{S}}(\mathpzc{E})$and $gr(A)_{top}\in\mathfrak{O}$. Then for each $i\in\mathbb{Z}$ we have $A_{i}\in\mathfrak{O}$. In particular if $(A_{top},\alpha_{i},a_{i})\in\overline{\mathpzc{Filt}}_{\mathcal{S}}(\mathpzc{E})$, and $\mathfrak{O}$ is closed under  $(\aleph_{0};\mathcal{S})$-extensions, then $(A_{top},\alpha_{i},a_{i})\in\overline{\mathpzc{Filt}}_{\mathcal{S}}(\mathfrak{O})$ if and only if $\textrm{gr}(A)\in\mathpzc{Gr}(\mathfrak{O})$. If further $\mathfrak{O}$ is closed under taking cokernels of maps in $\mathcal{S}$, then $(A_{top},\alpha_{i},a_{i})\in\overline{\mathpzc{Filt}}^{+}_{\mathcal{S}}(\mathfrak{O})$ if and only if $A_{i}\in\mathfrak{O}$ for each $0\le i<\infty$.  
\end{prop}
\begin{proof}
Without loss of generality we may assume that $A_{i}=0$ for $i<0$. Suppose that $\textrm{gr}(A)\in\mathpzc{Gr}(\mathfrak{O})$. In particular  $A_{0}\in\mathfrak{O}$. Moreover for each $i\ge0$ there is an exact sequence
$$0\rightarrow A_{i}\rightarrow A_{i+1}\rightarrow \textrm{gr}_{i}(A)\rightarrow 0$$
Since $\textrm{gr}_{i+1}(A)\in\mathfrak{O}$, an easy induction gives $A_{j}\in\mathfrak{O}$ for each $0\le j<\infty$. If $A$ is exhaustively filtered then $A_{top}$ is an $(\aleph_{0};\mathcal{S})$-extension of objects in $\mathfrak{O}$, so the second claim follows. The final claim follows again from the exact sequence above. 
\end{proof}
Let $(\mathfrak{L},\mathfrak{R})$ be a cotorsion pair on $\mathpzc{E}$. In particular $\mathfrak{L}$ and $\mathfrak{R}$ are extension closed. Since $\mathfrak{L}$ is closed under transfinite extensions by Proposition \ref{filteredthick} we get the following.
\begin{cor}
Suppose that $\mathpzc{E}$ is weakly $(\aleph_{0};\mathcal{S})$-elementary. Let $A\in\overline{\mathpzc{Filt}}^{+}_{\mathcal{S}}(\mathpzc{E})$. Then $\textrm{gr}(A)\in\mathpzc{Gr}(\mathfrak{L})$ if and only if $A\in\overline{\mathpzc{Filt}}_{\mathcal{S}}(\mathfrak{L})$. 
\end{cor}
Moreover Lemma \ref{transgeneral} implies the following.
\begin{prop}\label{filteredsplit}
Let $A\in\overline{\mathpzc{Filt}}^{+}_{\mathcal{S}}(\mathfrak{L})$ and $B=(B_{top},\beta_{i},b_{i})\in\overline{\mathpzc{Filt}}^{+}_{\mathcal{S}}(\mathpzc{E})$ be such that $B_{i}\in\mathfrak{R}$ for each $i\in\mathbb{Z}$. Then any filtered exact sequence
$$0\rightarrow B\rightarrow C\rightarrow A\rightarrow 0$$
splits. 
\end{prop}
\begin{defn}
We say that a morphism $f$ in $\overline{\mathpzc{Filt}}_{\mathcal{S}}(\mathpzc{E})$ (resp. $\mathpzc{Filt}_{\mathcal{S}}(\mathpzc{E})$) is a \textbf{filtered quasi-isomorphism} if $f_{i}$ is a weak equivalence for each ${i\in\mathbb{Z}}$ (resp. for each $i\in\mathbb{Z}$ and $i=top$).
\end{defn}
\begin{defn}
A complex $A\in\mathpzc{Filt}_{\mathcal{S}}(\mathpzc{E})$ is said to be \textbf{quotient acyclic} if $A_{top}\big\slash A_{n}$ is acyclic for each $n$.
\end{defn}
This will allow us to detect acyclic complete objects using the associated graded functors under mild assumptions. Note that any acyclic complex is completion acyclic.
\begin{example}
If $A\in\mathpzc{Filt}_{\textbf{AdMon}}(\mathpzc{E})$ is bounded below then it is acyclic if and only if it is quotient acyclic.
\end{example}
\begin{prop}
\begin{enumerate}
\item
Let $\mathpzc{E}$ be an exact category with enough projectives. Then a complex $A$ in $\reallywidehat{\mathpzc{Filt}}_{\textbf{AdMon}}(Ch(\mathpzc{E}))$ is quotient acyclic if and only if it is acyclic. 
\item
If $\mathpzc{E}$ is weakly $(\aleph_{0};\textbf{AdMon})$-elementary then a complex $A$ in $\overline{\mathpzc{Filt}}_{\textbf{AdMon}}(Ch(\mathpzc{E}))$ is quotient acyclic if and only if $gr(A)$ is acyclic.
\end{enumerate}
\end{prop}
\begin{proof}
\begin{enumerate}
\item
Suppose $A$ is quotient acyclic. Then each $A\big\slash A_{n}$ is acyclic. Moreover each $A_{top}\big\slash A_{n}\rightarrow A\big\slash A_{n+1}$ is an admissible epimorphism. Thus for each projective $P$, $Hom(P,\lim_{\leftarrow_{n}}A_{top}\big\slash A_{n})\cong\lim_{\leftarrow_{n}} Hom(P,A_{top}\big\slash A_{n})$ is acyclic. Hence $A_{top}\cong \lim_{\leftarrow_{n}}A_{top}\big\slash A_{n}$ is acyclic. This also implies that each $A_{n}$ is also acyclic.
\item
This proof is a generalisation is a modification of \cite{calaque2021lie} Lemma 2.16 (1). Suppose $\mathpzc{E}$ is weakly $(\aleph_{0};\textbf{AdMon})$-elementary, and let $A$ be such that $gr(A)$ is acyclic. Each $A_{n}\rightarrow A_{n+1}$ is an admissible monomorphism and an equivalence. Since $\mathpzc{E}$ is weakly $(\aleph_{0};\textbf{AdMon})$-elementary  the transfinite composition $A_{n}\rightarrow A_{top}$ is an admissible monomorphism which is an equivalence. Thus $A_{top}\big\slash A_{n}$ is acyclic.
\end{enumerate}
\end{proof}
\begin{cor}\label{cor:grequivcomp}
Let $f:X\rightarrow Y$ be a map in $\reallywidehat{\overline{\mathpzc{Filt}}}_{\textbf{AdMon}}(Ch(\mathpzc{E}))$. Suppose that $\mathpzc{E}$ has enough projectives and is weakly $(\aleph_{0};\textbf{AdMon})$-elementary. Then $f$ is an equivalence if and only if $gr(f)$ is an equivalence.
\end{cor}
\begin{cor}\label{prop:samegraded}
Suppose that $\mathpzc{E}$ has enough projectives and is weakly $(\aleph_{0};\textbf{AdMon})$-elementary. Let $f:X\rightarrow Y$ be a graded equivalence $\overline{\mathpzc{Filt}}_{\textbf{AdMon}}(Ch(\mathpzc{E}))$. Then $\reallywidehat{f}:\reallywidehat{X}\rightarrow \reallywidehat{Y}$ is an equivalence in $\reallywidehat{\overline{\mathpzc{Filt}}}_{\textbf{AdMon}}(Ch(\mathpzc{E}))$.
\end{cor}
\begin{proof}
This follows immediately from the fact that the natural map $gr(A)\rightarrow gr(\reallywidehat{A})$ is an isomorphism.
\end{proof}
Let $(\mathfrak{L},\mathfrak{R})$ be a $dg_{*}$-compatible cotorsion pair on $\mathpzc{E}$ and consider the induced model structure on $Ch_{*}(\mathpzc{E})$. 
\begin{thm}\label{filtmod}
Suppose that $\mathpzc{E}$ is a weakly $\mathcal{S}$-elementary exact category and that $(\mathfrak{L},\mathfrak{R})$ is a $dg_{*}$-compatible cotorsion pair for $*\in\{\emptyset,\ge0\}$ and that acyclic cofibrations are degree-wise split. Then the transferred model structure exists on $\mathpzc{Filt}_{Ch_{*}(\mathcal{S})}(Ch_{*}(\mathpzc{E}))$ and $\overline{\mathpzc{Filt}}_{Ch_{*}(\mathcal{S})}(Ch_{*}(\mathpzc{E}))$. If $g$ is a filtered (acyclic) cofibration then $\textrm{gr}(g)$ is a graded (acyclic) cofibration. 
\end{thm}
\begin{proof}
We prove the claim for exhaustive filtered objects, the non-exhaustive case being similar. Let us first show that transfinite compositions of pushouts of coproducts of maps of the form $\textrm{filt}(f)$, where $f$ is a generating cofibration in $Ch_{*}(\mathpzc{Gr}(\mathpzc{E}))$, exist in $\overline{\mathpzc{Filt}}_{Ch_{*}(\mathcal{S})}(Ch_{*}(\mathpzc{E}))$. Indeed such a map $f$ is in each homological degree a split monomorphism. Since by assumption $\mathcal{S}$ is closed under direct sums, by Proposition \ref{filtdiag} these colimits exist and the functor $\sum_{n\in\mathbb{Z}}(-)_{n}$ commutes with them. It therefore suffices to show that transfinite compositions of pushouts of coproducts of maps of the form $\textrm{filt}(f)$ where $f$ is a generating acyclic cofibration is a weak eauivalence. But since such colimits are computed degree-wise and $\sum_{n\in\mathbb{Z}}(-)_{n}$ commutes with such colimits this is clear.  The claim about the associated graded functor follows from the fact that $\textrm{gr}$ sends $cell(\textrm{filt}(I))$ and $cell(\textrm{filt}(J))$ to $cell(I)$ and $cell(J)$ respectively, where $I$ is a collection of generating cofibrations in $Ch_{*}(\mathpzc{Gr}(\mathpzc{E}))$, and $J$ is a collection of generating acyclic cofibrations in $Ch_{*}(\mathpzc{Gr}(\mathpzc{E}))$. 
\end{proof}
\begin{prop}
The class of cofibrant/ trivially cofibrant) objects in the transferred model structure of Theorem \ref{filtmod} on $\overline{\mathpzc{Filt}}_{\mathcal{S}}(Ch_{*}(\mathpzc{E}))$ is $\overline{\mathpzc{Filt}}_{\mathcal{S}}(\mathfrak{C})$ (resp. $\overline{\mathpzc{Filt}}_{\mathcal{S}}(\mathfrak{C}\cap\mathfrak{W})$) where $\mathfrak{C}$ (resp $\mathfrak{W}\cap\mathfrak{C}$) is the class of cofibrant (resp. trivially cofibrant) objects in $Ch_{*}(\mathpzc{E})$. In particular a bounded below exhaustively filtered object $A$ is cofibrant/ trivially cofibrant if and only $gr(A)$ is cofibrant/ trivially cofibrant as a graded object.
\end{prop}
\begin{proof}
The first claim follows from the fact that the objects described are the cofibrations (trivial cofibrations) in the projective model structure (for functors) on $\mathpzc{Seq}(Ch_{*}(\mathpzc{E}))$ (resp. $\mathpzc{Seq}_{top}(Ch_{*}(\mathpzc{E}))$), and the computation of a transfinite composition $\textrm{lim}_{\rightarrow_{\alpha<\lambda}}X_{\alpha}$ where $X_{0}\in\overline{\mathpzc{Filt}}_{\mathcal{S}}(\mathpzc{E})$ (resp. $\mathpzc{Filt}_{\mathcal{S}}(\mathpzc{E})$) and each $X_{\alpha}\rightarrow X_{\alpha+1}$ is a pushout of a generating cofibration/ acyclic cofibration is the same in both $\mathpzc{Seq}(Ch_{*}(\mathpzc{E}))$ (resp. $\mathpzc{Seq}_{top}(Ch_{*}(\mathpzc{E}))$) and  $X_{0}\in\overline{\mathpzc{Filt}}_{\mathcal{S}}(\mathpzc{E})$ (resp. $\mathpzc{Filt}_{\mathcal{S}}(\mathpzc{E})$).
\end{proof}
\begin{thm}\label{thm:filtmonoidaxiom}
Suppose that $\mathpzc{E}$ is a weakly $\mathcal{S}$-elementary exact category, with $\mathcal{S}$ $\aleph_{0}$-closed, left cancelling, and such that functorial $\mathcal{S}$-images exist. Suppose further that $(\mathfrak{L},\mathfrak{R})$ is a monoidally $dg_{*}$-compatible cotorsion pair for $*\in\{\emptyset,\ge0\}$. Suppose further that $\mathpzc{Filt}_{\mathcal{S}}(\mathpzc{E})$ is a monoidal category with the induced monoidal structure, and that the induced model structure on $Ch_{*}(\mathpzc{E})$ has a collection of generating cofibrations which are split exact in each degree. The transferred model structure on $\overline{\mathpzc{Filt}}_{Ch_{*}(\mathcal{S})}(Ch_{*}(\mathpzc{E}))$ is monoidal. If $Ch_{*}(\mathpzc{E})$ has maps of the form $0\rightarrow D^{n}(F)$, with $F\in\mathfrak{L}$, as generating acyclic cofibrations, then $\overline{\mathpzc{Filt}}{Ch_{*}(\mathcal{S})}(Ch_{*}(\mathpzc{E}))$ satisfies the monoid axiom.
\end{thm}

\begin{proof}
We have already explained why $\overline{\mathpzc{Filt}}_{Ch_{*}(\mathcal{S})}(Ch_{*}(\mathpzc{E}))$ is a monoidal model category. Now let us prove the monoid axiom. A transfinite composition of pushouts of tensor products of objects with generating acyclic cofibrations will be of the form $A\rightarrow A\oplus Y$, where $Y$ is a direct sum of objects of the form $X\otimes F_{i}(D^{n}(L))$ for some $L\in\mathfrak{L}$. $Y$ is clearly trivially cofibrant, so $A\rightarrow A\oplus Y$ is a trivial cofibration. 
\end{proof}
\subsubsection{Model Structures for Complete Objects}
Let $\mathpzc{E}$ be a weakly idempotent complete exact category with enough projectives. Since projective limits of admissibly epimorphic projective sequences in $\mathpzc{E}$ are exact by Proposition \ref{prop:projML}, we immediately get the following.
\begin{prop}
Let $\mathpzc{E}$ be a weakly idempotent complete exact category with enough projectives. Let $A=(A_{top},\alpha_{i},a_{i})$ be an object of $Ch(\mathpzc{Filt}_{\textbf{AdMon}}(\mathpzc{E}))$. Then the map $\textrm{lim}_{\leftarrow_{n}}A_{top}\big\slash A_{n}\rightarrow\textrm{holim}_{\leftarrow_{n}}A_{top}\big\slash A_{n}$ is an equivalence. In particular if $A$ is complete then it is homotopically complete, and $A$ is homotopically complete if and only if $\textrm{lim}_{\leftarrow_{n}}A_{top}\big\slash A_{n}$ is acyclic. 
\end{prop}
\begin{prop}
Let $\mathpzc{E}$ be an elementary exact category, and let $A=(A_{top},\alpha_{i},a_{i})$ be an object of $Ch_{*}(\mathpzc{Filt}_{\textbf{AdMon}}(\mathpzc{E}))$ for $*\in\{\ge0,\emptyset\}$. Then $A$ is homotopically complete if and only if $\textrm{holim}_{\leftarrow_{n}} A_{n}$ is trivial.
\end{prop}
\begin{proof}
For the stable model category $Ch(\mathpzc{E})$ this follows from \cite{gwilliam2018enhancing} Section 2.7. For $Ch_{\ge0}(\mathpzc{E})$, it follows because all the necessary computations happen in $Ch(\mathpzc{E})$. 
\end{proof}
Let $\mathpzc{E}$ be an elementary quasi-abelian category. Then $Ch(\overline{\mathpzc{Filt}}_{\textbf{AdMon}}(\mathpzc{E}))$ is a combinatorial, left proper model category when equipped with the projective model structure. The left Bousfield localisation of $Ch(\overline{\mathpzc{Filt}}_{\textbf{AdMon}}(\mathpzc{E}))$ at maps of the form $0\rightarrow const(P)$ where $P$ is a projective generator, and $const(P)$ is the constant filtered object with $const(P)_{i}=const(P)_{top}=P$ for all $\in\mathbb{Z}$ exists. The adjunction
$$\adj{I_{\textbf{AdMon};\mathpzc{Mor}(\mathpzc{E})}}{\mathpzc{Seq}(Ch(\mathpzc{E}))}{\overline{\mathpzc{Filt}}_{\textbf{AdMon}}(Ch(\mathpzc{E}))}{R_{\textbf{AdMon};\mathpzc{Mor}(\mathpzc{E})}}$$
is a Quillen equivalence. $I_{\textbf{AdMon};\mathpzc{Mor}(\mathpzc{E})}$ is the identity on the maps $0\rightarrow const(P)$, so there is an induced Quillen equivalence of the localised model structures by \cite{hirschhorn2009model}.  The equivalences for the localised model structure on $\mathpzc{Seq}(Ch(\mathpzc{E}))$ are the derived graded equivalences by Theorem \ref{thm:htpcompletegwilliam}.  Since the adjunction is a Quillen equivalence, the equivalences for the localisation of right-hand side are also the derived graded equivalences. Since in $\overline{\mathpzc{Filt}}_{\textbf{AdMon}}(Ch(\mathpzc{E}))$ the associated graded functor preserves equivalences, the weak equivalences are in fact all graded equivalences. We have proven the following.
\begin{thm}
The homotopically complete filtered model structure exists on $\overline{\mathpzc{Filt}}_{\textbf{AdMon}}(Ch(\mathpzc{E}))$. A map $f$ is an equivalence if and only if $gr(f)$ is an equivalence. 
\end{thm}
 On the other hand, we know that  $\reallywidehat{\overline{\mathpzc{Filt}}}_{\textbf{AdMon}}(\mathpzc{E})$ is an extension closed subcategory of $\overline{\mathpzc{Filt}}_{\textbf{AdMon}}(\mathpzc{E})$. Thus it is an exact category. Moreover  $F_{i}P$ is projective in $\reallywidehat{\mathpzc{Filt}}_{\textbf{AdMon}}(\mathpzc{E})$ for $P$ projective and $i\in\mathbb{Z}$ or $i=top$. Using that $P$ is projective and so $Hom(P,-)$ commutes with completion, it is easy to see that if $P$ is tiny, then $F_{i}$ is tiny in $\reallywidehat{\mathpzc{Filt}}_{\textbf{AdMon}}(\mathpzc{E})$. Thus $\reallywidehat{\mathpzc{Filt}}_{\textbf{AdMon}}(\mathpzc{E})$ and $\reallywidehat{\overline{\mathpzc{Filt}}}_{\textbf{AdMon}}(\mathpzc{E})$ are elementary quasi-abelian categories. 
\begin{thm}
Let $\mathpzc{E}$ be an elementary quasi-abelian category. Then the complete model structure exists on $Ch(\reallywidehat{\overline{\mathpzc{Filt}}}_{\textbf{AdMon}}(\mathpzc{E}))$, and it is combinatorial. Moreover the adjunction
$$\adj{\reallywidehat{(-)}}{Ch(\overline{\mathpzc{Filt}}_{\textbf{AdMon}}(\mathpzc{E}))}{Ch(\reallywidehat{\overline{\mathpzc{Filt}}}_{\textbf{AdMon}}(\mathpzc{E}))}{J}$$
 is a Qullen equivalence when the category on the left-hand side is equipped with the homotopically complete monoidal structure.
\end{thm}
\begin{proof}
$J$ clearly preserves fibrations. By Corollary \ref{cor:grequivcomp} and Proposition \ref{prop:samegraded}, $\reallywidehat{(-)}$ and $J$ both preserve and reflect all equivalences. It follows that the Quillen adjunction is in fact a Quillen equivalence.
\end{proof}
Now suppose that $\mathpzc{E}$ is closed symmetric monoidal elementary quasi-abelian category. Then $\overline{\mathpzc{Filt}}_{\textbf{AdMon}}(\mathpzc{E})$ is closed symmetric monoidal elementary. $R(\underline{Hom}_{filt}(-,-))$ defines an internal hom on $\overline{\mathpzc{Filt}}_{\textbf{AdMon}}(\mathpzc{E})$. Indeed the functor $\underline{Hom}(X,-)$ is kernel preserving, so for $A$ and $B$ objects of $\overline{\mathpzc{Filt}}_{\textbf{AdMon}}(\mathpzc{E})$, $R(\underline{Hom}_{filt}(A,B))$ is in $\overline{\mathpzc{Filt}}_{\textbf{AdMon}}(\mathpzc{E})$. As for vector spaces over a field in \cite{calaque2021lie} Section 2.1, we then have the following.
\begin{prop}
Let $\mathpzc{E}$ be a closed symmetric monoidal elementary quasi-abelian category. If $Y$ is an object in $\reallywidehat{\overline{\mathpzc{Filt}}}_{\textbf{AdMon}}(\mathpzc{E})$ and $X$ is any object in $\overline{\mathpzc{Filt}}_{\textbf{AdMon}}(\mathpzc{E})$ then $\underline{Hom}_{filt}(X,Y)$ is in $\reallywidehat{\overline{\mathpzc{Filt}}}_{\textbf{AdMon}}(\mathpzc{E})$. In particular $(\reallywidehat{\overline{\mathpzc{Filt}}}_{\textbf{AdMon}}(\mathpzc{E}),\reallywidehat{-\otimes-},F_{0}(k),\underline{Hom}_{filt})$ is a closed symmetric monoidal elementary quasi-abelian category.
\end{prop}
\begin{proof}
Since $Y$ is complete, we may write it as $Y\cong\textrm{lim}_{\leftarrow_{n}}Y^{\ge n}$. \newline
Then $\underline{Hom}_{filt}(X,Y)\cong\textrm{lim}_{\leftarrow_{n}}\underline{Hom}_{filt}(X,Y^{\ge n})$. $\underline{Hom}_{filt}(X,Y^{\ge n})$ is also a bounded-below filtered object, and hence complete. As a reflective subcategory $\reallywidehat{\overline{\mathpzc{Filt}}}_{\textbf{AdMon}}(\mathpzc{E})$ is closed under colimits in $\overline{\mathpzc{Filt}}_{\textbf{AdMon}}(\mathpzc{E})$. The result follows by abstract nonsense. 
\end{proof}
\chapter{Homotopical Algebra in Exact Categories}\label{homtopalgexact}
In this final chapter we show that monoidal elementary exact categories are good settings in which to do homotopical algebra. In particular we show that they are naturally homotopical algebra contexts in the sense of \cite{toen2004homotopical}. e also establish Dold-Kan equivalences for algebras over operads. As will be explained in future work, over $\mathbb{Q}$ this essentially implies that for derived geometry relative to $\mathpzc{E}$ one can work with either simplicial objects are non-negatively graded complexes, and over $\mathbb{Z}$ justifies working with simplicial objects.
\section{Algebra in Monoidal (Model) Categories}
Before specialising to complexes in exact categories, let us establish some general results concerning the existence of model structures on categories of algebras over operads. We first recall basic facts concerning algebra in arbitrary monoidal categories.
Throughout this section $(\mathpzc{C},\otimes,k)$ is a monoidal category, with monoidal functor $\otimes$ and $(\mathpzc{M},\bullet_{\mathpzc{M}})$ is a unital left $\mathpzc{C}$-module.  Because we want to deal with operads as associative monoids in a certain non-symmetric monoidal category, we \textit{will not assume} that the monoidal structure is symmetric. We shall assume that $\mathpzc{C}$ is finitely complete and cocomplete, and that the tensor product commutes with coproducts. What follows is largely standard. Much of it can be found in \cite{koren} for example.

\subsection{Associative Monoids}
We denote the category of (unital) associative monoids internal to $\mathpzc{C}$ by $\mathpzc{Alg}_{\mathfrak{Ass}}(\mathpzc{C})$. There is a faithful forgetful functor $|-|_{\textit{Ass}}:\mathpzc{Alg}_{\mathfrak{Ass}}(\mathpzc{C})\rightarrow\mathpzc{C}$. If $\mathpzc{C}$ has countable products then $|-|$ has a left adjoint $T$ which can be constructed explicitly. Namely for $V\in\mathpzc{C}$, set 
$$T_{n}(V)=V^{\otimes n}$$
$$T(V)=\bigoplus_{n=0}^{\infty} T_{n}(V)$$
where by definition $T_{0}(V)=V^{\otimes 0}=k$. Now $\otimes$ preserves colimits in each variable, so
$$T(V)\otimes T(V)\cong\bigoplus_{m,n=0}^{\infty}T_{m}(V)\otimes T_{n}(V)$$
The multiplication 
$$m:T(V)\otimes T(V)\rightarrow T(V)$$
is defined on the summand $T_{m}(V)\otimes T_{n}(V)$ by the composition
$$T_{m}(V)\otimes T_{n}(V)\cong T_{m+n}(V)\rightarrow T(V)$$
where the isomorphism $T_{m}(V)\otimes T_{n}(V)=V^{\otimes m}\otimes V^{\otimes n}\cong V^{\otimes(m+n)}= T_{m+n}(V)$ is the natural isomorphism. The identity is given by the inclusion $e:k=T_{0}(V)\rightarrow T(V)$. $m$ and $e$ endow $T(V)$ with the structure of a unital associative monoid. It is clear that $V\rightarrow T(V)$ is functorial in $V$, and it is straightforward to check that $T$ is left adjoint to $|-|$.

\subsubsection{Commutative Monoids}
If $\mathpzc{C}$ is symmetric monoidal then we denote the category of (unital) commutative monoids by $\mathpzc{Alg}_{\mathfrak{Comm}}(\mathpzc{C})$.  If $\mathpzc{C}$ has finite coequalizers and countable coproducts then the forgetful functor $|-|_{\textit{Comm}}:\mathpzc{Alg}_{\mathfrak{Comm}}(\mathpzc{C})\rightarrow\mathpzc{C}$ has a left-adjoint, which can be constructed explicitly as follows. The symmetric group on $n$ letters $\Sigma_{n}$ acts on $T_{n}(V)=V^{\otimes n}$. Let $S_{n}(V)=T_{n}(V)_{\Sigma_{n}}$ be the coinvariants for this action. We then set
$$S(V)=\bigoplus_{n=0}^{\infty} S_{n}(V)$$
The associative monoid structure on $T(V)$ descends to an associative monoid structure on $S(V)$. One checks easily that it is commutative and that it is a left adjoint.

\subsection{Modules}

Fix objects $A$ and $B$ of $\mathpzc{Alg}_{\mathfrak{Ass}}(\mathpzc{C})$. We denote by ${}_{A}\mathpzc{Mod}(\mathpzc{M})$ the category of left modules for $A$ in $\mathpzc{M}$. There is a forgetful functor $|-|_{{}_{A}}:{}_{A}\mathpzc{Mod}(\mathpzc{M})\rightarrow\mathpzc{M}$. This functor has a left adjoint. It sends an object $E$ to the object $ A\bullet E$ with the obvious left action of $A$. The adjunction
$$\adj{A\bullet(-)}{\mathpzc{M}}{{}_{A}\mathpzc{Mod}(\mathpzc{M})}{|-|_{A}}$$
is called the \textbf{free-forgetful adjunction}.

If $\mathpzc{M}$ is a right $\mathpzc{C}$-module we denote by  $\mathpzc{Mod}_{A}(\mathpzc{M})$ the category of right modules for $A$  in $\mathpzc{M}$. If $\mathpzc{M}$ is a $\mathpzc{C}$-bimodule, we denote by ${}_{A}\mathpzc{Mod}_{B}(\mathpzc{M})$ the category of $A-B$ bimodules.
Regard $\mathpzc{C}$ as a bimodule over itself, and let $\mathpzc{M}$ be a left $\mathpzc{C}$-module.
Let $E$ be a right $A$-module in $\mathpzc{C}$ with action morphism
$$a_{E}:E\otimes_{\mathpzc{C}} A\rightarrow E$$
and $F$ a left $A$-module in $\mathpzc{M}$ with action morphism
$$a_{F}:A\bullet_{\mathpzc{M}} F\rightarrow F$$
If the category $\mathpzc{C}$ has finite coequalisers, then we define
$$E\bullet_{A}F$$
to be the coequaliser of the maps
\begin{displaymath}
\xymatrix{
E\otimes_{\mathpzc{C}} A\bullet_{\mathpzc{M}} F\ar@/^1.0pc/[rr]^{a_{E}}\ar@/_1.0pc/[rr]^{a_{F}}& & E\bullet_{\mathpzc{M}} F
}
\end{displaymath}
This defines a bifunctor
$$\bullet_{A}:\mathpzc{Mod}_{A}(\mathpzc{C})\times {}_{A}\mathpzc{Mod}(\mathpzc{M})\rightarrow\mathpzc{M}$$
If $E$ is a $B-A$ bimodule in $\mathpzc{C}$, $F$ an $A$ bimodule in $\mathpzc{M}$, then $E\otimes_{A}F$ is naturally a $B$-module, i.e. $\bullet_{A}$ gives a bifunctor
$${}_{B}\mathpzc{Mod}_{A}(\mathpzc{C})\times {}_{A}\mathpzc{Mod}(\mathpzc{M})\rightarrow{}_{B}\mathpzc{Mod}(\mathpzc{M})$$
In particular if $\mathpzc{C}$ is symmetric monoidal, $\mathpzc{M}$ a left $\mathpzc{C}$-module, and $A$ is a commutative monoid then this gives a bifunctor
$${}_{A}\mathpzc{Mod}(\mathpzc{C})\times {}_{A}\mathpzc{Mod}(\mathpzc{M})\rightarrow {}_{A}\mathpzc{Mod}(\mathpzc{M})$$
If we regard $\mathpzc{C}$ as a bimodule over itself, then this endows ${}_{A}\mathpzc{Mod}(\mathpzc{C})$ with a symmetric monoidal structure, and ${}_{A}\mathpzc{Mod}(\mathpzc{M})$ with a left ${}_{A}\mathpzc{Mod}(\mathpzc{C})$-module structure.
Suppose that $\alpha:A\rightarrow B$ is a map of commutative monoids in $\mathpzc{C}$. Then $B$ is an object of $\mathpzc{Alg}_{\mathfrak{Comm}}({}_{A}\mathpzc{Mod}(\mathpzc{C}))$. Thus we get an extension of scalars adjunction.
$$\adj{B\bullet_{A}(-)}{{}_{A}\mathpzc{Mod}(\mathpzc{M})}{{}_{B}\mathpzc{Mod}(\mathpzc{M})}{|-|_{\alpha}}$$
Suppose further that the monoidal structure on $\mathpzc{C}$ is closed, and let $\underline{\textrm{Hom}}(-,-)$ denote the internal hom functor. Then one can also construct an internal hom, $\underline{\textrm{Hom}}_{A}(-,-)$ functor on ${}_{A}\mathpzc{Mod}(\mathpzc{C})$ by a similar method as used to construct $\otimes_{A}$. This makes $({}_{A}\mathpzc{Mod}(\mathpzc{C}),\otimes_{A},\underline{\textrm{Hom}}_{A}(-,-),A)$ a closed monoidal category. See for example \cite{koren} for details.

\subsubsection{Adjunctions of Categories of Modules}
In this subsection we let $(\mathpzc{C},\otimes_{\mathpzc{C}},k_{\mathpzc{C}})$ and $(\mathpzc{D},\otimes_{\mathpzc{D}},k_{\mathpzc{D}})$ be monoidal categories, $(\mathpzc{M},\bullet_{\mathpzc{M}})$ a left $\mathpzc{C}$-module, and $(\mathpzc{N},\bullet_{\mathpzc{N}})$ a left $\mathpzc{D}$-module. 
\begin{defn}
A pair of functors
$$R:\mathpzc{D}\rightarrow\mathpzc{C},\; \overline{R}:\mathpzc{N}\rightarrow\mathpzc{M}$$
is said to be \textbf{lax monoidal} if there are natural transformations 
$$\epsilon:k_{\mathpzc{C}}\rightarrow R(k_{\mathpzc{D}})$$
$$\mu:R(-)\otimes_{\mathpzc{C}}R(-)\rightarrow R(-\otimes_{\mathpzc{N}}-)$$
$$\nu:R(-)\bullet_{\mathpzc{M}}\overline{R}(-)\rightarrow\overline{R}(-\bullet_{\mathpzc{N}}-)$$
satisfying obvious associativity and unitality axioms (see e.g. \cite{kerodon} Section 2.1.5). If $\epsilon$ and $\mu$, and $\nu$ are isomorphisms then the pair of functors is said to be \textbf{strong monoidal}.  One defines an \textbf{oplax monoidal pair of functors} dually. 
\end{defn}
Note that in particular $R$ is a lax monoidal functor between monoidal categories. 
\begin{defn}
A pair of adjunctions
$$\adj{L}{\mathpzc{C}}{\mathpzc{D}}{R}$$
$$\adj{\overline{L}}{\mathpzc{M}}{\mathpzc{N}}{\overline{R}}$$
is said to be \textbf{lax monoidal} if the pair $(R,\overline{R})$ is lax monoidal.
\end{defn}
Let 
$$\adj{L}{\mathpzc{C}}{\mathpzc{D}}{R}$$
$$\adj{\overline{L}}{\mathpzc{M}}{\mathpzc{N}}{\overline{R}}$$
be a lax monoidal pair of adjunctions. By doctrinal adjunction (\cite{kelly1974doctrinal}) $(L,\overline{L})$ is an oplax monoidal pair of functors - we spell out the argument here. Let $\epsilon$, $\mu$, and $\nu$ be the maps realising $(R,\overline{R})$ as lax monoidal. Define $\epsilon^{\vee}:L(k_{\mathpzc{C}})\rightarrow k_{\mathpzc{D}}$ to be the adjoint of $\epsilon$, $\mu^{\vee}$ to be the composition
\begin{displaymath}
\xymatrix{
L(-\otimes_{\mathpzc{C}}-)\ar[r] & L(RL(-)\otimes_{\mathpzc{C}}RL(-))\ar[rr]^{L(\mu_{L(-),L(-)})} & &LR(L(-)\otimes_{\mathpzc{D}} L(-))\ar[r] & L(-)\otimes_{\mathpzc{D}} L(-)
}
\end{displaymath}
 and $\nu^{\vee}$ to be the composition
\begin{displaymath}
\xymatrix{
\overline{L}(-\bullet_{\mathpzc{M}}-)\ar[r] & \overline{L}(RL(-)\bullet_{\mathpzc{M}}\overline{R}\overline{L}(-))\ar[rr]^{L(\nu_{\overline{L}(-),L(-)})} & &\overline{L}\overline{R}(L(-)\bullet_{\mathpzc{N}} \overline{L}(-))\ar[r] & L(-)\bullet_{\mathpzc{N}} \overline{L}(-)
}
\end{displaymath}
where the first map comes from the unit transformations and the last map from the counit transformation. \newline
\\
Now let 
$$\adj{L}{\mathpzc{C}}{\mathpzc{D}}{R}$$
$$\adj{\overline{L}}{\mathpzc{M}}{\mathpzc{N}}{\overline{R}}$$
be a lax monoidal pair of adjunctions, and $T$ an associative monoid in $\mathpzc{D}$. By lax monoidality, $R(T)$ is an associative monoid in $\mathpzc{C}$. Moreover $\overline{R}$ induces a well-defined functor
$$\overline{R}_{T}:{}_{T}\mathpzc{Mod}(\mathpzc{N})\rightarrow{}_{R(T)}\mathpzc{Mod}(\mathpzc{M})$$
Exactly as in \cite{schwede} Page 305, this functor has a left adjoint constructed as follows. For $M$ an $R(T)$-module in $\mathpzc{M}$ let 
$$\alpha_{M}:T\bullet_{\mathpzc{N}}(\overline{L}(R(T)\bullet_{\mathpzc{M}} M))\rightarrow T\bullet_{\mathpzc{N}}\overline{L}(M)$$
be induced by the multiplication map $R(T)\bullet_{\mathpzc{M}} M\rightarrow M$, and let $\beta_{M}$ denote the composition
\begin{displaymath}
\xymatrix{
\overline{L}(R(T)\bullet_{\mathpzc{M}}M)\ar[r]^{\nu^{\vee}_{R(T),M}} & LR(T)\bullet_{\mathpzc{N}}\overline{L}(M)\ar[r] & T\bullet_{\mathpzc{N}}\overline{L}(M)
}
\end{displaymath}
Denote by $\overline{L}_{T}(M)$ the coequalizer of the maps $\alpha_{M}$ and $\beta_{M}$. $\overline{L}_{T}$ is a functorial construction, and is left adjoint to $\overline{R}_{T}$.
\begin{rem}
If $R(T)\bullet_{\mathpzc{M}}M$ is a free $S$-module, then $\overline{L}_{T}(R(T)\bullet_{\mathpzc{M}}M)\cong T\bullet_{\mathpzc{N}}M$.  
\end{rem}
If $\alpha:S\rightarrow R(T)$ is a map of associative monoids, then by composition with the extension of scalars adjunction we get a composite adjunction
$$\adj{L_{\alpha}}{{}_{S}\mathpzc{Mod}(\mathpzc{M})}{{}_{T}\mathpzc{Mod}(\mathpzc{N})}{R_{\alpha}}$$
\section{Operads}
Here we will recall some facts about operads and their algebras. We shall fix a \textit{symmetric} monoidal category $(\mathpzc{C},\otimes,k)$, and assume now that the tensor product commutes with \textit{all colimits} in each variable. Most of the definitions and claims in this section can be found in \cite{loday2012algebraic}.
\subsection{Non-Symmetric Sequences}
Consider the category $\mathpzc{Gr}_{\mathbb{N}_{0}}(\mathpzc{C})\defeq\mathpzc{Fun}(\mathbb{N}_{0},\mathpzc{C})$ where $\mathbb{N}_{0}$ is the discrete category of non-negative integers. 
\begin{defn}
Let $M$ and $N$ be objects of $\mathpzc{Gr}_{\mathbb{N}_{0}}(\mathpzc{C})$. The \textbf{tensor product} of $M$ and $N$, denoted $M\otimes N$, is the graded object defined by
$$(M\otimes N)(n)=\bigoplus_{i+j=n}(M(i)\otimes N(j))$$
\end{defn}
\begin{defn}
\begin{enumerate}
\item
The graded object $I_{0}$ is defined by $I_{0}(i)=0$ for $i\neq 0$ and $I_{0}(0)=k$.
\item
The graded object $I$ is defined by $I(i)=0$ for $i\neq 1$ and $I(1)=k$.
\end{enumerate}
\end{defn}
\begin{defn}
Let $M$ and $N$ be objects of $\mathpzc{Gr}_{\mathbb{N}_{0}}(\mathpzc{C})$. The \textbf{composite product} of $M$ and $N$ is the graded object defined by
$$M\circ_{ns}N\defeq\bigoplus_{k\ge0}(M(k)\otimes N^{\otimes k}(n))$$
\end{defn}
\begin{prop}
$(\mathpzc{Gr}_{\mathbb{N}_{0}}(\mathpzc{C}),\otimes,I_{0})$ and $(\mathpzc{Gr}_{\mathbb{N}_{0}}(\mathpzc{C}),\circ,I)$ are monoidal categories.
\end{prop}
 \subsection{Symmetric Sequences}
Let $(\mathpzc{C},\otimes, k)$ be a symmetric monoidal category with all small coproducts. Note that in this case $(\mathpzc{Gr}_{\mathbb{N}_{0}}(\mathpzc{C}),\otimes,I_{0})$ is a symmetric monoidal category.
\subsubsection{Discrete Groups in Monoidal Categories}
We denote by $k[-]:\mathpzc{Set}\rightarrow\mathpzc{C}$ the functor which sends a set $S$ to the object $k[S]=\bigoplus_{S}k$. If $f:S\rightarrow T$ is a map of sets, then $k[f]:k[S]\rightarrow k[T]$ is the morphism which sends the copy of $k$ indexed by $s\in S$ to the copy indexed by $f(s)\in T$. Objects and morphisms in the essential image of the functor $k[-]:\mathpzc{Set}\rightarrow\mathpzc{C}$ will be called \textbf{discrete}.
\begin{prop}
Let $(\mathpzc{C},\otimes,k)$ be a monoidal category. Suppose that $\otimes$ preserves all coproducts. Endow $\mathpzc{Set}$ with its Cartesian monoidal structure. Then the functor $k[-]:\mathpzc{Set}\rightarrow\mathpzc{C}$ is strong monoidal.
\end{prop}
\begin{proof}
Let $S$ and $T$ be sets. Then
$$
\Bigr(\coprod_{S}k\Bigr)\otimes\Bigr(\coprod_{T}k\Bigr)
\cong\coprod_{S}\coprod_{T} k\otimes k
\cong\coprod_{S}\coprod_{T} k
\cong\coprod_{S\times T}k
$$
\end{proof}
In particular $\mathpzc{Set}\rightarrow\mathpzc{C}$ sends groups to Hopf monoids. If $G$ is a group we call $k[G]$ the group monoid of $G$ in $\mathpzc{C}$. 
\subsection{The Category of $\Sigma$-Modules}
We denote by $k[\Sigma]$ the monoid in $(\mathpzc{Gr}_{\mathbb{N}_{0}}(\mathpzc{C}),\otimes,I_{0})$ defined as follows. In degree $n$ it is given by the monoid $k[\Sigma_{n}]$, the free monoid on the symmetric group in $n$ letters. It is an associative monoid in the monoidal category $(\mathpzc{Gr}_{\mathbb{N}_{0}}(\mathpzc{C}),\otimes,I)$.
\begin{defn}
The \textbf{category of } $\Sigma$-\textbf{modules} in $\mathpzc{C}$, denoted $\mathpzc{Mod}_{\Sigma}(\mathpzc{C})$  is the category of right $k[\Sigma]$-modules.
\end{defn}
\begin{defn}
Let $M$ and $N$ be two $\Sigma$-modules. The \textbf{composite product} of $M$ and $N$, denoted $M\circ N$ is defined by 
$$M\circ N(n)=\bigoplus_{k\ge0}(M(k)\otimes_{\Sigma_{k}}N^{\otimes k}(n)$$
\end{defn}
 \begin{prop}
 $(\mathpzc{Mod}_{\Sigma},\circ,I)$ is a monoidal category.
 \end{prop}
 \subsection{Operads and Algebras}
 \begin{defn}
 \begin{enumerate}
 \item
The category of \textbf{non-symmetric operads} is the category of associative monoids $\mathpzc{Alg}_{\mathfrak{Ass}}(\mathpzc{Gr}_{\mathbb{N}_{0}}(\mathpzc{C}),\circ_{ns},I)$. 
 \item
The category of \textbf{symmetric operads} is the category of associative monoids $\mathpzc{Alg}_{\mathfrak{Ass}}(\mathpzc{Mod}_{\Sigma}(\mathpzc{C}),\circ,I)$. 
 \end{enumerate}
 \end{defn}
 \begin{enumerate}
 \item
 $\mathpzc{Alg}_{\mathfrak{Ass}}(\mathpzc{C})$ is the category of algebras over the non-symmetric operad with $\mathfrak{Ass}(n)=k$ for all $n$.
 \item
  $\mathpzc{Alg}_{\mathfrak{Comm}}(\mathpzc{C})$ is the category of algebras over the symmetric operad with $\mathfrak{Comm}(n)=k$ for all $n$, regarded as a trivial $\Sigma_{n}$-module.
  \item
  In the additive setting there is an operad $\mathfrak{Lie}$ such that $\mathpzc{Alg}_{\mathfrak{Lie}}(\mathpzc{C})$ is the category of Lie algebras.
  \end{enumerate}
By thinking of objects of $\mathpzc{C}$ as graded objects concentrated in degree $0$, we may regard $\mathpzc{C}$ as a full subcategory of both $\mathpzc{Gr}_{\mathbb{N}_{0}}(\mathpzc{C})$ and $\mathpzc{Mod}_{\Sigma}(\mathpzc{C})$. This makes $\mathpzc{C}$ into a left $(\mathpzc{Gr}_{\mathbb{N}_{0}}(\mathpzc{C}),\circ_{ns},I)$-module and a left $(\mathpzc{Mod}_{\Sigma}(\mathpzc{C}),\circ,I)$-module 
 \begin{defn}
 Let $\mathfrak{P}$ be either a symmetric or non-symmetric operad in $\mathpzc{C}$. The \textbf{category of} $\mathfrak{P}$-\textbf{algebras}, denoted $\mathpzc{Alg}_{\mathfrak{P}}(\mathpzc{C})$ is the category ${}_{\mathfrak{P}}\mathpzc{Mod}(\mathpzc{C})$ (where in the non-symmetric case $\mathpzc{C}$ is a left $(\mathpzc{Gr}(\mathpzc{C}),\circ_{ns},I)$-module, and in the symmetric case $\mathpzc{C}$ is a left $(\mathpzc{Mod}_{\Sigma}(\mathpzc{C}),\circ,I)$-module ).
\end{defn}
\subsubsection{Colimits of Algebras}
In this section we recall from \cite{harper2010homotopy} how to compute certain colimits in the category of aglebras over an operad. 
 \begin{prop}[ \cite{harper2010homotopy} Proposition 7.28]
 Let $\mathfrak{P}$ be a non-symmetric operad and $X$ a $\mathfrak{P}$-algebra. There exists an object $\mathfrak{P}_{X}$ in $\mathpzc{Gr}_{\mathbb{N}_{0}}(\mathpzc{C})$ together with, for any $Y\in\mathpzc{Gr}_{\mathbb{N}_{0}}(\mathpzc{C})$, an isomorphism, natural in $X$ and $Y$,
 $$X\coprod(\mathfrak{P}\circ_{ns}Y)\cong\mathfrak{P}_{X}\circ_{ns}(Y)$$
 \end{prop}
 As in \cite{harper2010homotopy} Definition 7.31, for $s:A\rightarrow B$ a map in $\mathpzc{Gr}_{\mathbb{N}_{0}}(\mathpzc{C})$, we define $Q^{t}_{q}(s)$ for $t\ge 1$ and $0\le q\le t$ as follows. $Q_{0}^{t}(s)\defeq A^{\otimes t}$, $Q^{t}_{t}(s)\defeq B^{\otimes t}$ and for $0<q<t$, $Q^{t}_{q}(s)$ is defined by the pushout. 
 \begin{displaymath}
 \xymatrix{
 (X^{\otimes(t-q)}\otimes Q^{q}_{q-1}(s))^{\oplus\binom{t}{q}}\ar[d]\ar[r] & Q_{q-1}^{t}(s)\ar[d]\\
  (X^{\otimes(t-q)}\otimes B^{\otimes q})^{\oplus\binom{t}{q}}\ar[r] & Q^{t}_{q}(s)
 }
 \end{displaymath}
 where the top map is the obvious projection, and the left-hand map is induced by natural map $Q^{q}_{q-1}(s)\rightarrow B^{\otimes q}$. 
 \begin{prop}[\cite{harper2010homotopy}  Proposition 7.32]
 Let $s:A\rightarrow B$ be a map in $\mathpzc{C}$, and let $X$ be a $\mathfrak{P}$-algebra. Consider a pushout diagram
 \begin{displaymath}
 \xymatrix{
 \mathfrak{P}(A)\ar[d]^{\mathfrak{P}(s)}\ar[r] & X\ar[d]\\
 \mathfrak{P}(B)\ar[r] & P
 }
 \end{displaymath}
 Then $P$ is naturally isomorphic to a filtered colimit
 $$P\cong\textrm{lim}_{\rightarrow_{n}}X_{n}$$
 where $X_{0}=X$, and for $n\ge 1$ the map $X_{n-1}\rightarrow X_{n}$ is given by the pushout diagram in $\mathpzc{C}$
 \begin{displaymath}
 \xymatrix{
 \mathfrak{P}_{X}(n)\otimes Q^{n}_{n-1}(s)\ar[d]\ar[r] & X_{n-1}\ar[d]\\
 \mathfrak{P}_{X}(n)\otimes B^{\otimes n}\ar[r] & X_{n}
 }
 \end{displaymath}
 \end{prop}
 The filtration $P\cong\textrm{lim}_{\rightarrow_{n}}X_{n}$ will be called the \textbf{standard filtration}. 
 There is also a version of this when $\mathfrak{P}$ is a symmetric operad in $\mathpzc{C}$. As in \cite{harper2010homotopy} Proposition 7.6 for $X$ a $\mathfrak{P}$-algebra, there is a $\Sigma$-module $\mathfrak{P}_{X}$ in $\mathpzc{C}$ such that for any $\Sigma$-module $Y$ in $\mathpzc{C}$, there is an isomorphism, natural in $X$ and $Y$,
 $$X\coprod(\mathfrak{P}\circ Y)\cong\mathfrak{P}_{X}\circ Y$$
 The pushout $P$ of a map $\mathfrak{P}\circ s:\mathfrak{P}\circ A\rightarrow\mathfrak{P}\circ B$ along a map $\mathfrak{P}\circ A\rightarrow X$ can also be computed using a standard filtration 
  $$P\cong\textrm{lim}_{\rightarrow_{n}}X_{n}$$ where the map $X_{n-1}\rightarrow X_{n}$ is given by pushout along $ \mathfrak{P}_{X}(n)\otimes_{\Sigma_{n}} Q^{n}_{n-1}(s)\rightarrow  \mathfrak{P}_{X}(n)\otimes_{\Sigma_{n}} B^{\otimes n}$
  \subsubsection{Adjunctions of Operads and Algebras}
  Let $\mathpzc{C}$ and $\mathpzc{D}$ be symmetric monoidal categories, and let 
  $$\adj{L}{\mathpzc{C}}{\mathpzc{D}}{R}$$
  be a lax monoidal adjunction. There are two obvious induced lax monoidal pairs of adjunctions.
  $$\adj{L_{ns}}{(\mathpzc{Gr}_{\mathbb{N}_{0}}(\mathpzc{C}),\otimes,I_{0})}{(\mathpzc{Gr}_{\mathbb{N}_{0}}(\mathpzc{D}),\otimes,I_{0})}{R_{ns}},\;\;\; \adj{L}{\mathpzc{C}}{\mathpzc{D}}{R}$$
    $$\adj{L_{ns}}{(\mathpzc{Gr}_{\mathbb{N}_{0}}(\mathpzc{C}),\circ_{ns},I)}{(\mathpzc{Gr}_{\mathbb{N}_{0}}(\mathpzc{D}),\circ_{ns},I)}{R_{ns}},\;\;\; \adj{L}{\mathpzc{C}}{\mathpzc{D}}{R}$$
    Thus if $\mathfrak{T}$ is a non-symmetric operad in $\mathpzc{D}$, $\mathfrak{P}$ a non-symmetric operad in $\mathpzc{C}$, and $\alpha:\mathfrak{P}\rightarrow R(\mathfrak{T})$ a map of operads, then we get an adjunction
    $$\adj{L_{\alpha}}{\mathpzc{Alg}_{\mathfrak{P}}(\mathpzc{C})}{\mathpzc{Alg}_{\mathfrak{T}}(\mathpzc{D})}{R_{\alpha}}$$
    Consider the monoid $k_{\mathpzc{D}}[\Sigma]$ in $(\mathpzc{Gr}_{\mathbb{N}_{0}}(\mathpzc{D}),\otimes,I_{0})$. Using the map $k_{\mathpzc{C}}[\Sigma]\rightarrow R(k_{\mathpzc{D}}[\Sigma])$ we get an adjunction.
    $$\adj{L_{\Sigma}}{\mathpzc{Mod}_{\Sigma}(\mathpzc{C})}{\mathpzc{Mod}_{\Sigma}(\mathpzc{D})}{R_{\Sigma}}$$
    If $R$ is a symmetric monoidal functor, then $R_{\Sigma}$ is in fact a lax monoidal functor  $(\mathpzc{Mod}_{\Sigma}(\mathpzc{D}),\circ,I)\rightarrow(\mathpzc{Mod}_{\Sigma}(\mathpzc{C}),\circ,I)$, and we get the following lax monoidal pair of adjunctions.
        $$\adj{L_{\Sigma}}{(\mathpzc{Mod}_{\Sigma}(\mathpzc{C}),\circ,I)}{(\mathpzc{Mod}_{\Sigma}(\mathpzc{D}),\circ,I)}{R_{\Sigma}},\;\;\; \adj{L}{\mathpzc{C}}{\mathpzc{D}}{R}$$
So again if $\mathfrak{T}$ is a symmetric operad in $\mathpzc{D}$, $\mathfrak{P}$ a symmetric operad in $\mathpzc{C}$, and $\alpha:\mathfrak{P}\rightarrow R(\mathfrak{T})$ a map of operads, then we get an adjunction
    $$\adj{L_{\alpha}}{\mathpzc{Alg}_{\mathfrak{P}}(\mathpzc{C})}{\mathpzc{Alg}_{\mathfrak{T}}(\mathpzc{D})}{R_{\alpha}}$$
  \begin{rem}
  It is important to keep in mind that $L$ is in general not lax symmetric monoidal. However it is oplax symmetric monoidal. In particular if $X$ is a right $\Sigma_{n}$-module and $Y$ a left $\Sigma_{n}$-module in $\mathpzc{C}$, we get a map $L(X\otimes_{\Sigma_{n}}Y)\rightarrow L(X)\otimes_{\Sigma_{n}}L(Y)$. This is the map in Definition 2.3.2 of \cite{white2019homotopical}.
  \end{rem}

\section{Model Categories of Monoids and Algebras}
Let $(\mathpzc{C},\otimes_{\mathpzc{C}},k)$ be a monoidal category, and $(\mathpzc{M},\bullet_{\mathpzc{M}})$ be a left $\mathpzc{C}$-module which is also a combinatorial model category. Again, we are not assuming that $\otimes$ is symmetric, nor that $\bullet$ commutes with any colimits a priori.
\subsection{Existence of Model Structures}
\begin{defn}
An associative monoid $R$ in $\mathpzc{C}$ is said to be \textbf{admissible in $\mathpzc{M}$} if the transferred model structure along the free-forgetful adjunction
$$\adj{R\otimes(-)}{\mathpzc{M}}{{}_{R}\mathpzc{Mod}(\mathpzc{M})}{|-|_{R}}$$
exists on ${}_{R}\mathpzc{Mod}(\mathpzc{M})$. 
\end{defn}
We will use the following slightly generalised version of \cite{schwede} Definition 3.3.
\begin{defn}
Let $R$ be an associative monoid in $\mathpzc{C}$.  A collection $\mathcal{S}$ of weak equivalences in $\mathpzc{M}$ is said to satisfy the $R$-\textbf{monoid axiom} if 
\begin{enumerate}
\item
every map of the form $R\bullet_{\mathpzc{M}} s$ where $s\in\mathcal{S}$ is a $h$-cofibration.
\item
$\mathpzc{M}$ is weakly $\mathcal{S}^{R}$-elementary, where $\mathcal{S}$ is the class of pushouts of maps of the form $R\bullet_{\mathpzc{M}} s$ for $s\in\mathcal{S}$.
\end{enumerate}
\end{defn}
In particular, a monoidal model category satisfies the monoid axiom of \cite{schwede} precisely if the class of acyclic cofibrations satisfies the $R$-monoid axiom for any associative monoid $R$, where we regard $\mathpzc{C}$ as a left-module over itself. 
\begin{prop}
Suppose that the class of acyclic cofibrations satisfies the $R$-monoid axiom, and that the forgetful functor $|-|_{R}:{}_{R}\mathpzc{Mod}(\mathpzc{M})\rightarrow\mathpzc{M}$ commutes with transfinite compositions. Then $R$ is admissible. Suppose further that $|-|_{R}$ commutes with all colimits and the class of cofibrations satisfies the $R$-monoid axiom. Then ${}_{R}\mathpzc{Mod}(\mathpzc{M})$ is left proper.
\end{prop}
\begin{proof}
The first claim is essentially \cite{schwede} Theorem 4.1/ Remark 4.2, and follows immediately from Theorem \ref{transfer}. The second follows from Proposition \ref{prop:leftproppushouttrans}.
\end{proof}
\subsubsection{Model Structures on Algebras Over Operads}
We now specialise to the situation that $\mathpzc{C}$ is a monoidal category which is also a monoidal category, and we consider $\mathpzc{C}$ as a left $(\mathpzc{Gr}_{\mathbb{N}_{0}},\circ_{ns},I)$-module or a left $(\mathpzc{Mod}_{\Sigma},\circ,I)$-module.
\begin{notation}
Let $\mathfrak{P}$ be an operad, $X$ a $\mathfrak{P}$-algebra, and $\mathcal{S}$ be a class of maps in $\mathpzc{C}$. 
\begin{enumerate}
\item
Write $\textrm{Sull}_{\mathfrak{P}}(\mathcal{S};X)$ for the class of algebras which can be obtained as a transfinite composition 
$$X_{0}\rightarrow X_{1}\rightarrow\ldots\rightarrow X_{\alpha}\rightarrow\ldots$$
where $X_{0}=X$, and each $X_{\alpha}\rightarrow X_{\alpha+1}$ is a pushout of a map of the form $\mathfrak{P}\circ s$, where $s\in\mathcal{S}$. 
\item
For a functor $L:\mathpzc{C}\rightarrow\mathpzc{D}$, denote by $L(\mathcal{S})^{\mathfrak{P};X}$ the class of maps of the form $L(Y_{n-1})\rightarrow L(Y_{n})$ where $Y_{n-1}\rightarrow Y_{n}$ is a map appearing in the standard filtration of the pushout of an algebra $Y\in\textrm{Sull}_{\mathfrak{P}}(\mathcal{S};X)$ along a map of the form $\mathfrak{P}(s)$, for $s\in\mathcal{S}$.
\end{enumerate}
\end{notation}
The following definition is based on a comment after Remark 6.1.3 on Page 36 \cite{white2018bousfield}, where we are also relaxing the condition that $\mathpzc{C}$ be a monoidal model category.
\begin{defn}\label{defn:weakPalgeb}
Let $\mathfrak{P}$ be a non-symmetric (symmetric) operad, and $\mathcal{S}$ a class of maps in $\mathpzc{C}$ and $L:\mathpzc{C}\rightarrow\mathpzc{D}$ a functor where $\mathpzc{D}$ is also a combinatorial model category. A collection of maps $\mathcal{S}$ in $\mathpzc{C}$ is said to \textbf{satisfy the weak }$\mathfrak{P}$-\textbf{algebra axiom relative to }$(L;X)$ if for any $s\in\mathcal{S}$
\begin{enumerate}
\item
the square below is a homotopy pushout for any $n\ge1$.
 \begin{displaymath}
\vcenter{ \xymatrix{
L(\mathfrak{P}_{X}(n)\otimes Q_{n-1}^{n}(s))\ar[d]\ar[r] & L(X_{n-1})\ar[d] \\
L(\mathfrak{P}_{X}(n)\otimes B^{\otimes n})\ar[r] & L(X_{n}) &
 }}
 \left(
 \vcenter{\xymatrix{
 L(\mathfrak{P}_{X}(n)\otimes_{\Sigma_{n}} Q_{n-1}^{n}(s))\ar[d]\ar[r] & L(X_{n-1})\ar[d]\\
 L(\mathfrak{P}_{X}(n)\otimes_{\Sigma_{n}} B^{\otimes n})\ar[r] & L(X_{n})}
 }
   \right)
 \end{displaymath}
 \item
$\mathpzc{D}$ is weakly $L(\mathcal{S})^{\mathfrak{P};X}$-elementary. 
\end{enumerate}
$\mathcal{S}$ is said to \textbf{satisfy the weak }$\mathfrak{P}$-\textbf{algebra axiom relative to $X$} if it satisfies the weak $\mathfrak{P}$-algebra axiom relative to $(Id_{\mathpzc{C}};X)$ and it is said to \textbf{satisfy the weak }$\mathfrak{P}$-\textbf{algebra axiom} if it satisfies the weak $\mathfrak{P}$-algebra axiom relative to $X$ for all $X\in\mathpzc{Alg}_{\mathfrak{P}}(\mathpzc{C})$.
\end{defn}
White and Yau prove in  \cite{white2018bousfield} Theorem 6.1.1 that when a monoidal model category satisfies a slightly stronger condition that the weak $\mathfrak{P}$-algebra axiom relative to the class of cofibrant $\mathfrak{P}$-algebras (called the $\mathfrak{P}$-algebra axiom in  \cite{white2017model} Definition C.1), then the operad $\mathfrak{P}$ is semi-admissible, that is, there is a transferred \textit{semi-model} structure on $\mathpzc{Alg}_{\mathfrak{P}}(\mathpzc{C})$ (for a semi-model structure, one only requires the left lifting property of trivial cofibrations with cofibrant domain against fibrations, and only maps with cofibrant domain can be factored as a trivial cofibration followed by a fibration). White and Yau actually prove this for coloured operads, and mention that their proof holds under the more general hypotheses of Definition \ref{defn:weakPalgeb} (though we also relax the requirement that $\mathpzc{M}$ be a monoidal model category). For completeness we give the proof here.
%
\begin{thm}
Suppose that $\mathpzc{C}$ is a combinatorial model category, which is also a symmetric monoidal category. Let $\mathfrak{P}$ be either a symmetric or non-symmetric operad in $\mathpzc{C}$ such that the class of acyclic cofibrations in $\mathpzc{C}$ satisfies the weak $\mathfrak{P}$-algebra axiom. Then $\mathfrak{P}$ is admissible.
\end{thm}
\begin{proof}
 Let $\mathcal{S}$ denote the class of acyclic cofibrations. We need to show any map of $\mathfrak{P}$-algebras of the form $X\rightarrow Y$ which is obtained as a transfinite composition of pushouts of maps of the form $\mathfrak{P}(f)$, for $f\in\mathcal{S}$, is an equivalence. Such a map can be written, as a transfinite composition, computed in $\mathpzc{C}$, of maps in $\mathcal{S}^{\mathfrak{P}}$. Since homotopy pushouts of equivalences are equivalences, any map in $\mathcal{S}^{\mathfrak{P}}$ is an equivalence. Moreover homotopy colimits of equivalences are equivalences, and the result is proven.
\end{proof}
One can also show that if the class of acyclic cofibrations satisfies the weak $\mathfrak{P}$-algebra axiom relative to cofibrant $\mathfrak{P}$-algebras then one gets a semi-model structure, as in \cite{white2018bousfield}. One just needs to repeat the proof of the theorem assuming that $X$ is a cofibrant $\mathfrak{P}$-algebra.\newline
\\
We have the following useful trick for categories enriched over $\mathbb{Q}$. It follows immediately from the fact that for any right $\Sigma_{n}$ module $X$ and any map $f$ of left $\Sigma_{n}$-modules, $X\otimes_{\Sigma_{n}}f$ is a retract of $X\otimes f$.
\begin{prop}
Let $\mathpzc{C}$ be a combinatorial model category which is enriched over $\mathbb{Q}$, and is also a symmetric monoidal category. Let $\mathfrak{P}$ be a symmetric operad in $\mathpzc{C}$. Suppose when regarded as a non-symmetric operad the class of acyclic cofibrations in $\mathpzc{C}$ satisfies the weak $\mathfrak{P}$-algebra axiom. Then the class of acyclic cofibrations satisfies the weak $\mathfrak{P}$-algebra axiom when it is regarded as a symmetric operad. 
\end{prop}
\subsection{Equivalences of Model Categories of Algebras}
In this section we establish when Quillen equivalences of model categories with monoidal structures lift to Quillen equivalences modules. Our proofs are modifications of the proofs of \cite{schwede2003equivalences}  Theorem 3.12, Proposition 3.16, and particularly for the case of algebras over operads, \cite{white2019homotopical} Theorem A.
\begin{thm}\label{thm:monoidequiv}
Let $(\mathpzc{C},\otimes_{\mathpzc{C}})$ and $(\mathpzc{D},\otimes_{\mathpzc{D}})$ be monoidal categories, $(\mathpzc{M},\bullet_{\mathpzc{M}})$ a $\mathpzc{C}$-module and $(\mathpzc{M},\bullet_{\mathpzc{N}})$ a $\mathpzc{D}$-module where $\mathpzc{M}$ and $\mathpzc{N}$ are combinatorial model categories. Let
$$\adj{L}{\mathpzc{C}}{\mathpzc{D}}{R}$$
$$\adj{\overline{L}}{\mathpzc{M}}{\mathpzc{N}}{\overline{R}}$$
be a lax monoidal pair of adjunctions such that the bottom adjunction is a Quillen equivalence of model categories. Let $S$ be an associative monoid in $\mathpzc{C}$, $T$ an associative monoid in $\mathpzc{D}$, and $\alpha:S\rightarrow R(T)$ a map of monoids such that 
\begin{enumerate}
\item
$S$ is admissible in $\mathpzc{M}$
\item
$T$ is admissible in $\mathpzc{N}$
\item
For any cofibrant $B$ in ${}_{S}\mathpzc{Mod}(\mathpzc{M})$ the map $\chi_{B}:\mathbb{L}L|B|_{S}\rightarrow |L_{\alpha}B|_{T}$ which is adjoint to the underlying map in $\mathpzc{M}$ of the unit $|B|_{S}\rightarrow R(|L_{\alpha}B|_{T})$ is an equivalence.
\end{enumerate}
Then the adjunction
$$\adj{L_{\alpha}}{{}_{S}\mathpzc{Mod}(\mathpzc{M})}{{}_{T}\mathpzc{Mod}(\mathpzc{M})}{R_{\alpha}}$$
is a Quillen equivalence. 
\end{thm}
\begin{proof}
It is clearly a Quillen adjunction. Let $B$ be a cofibrant $S$-module, and $Y$ a fibrant $T$-module. We need to show that a map $L_{\alpha}B\rightarrow Y$ is an equivalence if and only if $B\rightarrow R_{\alpha}Y$ is an equivalence. Let $C\rightarrow|B|_{S}$ be a cofibrant replacement. Then $LC\rightarrow|L_{\alpha}B|_{T}$ is an equivalence by assumption. Note also that $|Y|_{T}$ is fibrant in $\mathpzc{N}$. Now since 
$$\adj{\overline{L}}{\mathpzc{M}}{\mathpzc{N}}{\overline{R}}$$
is a Quillen equivalence, we have $C\rightarrow R|Y|_{T}\cong |R_{\alpha}Y|_{S}$ is an equivalence if and only if $LC\rightarrow |Y|_{T}$ is an equivalence. By the two-out-of-three property this implies that $|L_{\alpha}B\rightarrow Y|_{T}$ is an equivalence if and only if $|B\rightarrow R_{\alpha}Y|_{S}$ is an equivalence, which completes the proof.
\end{proof}
We immediately get the following.
\begin{cor}\label{cor:strictequiv}
Let $(\mathpzc{C},\otimes_{\mathpzc{C}})$ and $(\mathpzc{D},\otimes_{\mathpzc{D}})$ be monoidal categories, $(\mathpzc{M},\bullet_{\mathpzc{M}})$ a left $\mathpzc{C}$-module and $(\mathpzc{N},\bullet_{\mathpzc{N}})$ a left $\mathpzc{D}$-module which are both combinatorial model categories. Let
$$\adj{L}{\mathpzc{M}}{\mathpzc{N}}{R}$$
$$\adj{\overline{L}}{\mathpzc{M}}{\mathpzc{N}}{\overline{R}}$$
be a lax monoidal pair of adjunctions such that the bottom adjunction is a Quillen equivalence of model categories. Suppose that $\overline{L}$ and $L$ are strictly monoidal. Let $\alpha:S\rightarrow R(T)$  be an isomorphism of  associative monoids in $\mathpzc{C}$, where $S$ and $T$ are admissible. Then the adjunction
$$\adj{L_{\alpha}}{{}_{S}\mathpzc{Mod}(\mathpzc{M})}{{}_{T}\mathpzc{Mod}(\mathpzc{N})}{R_{\alpha}}$$
is a Quillen equivalence.
\end{cor}
\subsubsection{Equivalences of Categories of Alegbras over Operads}
Here we once again specialise to algebras over operads. Let $(\mathpzc{C},\otimes,k)$ be a closed monoidal category which is also a combinatorial model category. Our results here are adaptations of  \cite{white2019homotopical} to slightly more general conditions. Essentially, where in the commutative cubes below White/ Yau's axioms implies that the front and back faces are homotopy pushouts, we shall impose these conditions.
\begin{lem}\label{lem:unitequivop}
Let  $\mathfrak{P}$ be a non-symmetric (resp. symmetric) operad $Y$ be an object of $\textrm{Sull}_{\mathfrak{P}}(\mathcal{S};X)$ where
\begin{enumerate}
\item
$LX\rightarrow L_{\alpha}(X)$ is an equivalence.
\item
The collection of maps $\mathcal{S}$ satisfies the weak $\mathfrak{P}$-algebra axiom relative to $X$, and relative to $(L;X)$.
\item
The collection of maps $\{L(s):s\in\mathcal{S}\}$satisfies the weak $\mathfrak{T}$-algebra axiom relative to $L(X)$.
\item
For any $s:A\rightarrow B\in\mathcal{S}$ any $X'\in\textrm{Sull}_{\mathfrak{P}}(\mathcal{S};X)$, and any $n\ge0$ the maps
$$L(\mathfrak{P}_{X'}(n+1)\otimes Q(s)^{n+1}_{n})\rightarrow\mathfrak{T}_{X'}(n+1)\otimes Q^{n+1}_{n}(L(s))\textrm{ (resp.  }L(\mathfrak{P}_{X'}(n+1)\otimes_{\Sigma_{n+1}} Q(s)^{n+1}_{n})\rightarrow\mathfrak{T}_{X'}(n+1)\otimes Q^{n+1}_{n}(L(s))\textrm{)}$$
$$L(\mathfrak{P}_{X'}(n+1)\otimes B^{\otimes(n+1)})\rightarrow\mathfrak{T}_{X'}(n+1)\otimes L(B)^{\otimes(n+1)}\textrm{ (resp.  }L(\mathfrak{P}_{X'}(n+1)\otimes_{\Sigma_{n+1}} B^{\otimes(n+1)})\rightarrow\mathfrak{T}_{X'}(n+1)\otimes B^{\otimes(n+1)}\textrm{)}$$
are equivalences.
\end{enumerate}
Then the map 
$$LY\rightarrow L_{\alpha}Y$$
is an equivalence.
\end{lem}
\begin{proof}
We first prove that for any map $s:A\rightarrow B$ in $\mathcal{S}$, and any object $X$ such that $LX\rightarrow L_{\alpha}(X)$ is an equivalence, in any pushout diagram
\begin{displaymath}
\xymatrix{
\mathfrak{P}(A)\ar[r]\ar[d]^{\mathfrak{P}(s)} & X\ar[d]^{s'}\\
\mathfrak{P}(B)\ar[r] & X'
}
\end{displaymath}
we have that $LX'\rightarrow L_{\alpha}X'$ is an equivalence.  In this case $L_{\alpha}X'$ is given by the pushout diagram
\begin{displaymath}
\xymatrix{
\mathfrak{T}(A)\ar[r]\ar[d]^{\mathfrak{P}(s)} & L_{\alpha}(X)\ar[d]^{s'}\\
\mathfrak{T}(B)\ar[r] & L_{\alpha}(X')
}
\end{displaymath}
Consider the standard filtrations $\textrm{lim}_{\rightarrow_{n}}X_{n}$ and $\textrm{lim}_{\rightarrow_{n}}\widetilde{X}_{n}$ of $X'$ and $L_{\alpha}X'$ respectively.  $L(X')$ is given by a transfinite composition $\textrm{lim}_{\rightarrow_{n}}L(X_{n})$ where $L(X_{n})\rightarrow L(X_{n+1})$ is given by the pushout
\begin{displaymath}
\xymatrix{
L(\mathfrak{P}_{X}(n+1)\otimes Q(s)_{n}^{n+1}\ar[d]\ar[r] & L(X_{n})\ar[d]\\
L(\mathfrak{P}_{X}(n+1)\otimes B^{n+1})\ar[r]  & L(X_{n+1})\\
}
\end{displaymath}
The map 
$$LX'\rightarrow L_{\alpha}X'$$ is given by the colimit
\begin{displaymath}
\xymatrix{
L(X_{0})\ar[d]\ar[r] & L(X_{1})\ar[d]\ar[r] & \ldots\ar[r] & L(X_{n})\ar[r] \ar[d]& \ldots\\
\widetilde{X}_{0}\ar[r] & \widetilde{X}_{1}\ar[r] & \ldots\ar[r] & \widetilde{X}_{n}\ar[r] & \ldots
}
\end{displaymath}
where $X_{0}=X$ and $\widetilde{X}_{0}=L_{\alpha}X$. We prove by induction that each map $L(X_{n})\rightarrow\widetilde{X}_{n}$ is an equivalence. Then since $\mathpzc{N}$ is both weakly $\{L(s):s\in\mathcal{S}\}^{\mathfrak{T}}$-elementary and weakly $L(\mathcal{S})^{\mathfrak{P}}$-elementary, this would give that $LX'\rightarrow L_{\alpha}X'$ is an equivalence. Note that $L(X_{0})\rightarrow\widetilde{X}_{0}$ is an equivalence by assumption. Also by assumption, in the commutative cube below
\begin{displaymath}
\xymatrix{
&L(\mathfrak{P}_{X}(n+1)\otimes Q(s)_{n}^{n+1})\ar[dl]\ar[rr]\ar[dd] & & L(X_{n})\ar[dl]\ar[dd]\\
\mathfrak{T}_{X}(n+1)\otimes Q(L(s))_{n}^{n+1}\ar[rr]\ar[dd] &  &\widetilde{X}_{n}\ar[dd]\\
& L(\mathfrak{P}_{X}(n+1)\otimes B^{\otimes(n+1)})\ar[rr]\ar[dl] &  & L(X_{n+1})\ar[dl]\\
\mathfrak{T}_{X}(n+1)\otimes L(B)^{\otimes(n+1)}\ar[rr]& &\widetilde{X}_{n+1}
}
\end{displaymath}
the front and back faces are homotopy pushouts. By assumption the back-to-front maps on the left-hand face are equivalences. By the induction hypothesis the top right back-to-front map is an equivalence. Thus the bottom right back-to-front map is also an equivalence.
Now since $L$ and $L_{\alpha}$ both commute with filtered (in fact all) colimits, the forgetful functor from $\mathpzc{Alg}_{\mathfrak{T}}(\mathpzc{N})$ to $\mathpzc{N}$ commutes with filtered colimits, and $\mathpzc{N}$ is both weakly $\{L(s):s\in\mathcal{S}\}^{\mathfrak{T}}$-elementary and weakly $L(\mathcal{S})^{\mathfrak{P}}$-elementary, $L_{\alpha}Y$ is a transfinite composition of maps in the first class, and $LY$ a transfinite composition of maps in the second, the map $LY\rightarrow L_{\alpha}Y$ is an equivalence.
\end{proof}
\begin{cor}
Let $(\mathpzc{C},\otimes_{\mathpzc{C}})$ and $(\mathpzc{D},\otimes_{\mathpzc{D}})$ be (symmetric) monoidal categories, which are both combinatorial
model categories. Let
$$\adj{L}{\mathpzc{C}}{\mathpzc{D}}{R}$$
be a lax monoidal adjunction which is a Quillen equivalence of model categories. Let $\mathfrak{P}$ be an admissible (symmetric) operad in $\mathpzc{C}$, $\mathfrak{T}$ an admissible (symmetric) operad in $\mathpzc{D}$, and $\alpha:\mathfrak{P}\rightarrow R(\mathfrak{T})$ a map of operads. Finally, suppose that there is a generating set $I$ of cofibrations in $\mathpzc{C}$ such that 
\begin{enumerate}
\item
$L$ preserves equivalence between the underlying objects of cofibrant $\mathfrak{P}$-algebras.
\item
The map $\mathbb{L}L(\mathfrak{P}(0))\rightarrow\mathfrak{T}(0)$ is an equivalence.
\item
$I$ satisfies the weak $\mathfrak{P}$-algebra axiom relative to $\mathfrak{P}(0)$ and relative to $(L;\mathfrak{P}(0))$.
\item
$\{L(f):f\in I\}$ satisfies the weak $\mathfrak{T}$-algebra axiom relative to $\mathfrak{T}(L(0))$. 
\item
Whenever $X'$ is a cofibrant $\mathfrak{P}$-algebra, $s\in I$, and $n\ge 0$, the maps
$$L(\mathfrak{P}_{X'}(n+1)\otimes Q(s)^{n+1}_{n})\rightarrow\mathfrak{T}_{X'}(n+1)\otimes Q^{n+1}_{n}(L(s))\textrm{(resp.  }L(\mathfrak{P}_{X'}(n+1)\otimes_{\Sigma_{n+1}} Q(s)^{n+1}_{n})\rightarrow\mathfrak{T}_{X'}(n+1)\otimes Q^{n+1}_{n}(L(s))\textrm{)}$$
$$L(\mathfrak{P}_{X'}(n+1)\otimes B^{\otimes(n+1)})\rightarrow\mathfrak{T}_{X'}(n+1)\otimes L(B)^{\otimes(n+1)}\textrm{(resp.  }L(\mathfrak{P}_{X'}(n+1)\otimes_{\Sigma_{n+1}} B^{\otimes(n+1)})\rightarrow\mathfrak{T}_{X'}(n+1)\otimes B^{\otimes(n+1)}\textrm{)}$$
are equivalences.
\end{enumerate}
Then the adjunction
$$\adj{L_{\alpha}}{\mathpzc{Alg}_{\mathfrak{P}}(\mathpzc{M})}{\mathpzc{Alg}_{\mathfrak{T}}(\mathpzc{N})}{R_{\alpha}}$$
is a Quillen equivalence.
\end{cor}
\begin{rem}
We are implicitly using here the natural transformation $L(Q^{n}_{t}(s))\rightarrow Q^{n}_{t}(L(s))$ coming from the oplax monoidal structure of $L$.
\end{rem}
\begin{proof}
The assumptions and Lemma \ref{lem:unitequivop} imply that for any cofibrant algebra $X$, the map
$$\mathbb{L}L|X|_{\mathfrak{P}}\rightarrow |L_{\alpha}X|_{\mathfrak{T}}$$
is an equivalence. Then the result follows from Theorem \ref{thm:monoidequiv}.
\end{proof}
\subsubsection{Commutative Monoids and Properness}
Consider the commutative operad $\mathfrak{Comm}$. For $X$ a commutative monoid, the $\Sigma$-module $\mathfrak{Comm}_{A}$ is given by  $\mathfrak{Comm}_{X}(n)=X$. In particular for $Z$ a left $\Sigma_{n}$-module, $\mathfrak{Comm}_{X}(n)\otimes_{\Sigma_{n}}Z\cong (Z)_{\Sigma_{n}}$. Thus a class of maps $\mathcal{S}$ satisfies the weak $\mathfrak{Comm}$-axiom relative to $X$ if the square below is a homotopy pushout for any $n\ge1$ and any $0<q<n$.
 \begin{displaymath}
 \xymatrix{
 X\otimes (Q_{n-1}^{n}(s))_{\Sigma_{n}}\ar[d]\ar[r] & X_{n-1}\ar[d]\\
X\otimes (B^{\otimes n})_{\Sigma_{n}}\ar[r] & X_{n}
 }
 \end{displaymath}
 and $\mathpzc{C}$ is $\mathcal{S}^{\mathfrak{Comm};X}$-elementary.
 For $\mathcal{S}$ the class of acyclic cofibrations, this is implied by White's commutative monoid axiom, and for $\mathcal{S}$ the class of cofibrations, this is implied by White's strong commutative monoid axiom \cite{white2017model} Definition 3.1. In \cite{white2017model} Theorem 4.17, White gives conditions under which the model category $\mathpzc{Alg}_{\mathfrak{Comm}}(\mathpzc{C})$ is left proper. Following the strategy of White's proof, let us prove a generalisation. First we will need a definition.
 \begin{defn}
 Let $\mathpzc{C}$ be a symmetric monoidal category equipped which is also a model category. An object $X$ of $\mathpzc{C}$ is said to be $K$-\textbf{flat} if for any equivalence $f$, $X\otimes f$ is also an equivalence.
 \end{defn}
\begin{thm}
Suppose
\begin{enumerate}
\item
$\mathfrak{Comm}$ is admissible.
\item
$\mathpzc{E}$ is left proper.
\item
If $f:X\rightarrow Y$ is an acyclic cofibration, then $(Q^{n+1}_{n}(f))_{\Sigma_{n+1}}$ and $(Y^{\otimes n+1})_{\Sigma_{n+1}}$ are $K$-flat for all $n\ge 0$.
\item
The class $\mathcal{S}$ of cofibrations in $\mathpzc{C}$ satisfies the weak $\mathfrak{Comm}$-algebra axiom.
\end{enumerate}
Then the category $\mathpzc{Alg}_{\mathfrak{Comm}}(\mathpzc{E})$ is left proper.
\end{thm}
\begin{proof}
It suffices to prove that for $g:X\rightarrow Y$ a cofibration in $\mathpzc{E}$, the map $S(g):S(X)\rightarrow S(Y)$ is a $h$-cofibration. Cosider a commutative diagram
\begin{displaymath}
\xymatrix{
S(X)\ar[d]^{S(g)}\ar[r] & A\ar[d]^{g'}\ar[r]^{f} & B\ar[d]\\
S(Y)\ar[r] & A'\ar[r]^{f'} & B'
}
\end{displaymath}
in which both squares are pushouts and $f$ is an equivalence. We need to show that $f'$ is an equivalence. The map $f'$ is given by the colimit of the maps on standard filtrations
\begin{displaymath}
\xymatrix{
A_{0}\ar[d]^{f}\ar[r] & A_{1}\ar[d]^{f_{1}}\ar[r] & \ldots\ar[r] & A_{n}\ar[d]^{f_{n}}\ar[r] & \ldots\\
B_{0}\ar[r] & B_{1}\ar[r] & \ldots\ar[r] & B_{n}\ar[r] & \ldots
}
\end{displaymath}
where $A_{0}=A$, $B_{0}=B$, and $A_{n+1}$, $B_{n+1}$ fit into commutative cubes.
\begin{displaymath}
\xymatrix{
& A\otimes (Q^{n+1}_{n}(g))_{\Sigma_{n+1}}\ar[dl]\ar[rr]\ar[dd]& & A_{n}\ar[dl]\ar[dd]\\
B\otimes (Q^{n+1}_{n}(g))_{\Sigma_{n+1}}\ar[rr]\ar[dd]&  &B_{n}\ar[dd]\\
&A\otimes(Y^{\otimes n+1})_{\Sigma_{n+1}}\ar[rr]\ar[dl] &  & A_{n+1}\ar[dl]\\
B\otimes(Y^{\otimes n+1})_{\Sigma_{n+1}}\ar[rr]& &B_{n+1}
}
\end{displaymath}
Since $(Q^{n+1}_{n}(g))_{\Sigma_{n+1}}$ and $(Y^{\otimes n+1})_{\Sigma_{n+1}}$ are $K$-flat, the horizontal maps on the left face are equivalences. The map $A_{n}\rightarrow B_{n}$ is an equivalence by the inductive step. The back-left and front-left maps are left proper by assumption, and the front and back faces are pushouts, and therefore homotopy pushouts. Hence $A_{n+1}\rightarrow B_{n+1}$ is an equivalence. 
\end{proof}
Suppose for a moment that $\mathpzc{C}$ be a pointed model category. The following result will be important for us.
\begin{lem}
Let $\mathpzc{C}$ is a pointed model category. Suppose that the category $\mathpzc{Alg}_{\mathfrak{Comm}}(\mathpzc{C})$ is equipped with the tranferred model structure (resp. is left proper). Then the category $\mathpzc{Alg}_{\mathfrak{Comm}}^{nu}(\mathpzc{C})$ of non-unital commutative monoids is equipped with the trasferred model structure (resp. is left proper).
\end{lem}
\begin{proof}
Let $X_{0}\rightarrow \textrm{lim}_{\rightarrow_{\alpha<\lambda}}X_{\alpha}$ be a transfinite composition of pushouts in $\mathpzc{Alg}_{\mathfrak{Comm}}^{nu}(\mathpzc{C})$ of maps of the form $S^{nu}(f)$, with $f$ a trivial cofibration. Then $k\coprod X_{0}\rightarrow \textrm{lim}_{\rightarrow_{\alpha<\lambda}}k\coprod X_{\alpha}$ is a transfinite composition of pushouts in $\mathpzc{Alg}_{\mathfrak{Comm}}(\mathpzc{C})$ of maps of the form $S(f)$, with $f$ a trivial cofibration. Thus $k\coprod X_{0}\rightarrow  \textrm{lim}_{\rightarrow_{\alpha<\lambda}}k\coprod X_{\alpha}$ is an equivalence. As a retract of this map, $X_{0}\rightarrow \textrm{lim}_{\rightarrow_{\alpha<\lambda}}X_{\alpha}$ is an equivalence.\newline
\\
Suppose now $\mathpzc{Alg}_{\mathfrak{Comm}}(\mathpzc{C})$ is left proper. Consider a pushout.
\begin{displaymath}
\xymatrix{
S^{nu}(X)\ar[d]^{S^{nu}(g)}\ar[r] & A\ar[d]^{g'}\ar[r]^{f} & B\ar[d]\\
S^{nu}(Y)\ar[r] & A'\ar[r]^{f'} & B'
}
\end{displaymath}
in which both squares are pushouts and $f$ is an equivalence.
Then 
\begin{displaymath}
\xymatrix{
S(X)\ar[d]^{S(g)}\ar[r] &k\coprod A\ar[d]^{k\coprod g'}\ar[r]^{k\coprod f} & k\coprod B\ar[d]\\
S(Y)\ar[r] & k\coprod A'\ar[r]^{k\coprod f'} & k\coprod B'
}
\end{displaymath}
is a diagram in $\mathpzc{Alg}_{\mathfrak{Comm}}(\mathpzc{C})$ with both squares pushouts and $k\coprod f$ an equivalence. Thus $k\coprod f'$ is an equivalence. So, as a retract of $k\coprod f'$, $f'$ is an equivalence.
\end{proof}
\section{Homotopical Algebra in Exact Categories}
\textcolor{red}{
In this section we let $(\mathpzc{E},\otimes,k)$ be a complete and cocomplete symmetric monoidal exact category which is quasi-elementary. We do not assume that projectives are flat, nor that the tensor product of two projectives is projective. Then the projective model structure exists on $Ch_{\ge0}(\mathpzc{E})$, $Ch(\mathpzc{E})$, $Ch_{\ge0}(\mathpzc{E})$, $\textbf{s}\mathpzc{E}$, and $\textbf{cs}\mathpzc{E}$. These are not, in general, monoidal model categories. However the cofibrations in all of these categories are degree-wise/ level-wise split, and hence $h$-cofibrations (and therefore left proper). Moreover the tensor product of a trivially cofibrant object with any object is still trivial. Thus if $f$ is an acyclic cofibration, then any pushout of a tensor product of $f$ with any object will still be a weak equivalence. Finally, if $\mathpzc{E}$ is weakly $\textbf{AdMon}$-elementary as an exact category, then they are all weakly $\textbf{AdMon}$-elementary as model categories. Thus we get the following.
\begin{thm}\label{thm:adopch}
Let $\mathpzc{M}$ be one of the additive model categories $Ch_{\ge0}(\mathpzc{E})$, $Ch(\mathpzc{E})$, $Ch_{\ge0}(\mathpzc{E})$, $\textbf{s}\mathpzc{E}$, $\textbf{cs}\mathpzc{E}$, and let $\mathfrak{P}$ be a non-symmetric operad in $\mathpzc{M}$. Then $\mathfrak{P}$ is admissible. If $\mathpzc{M}$ is enriched over $\mathbb{Q}$ then any symmetric operad is also admissible. 
\end{thm}
}
%
%
%

\subsection{The Monoidal Dold-Kan Correspondence}
Following \cite{schwede2003equivalences} closely, in this section we show that the Dold-Kan equivalence is a monoidal Quillen equivalence, and deduce equivalences of categories of algebras over operads. We will also generalise the work of \cite{castiglioni2004cosimplicial} on the cosimplicial Dold-Kan correspondence.
Let $\mathpzc{E}$ be an idempotent complete additive category. Suppose that it is equipped with an additive monoidal structure $\otimes:\mathpzc{E}\times\mathpzc{E}\rightarrow\mathpzc{E}$. As explained in the introduction $Ch_{\ge0}(\mathpzc{E})$ has a natural monoidal structure. $\textbf{s}\mathpzc{E}$ also has a monoidal structure defined level-wise: if $X_{\bullet}$ and $Y_{\bullet}$ are objects of $\textbf{s}\mathpzc{E}$ then $X_{\bullet}\otimes Y_{\bullet}$ is the simplicial object which in level $n$ is $X_{n}\otimes Y_{n}$. In the Dold-Kan adjunction
$$\adj{\Gamma}{Ch_{\ge0}(\mathpzc{E})}{\textbf{s}\mathpzc{E}}{N}$$
both the left and right adjoints are lax and oplax monoidal. Indeed it suffices to show that the right adjoint is both lax and oplax monoidal, then $\Gamma$ inherits the same structures via doctrinal adjunction. For $A$ and $B$ objects of $\textbf{s}\mathpzc{E}$, define the \textbf{Alexander-Whitney map}
$$\Delta_{A,B}:C(A\otimes B)\rightarrow C(A)\otimes C(B)$$
on the factor $A_{n}\otimes B_{n}$ by
$$\bigoplus_{p+q=n}\widetilde{d^{p}}\otimes d_{0}^{q}$$
where $\widetilde{d^{p}}$ is induced by the map 
$$[p]\rightarrow [p+q], i\mapsto i$$
 and $d_{0}^{q}$ is induced by the map 
 $$[q]\rightarrow[p+q], i\mapsto i+p$$
 Define the \textbf{Eilenberg-Zilber map}
 $$\nabla_{A,B}:C(A)\otimes C(B)\rightarrow C(A\otimes B)$$
 on the factor $A_{p}\otimes B_{q}$ by
 $$\sum_{(\mu,\nu)}\textrm{sign}(\mu,\nu)s_{\nu}\otimes s_{\mu}$$
 where the sum is over all $(p,q)$-shuffles $(\mu,\nu)=(\mu_{1},\ldots,\mu_{p},\nu_{1},\ldots,\nu_{q})$, and 
 $$s_{\mu}\defeq s_{\mu_{p}-1}\circ \ldots\circ s_{\mu_{2}-1}\circ s_{\mu_{1}-1}$$
and similarly for $s_{\nu}$. Both $\Delta_{A,B}$ and $\nabla_{A,B}$ are functorial in $A$ and $B$, and restrict to maps
$$\overline{\Delta}_{A,B}:N(A\otimes B)\rightarrow N(A)\otimes N(B)$$
 $$\overline{\nabla}_{A,B}:N(A)\otimes N(B)\rightarrow N(A\otimes B)$$
 \begin{rem}\label{rem:laxsymm}
 It is shown in e.g. \cite{Weibel} 8.5.4 that for simplicial abelian groups $\overline{\nabla}$ is lax symmetric monoidal, and the proof works for arbitrary additive categories with kernels.
 \end{rem}
\begin{thm}[\cite{kerodon}]
Let $A$ and $B$ be simplicial objects.
\begin{enumerate}
\item
$\overline{\Delta}_{A,B}\circ\overline{\nabla}_{A,B}=Id_{N(A)\otimes N(B)}$
\item
 The Eilenberg-Zilber map 
$$\overline{\Delta}_{A,B}:N(A\otimes B)\rightarrow N(A)\otimes N(B)$$
is a homotopy equivalence.
\end{enumerate}
\end{thm}
\begin{proof}
\begin{enumerate}
\item
This can be proven exactly as in \cite{kerodon} Section 2.5.8. 
\item
This can be proven  as in \cite{kerodon} Section 2.5.8, with some small adjustments. It suffices to show that for any complex $M$, the map
\begin{displaymath}
\xymatrix{
N(A)\otimes M\cong N(A)\otimes N(\Gamma(M))\ar[rr]^{\;\;\;\;\overline{\nabla}_{A,\Gamma(M)}}& & N(A\otimes\Gamma(M))
}
\end{displaymath}
is a homotopy equivalence. $M$ can be written as an $(\aleph_{0};\textbf{SplitMon})$-colimit of bounded above complexes $M^{n}$, where $M^{n}_{k}=M_{k}$ for $k\le n$, and $M^{n}_{k}=0$ for $k>n$. Thus we may assume that $M$ is bounded above, and prove the claim by induction on the length of $M$. Let $M$ be of length $n$. Denote by $M_{n}[n]$ the complex given by $M_{n}$ concentrated in degree $n$, and let $M'$ denote the kernel of the map $M\rightarrow M_{n}[n]$. We get degree-wise split exact sequence of complexes and simplicial objects respectively
$$0\rightarrow M'\rightarrow M\rightarrow M_{n}[n]\rightarrow 0$$
$$0\rightarrow \Gamma(M')\rightarrow \Gamma(M)\rightarrow \Gamma(M_{n}[n])\rightarrow 0$$
Thus tensoring the first sequence with a complex and the second with a simplicial object will still result in a degree-wise split exact sequence. In particular get a commutative diagram of degree-wise split exact sequences
\begin{displaymath}
\xymatrix{
0\ar[r] & N(A)\otimes M'\ar[d]\ar[r] & N(A)\otimes M\ar[d]\ar[r] & N(A)\otimes M_{n}[n]\ar[d]\ar[r] & 0\\
0\ar[r] & N(A\otimes\Gamma(M'))\ar[r] & N(A\otimes\Gamma(M))\ar[r] & N(A\otimes\Gamma(M_{n}[n]))\ar[r] & 0
}
\end{displaymath}
If both the first and last vertical maps are homotopy equivalences then the middle one will be as well. It follows by induction that we may assume that $M$ is of the form $Y[n]\cong N(\mathbb{Z}[\Delta^{n})]\otimes Y$ for some $n\ge0$ and some object $Y$ of $\mathpzc{E}$. By a similar inductive argument, we may also assume that $A$ is of the form $N(\mathbb{Z}[\Delta^{m}])\otimes X$ for some $m\ge0$ and some object $X$ of $\mathpzc{E}$. Thus we reduce to showing that the map $\overline{\nabla}_{N(\mathbb{Z}[\Delta^{m}])\otimes X,N(\mathbb{Z}[\Delta^{n}])\otimes Y}$, which is given by the composition $$(N(\mathbb{Z}[\Delta^{m}])\otimes X)\otimes (N(\mathbb{Z}[\Delta^{n}])\otimes Y)\cong N(\mathbb{Z}[\Delta^{m}])\otimes N(\mathbb{Z}[\Delta^{n}])\otimes X\otimes Y\rightarrow N(\mathbb{Z}[\Delta^{m}\otimes\Delta^{n}])\otimes X\otimes Y$$
is a homotopy equivalence. By the classical result for abelian groups, e.g. \cite{kerodon} Theorem 2.5.7.14, the map 
$$N(\mathbb{Z}[\Delta^{m}])\otimes N(\mathbb{Z}[\Delta^{n}])\rightarrow N(\mathbb{Z}[\Delta^{m}\otimes\Delta^{n}])$$
is a homotopy equivalence. Thus 
$$(N(\mathbb{Z}[\Delta^{m}])\otimes X)\otimes (N(\mathbb{Z}[\Delta^{n}])\otimes Y)\cong N(\mathbb{Z}[\Delta^{m}])\otimes N(\mathbb{Z}[\Delta^{n}])\otimes X\otimes Y\rightarrow N(\mathbb{Z}[\Delta^{m}\otimes\Delta^{n}])\otimes X\otimes Y$$
is also a homotopy equivalence
\end{enumerate}
\end{proof}
Recall the oplax structure inherited by the functor $\Gamma$, which we denote by $\nabla^{c}$, is defined on complexes $X$ and $Y$ by the composition
\begin{displaymath}
\xymatrix{
\Gamma(X\otimes Y)\ar[r]&\Gamma(N\Gamma(X)\otimes N\Gamma(Y))\ar[rr]^{\Gamma(\nabla_{\Gamma(X),\Gamma(Y)}) }& & \Gamma N(\Gamma(X)\otimes\Gamma(Y))\ar[r] & \Gamma(X)\otimes\Gamma(Y)
}
\end{displaymath}
Since all maps in the composite are homotopy equivalences, we get the following.
\begin{cor}
Let $X$ and $Y$ be objects in $Ch_{\ge0}(\mathpzc{E})$. Then the map
$$\nabla^{c}_{X,Y}:\Gamma(X\otimes Y)\rightarrow \Gamma(X)\otimes\Gamma(Y)$$
is a homotopy equivalence. 
\end{cor}
\subsubsection{The Cosimplicial Monoidal Dold-Kan equivalence}
Let $\mathpzc{E}$ be an cocomplete additive category equipped with a closed monoidal structure, and let $A,B\in Ch_{\le0}(\mathpzc{E})$. Define 
$$\nu:QA\otimes QB\rightarrow Q(A\otimes B)$$
on $A_{n}\otimes T^{i}(V^{n})\otimes B_{m}\otimes T^{j}(V^{m}) $ 
by
$$\nu=Id_{A_{n}}\otimes Id_{B_{m}}\otimes\mu+(-1)^{n}Id_{A_{n}}\otimes d_{B}\otimes\mu\circ(\theta\otimes Id)$$
The following is proven formally as in \cite{castiglioni2004cosimplicial}, Theorem 5.2.
\begin{lem}
$\nu$ is a natural isomorphism, and makes $Q:Ch_{\le0}(\mathpzc{E})\rightarrow\textbf{cs}\mathpzc{E}$ a strong monoidal functor.
\end{lem}
\subsubsection{Dold-Kan Equivalences for Algebras}
The results in this section are obtained in \cite{white2019homotopical} for the category of $k$-modules where $k$ is a unital commutative ring. The Dold-Kan equivalence for algebras was of course famously proved in \cite{schwede2003equivalences} for the category of abelian groups, and the dual Dold-Kan correspondence for abelian groups in \cite{castiglioni2004cosimplicial}. 
Let $(\mathpzc{E},\otimes,\underline{Hom},k)$ be a complete and cocomplete closed symmetric monoidal $\textbf{AdMon}$-elementary exact category. We suppose that $Ch_{\ge0}(\mathpzc{E})$ and $\textbf{s}\mathpzc{E}$ are equipped with the projective model structures. Let $\mathfrak{T}$ be an admissible non-symmetric operad in $\textbf{s}\mathpzc{E}$, $\mathfrak{P}$ an admissible non-symmetric operad in $Ch_{\ge0}(\mathpzc{E})$, and $\alpha:\mathfrak{P}\rightarrow N(\mathfrak{T})$ a homotopy equivalence. Then we have the following.
\begin{thm}[The Simplicial Dold-Kan Equivalence  for Algebras]\label{thm:simpDKA}
The adjunction
$$\adj{\Gamma_{\alpha}}{\mathpzc{Alg}_{\mathfrak{P}}(Ch_{\ge0}(\mathpzc{E}))}{\mathpzc{Alg}_{\mathfrak{T}}(\textbf{s}\mathpzc{E})}{N_{\alpha}}$$
is a Quillen equivalence. 
\end{thm}
Now suppose that  $\mathpzc{E}$ is enriched over $\mathbb{Q}$, let $\mathfrak{T}$ be an admissible symmetric operad in $\textbf{s}\mathpzc{E}$, $\mathfrak{P}$ an admissible symmetric operad in $Ch_{\ge0}(\mathpzc{E})$, and $\alpha:\mathfrak{P}\rightarrow N(\mathfrak{T})$ a homotopy equivalence of symmetric operads.  The maps
$$L(\mathfrak{P}_{X'}(n+1)\otimes Q(s)^{n+1}_{n})\rightarrow\mathfrak{T}_{X'}(n+1)\otimes Q^{n+1}_{n}(L(s))\;\; L(\mathfrak{P}_{X'}(n+1)\otimes B^{\otimes(n+1)}\rightarrow\mathfrak{T}_{X'}(n+1)\otimes L(B)^{\otimes(n+1)}$$
are homotopy equivalences.
Since $\mathpzc{E}$ is enriched over $\mathbb{Q}$, for any $n$ and any $\mathfrak{P}$-algebra $X$, the maps
$$L(\mathfrak{P}_{X'}(n+1)\otimes_{\Sigma_{n+1}} Q(s)^{n+1}_{n})\rightarrow\mathfrak{T}_{X'}(n+1)\otimes Q^{n+1}_{n}(L(s))\;\; L(\mathfrak{P}_{X'}(n+1)\otimes_{\Sigma_{n+1}} B^{\otimes(n+1)}\rightarrow\mathfrak{T}_{X'}(n+1)\otimes L(B)^{\otimes(n+1)}$$
are retracts of the first pair of maps, and are therefore also homotopy equivalences. Thus we get the following
\begin{thm}[The Simplicial Dold-Kan Equivalence  for Symmetric Algebras]\label{thm:simpDKA}
The adjunction
$$\adj{\Gamma_{\alpha}}{\mathpzc{Alg}_{\mathfrak{P}}(Ch_{\ge0}(\mathpzc{E}))}{\mathpzc{Alg}_{\mathfrak{T}}(\textbf{s}\mathpzc{E})}{N_{\alpha}}$$
is a Quillen equivalence for $\mathfrak{P}$ and $\mathfrak{T}$ symmetric operads as above.
\end{thm}
Just as for $\Sigma$-cofibrant (coloured) operads in  \cite{white2019homotopical} Section 4.3, in general when $\mathpzc{E}$ is not enriched over $\mathbb{Q}$, one can work with $\Sigma$-K-flat symmetric operads, though we will not go into this here.

$Ch_{\le0}(\mathpzc{E})$ is equipped with a combinatorial model structure. Again we assume that homotopy equivalences are weak equivalences.  We equip $\textbf{cs}\mathpzc{E}$ with the model structure transferred along the dual Dold-Kan equivalence. The functor $Q$ is strongly monoidal. Thus by Corollary \ref{cor:strictequiv} we get the following.
\begin{thm}\label{thm:simpCDKA}
Let $\mathfrak{P}$ be an admissible non-symmetric operad in $Ch_{\le0}(\mathpzc{E})$, $\mathfrak{T}$ be an admissible non-symmetric operad in $\textbf{cs}(\mathpzc{E})$, and $\alpha:\mathfrak{P}\rightarrow H(\mathfrak{T})$ a homotopy equivalence of operads. Then the adjunction
$$\adj{Q_{\alpha}}{\mathpzc{Alg}_{\mathfrak{P}}(Ch_{\le0})(\mathpzc{E})}{\mathpzc{Alg}_{\mathfrak{T}}(\textbf{cs}\mathpzc{E})}{H_{\alpha}}$$
is a Quillen equivalence. 
\end{thm}
\subsection{Homotopical Algebra Contexts}
Before concluding let us make a connection with geometry. Recall that in \cite{toen2004homotopical} To{\"e}n and Vezzosi introduce an abstract categorical framework in which one can `do' homotopical algebra, namely a homotopical algebra context. Let us recall the truncated definition (for the category $\mathsf{C}_{0}$ in \cite{toen2004homotopical} we always take $\mathsf{C}=\mathsf{C}_{0}$).
\begin{defn}\label{Defn:HA context}
Let $\mathpzc{M}$ be a combinatorial symmetric monoidal model category. We say that $\mathpzc{M}$ is an \textbf{homotopical algebra context} (or HA context) if for any $A\in\mathpzc{Alg}_{\mathfrak{Comm}}(\mathpzc{M})$.
\begin{enumerate}
\item
The model category $\mathpzc{M}$ is proper, pointed and for any two objects $X$ and $Y$ in $\mathpzc{M}$ the natural morphisms
$$QX\coprod QY\rightarrow X\coprod Y\rightarrow RX\times RY$$
are equivalences.
\item
$Ho(\mathpzc{M})$ is an additive category.
\item
With the transferred model structure and monoidal structure $-\otimes_{A}-$, the category ${}_{A}\mathpzc{Mod}$ is a combinatorial, proper, symmetric monoidal model category.
\item
For any cofibrant object $M\in{}_{A}\mathpzc{Mod}$ the functor
$$-\otimes_{A}M:{}_{A}\mathpzc{Mod}\rightarrow{}_{A}\mathpzc{Mod}$$
preserves equivalences.
\item
With the transferred model structures $\mathpzc{Alg}_{\mathfrak{Comm}}({}_{A}\mathpzc{Mod})$ and $\mathpzc{Alg}_{\mathfrak{Comm}_{nu}}({}_{A}\mathpzc{Mod})$ are combinatorial proper model categories.
\item
If $B$ is cofibrant in $\mathpzc{Alg}_{\mathfrak{Comm}}({}_{A}\mathpzc{Mod})$ then the functor
$$B\otimes_{A}-:{}_{A}\mathpzc{Mod}\rightarrow{}_{B}\mathpzc{Mod}$$
preserves equivalences.
\end{enumerate}
\end{defn}
\begin{thm}\label{HAcont}
Let $(\mathpzc{E},\otimes,\underline{\textrm{Hom}},k)$ be a locally presentable closed projectively monoidal exact category which is $\textbf{AdMon}$-elementary. Then $Ch_{\ge0}(\mathpzc{E})$ and $Ch(\mathpzc{E})$ are homotopical algebra context. \end{thm}
\begin{proof}
The natural maps
$$QX\coprod QY\rightarrow X\coprod Y\rightarrow RX\times RY$$
are clearly equivalences. All that remains to prove is the final property. Now if $B$ is a cofibrant $A$-algebra then it is a retract of the free $A$-algebra on a cofibrant $A$-module. But the free $A$-algebra on a cofibrant $A$-module is cofibrant as an $A$-module. Hence $B\otimes_{A}(-)$ preserves equivalences by 4). 
\end{proof}

\appendix

\chapter{Model Categories}\label{modelapp}
\section{Weak factorisation Systems and Model Structures}
Here we briefly recall the definition of a model structure by means of weak factorisation systems. Details can be found in \cite{Riehl}.

\begin{defn}
Let $\mathcal{C}$ be a class of morphisms in a category $\mathpzc{M}$. A morphism $f$ in $\mathpzc{M}$ is said to have the \textbf{left lifting property} with respect to $\mathcal{C}$ if in any diagram of the form
\begin{displaymath}
\xymatrix{
A\ar[d]^{f}\ar[r] & C\ar[d]^{c}\\
B\ar[r] & D
}
\end{displaymath}
with $c\in C$, there exists a morphism $h:B\rightarrow C$ such that the following diagram commutes
\begin{displaymath}
\xymatrix{
A\ar[d]^{f}\ar[r] & C\ar[d]^{c}\\
B\ar[r]\ar@{-->}[ur]^{h} & D
}
\end{displaymath}
We denote the class of all morphisms which have the left lifting property with respect to $\mathcal{C}$ by $\;^{\llp}\mathcal{C}$. Dually one defines the morphisms having the right lifting property with respect to $\mathcal{C}$. The class of all such morphisms is denoted $\mathcal{C}^{\llp}$.
\end{defn}

The following is straightforward

\begin{prop}\label{liftingwf}
Let $\mathcal{C}$ be a class of morphisms in a category $\mathpzc{M}$. Then $\;^{\llp}\mathcal{C}$ is closed under retracts, pushouts and transfinite composition (whenever they exist). 
\end{prop}

\begin{proof}
See \cite{Riehl} Lemma 11.1.4.
\end{proof}

\begin{defn}
A \textbf{weak factorisation system} on a category $\mathcal{C}$ is a pair $(\mathcal{L},\mathcal{R})$ such that
\begin{enumerate}
\item
Any map in $\mathcal{C}$ can be factored as a map in $\mathcal{L}$ followed by a map in $\mathcal{R}$.
\item
$\mathcal{L}=\;^{\llp}\mathcal{R}$ and $\mathcal{R}=\mathcal{L}^{\llp}$.
\end{enumerate}
A weak factorisation system is said to be \textbf{functorial} if the factorisation in $(1)$ can be made functorial.
\end{defn}

We can now give a definition of the notion of a model structure in terms of weak factorisation systems.

\begin{defn}
A \textbf{model structure} on a category $\mathpzc{M}$ is a collection of three wide subcategories $(\mathcal{C},\mathcal{F},\mathcal{W})$ such that
\begin{enumerate}
\item
The class $\mathcal{W}$ satisfies the $2$-out-of-$6$ property (see \cite{Riehl}).
 \item
 Both $(\mathcal{C}\cap\mathcal{W},\mathcal{F})$ and $(\mathcal{C},\mathcal{W}\cap\mathcal{F})$ are weak factorisation systems.
 \end{enumerate}
\end{defn}

We do not assume completeness or cocompleteness of $\mathpzc{M}$.

\begin{defn}
A model structure on a category $\mathpzc{M}$ is said to be \textbf{functorial} if the factorisation systems are functorial.
\end{defn}

\begin{defn}
A \textbf{(functorial) model category} is a category together with a (functorial) model structure.
\end{defn}

\section{Cofibrant Generation}

We state here our conventions regarding cofibrant generation. These are largely slightly modified definitions from \cite{hoveybook} Chapter 2.

\begin{defn}
If $I$ is a collection of maps in category $\mathpzc{C}$, we denote by $cell(I)$ the collection of transfinite compositions of pushouts of maps in $I$. 
\end{defn}

\begin{defn}
If $I$ is a collection of maps in category $\mathpzc{C}$, we say that $I$ satisfies the small object argument if any transfinite composition of pushouts of morphisms in $I$ exists, and any morphism $f$ has a factorisation $f=h\circ g$ where $g\in cell(I)$ and $h\in I^{\llp}$. 
\end{defn}

\begin{defn}
Let $\mathpzc{C}$ be a category. A weak factorisation system $(\mathcal{L},\mathcal{R})$ on $\mathpzc{C}$ is said to be \textbf{cofibrantly small} if there is a set $I$ of maps in $\mathcal{L}$ such that $\mathcal{R}=I^{\llp}$. $I$ is called a set of \textbf{generating morphisms}. If in addition $I$ admits the small object argument then the weak factorisation system is said to be \textbf{cofibrantly generated}. If $\mathpzc{C}$ is locally presentable and cofibrantly generated, then the weak factorisation system is said to be \textbf{combinatorial}. A model category $(\mathcal{C},\mathcal{W},\mathcal{F})$ is said to be cofibrantly small/ cofibrantly generated/ combinatorial if both the weak factorisation systems $(\mathcal{C},\mathcal{F}\cap\mathcal{W})$ and $(\mathcal{C}\cap\mathcal{W},\mathcal{F})$ are cofibrantly small/ cofibrantly generated/ combinatorial.
\end{defn}

\begin{rem}
A cofibrantly small weak factorisation system (resp. model structure) on a locally presentable category is automatically combinatorial. 
\end{rem}

\section{Homotopy Colimits}
Let $\mathpzc{C}$ be a combinatorial model category, and $\mathcal{I}$ a small category. The category of functors $\mathpzc{Fun}(\mathcal{I},\mathpzc{C})$ can be equipped with the projective model structure \cite{Riehl} Theorem 12.3.2. The functor $\textrm{colim}_{\mathcal{I}}:\mathpzc{Fun}(\mathcal{I},\mathpzc{C})\rightarrow\mathpzc{C}$ is left Quillen (\cite{Riehl}  Corollary 5.1.3). For a functor $F\in\mathpzc{Fun}(\mathcal{I},\mathpzc{C})$, the \textbf{homotopy colimit} of $F$, denoted $\textrm{hocolim}_{\mathcal{I}}$ (or $\textrm{holim}_{\rightarrow_{\mathcal{I}}}$ for filtered $\mathcal{I}$) is the derived functor of $\textrm{colim}_{\mathcal{I}}$ applied to $F$. One defines homotopy limits dually, and denotes them by $\textrm{holim}_{\mathcal{I}}$,  (or $\textrm{holim}_{\leftarrow_{\mathcal{I}}}$ for cofiltered $\mathcal{I}$). 
\subsection{Homotopy Pushouts and Properness}
Properness is a useful condition to have on model categories and on maps in model categories.
\begin{defn}[\cite{white2014monoidal} Definition 8.1.]\label{defn:hcof}
\begin{enumerate}
\item
A map $f:X\rightarrow Y$ in a model category $\mathpzc{C}$
is said to be a $h$-\textbf{cofibration} if whenever 
\begin{displaymath}
\xymatrix{
X\ar[d]^{f}\ar[r]^{g} & A\ar[d]\ar[r]^{w}.& B\ar[d]\\
Y\ar[r]^{g'} & A'\ar[r]^{w'} & B'
}
\end{displaymath}
is a commutative diagram in which both squares are pushouts, and $w$ a weak equivalence, then $w'$ is a weak equivalence. 
\item
A model category is said to be \textbf{left proper} if cofibrations are $h$-cofibrations.
\end{enumerate}
One defines $h$-\textbf{fibrations} of maps and \textbf{right properness} of model categories dually.
\end{defn}
\begin{prop}[\cite{batanin2013homotopy} Proposition 1.6]
A map $f:X\rightarrow Y$ in a left proper model category $\mathpzc{C}$ is a $h$-cofibration if and only if any pushout diagram
\begin{displaymath}
\xymatrix{
X\ar[d]^{f}\ar[r] & A\ar[d]\\
Y\ar[r] & P
}
\end{displaymath}
is a homotopy pushout.
\end{prop}
This motivates the following definition.
\begin{defn}\label{defn:lprop}
Let $\mathpzc{C}$ be a model category and $\mathcal{P}\subset\mathpzc{C}$ a subcategory $\mathpzc{C}$. A map $f:X\rightarrow Y$ in $\mathpzc{C}$ is said to be \textbf{left proper relative to }$\mathcal{P}$ if any pushout diagram
\begin{displaymath}
\xymatrix{
X\ar[d]^{f}\ar[r] & A\ar[d]\\
Y\ar[r] & P
}
\end{displaymath}
with $A\in\mathcal{P}$ is a homotopy pushout. If $\mathcal{P}=\mathpzc{C}$, then $f$ is said to be \textbf{left proper}. One defines relative right-properness dually.
\end{defn}
In particular, in a proper model category $h$-cofibrations are left proper. Let us make some straightforward observations.
\begin{prop}\label{prop:pushoutstableftprop}
Let $\mathcal{S}$ be a pushout-stable class of maps such that any pushout of a weak equivalence along a map in $\mathcal{S}$ is still a weak equivalence. Then any map in $\mathcal{S}$ is a $h$-cofibration.
\end{prop}
The following is just a consequence of the $2$-out-of-$3$ property of weak equivalences.
\begin{prop}\label{prop:pushoutstabweq}
Let $f$ be a weak equivalence such that any pushout of $f$ is a weak equivalence. Then $f$ is left proper, and a $h$-cofibration.
\end{prop}
\subsubsection{Stable Model Categories}
Let $\mathpzc{C}$ be a combinatorial pointed model category with initial-terminal object $0$. The \textbf{suspension functor} $\Sigma:\textrm{Ho}(\mathpzc{C})\rightarrow\textrm{Ho}(\mathpzc{C})$ assigns to an object $X$ the homotopy pushout $0\coprod^{\mathbb{L}}_{X}0$.
\begin{defn}[\cite{hoveybook}, Definition 7.1.1]\label{defn:stable}
A pointed model category $\mathpzc{C}$ is said to be \textbf{stable} if the functor $\Sigma$ is an auti-equivalence.
\end{defn}

\subsection{Homotopy Transfinite Compositions}
\begin{defn}\label{defn:weaklyelmodel}
Let $\mathcal{S}$ be a collection of maps in a co-complete model category $\mathpzc{C}$. $\mathpzc{C}$ is said to be \textbf{weakly }$\mathcal{S}$-\textbf{elementary} if for any ordinal $\lambda$, and any functor $X\in\mathpzc{Fun}_{\mathcal{S}}(\lambda,\mathpzc{C})^{cocont}$, the colimit $\textrm{colim}_{\alpha<\lambda}X_{\alpha}$ is a homotopy colimit. 
\end{defn}
\begin{example}
\begin{enumerate}
\item
If $\mathcal{S}$ is the class of cofibrations then $\mathpzc{C}$ is weakly $\mathcal{S}$-elementary.
\item
If $\mathcal{S}$ is a class of weak equivalences such that any transfinite composition of maps in $\mathcal{S}$ is a weak equivalence, then it follows from the $2$-out-of-$3$ property that $\mathpzc{C}$ is weakly $\mathcal{S}$-elementary.
\end{enumerate}
\end{example}
The following is straightforward.
\begin{prop}\label{prop:leftproppushouttrans}
Let $\mathcal{S}$ be a class of maps in a model category $\mathpzc{C}$. Suppose that any map in $\mathcal{S}$ is a $h$-cofibration, and that $\mathpzc{C}$ is weakly $\mathcal{S}^{p}$-elementary, where $\mathcal{S}^{p}$ is the class of pushouts of maps in $\mathpzc{C}$. Then any map obtained as a transfinite composition of pushouts of maps in $\mathcal{S}$ is a $h$-cofibration.
\end{prop}

\section{Monoidal Model Categories}\label{monoidalmod}
\begin{defn}
Let $\mathpzc{M},\mathpzc{N},\mathcal{P}$ be model categories. A bifunctor $-\otimes-:\mathpzc{M}\times\mathpzc{N}\rightarrow\mathcal{P}$ is said to be \textbf{left Quillen} if whenever $i:m\rightarrow m'$ and $j:n\rightarrow n'$ are cofibrations then so is $i\hat{\otimes}j$, and it is an acyclic cofibration if either $i$ or $j$ is. Here $i\hat{\otimes}j$ is the following map
\begin{displaymath}
\xymatrix{
m\otimes n\ar[r]^{i\otimes 1}\ar[d]^{1\otimes j} & m'\otimes n\ar[d] \ar@/^2.0pc/[ddr]^{1\otimes j}&\\
m\otimes n'\ar[r]\ar@/_2.0pc/[drr]^{i\otimes 1} & P\ar[dr]^{i\hat{\otimes}j} &\\
& & m'\otimes n'
}
\end{displaymath}
where the square is a pushout.
\end{defn}

\begin{defn}
A \textbf{(closed) monoidal model category} is a (closed) symmetric monoidal category $(\mathcal{V},\otimes,k)$ ($(\mathcal{V},\otimes,k,\underline{Hom})$) with a model structure so that the monoidal product is a left Quillen bifunctor, and the maps 
$$Q(k)\otimes v\rightarrow k\otimes v\cong v$$
and
$$v\otimes Q(k)\rightarrow v\otimes k\cong v$$
are weak equivalences whenever $v$ is cofibrant. Here $Q$ is the cofibrant replacement functor.
\end{defn}

Another condition that is often asked of a monoidal model category is that it satisfies the so-called monoid axiom. Under certain additional technical assumptions on the model category, this guarantees the existence of a model structure on the category of algebras over any cofibrant operad. 

\begin{defn}
A monoidal model category $(\mathcal{V},\otimes,k)$ is said to satisfy the \textbf{monoid axiom} if every morphism which is obtained as a transfinite composition of pushouts of tensor products of acyclic cofibrations with any object is a weak equivalence.
\end{defn} 

\subsection{Two-Variable Quillen Adjunctions}\label{twovarquill}
Let $(\mathpzc{C},\mathpzc{D},\mathpzc{E})$ be a triple of model categories. Recall (Section 10.1 \cite{Riehl}) that a two-variable adjunction on $(\mathpzc{C},\mathpzc{D},\mathpzc{E})$ is a triple of functors
$$\otimes:\mathpzc{C}\times \mathpzc{D}\rightarrow\mathpzc{E}\;\;\; Hom_{l}:\mathpzc{C}^{op}\times\mathpzc{E}\rightarrow\mathpzc{D}\;\;\; Hom_{r}:\mathpzc{D}{^op}\times\mathpzc{E}\rightarrow\mathpzc{C}$$
together with natural isomorphisms
$$Hom_{\mathpzc{E}}(c\otimes d,e)\cong Hom_{\mathpzc{C}}(c,Hom_{l}(d,e))\cong Hom_{\mathpzc{D}}(d,Hom_{r}(c,e))$$
The following is from Section 11.4 in \cite{Riehl}.
\begin{defn}
Let $(\mathpzc{C},\mathpzc{D},\mathpzc{E})$ be model categories. A two-variable ajdunction on  $(\otimes,Hom_{l},Hom_{r})$ on $(\mathpzc{C},\mathpzc{D},\mathpzc{E})$ is said to be a \textbf{Quillen two-variable adjunction} if $\otimes$ is a left Quillen bifunctor.
\end{defn}
\begin{defn}
Let $\mathpzc{C},\mathpzc{D}$ be model categories. If there exists a Quillen two-variable adjunction on $(\mathpzc{C},\mathpzc{D},\mathpzc{D})$ then $\mathpzc{D}$ is said to be a $\mathpzc{C}$-model category.
\end{defn}
In particular, following Definition 11.4.4 of \cite{Riehl}, a simplicial model category is just a $\mathpzc{sSet}$-model category. The following is clear.
\begin{prop}\label{twovarcont}.
Let $(\otimes,Hom_{l},Hom_{r})$ be a two-variable Quillen adjunction on model categories $(\mathpzc{C},\mathpzc{D},\mathpzc{E})$. Let 
$$\adj{L}{\mathpzc{C}'}{\mathpzc{C}}{R}$$
be a Quillen adjunction. Then $(L(-)\otimes L(-),Hom_{l}(L,-),R\circ Hom_{r})$ is a Quillen two-variable adjunction on $(\mathpzc{C}',\mathpzc{D},\mathpzc{E})$. In particular if $\mathpzc{D}$ is a $\mathpzc{C}$-model category then it is also a $\mathpzc{C'}$-model category. 
\end{prop}

\section{Transferred Model Structures}\label{transfersec}

\begin{defn}\label{transferdef}
Let $\mathpzc{C}$ and $\mathpzc{D}$ be categories with $\mathpzc{C}$ a model category. Suppose $F:\mathpzc{C}\rightarrow\mathpzc{D}$ and $G:\mathpzc{D}\rightarrow\mathpzc{C}$ are functors with $F\dashv G$. If it exists, the \textbf{transferred model structure} on $\mathpzc{D}$ is the one defined as follows.
\begin{enumerate}
\item
A map $f$ in $\mathpzc{D}$ is a weak equivalence precisely if $G(f)$ is a weak equivalence in $\mathpzc{C}$.
\item
A map $f$ in $\mathpzc{D}$ is a fibration precisely if $G(f)$ is a fibration in $\mathpzc{C}$.
\item
A map $f$ in $\mathpzc{D}$ is a cofibration precisely if it has the left lifting property with respect to acyclic cofibrations.
\end{enumerate}
\end{defn}

\begin{rem}
If the transferred model structure exists on $\mathpzc{D}$ then $F\dashv G$ is a Quillen adjunction.
\end{rem}

We need the following important result, which is essentially Theorem 3.3 in \cite{crans1995quillen}. Although the conditions are marginally different, the proof is identical.
\begin{thm}\label{transfer}
 Suppose $F:\mathpzc{C}\rightarrow\mathpzc{D}$ and $G:\mathpzc{D}\rightarrow\mathpzc{C}$ are functors with $F\dashv G$. Suppose that $\mathpzc{C}$ is finitely complete and cocomplete cofibrantly generated model category with generating cofibrations $I$ and generating acyclic cofibrations $J$, and that $\mathpzc{D}$ is a category having finite limits, sequential colimits, and pushouts along maps of the form $F(f)$ where $f$ is a coproduct of maps in $I\cup J$. Let $\lambda$ (resp $\lambda'$) be an ordinal such that domains of generating cofibrations (resp. domains of generating acyclic cofibrations) are $\lambda$-presented (resp. $\lambda'$-presented) relative to pushouts of coproducts of maps in $I$ (resp. $J$). Suppose that if $c$ is a generating cofibration (resp. generating acyclic cofibration) then the domain of $F(c)$ is $\lambda$-presented (resp. $\lambda'$-presented ) relative to pushouts of coproducts of maps of the form $F(f)$ where $f$ is a generating cofibration (resp. generating acyclic cofibration). Then the transferred model structure on $\mathpzc{E}$ exists if and only if the weak equivalences in $\mathpzc{D}$ contain any sequential colimit of pushouts of maps of the form $F(g)$, where $g$ is a generating trivial cofibration in $\mathpzc{C}$. 
 \end{thm}

\begin{cor}\label{transferalg}
Let $\mathpzc{C}$ and $\mathpzc{D}$ be categories, with $\mathpzc{C}$ a cococomplete combinatorial model category and $\mathpzc{D}$ having finite limits and all colimits. Suppose $F:\mathpzc{C}\rightarrow\mathpzc{D}$ and $G:\mathpzc{D}\rightarrow\mathpzc{C}$ are functors with $F\dashv G$. If  $G$ preserves filtered colimits, then the transferred model structure on $\mathpzc{D}$ exists if and only if the weak equivalences in $\mathpzc{D}$ contain any map of the form $F(g)$, where $g$ is a generating trivial cofibration in $\mathpzc{C}$. Moreover the transferred model structure is cofibrantly generated.
\end{cor}

\begin{rem}
Note that in \cite{crans1995quillen} it is actually proved that if an adjunction satisfying the above condition then the transferred model structure exists. The converse is clear however since as a left Quillen functor $F$ preserves acyclic cofibrations and colimits. 
\end{rem}
\backmatter
\bibliographystyle{amsalpha}
\bibliography{revisedarxiv.bib}

\end{document}